\documentclass[12pt, reqno]{amsart}
\usepackage{amssymb, mathrsfs}
\usepackage{latexsym}
\usepackage{amsmath}
\usepackage{amsthm}
\usepackage{amsfonts}
\usepackage{dsfont, upgreek}
\usepackage{mathtools}
\usepackage{epsfig}
\usepackage{amscd}
\usepackage{graphicx}
\usepackage{tikz-cd }

\usepackage{enumitem}
\setlist[enumerate]{itemsep=0.5ex}

\usepackage{stmaryrd}

\usepackage{color}
\usepackage{tikz}
\usetikzlibrary{patterns}

\usepackage[left=1.4in,right=1.4in,top=1.2in,bottom=1.2in]{geometry}

\usetikzlibrary{decorations.markings}
\usetikzlibrary{arrows.meta}

\usepackage{extarrows}

\usepackage{adjustbox}
\usepackage{pgfplots}
\usepgfplotslibrary{colormaps}
\usepackage{slashed}

\usepackage[colorlinks=true,linkcolor=blue,citecolor=blue]{hyperref}

\usepackage[colorinlistoftodos,prependcaption,textsize=tiny]{todonotes}  

\usepackage[all,cmtip]{xy}
\theoremstyle{plain}
\newtheorem{theorem}{Theorem}[section]
\newtheorem{proposition}[theorem]{Proposition}
\newtheorem{lemma}[theorem]{Lemma}
\newtheorem{corollary}[theorem]{Corollary}

\newtheorem{conjecture}[theorem]{Conjecture}

\theoremstyle{definition} 
\newtheorem{definition}[theorem]{Definition}

\newtheorem*{claim*}{Claim}
\newtheorem{claim}[theorem]{Claim}
\newtheorem{setup}[theorem]{Geometric Setup}

\theoremstyle{remark} 
\newtheorem{remark}[theorem]{Remark}

\numberwithin{equation}{section}

\newcommand{\Sc}{\mathrm{Sc}}
\newcommand{\cl}{\mathbb{C}\ell}
\newcommand{\Endo}{\mathrm{End}}
\newcommand{\Bigwedge}{\mathord{\adjustbox{raise=.4ex, totalheight=.7\baselineskip}{$\bigwedge$}}}

\newcommand{\fiber}{\mathbb F}
\newcommand{\link}{\mathbb L}
\newcommand{\hotimes}{\mathbin{\widehat{\otimes}}}

\newcommand{\ind}{\textup{Ind}}
\newcommand{\id}{\mathrm{id}}

\newcommand{\dist}{\mathrm{dist}}

\newcommand{\R}{\mathbb{R}}
\newcommand{\tr}{\mathrm{tr}}

\newcommand{\Z}{\mathbb{Z}}
\newcommand{\N}{\mathbb{N}}

\newcommand{\dR}{\mathrm{dR}}
\newcommand{\dom}{\mathrm{dom}}
\newcommand{\sph}{\mathbb{S}}
\newcommand{\normal}{\mathbf{n}}
\newcommand{\normalnew}{\mathbf{m}}
\newcommand{\Y}{\mathbf{e}}

\newcommand{\interior}[1]{%
	{\kern0pt#1}^{\mathrm{\,o}}%
}

\makeatletter
\let\save@mathaccent\mathaccent
\newcommand*\if@single[3]{%
	\setbox0\hbox{${\mathaccent"0362{#1}}^H$}%
	\setbox2\hbox{${\mathaccent"0362{\kern0pt#1}}^H$}%
	\ifdim\ht0=\ht2 #3\else #2\fi
}
\newcommand*\rel@kern[1]{\kern#1\dimexpr\macc@kerna}
\newcommand*\wideaccent[2]{\@ifnextchar^{{\wide@accent{#1}{#2}{0}}}{\wide@accent{#1}{#2}{1}}}
\newcommand*\wide@accent[3]{\if@single{#2}{\wide@accent@{#1}{#2}{#3}{1}}{\wide@accent@{#1}{#2}{#3}{2}}}
\newcommand*\wide@accent@[4]{%
	\begingroup
	\def\mathaccent##1##2{%
		\let\mathaccent\save@mathaccent
		\if#42 \let\macc@nucleus\first@char \fi
		\setbox\z@\hbox{$\macc@style{\macc@nucleus}_{}$}%
		\setbox\tw@\hbox{$\macc@style{\macc@nucleus}{}_{}$}%
		\dimen@\wd\tw@
		\advance\dimen@-\wd\z@
		\divide\dimen@ 3
		\@tempdima\wd\tw@
		\advance\@tempdima-\scriptspace
		\divide\@tempdima 10
		\advance\dimen@-\@tempdima
		\ifdim\dimen@>\z@ \dimen@0pt\fi
		\rel@kern{0.6}\kern-\dimen@
		\if#41
		#1{\rel@kern{-0.6}\kern\dimen@\macc@nucleus\rel@kern{0.4}\kern\dimen@}%
		\advance\dimen@0.4\dimexpr\macc@kerna
		\let\final@kern#3%
		\ifdim\dimen@<\z@ \let\final@kern1\fi
		\if\final@kern1 \kern-\dimen@\fi
		\else
		#1{\rel@kern{-0.6}\kern\dimen@#2}%
		\fi
	}%
	\macc@depth\@ne
	\let\math@bgroup\@empty \let\math@egroup\macc@set@skewchar
	\mathsurround\z@ \frozen@everymath{\mathgroup\macc@group\relax}%
	\macc@set@skewchar\relax
	\let\mathaccentV\macc@nested@a
	\if#41
	\macc@nested@a\relax111{#2}%
	\else
	\def\gobble@till@marker##1\endmarker{}%
	\futurelet\first@char\gobble@till@marker#2\endmarker
	\ifcat\noexpand\first@char A\else
	\def\first@char{}%
	\fi
	\macc@nested@a\relax111{\first@char}%
	\fi
	\endgroup
}
\makeatother

\newcommand\overbar{\wideaccent\overline}

\makeatletter
\newcommand*{\transpose}{%
	{\mathpalette\@transpose{}}%
}
\newcommand*{\@transpose}[2]{%
	\raisebox{\depth}{$\m@th#1\intercal$}%
}
\makeatother

\begin{document}
	
	\title{On Gromov's dihedral extremality and rigidity conjectures}
	
	\author{Jinmin Wang}
	\address[Jinmin Wang]{Department of Mathematics, Texas A\&M University}
	\email{jinmin@tamu.edu}
	\thanks{The first author is partially supported by  NSF 1800737 and 1952693.}
	\author{Zhizhang Xie}
	\address[Zhizhang Xie]{ Department of Mathematics, Texas A\&M University }
	\email{xie@tamu.edu}
	\thanks{The second author is partially supported by NSF 1800737 and 1952693.}
	\author{Guoliang Yu}
	\address[Guoliang Yu]{ Department of
		Mathematics, Texas A\&M University}
	\email{guoliangyu@tamu.edu}
	\thanks{The third author is partially supported by NSF 1700021, 2000082, and the Simons Fellows Program.}
	
	\begin{abstract}
		In this paper, we develop a new index theory for manifolds with polyhedral boundary. As an application, we prove  Gromov's dihedral extremality conjecture  regarding comparisons of scalar curvatures, mean curvatures and dihedral angles between   two compact manifolds with polyhedral boundary in all dimensions. We also prove Gromov's  dihedral rigidity conjecture for a class of positively curved manifolds with polyhedral boundary in all dimensions.
	\end{abstract}
	\maketitle

\section{Introduction}

In the past several years, Gromov has formulated an extensive list of conjectures and
open questions on  scalar curvature
\cite{GromovDiracandPlateau, Gromovadozen, Gromovinequalities2018, 	Gromov4lectures2019}. This has 
given rise to new perspectives on scalar curvature and inspired a wave of recent activity 
in this area. The main purpose of this paper is to prove a new index theorem for manifolds with polyhedral boundary and apply it to investigate the following two conjectures of Gromov:  the  dihedral extremality conjecture  (Conjecture \ref{conj:extremal}) \cite{Gromovadozen} and  the dihedral rigidity conjecture (Conjecture \ref{conj:dihedral}) \cite{ GromovDiracandPlateau} about comparisons of scalar curvatures, mean curvatures and dihedral angles for compact manifolds with polyhedral boundary, which can be viewed as scalar curvature analogues of the Alexandrov's triangle comparisons for spaces whose sectional curvature is bounded below \cite{MR0049584,MR854238}. These two  conjectures have profound implications in  geometry and mathematical physics  such as the positive mass theorem, a foundational result in general relativity and differential geometry
\cite{MR535700,MR612249} \cite{MR626707} (cf. \cite[Section 5]{Li:2019tw} and also the discussion after Theorem \ref{thm:polyhedra}). In this paper, we answer positively Gromov's dihedral extremality conjecture  in all dimensions.  We also prove Gromov's  dihedral rigidity conjecture for a class of positively curved manifolds in all dimensions (cf. (I) \& (II) of  Theorem \ref{thm:special}).   In fact, we shall  prove a  general theorem  (Theorem \ref{thm:extremal-rigid}) on comparisons of scalar curvatures, mean curvatures and dihedral angles between  two compact manifolds with polyhedral boundary possibly of different dimensions, from which we derive the above results. Roughly speaking, manifolds with polyhedral boundary are manifolds that locally modeled on Euclidean polyhedra.   See Definition \ref{def:polytopeboundary} for the precise definition of manifolds with polyhedral boundary.  

Given a Riemannian metric $g$ on an oriented manifold $M$ with polyhedral boundary, we shall denote the scalar curvature of $g$ by $\Sc(g)$, the mean curvature of each face $F_i$ of $M$ by $H_g(F_i)$, and the dihedral angle function of two adjacent faces $F_i$ and $F_j$ by $\theta_{ij}(g)$. Here the dihedral angle $\theta_{ij}(g)_x$ at a point $x\in F_i\cap F_j$ is defined as follows.  
\begin{definition}\label{def:exterior}
	Write $F_{ij} = F_i\cap F_j$. Let $\nu_i$ and $\nu_j$ be the unit inner normal vector of $F_{ij}$ with respect to $F_i$ and $F_j$ at $x\in F_{ij}$, respectively.  Let $\theta_{ij}(g)_x$ be either the angle of $\nu_i$ and $\nu_j$, or $\pi$ plus this angle, depending on the vector $(\nu_i+\nu_j)/2$ points inward or outward, respectively. See Figure \ref{fig:dihedralangles}.
	
	 Equivalently, we can also describe the dihedral angles in terms of the angle $\alpha_{ij}$ between  the unit inner normal vectors $u_i$ and $u_j$ of $F_i$ and $F_j$ (with respect to $M$) at $x$. Indeed, let $\alpha_{ij}$ be the angle between $u_i$ and $u_j$. The dihedral angle $\theta_{ij}(g)_x$ equals to either $\pi- \alpha_{ij}$ or $\pi + \alpha_{ij}$, depending on the vector $(\nu_i+\nu_j)/2$ points inward or outward, respectively. 
\end{definition} 

\begin{figure}
	\begin{tikzpicture}[scale=1]
		\fill[gray!20] ({7-sqrt(3)},2) -- (7,1) -- ({7+sqrt(3)},2) -- ({7+sqrt(3)},0) -- ({7-sqrt(3)},0) -- cycle;
			\fill[gray!20] (0,0) -- (3,0) -- ({3*sin(50)},{3*cos(50)}) -- (0,0);
	\draw[very thick,black] (0,0) -- (3,0);
	\draw[very thick,black] (0,0) -- ({3*sin(50)},{3*cos(50)});
	\filldraw (1.5,-0.7) node {Angle in $(0,\pi)$.};
	\filldraw (7,-0.7) node {Angle in $(\pi,2\pi)$};
	\draw[very thick,black] ({7-sqrt(3)},2) -- (7,1) -- ({7+sqrt(3)},2) ;
	\draw[very thick,black] (0,0) -- ({3*sin(50)},{3*cos(50)});
	\fill[gray!20] (0,0) -- (3,0) -- ({3*sin(50)},{3*cos(50)}) -- (0,0);
	\draw[very thick, blue,-stealth] (0,0) -- (1.5,0);
\draw[very thick, blue,-stealth] (0,0) -- ({1.5*sin(50)},{1.5*cos(50)});
\draw[blue,-stealth] (0,0) -- ({(1.5*sin(50)+1.5)/2},{1.5*cos(50)/2});
\draw[blue,dotted]({1.5*sin(50)},{1.5*cos(50)}) -- (1.5,0);
			\draw[black] (0.5,0) arc (0:40:0.5);
		\draw[blue] ({1.5*sin(50)/2-0.2},{1.5*cos(50)/2+0.2}) node {$\nu_i$};
		\draw[blue] ({1.5/2},-0.25) node {$\nu_j$};
\draw[black] (0.6,0) arc (0:40:0.6);
\draw[black] ({7+0.15*sqrt(3)},{1+0.15}) arc (30:-210:0.3);
\draw[black] ({7+0.175*sqrt(3)},{1+0.175}) arc (30:-210:0.35);
	\draw[very thick, blue,-stealth] (7,1) -- ({7+sqrt(3)*0.8},1.8) node[anchor=north west] {$\nu_j$};
	\draw[very thick, blue,-stealth] (7,1) -- ({7-sqrt(3)*0.8},1.8) node[anchor=north east] {$\nu_i$};
	\draw[blue,-stealth] (7,1) -- ({7},1.8) ;
	\draw[blue,dotted]({7+sqrt(3)*0.8},1.8) -- ({7-sqrt(3)*0.8},1.8);
	\end{tikzpicture}
	\label{fig:dihedralangles}
	\caption{Dihedral angles.}
\end{figure}
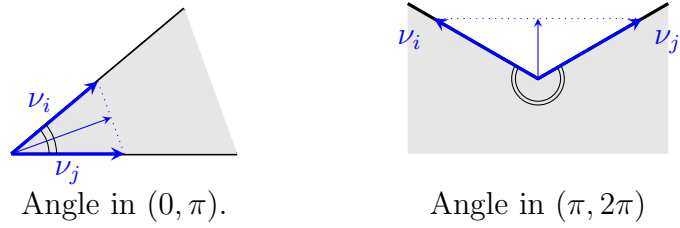
Here the angle $\theta_{ij}(g)_x$ takes values in $(0,\pi)\cup(\pi,2\pi)$. 
Furthermore, our sign convention for the mean curvature is that the mean curvature of the standard round sphere,  viewed as the boundary of the Euclidean unit ball,  is positive.

\begin{conjecture}[{Gromov's dihedral extremality conjecture for convex polyhedra, \cite[Section 7]{Gromovadozen}}]\label{conj:extremal}
	Let $P$ be a convex polyhedron in $\R^n$ and $g_0$ the Euclidean metric on $P$. If $g$ is a smooth Riemannian metric on $P$ such that 
	\begin{enumerate}[label=$(\arabic*)$]
		\item $\Sc(g)\geq \Sc(g_0) = 0$,
		\item $H_g(F_i)\geq H_{g_0}(F_i) = 0$ for each face $F_i$ of $P$, and
		\item $\theta_{ij}(g)\leq \theta_{ij}(g_0)$ on each $F_{ij} = F_i\cap F_j$,	
	\end{enumerate}
	then we have 
	\[ \Sc(g)=0, H_g(F_i) = 0 \textup{ and }  \theta_{ij}(g) =  \theta_{ij}(g_0)\]
	for all $i$ and all $j\neq i$. 
\end{conjecture}	
In other words, the dihedral extremality conjecture  states that on a convex polyhedron, one cannot simultaneously increase the scalar curvature of the metric and the mean curvature of the faces while decreasing the dihedral angles. 

\begin{conjecture}[{Gromov's dihedral rigidity conjecture for convex polyhedra, \cite[Section 2.2]{GromovDiracandPlateau}}]\label{conj:dihedral}
	Let $P$ be a convex polyhedron in $\R^n$ and $g_0$ the Euclidean metric on $P$. If $g$ is a smooth Riemannian metric on $P$ such that 
	\begin{enumerate}[label=$(\arabic*)$]
		\item $\Sc(g)\geq \Sc(g_0) = 0$,
		\item $H_g(F_i)\geq H_{g_0}(F_i) = 0$ for each face $F_i$ of $P$, and
		\item $\theta_{ij}(g)\leq \theta_{ij}(g_0)$ on each $F_{ij} = F_i\cap F_j$,	
	\end{enumerate}
	then $g$ is also a flat metric.
\end{conjecture}

In dimension two, both conjectures are  immediate consequences of the classical Gauss--Bonnet theorem for compact surfaces with piecewise smooth boundary. In higher dimensions,  Gromov showed that  the dihedral extremality conjecture holds for a special class of convex polyhedra, under some extra restrictions on dihedral angles \cite{GromovDiracandPlateau}. In particular, Gromov showed that the the dihedral extremality conjecture holds for the standard Euclidean cube $[0, 1]^n$ for all $n\geq 2$ \cite{GromovDiracandPlateau}. More recently, Li proved the dihedral rigidity conjecture for the following special classes  of convex polyhedra  in dimension $\leq 7$ \cite{Li:2019tw, Lichaocomparison}: 
\begin{enumerate}[label=$(\arabic*)$]
	\item  cone type $3$-dimensional polyhedra satisfying some extra dihedral angle conditions,
	\item  prism-type polyhedra of dimension $n\leq 7$, that is, $P$ is the Cartesian product $P_0 \times [0, 1]^{n-2}$, where $P_0$ is $2$-dimensional polygon with non-obtuse dihedral angles. 
\end{enumerate}

In this paper, we shall prove Gromov's dihedral extremality conjecture (Conjecture  \ref{conj:extremal}) for  all convex polyhedra in all dimensions.   In fact,  our proof also proves a weaker version of  Gromov's dihedral rigidity conjecture  for convex polyhedra in all dimensions,  where we show that $g$ is Ricci flat instead of showing $g$ is flat. Note that Ricci flatness coincides with flatness in dimension three.  Consequently, we answer positively Gromov's  dihedral rigidity conjecture for all convex polyhedra  in dimension three.

 We shall derive Gromov's dihedral extremality conjecture and dihedral rigidity conjecture for convex polyhedra as a consequence of a more general theorem (Theorem \ref{thm:special} or Theorem \ref{thm:extremal-rigid}) on comparisons of scalar curvatures, mean curvatures and dihedral angles between  two compact manifolds with polyhedral boundary, which are possibly of different dimensions.

Before we state the theorems, let us first fix some notation. Let $f\colon (N, \overbar{g}) \to (M, g)$ be a smooth map between two Riemannian manifolds. We have the linear maps
\[  df\colon TN \to TM \textup { and more generally } \wedge^k df \colon \Bigwedge^k TN \to \Bigwedge^k TM\]  
where $df$ is the tangent map and $\wedge^k TN$ is the $k$-th exterior product of $TN$. 

\begin{definition}
	We define $\|df\|_x$ to be the norm of the map 
	\[ df \colon T_x N \to T_{f(x)}M \]
	and more generally $\|\wedge^k df\|_x$ to be the norm of the map
	\[ \wedge^k df \colon \Bigwedge^k T_x N \to \Bigwedge^k T_{f(x)}M. \]
	We say $f$ is distance-non-increasing if $\|df\|_x\leq 1$ for all $x\in N$. Similarly, we say $f$ is area-non-increasing if $\|\wedge^2 df\|_x\leq 1$ for all $x\in N$.
\end{definition}
\begin{definition}\label{def:spin}
	A smooth map $f\colon N \to M$ is called a spin map if the second Stiefel--Whitney classes of $TM$ and $TN$ are related by 
	\[  w_2(TN) = f^\ast(w_2(TM)).\]
	Equivalently, $f\colon N \to M$ is a spin map   if $TN\oplus f^\ast TM$ admits a spin structure.
\end{definition}  Note that here we do \emph{not} require either $M$ or $N$ to be a spin manifold. On the other hand, if $M$ happens to be a spin manifold, then $f$ being a spin map implies that $N$ is also a spin manifold.

Let $f\colon (N, \overbar g) \to (M, g)$ be a map between two manifolds with polyhedral boundary. Throughout the paper, we shall always assume $f$ is a polytope map, which roughly speaking means that $f$ is Lipschitz, smooth away from the codimension three faces of $N$,  and maps codimension $k$ faces of $N$ to codimension $k$ faces of $M$. See Definition \ref{def:polytopeMap} for the precise definition of polytope maps. We have the following main theorem of the paper. 
\begin{theorem}\label{thm:special}
	Let $(N,\overbar{g})$ and $(M, g)$ be compact oriented Riemannian manifolds with polyhedral boundary. Suppose  
	\begin{enumerate}[label=$(\alph*)$]
		\item the curvature operator of $g$ is non-negative,
		\item each codimension one face ${F}_i$ of $M$ is convex, that is, the second fundamental form of  ${F}_i$  is non-negative, 
		\item all dihedral angles $\theta_{ij}(g)$ of $M$ are less than $\pi$. 
	\end{enumerate}
	Let $f\colon (N, \overbar{g}) \to (M, g)$ be a  spin and polytope\footnote{See Definition \ref{def:polytopeMap} for the precise definition of polytope maps.} map such that 
	\begin{enumerate}[label=$(\arabic*)$]
		\item $f$ is area-non-increasing on $N$, and $f^\partial = f|_{\partial N}$ is distance-non-increasing on the boundary $\partial N$,  
		\item 
		$\Sc(\overbar{g})_x \geq \Sc(g)_{f(x)}$ for all $x\in N$, 
		\item $H_{\overbar{g}}(\overbar{F}_i)_y \geq  H_{g}(F_i)_{f(y)}$ for all $y$ in each codimension one face\footnote{Here the notations $\overbar{F}_i$ and $F_i$ are chosen  so that the map $f$ takes the face $\overbar{F}_i$ of $N$ to the face $F_i$ of $M$.} $\overbar{F}_i$ of $N$, 
		\item $\theta_{ij}(\overbar{g})_z\leq  \theta_{ij}(g)_{f(z)}$ for all $\overbar F_i, \overbar F_j$ and all $z \in \overbar F_i\cap \overbar{F}_j$, 
		\item $M$  has nonzero Euler characteristic,
		\item the $\widehat A$-degree $\deg_{\widehat A}(f)$ of $f$ is nonzero, that is, 
		\[  \deg_{\widehat A}(f)\coloneqq \int_{N} \widehat A(N) \wedge f^\ast[M] \neq 0. \] 
	\end{enumerate} 
	Then we have 
	\begin{enumerate}[label=$(\roman*)$]
		\item 
		$\Sc(\overbar{g})_x =\Sc(g)_{f(x)}$ for all $x\in N$, 
		\item $H_{\overbar{g}}(\overbar{F}_i)_y  =  H_{g}(F_i)_{f(y)}$ for all $y\in \overbar{F}_i$, 
		\item $\theta_{ij}(\overbar{g})_z = \theta_{ij}(g)_{f(z)}$ for all $\overbar F_i, \overbar F_j$ and all $z \in \overbar F_i\cap \overbar{F}_j$, 
	\end{enumerate}
	and the following are true. 
	\begin{enumerate}[label=\textup{(\Roman*)}]
		\item If  $\mathrm{Ric}(g)>0$ and $f$ is distance-non-increasing on the whole $N$,  then $f$ is a Riemannian submersion. Here $\mathrm{Ric}(g)$ is the Ricci curvature of the metric $g$ on $M$.
		\item If  $0<\mathrm{Ric}(g)<\frac{1}{2} \Sc(g)\cdot g$, then $f$ is a Riemannian submersion. 
		\item If $(M, g)$ is flat, then $(N, \overbar{g})$ is Ricci flat. Consequently, $(N, \overbar{g})$ is flat if $\dim N = 3$.
	\end{enumerate} 
\end{theorem}
In the above, $\widehat A(N)$ is the $\widehat A$-class of  $N$ and $[M] \in H^{\dim M}(M, \partial M)$ is the fundamental class of $M$ (with respect to the given orientation on $M$). In fact, with a little more care, we can improve the estimates to obtain the following strengthening of Theorem \ref{thm:special}. 

\begin{theorem}\label{thm:extremal-rigid}
	Let $(N,\overbar{g})$ and $(M, g)$ be compact oriented Riemannian manifolds with polyhedral boundary. Suppose  
	\begin{enumerate}[label=$(\alph*)$]
		\item the curvature operator of $g$ is non-negative,
		\item each codimension one face ${F}_i$ of $M$ is convex, that is, the second fundamental form of  ${F}_i$  is non-negative, 
		\item all dihedral angles $\theta_{ij}(g)$ of $M$ are less than $\pi$. 
	\end{enumerate}
	Let $f\colon (N, \overbar{g}) \to (M, g)$ be a  spin and  polytope map such that 
	\begin{enumerate}[label=$(\arabic*)$]
		\item 
		$\Sc(\overbar{g})_x \geq \|\wedge^2 df\| \cdot \Sc(g)_{f(x)}$ for all $x\in N$, 
		\item $H_{\overbar{g}}(\overbar{F}_i)_y \geq \|df^\partial\| \cdot  H_{g}(F_i)_{f(y)}$  for all codimension one faces\footnote{Here $f^\partial = f|_{\partial N}$ and $\|df^\partial\|$ is the norm of the map $df^\partial\colon T_y(\partial N) \to T_{f(y)}(\partial M)$.} $\overbar{F}_i$ of $N$ and all $y\in \overbar{F}_i$, 
		\item $\theta_{ij}(\overbar{g})_z\leq  \theta_{ij}(g)_{f(z)}$ for all $\overbar F_i, \overbar F_j$ and all $z \in \overbar F_i\cap \overbar{F}_j$, 
		\item $M$  has nonzero Euler characteristic,
		\item the $\widehat A$-degree $\deg_{\widehat A}(f)$ of $f$ is nonzero, 
	\end{enumerate}
	Then  we have 
	\begin{enumerate}[label=$(\roman*)$]
		\item 
		$\Sc(\overbar{g})_x =\|\wedge^2 df\| \cdot \Sc(g)_{f(x)}$ for all $x\in N$, 
		\item $H_{\overbar{g}}(\overbar{F}_i)_y  = \|df^\partial\| \cdot  H_{g}(F_i)_{f(y)}$ for all $y\in \overbar{F}_i$, 
		\item $\theta_{ij}(\overbar{g})_z = \theta_{ij}(g)_{f(z)}$ for all $\overbar F_i, \overbar F_j$ and all $z \in \overbar F_i\cap \overbar{F}_j$,
	\end{enumerate}
	and the following are true. 
	\begin{enumerate}[label=\textup{(\Roman*)}] 
		\item Suppose $\dim M = \dim N$. If $\mathrm{Ric}(g)>0$ and 
		\[ \Sc(\overbar{g})_x \geq \|df\|^2 \cdot \Sc(g)_{f(x)} \]
		for all $x\in N$, then $\|df\|\equiv a$ for some constant $a>0$ and $f\colon (N, a\cdot \overbar{g}) \to (M, g)$ is a Riemannian covering map. 
		\item If $\dim M = \dim N$ and  $0<\mathrm{Ric}(g)<\frac{1}{2} \Sc(g)\cdot g$, then $\|\wedge^2 df\|\equiv c$ for some constant $c >0$ and $f\colon (N, \sqrt{c}\cdot \overbar{g}) \to (M, g)$ is a Riemannian covering map. 
		\item If $(M, g)$ is flat, then $(N, \overbar{g})$ is Ricci flat. Consequently, $(N, \overbar{g})$ is flat if $\dim N = 3$.
	\end{enumerate} 
\end{theorem}

Theorem \ref{thm:special} and 	Theorem \ref{thm:extremal-rigid}  can be viewed as a  strengthening  of the well-known positive mass theorem (for spin manifolds). Recall that the positive mass theorem states that if $(X, g)$ is a complete asymptotically Euclidean manifold of dimension $n\geq 3$ such that its scalar curvature is non-negative, then the ADM mass of each end of $X$ is non-negative. We refer the reader to \cite{MR535700,MR612249} \cite{MR626707} for the precise meanings of ``asymptotically Euclidean" and ADM mass. Here we shall briefly indicate how one can deduce the positive mass theorem for spin manifolds from Theorem \ref{thm:extremal-rigid} (cf. \cite[Section 5]{Li:2019tw} for a similar discussion). Indeed, by a result of Lohkamp \cite[Lemma 6.2]{MR1678604}, if the ADM mass of an end of $X$ is negative, then one can reduce to the case where $X$ has only an end and the ADM mass of that end is negative. In this case, again by a result of Lohkamp  \cite[Proposition 6.1]{MR1678604},  there exists another complete Riemannian metric $g_1$ on $X$ with $\Sc(g_1) \geq 0$ and $\Sc(g_1)_x > 0 $ for some point $x\in X$ such that there is a compact set $K\subset X$ with $(X- K, g_1)$ being isometric to the standard  Euclidean space minus a ball. Choose a large flat cube  in $\mathbb R^n$ and denote it by $M$. Let $Z$ be the isometric copy of $\mathbb R^n- M$ in $X-K $, and define $N$ to be $X - Z$. Note that  $N$ is a spin manifold, since we have assumed $X$ is spin. The boundary $\partial N$ of $N$ is isometric to the boundary $\partial M$ by construction. Furthermore, since $M$ is a flat cube, clearly there exists a smooth map $f\colon N \to M$ such that $f$ equals the identity near the boundary and all conditions for $f$ in Theorem \ref{thm:extremal-rigid} are satisfied. For example, take $f\colon N \to M$ to be a map that is identity near the boundary and crashes $K\subset N$ to a point in $M$. However the scalar curvature on $N$ is strictly positive somewhere, we arrive at a contradiction. Therefore, we see that Theorem \ref{thm:extremal-rigid} implies the positive mass theorem for spin manifolds.

As an immediate consequence of either Theorem \ref{thm:special} or Theorem \ref{thm:extremal-rigid},  we obtain the following theorem, which solves Gromov's dihedral extremality conjecture (Conjecture \ref{conj:extremal}) for convex polyhedra in all dimensions, and Gromov's dihedral rigidity conjecture (Conjecture \ref{conj:dihedral})  for convex polyhedra in dimension three. 
\begin{theorem}\label{thm:polyhedra}
	Let $P$ be a convex polyhedron in $\R^n$ and $g_0$ the Euclidean metric on $P$. If $g$ is a smooth Riemannian metric on $P$ such that 
	\begin{enumerate}[label=$(\arabic*)$]
		\item $\Sc(g)\geq \Sc(g_0) = 0$,
		\item $H_g(F_i)\geq H_{g_0}(F_i) = 0$ for each face $F_i$ of $P$, and
		\item $\theta_{ij}(g)\leq \theta_{ij}(g_0)$ on each $F_{ij} = F_i\cap F_j$,	
	\end{enumerate}
	then we have 
	\[ \Sc(g)=0, H_g(F_i) = 0 \textup{ and }  \theta_{ij}(g) =  \theta_{ij}(g_0)\]
	for all $i$ and  all $j\neq i$. Furthermore, $g$ is Ricci flat. Consequently, $g$ is flat if $\dim P=3$.
\end{theorem}

Our strategy to prove Theorem \ref{thm:special} and Theorem \ref{thm:extremal-rigid} is to use Dirac type operators with appropriate elliptic boundary conditions. Let us briefly outline the key steps. First, we shall find a suitable elliptic boundary condition for the relevant twisted Dirac operator that naturally arises in our geometric setup, and use it to construct a self-adjoint Fredholm operator. As we are dealing with manifolds with polyhedral boundary, the Riemannian metric on the  boundary is not smooth but with singularities. In general, the analysis for elliptic boundary problems of differential operators on manifolds with singularities is rather delicate. The approach that we develop in the current paper takes advantage of the special features of both  the operators and the underlying geometry, and should be of independent interest on its own.  

Due to the presence of polyhedral corners, we are naturally led to the analysis of operators that arise from  Riemannian metrics of conical type. There is an extensive literature on this type of analysis since the work of Cheeger \cite{MR530173,MR573430,JC83}. For the geometric applications in the present paper, a key step is to prove the  relevant twisted Dirac operator  subject to a natural boundary condition is an essentially self-adjoint Fredholm operator. We show that the essential self-adjointness of the operator is completely determined by the dihedral angles. More precisely, we show in Theorem \ref{thm:ess-sa} that the relevant twisted Dirac operator is an essentially self-adjoint Fredholm operator when the dihedral angles are less than $\pi$ and satisfy a certain comparison condition (cf. Theorem \ref{thm:ess-sa}). We then develop a cutting-and-pasting formula and a deformation argument to compute the Fredholm index of  twisted Dirac operators (with boundary conditions) on manifolds with polyhedral boundarys.  In particular,   we  show that the Fredholm index of the relevant twisted Dirac operator is  precisely the Euler characteristic of $M$ multiplied by the $\widehat A$-degree of $f$, cf. Theorem 
\ref{thm:index} for the case of manifolds with corners and Theorem \ref{thm:index-poly} for the more general case of manifolds with polyhedral boundary. Since  Theorem \ref{thm:special} (or Theorem \ref{thm:extremal-rigid}) assumes that the Euler characteristic of $M$ multiplied by the $\widehat A$-degree of $f$ is nonzero, it follows  that the associated twisted Dirac operator admits a nontrivial solution. However, there is an extra technical difficulty that we need to overcome. The assumption on the dihedral angles in Theorem \ref{thm:ess-sa} is slightly stronger than the corresponding assumption in  Theorem \ref{thm:special} and Theorem \ref{thm:extremal-rigid}. Strictly speaking, Theorem \ref{thm:ess-sa} does not directly apply to the geometric setup of Theorem \ref{thm:special} and Theorem \ref{thm:extremal-rigid}.  To remedy this, we shall approximate the geometric setup of  Theorem \ref{thm:special} and Theorem \ref{thm:extremal-rigid} by a sequence of setups satisfying  the stronger comparison condition on dihedral angles so that  Theorem \ref{thm:ess-sa} applies to each approximation. We then apply the cutting-and-pasting formula and deformation argument to compute the Fredholm index of the associated twisted Dirac operator  for each approximation, which is again the Euler characteristic of $M$ multiplied by the $\widehat A$-degree of $f$. In particular, the twisted Dirac operator $D_n$ associated to each approximation admits a nontrivial solution $\varphi_n$. We then use this sequence  $\{\varphi_n\}$ to produce a nontrivial solution $\varphi$ of the original twisted Dirac operator. See Lemma \ref{lemma:solutionapproximate-poly} and the proof of Theorem \ref{thm:extremal-rigid} for details. 

Finally, we  apply the standard Bochner-Lichnerowicz-Weitzenbock formula  to this  nontrivial solution $\varphi$ to conclude the equalities for the scalar curvatures and mean curvatures (i.e. part $(i)$ and part $(ii)$ of Theorem \ref{thm:special} and Theorem \ref{thm:extremal-rigid}). The equality for dihedral angles (i.e. part $(iii)$ of Theorem \ref{thm:special} and Theorem \ref{thm:extremal-rigid}) requires a separate argument. In an early version of the paper, this was proved by a somewhat elaborate differential geometric computation. However, an observation from \cite{Wang:2022vf} allows us to greatly simplify this part of the proof. Therefore,   we shall prove the equality for dihedral angles in the current version of the paper by following the observation from \cite{Wang:2022vf}. In the end,  the remaining parts of Theorem \ref{thm:special} and Theorem \ref{thm:extremal-rigid} are proved by some local computations.

We point out that Lott proved Theorem  \ref{thm:special} for even dimensional manifolds with smooth boundary (in this case, there are no dihedral angles) \cite{Lottboundary}. In the case of odd dimensional manifolds with smooth boundary, a natural approach is to reduce it to the even dimensional case by taking direct products of the manifolds with the closed unit interval.  But such a product construction results in  manifolds with corners. This is one of the reasons that  studying manifolds with corners (more generally polyhedral boundary) is important and necessary.

This paper is organized as follows. In Section \ref{sec:boundaryconditions}, we establish some key  estimates for scalar curvature and  mean curvature on manifolds with polyhedral boundary. In Section \ref{sec:indextheory}, we introduce some natural elliptic boundary conditions for  twisted  Dirac operators   on  manifolds with polyhedral boundary. We use these boundary conditions to prove that the relevant Dirac operators are essentially self-adjoint under a suitable comparison condition on the  dihedral angles. We then prove an index theorem for these Dirac operators (subject to certain local boundary conditions) on manifolds with polyhedral boundary. It turns out the case of general manifolds with polyhedral boundary is more subtle than the case of manifolds with corners. In order to make our exposition more transparent, we first prove the index theorem for the case of manifolds with corners in Section \ref{sec:fredindex}, then we generalize it  to the general case of manifolds with polyhedral boundary  in Section \ref{sec:index-poly}.   Finally in Section \ref{sec:proofmain}, we prove the main results of this paper. 

\vspace{.5cm}
\textbf{Acknowledgments.} We would like to thank Christian B\"{a}r, Simone Cecchini, Bernhard Hanke, and Thomas Schick for their very helpful  comments.

\section{Estimates of scalar curvature and mean curvature on manifolds with polyhedral boundary}\label{sec:boundaryconditions}
In this section, we establish some estimates for scalar curvature and  mean curvature  on manifolds with polyhedral boundary. These estimates will be useful in  the proofs of our main theorems. 

\subsection{Estimates of scalar curvature and mean curvature on manifolds with smooth boundary}\label{sec:smoothboundary}

Let us first review some estimates for scalar curvature and mean curvature on manifolds with smooth boundary, cf. \cite{GoetteSemmelmann, Listing:2010te, Lottboundary}.

Let $(M, g)$ and $(N, \overbar{g})$ be two oriented compact Riemannian manifolds with smooth boundary.  Suppose  $f\colon N \to M$ is a spin  map. So the bundle $TN\oplus f^\ast TM$ admits a spin structure. Let us denote by $S_N\otimes f^\ast S_M = S_{TN\oplus f^\ast TM}$ the associated spinor bundle over $N$. 
The vector bundle $S_N\otimes f^* S_M$ carries a natural Hermitian metric and a unitary connection $\nabla$ compatible with Clifford multiplication by elements of $\cl (TN)\otimes f^\ast \cl(TM)$. We will denote the Clifford multiplication of a vector $\overbar v\in TN$ by $\overbar c(\overbar v)$ and the Clifford multiplication of a vector $v\in f^*TM$ by $ c(v)$.

Let $D$ be the Dirac operator on $S_N\otimes f^* S_M$, which  can be locally  expressed by
\begin{equation}\label{eq:Dirac}
D=\sum_{i} \overbar c(\overbar e_i)\nabla_{\overbar e_i},
\end{equation}
where $\{\overbar e_i\}$ is a local orthonormal basis with respect to the metric $\overbar{g}$ on $TN$. By the Bochner--Lichnerowicz--Weitzenbock formula, we have 
\begin{equation}\label{eq:D^2}
D^2=\nabla^*\nabla+\frac{\overbar{\Sc}}{4}+\frac{1}{8}\sum_{i,j}\sum_{k,l}\langle f^\ast R^M_{\overbar e_i,\overbar e_j}e_k,e_l\rangle_{M} \, \overbar c(\overbar e_i)\overbar c(\overbar e_j)\otimes c(e_k)c(e_l),
\end{equation}
where  $\overbar{\Sc} \coloneqq  \Sc(\overbar g)$ is the scalar curvature of $N$, $\{e_i\}$ is a local orthonormal basis of $f^\ast TM$, and  $f^\ast R^M$ is the curvature form of $f^\ast TM$.

The Riemannian curvature tensor $R^M$ on $M$ induces a self-adjoint curvature operator $\mathcal R^M$ on $\Bigwedge^2 TM$ by 
\begin{equation}\label{eq:curvatureoperator}
\langle\mathcal R^M(e_i\wedge e_j), e_k\wedge e_l\rangle = - \langle R^M_{e_i, e_j}e_k, e_l\rangle,
\end{equation}
where $\{e_i\}$ is a local orthonormal basis of $TM$. The sign has been chosen so that all sectional curvatures of $M$ is non-negative if $\mathcal R^M$ is a non-negative operator. Let us define Clifford multiplication by $2$-forms to be 
$$\overbar c(\overbar u\wedge \overbar v)=\overbar c(\overbar u)\overbar c(\overbar v),~c(u\wedge v)=c(u)c(v),
$$
for all $\overbar u, \overbar v \in T_xN$ and $u, v \in (f^\ast TM)_x$ with $\overbar g(\overbar u, \overbar v) = 0 = g(u, v)$. 
If  $\{\overbar w_j\}$ and $\{w_i\}$ are local orthonormal bases of $\Bigwedge^2TN$ and $f^\ast \Bigwedge^2TM$ respectively, 
then we can rewrite \eqref{eq:D^2} as
\begin{equation}\label{eq:D^2smooth}
D^2=\nabla^*\nabla+\frac{\overbar{\Sc}}{4}-\frac{1}{2}\sum_{i,j}\langle \mathcal R^M f_*\overbar w_j,w_i\rangle_{M} \, \overbar c(\overbar w_j)\otimes c(w_i),
\end{equation}

Let $\varphi$ be a smooth section of $S_N\otimes f^* S_M$ over $N$. By the Stokes formula, we have
\begin{equation}
\int_N\langle D\varphi,D\varphi\rangle=\int_N\langle D^2\varphi,\varphi\rangle+\int_{\partial N}\langle \overbar c(\overbar e_n) D\varphi,\varphi\rangle,
\end{equation}
where $\overbar e_n$ denotes the unit inner normal vector to $\partial N$. From line \eqref{eq:D^2smooth}, we have
\begin{align*}
\int_N\langle D^2\varphi,\varphi\rangle=&\int_N\langle\nabla^*\nabla \varphi,\varphi\rangle+\int_N\frac{\overbar{\Sc}}{4}|\varphi|^2\\ &-\frac{1}{2}\int_N\langle\sum_{i,j}\langle \mathcal R^M f_*\overbar w_j,w_i\rangle_{M} \overbar c(\overbar w_j)\otimes c(w_i)\varphi,\varphi\rangle.
\end{align*}

We have the following lemma (cf. \cite[Section 1.1]{GoetteSemmelmann}). 
\begin{lemma}\label{lemma:curvature>=}
	If the curvature operator of $M$ is non-negative, then
	\begin{equation}\label{eq:curvature>=}
	-\frac{1}{2}\sum_{i,j}\langle \mathcal R^M f_*\overbar w_j,w_i\rangle_{M} \overbar c(\overbar w_j)\otimes c(w_i)\geq -\|\wedge^2 df\|\cdot\frac{f^*\Sc}{4}.
	\end{equation}
\end{lemma}
\begin{proof}
	As the curvature operator $\mathcal R^M$ is non-negative, there exists a self-adjoint $L\in \Endo(\Bigwedge^2 TM)$  such that $\mathcal R^M=L^2$, that is,
	$\langle \mathcal R^M  w_j,w_i\rangle_{M}=\langle L  w_j,L w_i\rangle_{M}.$
	
	Set $$\overbar L w_k\coloneqq \sum_i\langle L w_k,f_*\overbar w_i\rangle_M\overbar w_i\in\Bigwedge^2 TN.$$
	If $\|\wedge^2 df\|=0$, that is, $\wedge^2 df=0$, then both sides of line \eqref{eq:curvature>=} are zero. So we assume that $\|\wedge^2 df\|>0$ and set $\alpha=\sqrt{\|\wedge^2 df\|}$.
	The left hand side of line \eqref{eq:curvature>=} can be written as
	\begin{align*}
	&-\frac{1}{2}\sum_{i,j}\langle \mathcal R^M f_*\overbar w_j,w_i\rangle_{M} \overbar c(\overbar w_j)\otimes c(w_i)\\
	=&-\frac{1}{2}\sum_{i,j,k}\langle L(f_*\overbar w_j),w_k\rangle_M \cdot 
	\langle Lw_i,w_k\rangle_M \cdot  \overbar c(\overbar w_j)\otimes c(w_i)\\
	=&-\frac{1}{2}\sum_{k}\overbar c(\overbar Lw_k)\otimes c(L w_k)\\
	=&\frac{1}{4}\sum_k\Big(\alpha^{-2}\overbar c(\overbar Lw_k)^2\otimes 1+1\otimes \alpha^2c(L w_k)^2-\big(\alpha^{-1}\overbar c(\overbar Lw_k)\otimes 1+1\otimes \alpha c(L w_k)\big)^2\Big)\\
	\geq&\frac 1 4 \sum_k\alpha^{-2}\overbar c(\overbar Lw_k)^2\otimes 1+\frac{1}{4}\sum_k 1\otimes \alpha^2c(L w_k)^2, 
	\end{align*}
	where the last inequality follows from the fact  that the element
	\[  \alpha^{-1}\overbar c(\overbar Lw_k)\otimes 1+1\otimes \alpha c(L w_k)\] is skew-symmetric, hence its square is non-positive.

	Now we consider the terms  $\sum_k\alpha^{-2}\overbar c(\overbar Lw_k)^2\otimes 1 $ and $\sum_k 1\otimes  \alpha^2c(L w_k)^2$. The same proof for the Lichnerowicz formula (cf. \cite[Theorem II.8.8]{spingeometry}) shows that
	$$\sum_k \alpha^2 c(L w_k)^2=-\alpha^2 \cdot \frac{f^*\Sc}{2}=-\|\wedge^2 df\|\cdot \frac{f^*\Sc}{2}.$$
	Similarly, by the definition of $\overbar L$, we have
	\begin{align*}
	\sum_k\overbar c(\overbar Lw_k)^2 & =  \sum_{i,j,k}\langle \overbar Lw_k, f_*\overbar w_i\rangle_M \cdot 
	\langle \overbar Lw_k, f_*\overbar w_j\rangle_M \cdot \overbar c(\overbar w_i)\otimes c(\overbar w_j) \\
	& =\sum_{i,j}\langle \mathcal R^M(f_*\overbar w_i),f_*(\overbar w_j)\rangle_M \cdot \overbar c(\overbar w_i)\overbar c(\overbar w_j).
	\end{align*}
	We choose a local $\overbar g$-orthonormal frame $\overbar e_1,\ldots,\overbar e_n$ of $TN$ and a local $g$-orthonormal frame $e_1,\ldots,e_m$ of $TM$ such that $f_*\overbar e_i=\mu_i e_i$ with $\mu_i\geq 0$ for any $i\leq \min\{m,n\}$, and $f_*\overbar e_i=0$ otherwise. This can be done by diagonalizing $f^*g$ with respect to the metric $\overbar g$. Then we have 
	$f_*(\overbar e_i\wedge\overbar e_j)=\mu_i\mu_j e_i\wedge e_j$.
	Clearly, we have $\mu_i\mu_j\leq\|\wedge^2 df\|$ for all $i, j$ with $i\neq j$. Therefore we have
	\begin{align*}
	\sum_k\alpha^{-2} \overbar c(\overbar Lw_k)^2=-\alpha^{-2} \sum_{i,j}\mu_i^2\mu_j^2(f^\ast R_{ijji}^M)\geq -\|\wedge^2 df\|\cdot \frac{f^*\Sc}{2}.
	\end{align*}
	This finishes the proof.
\end{proof}

Next we shall review a comparison formula for mean curvature. By the Stokes formula, we have 
\begin{equation}
\int_N\langle\nabla^*\nabla \varphi,\varphi\rangle=\int_N|\nabla\varphi|^2+\int_{\partial N}\langle\nabla_{\overbar e_n}\varphi,\varphi\rangle
\end{equation}
for all smooth sections $\varphi$ of $S_N\otimes f^\ast S_M$.
\begin{definition}
	We define $\overbar c_\partial$ to be the Clifford action of $T(\partial N)$ on the bundle  $S_N\otimes f^* S_M$ given by  
	\[ \overbar c_\partial(\overbar e_\lambda)=\overbar c(\overbar e_n)\overbar c(\overbar e_\lambda) \]
	for all $\overbar e_\lambda \in T(\partial N)$,  where $\overbar e_n$ is the unit inner normal vector to $\partial N$. Similarly, we define $ c_\partial$ to be the Clifford action of $f^\ast T(\partial M)$ on the bundle  $S_N\otimes f^* S_M$ given by  
	\[ c_\partial(e_\lambda)= c( e_n) c(e_\lambda), \] where $e_n$ is the unit inner normal vector to $\partial M$ in $M$. 
\end{definition}

The boundary Dirac operator $D^\partial$ acting on $S_N\otimes f^* S_M$ over $\partial N$ is given by 
\begin{equation}\label{eq:boundarydirac}
D^{\partial}\coloneqq \sum_{\lambda}\overbar c_\partial(\overbar e_\lambda)\nabla^{\partial}_{\overbar e_\lambda}, 
\end{equation}
where $\nabla^{\partial}$ is the connection on $S_N\otimes f^* S_M$ over  $\partial N$ defined by 
\begin{equation}\label{eq:nablapartial}
\begin{aligned}
\nabla^\partial & =\nabla-\frac 1 2 \sum_{\mu}\langle\prescript{N}{}\nabla \overbar e_n,\overbar e_\mu \rangle_N\cdot \overbar c(\overbar e_n)\overbar c(\overbar e_\mu)\otimes 1 \\
& \hspace{2cm} -\frac 1 2\sum_{\mu} \langle\prescript{M}{}\nabla  e_n, e_\mu \rangle_M \cdot 1\otimes  c( e_n) c(e_\mu),
\end{aligned}
\end{equation}
where $\prescript{N}{}\nabla$ and $\prescript{M}{}\nabla$ are the Levi--Civita connections on $N$ and $M$ respectively (cf. \cite[Theorem 2.7]{spinorialapproach}).
Here we use Greek symbols $\lambda$ and $\mu$ to indicate that the summation is taken over basis vectors tangential to the boundary.
We have
\begin{equation}\label{eq:D-Dpartial}
\overbar c(\overbar e_n) D+\nabla_{\overbar e_n}=\overbar c(\overbar e_n)\sum_{\lambda} \overbar c(\overbar e_\lambda)\nabla_{\overbar e_\lambda}=D^\partial+\sum_{\lambda}\overbar c_\partial(\overbar e_\lambda)(\nabla_{\overbar e_\lambda}-\nabla_{\overbar e_\lambda}^{\partial}).
\end{equation}

Let $\overbar A$ be the second fundamental form of $\partial N$ in $N$, that is,
\begin{equation*}
\prescript{N}{}{\nabla} _{\overbar e_\lambda}\overbar e_\mu-\prescript{\partial N}{}{\nabla}_{\overbar e_\lambda}\overbar e_\mu=\overbar A_{\lambda,\mu}\overbar e_n.
\end{equation*}
Let $\overbar H$ be the mean curvature of $\partial N$, that is,\footnote{Our convention of the  mean curvature is that the mean curvature is  the trace of the second fundamental form, or equivalently the sum of all principal curvatures. } $\overbar H \coloneqq \tr \overbar A$. Similarly, on $M$ we have 
\begin{equation*}
\prescript{M}{}{\nabla} _{ e_\lambda}e_\mu-\prescript{\partial M}{}{\nabla}_{ e_\lambda} e_\mu= A_{\lambda,\mu} e_n.
\end{equation*}
and $H=\tr A$. Note that
\begin{equation*}
\sum_{\lambda,\mu}\overbar c_{\partial}(\overbar e_\lambda)\langle\prescript{N}{}\nabla_{\overbar e_\lambda} \overbar e_n,\overbar e_\mu \rangle_N\cdot \overbar c(\overbar e_n)\overbar c(\overbar e_\mu)=
-\sum_{\lambda,\mu}\overbar A_{\lambda,\mu} \overbar c_\partial(\overbar e_\lambda)\overbar c_\partial(\overbar e_\mu)=\overbar H,
\end{equation*}
and
\begin{equation*}
\sum_{\lambda,\mu}\overbar c_{\partial}(\overbar e_\lambda)\otimes \langle\prescript{M}{}\nabla_{f_*\overbar e_\lambda}  e_n, e_\mu \rangle_M c( e_n) c(e_\mu)= -
\sum_{\lambda,\mu}A(f_*\overbar e_\lambda, e_\mu)\overbar c_\partial(\overbar e_\lambda)\otimes c_\partial(e_\mu).
\end{equation*}
We obtain that
\begin{equation}\label{eq:secondf}
\sum_{\lambda}\overbar c_\partial(\overbar e_\lambda)(\nabla_{\overbar e_\lambda}-\nabla_{\overbar e_\lambda}^{\partial})=\frac{\overbar H}{2}-\frac{1}{2}\sum_{\lambda,\mu}A(f_*\overbar e_\lambda, e_\mu)\overbar c_\partial(\overbar e_\lambda)\otimes c_\partial(e_\mu).
\end{equation}
For the last term on the right hand side of the above equation, we have the following lemma (cf. \cite[Lemma 2.1]{Lottboundary}).
\begin{lemma}\label{lemma:secondff>=}
	If the second fundamental form $A$ of $\partial M$ is non-negative, then
	\begin{equation}\label{eq:secondff>=}
	-\frac{1}{2}\sum_{\lambda,\mu}A(f_*\overbar e_\lambda, e_\mu)\overbar c_\partial(\overbar e_\lambda)\otimes c_\partial(e_\mu)\geq-\|df^\partial\|\cdot\frac{f^*H}{2}.
	\end{equation}
\end{lemma}
\begin{proof}
	The strategy is similar to that of Lemma \ref{lemma:curvature>=}. 
	As the second fundamental form $A$ is non-negative, there exists a self-adjoint operator $L \in \Endo(TM)$ such that $A=L^2$, that is,
	$$ A(e_\lambda,e_\mu)=\langle L  e_\lambda,L e_\mu\rangle_{M}.$$
	Let us define $$\overbar L e_\nu\coloneqq \sum_\lambda\langle L e_\nu,f_*\overbar e_\lambda\rangle_M\cdot \overbar e_\lambda.$$
	If $\|df^\partial\|=0$, that is, $df^\partial=0$, then both sides of line \eqref{eq:secondff>=} are zero. If $\|df^\partial\|>0$, we set $\alpha=\sqrt{\|df^\partial\|}$ and rewrite
	the left hand side of line \eqref{eq:secondff>=} as
	\begin{align*}
	&-\frac{1}{2}\sum_{\lambda,\mu}A(f_*\overbar e_\lambda, e_\mu)\overbar c_\partial(\overbar e_\lambda)\otimes c_\partial(e_\mu)\\
	=&-\frac{1}{2}\sum_{\lambda,\mu,\nu}\langle L(f_*\overbar e_\lambda),e_\nu\rangle_M 
	\langle L(e_\mu),e_\nu\rangle_M  \overbar c_\partial(\overbar e_\lambda)\otimes c_\partial(e_\mu)\\
	=&-\frac{1}{2}\sum_{\nu}\overbar c_\partial(\overbar Le_\nu)\otimes c_\partial(L e_\nu)\\
	=&\frac{1}{4}\sum_\nu\Big(\alpha^{-2}\overbar c_\partial(\overbar Le_\nu)^2\otimes 1+1\otimes \alpha^2c_\partial(L e_\nu)^2-\big(\alpha^{-1}\overbar c_\partial(\overbar Le_\nu)\otimes 1+1\otimes\alpha c_\partial(L e_\nu)\big)^2\Big)\\
	\geq&\frac 1 4 \sum_\nu \alpha^{-2}\overbar c_\partial(\overbar Le_\nu)^2\otimes 1+\frac{1}{4}\sum_\nu 1\otimes  \alpha^2 c_\partial(Le_\nu)^2,
	\end{align*}
	where the last inequality  follows from the fact  that the element
	\[  \big(\alpha^{-1}\overbar c_\partial(\overbar Le_\nu)\otimes 1+1\otimes \alpha c_\partial(L e_\nu)\big) \] is skew-symmetric, hence its square is non-positive.  
	
	If we write $L e_\nu=\sum_\lambda L_{\nu\lambda} \cdot e_\lambda$, then we have
	\begin{align*}
	\alpha^2\sum_\nu c_\partial(Le_\nu)^2=&\alpha^2\sum_{\nu,\lambda}L_{\nu\lambda}L_{\nu\lambda}\cdot c_\partial(e_\lambda)^2+\alpha^2\sum_{\nu}\sum_{\lambda\ne\mu}L_{\nu\lambda}L_{\nu\mu}\cdot c_\partial(e_\lambda)c_\partial(e_\mu)\\
	=&-\alpha^2\sum_\lambda A_{\lambda\lambda}+\alpha^2\sum_{\lambda\ne\mu}A_{\lambda\mu} \cdot c_\partial(e_\lambda)c_\partial(e_\mu)\\
	=&-\alpha^2\tr(A)=-\|df^\partial\|\cdot H.
	\end{align*}
	Similarly, since $\langle f_*\overbar e_\nu,f_*\overbar e_\nu\rangle_M\leq\|df^\partial\|^2\cdot \langle \overbar e_\nu,\overbar e_\nu\rangle$, we have
	\begin{align*}
	\alpha^{-2}\sum_\nu c_\partial(\overbar L e_\nu)^2=
	&\alpha^{-2}\sum_{\lambda,\mu}A(f_*\overbar e_\lambda, f_*\overbar e_\mu)\cdot c_\partial(e_\lambda)c_\partial(e_\mu)\\
	=&-\alpha^{-2}\sum_{\lambda} A(f_*\overbar e_\lambda, f_*\overbar e_\lambda) \\
	\geq & -\alpha^{-2}\|df^\partial\|^2\cdot \tr (A)=-\|df^\partial\|\cdot H.
	\end{align*}
	This finishes the proof.
\end{proof}

By combining Lemma \ref{lemma:curvature>=} and Lemma \ref{lemma:secondff>=}, we obtain the following proposition.
\begin{proposition}\label{prop:D^2smooth}
	Let 	$(M, g)$ and $(N, \overbar{g})$ be two oriented compact Riemannian manifolds with smooth boundary and $f\colon N \to M$ is a spin  map.  Assume that both the curvature operator of $M$ and the second fundamental form of $\partial M$ are non-negative. Then for a smooth section $\varphi$ of $S_N\otimes f^* S_M$ over $N$, we have
	\begin{equation}\label{eq:>=smooth}
	\begin{aligned}
	\int_N|D\varphi|^2\geq& \int_N |\nabla\varphi|^2 + \int_N\frac{\overbar{\Sc}}{4}|\varphi|^2-\int_N\|\wedge^2 df\|\cdot \frac{f^*\Sc}{4}|\varphi|^2\\
	&+\int_{\partial N}\langle D^\partial\varphi,\varphi\rangle+
	\int_{\partial N}\frac{\overbar H}{2}|\varphi|^2-\int_{\partial N}\|df^\partial\|\cdot \frac{f^*H}{2}|\varphi|^2.
	\end{aligned}\end{equation}
\end{proposition}
\begin{proof}
	By the Stokes formula, we have 
	\begin{align*}
	\int_N\langle D\varphi,D\varphi\rangle  & =\int_N\langle D^2\varphi,\varphi\rangle+\int_{\partial N}\langle \overbar c(\overbar e_n) D\varphi,\varphi\rangle.
	\end{align*}
	From Equation \eqref{eq:D^2smooth} and Lemma \ref{lemma:curvature>=}, we have
	\begin{align*}
	\int_N\langle D^2\varphi,\varphi\rangle\geq&
	\int_N |\nabla\varphi|^2 + \int_N\frac{\overbar{\Sc}}{4}|\varphi|^2-\int_N\|\wedge^2 df\|\cdot \frac{f^*\Sc}{4}|\varphi|^2+\int_{\partial N}\langle\nabla_{\overbar e_n}\varphi,\varphi\rangle.
	\end{align*}
	By applying line \eqref{eq:D-Dpartial}, line \eqref{eq:secondf} and Lemma \ref{lemma:secondff>=}, we obtain
	$$\int_{\partial N}\langle\nabla_{\overbar e_n}\varphi,\varphi\rangle+\int_{\partial N}\langle \overbar c(\overbar e_n) D\varphi,\varphi\rangle\geq
	\int_{\partial N}\langle D^\partial\varphi,\varphi\rangle+
	\int_{\partial N}\frac{\overbar H}{2}|\varphi|^2-\int_{\partial N}\|df^\partial\|\cdot \frac{f^*H}{2}|\varphi|^2.
	$$
	
\end{proof}

\subsection{Manifolds with polyhedral boundary}
In this subsection, we introduce a notion of manifolds with polyhedral boundary. 

Recall that $n$-dimensional smooth manifolds with corners are locally modeled on $[0, \infty)^k\times \mathbb R^{n-k}$ with $0\leq k \leq n$. More precisely, let $M$ be a Hausdorff space. A chart $(U, \varphi)$ (possibly with corners) for $M$ is a homeomorphism $\varphi$ from an open subset $U$ of $M$ to an open subset of  $[0, \infty)^k\times \mathbb R^{n-k}$ for some $0\leq k \leq n$. Two charts $(U_1, \varphi_1)$ and $(U_2, \varphi_2)$ are $C^\infty$-related if either $U_1\cap U_2$ is empty or the map 
\[ 	\varphi_2\circ \varphi_1^{-1}\colon \varphi_1(U_1\cap U_2) \to \varphi_2(U_1\cap U_2) \]
is a diffeomorphism (of open subsets in $[0, \infty)^{k_1}\times \mathbb R^{n-k_1}$ and $[0, \infty)^{k_2}\times \mathbb R^{n-k_2}$). A system of pairwise $C^\infty$-related charts of $M$ that covers $M$ is called an atlas of $M$. 
\begin{definition}\label{def:manifoldcorner}
	A smooth manifold with corners is a Hausdorff space equipped with a maximal atlas of charts. 
\end{definition}

Similarly, we introduce the following notion of \emph{manifolds with polyhedral  boundary}, which are locally modeled on $n$-dimensional polyhedra in $\mathbb R^n$. For a given Hausdorff space $X$, a polytope chart $(U, \varphi)$  for $X$ is a homeomorphism $\varphi$ from an open subset $U$ of $M$ to an open subset of an $n$-dimensional polyhedron in $\mathbb R^n$.  Two polytope charts $(U_1, \varphi_1)$ and $(U_2, \varphi_2)$ are $C^\infty$-related if either $U_1\cap U_2$ is empty or the map 
\[ 	\varphi_2\circ \varphi_1^{-1}\colon \varphi_1(U_1\cap U_2) \to \varphi_2(U_1\cap U_2) \]
is a diffeomorphism (of open subsets of $n$-dimensional polyhedra). Again, a system of pairwise $C^\infty$-related charts of $X$ that covers $X$ is called an atlas of $X$. 
\begin{definition}\label{def:polytopeboundary}
	A smooth manifold with polyhedral  boundary is a Hausdorff space equipped with a maximal atlas of polytope charts. 
\end{definition}

A Riemannian manifold with polyhedral boundary is a smooth manifold with polyhedral boundary equipped with a smooth Riemannian metric.  A main difference between manifolds with corners and manifolds with polyhedral boundary is the following: for an $n$-dimensional manifold with corners, there can be at most $n$ codimension one faces meeting at a given point; while there may be more than $n$ codimension one faces meeting at a given point in an $n$-dimensional manifold with polyhedral boundary.

\begin{definition}
	Let $N$ be an $n$-dimensional manifold with polyhedral boundary. We define the codimension $k$ stratum of $N$ to be the set of interior points of all codimension $k$ faces of $N$. 
\end{definition}

For each point $x$ in the codimension $k$ stratum of $N$, it admits a small neighborhood $U$ of the form: 
\[ \R^{n-k} \times P\]
such that $P$ is a polyhedral corner in $\R^k$ enclosed by hyperplanes passing through the origin of $\R^k$ and $x$ is the origin of $\R^n$.  In this case, we call the partial derivatives along $\mathbb R^{n-k}$ the base directions of the neighborhood $U$ of $x$.

\begin{definition}\label{def:polytopeMap} A  map $f\colon (N,\overbar g)\to (M,g)$ between Riemannian manifolds with polyhedral boundary is called a \emph{polytope map} if 
	\begin{enumerate}
		\item 	$f$ is Lipschitz\footnote{Here $f\colon (N,\overbar g)\to (M,g)$ is said to be Lipschitz if there exists $C>0$ such that $d_M(f(x), f(y))\leq C\cdot d_N(x, y)$  for all points in $x, y\in N$.},
		\item $f$ is smooth away from the codimension three faces of $N$,
		\item $f$ maps the codimension $k$ stratum of $N$ to the  codimension $k$ stratum of $M$, and 
		\item every point $x$ in $N$ has a small open neighborhood $U$ such that $f$ is smooth with respect to the base directions on $U$.  
	\end{enumerate}
\end{definition}

\begin{remark}
	 Condition (4) in the above definition of polytope maps is mainly added for technical reasons. It can certainly be weakened without affecting all the main results in the present paper. However, imposing condition (4) makes some of the proofs of this paper a little more transparent. In any case,  condition (4) is always satisfied  in the main geometric applications that we are concerned with. 
\end{remark}

We emphasize that a polytope map $f\colon N \to M$ is \emph{not} required to be smooth at the codimension three faces of $N$. Such a flexibility will be important when we consider $n$-dimensional polyhedral corners that have more than $n$ codimension one faces meeting at their vertices. For example, let $N$ and $M$ be two convex polyhedra in $\R^n$ with the same combinatorial type. Then there is always a smooth polytope map $f\colon N\to M$. One can construct $f$ smoothly near each codimension $2$ face of $N$, and extend $f$ radially to any higher codimensional faces.

Consider the vector bundle $f^*TM$ over $N$, which is equipped with the pull-back connection $f^*\nabla^M$ of the Levi--Civita connection on $M$. The smooth structure of $f^*TM$ is defined everywhere away from faces of codimension $\geq 3$. In particular, it makes sense to talk about the space of smooth sections of $f^\ast TM$ that vanishes near codimension two faces, which will be denoted by  $C^\infty_0(N,f^*TM)$. Moreover, the connection $f^*\nabla^M$ is well-defined away from codimension three  faces.

We define
$H^1(N,f^*TM)$ to be the completion of $C^\infty_0(N,f^*TM)$ with respect to the the following Sobolev $H^1$-norm:
\begin{equation}\label{eq:sobolevh1} 
	\|s\|_1\coloneqq \big(\|s\|^2+\|\widetilde \nabla s\|^2\big)^{1/2}
\end{equation}
for $s\in C^\infty_0(N,f^*TM)$, where $\widetilde \nabla = f^\ast \nabla^M$. 
\begin{lemma}\label{lemma:H1space}
	The space $H^1(N,f^*TM)$  is independent (up to bounded isomorphisms of Hilbert spaces) of the metric on $TM$, and coincides with the usual $H^1$-space if $f$ is smooth.
\end{lemma}
\begin{proof}
	Let  $\{U_\alpha\}$ be an open cover of $N$ consisting of polytope charts such that  $TN$ is trivial on $U_\alpha$ and $TM$ is trivial on $f(U_\alpha)$. Let $\{\phi_\alpha\}$ be a smooth partition of unity subordinate to $\{U_\alpha\}$. Set $m = \dim M$. For each $s\in C^\infty_0(N,f^*TM)$,  we may view $s_\alpha\coloneqq \varphi_\alpha s$ with a smooth function from $N$ to $\R^m$ after we identify $f^\ast TM $ with a trivial bundle over $N$. More precisely, we choose a smooth orthonormal basis $\{e_i\}$ of $TM$ over $f(U_\alpha)$. Then $s_\alpha$ is uniquely written as
	$$s_\alpha=\sum_{i=1}^n s^i_\alpha e_i$$
	where $s^i_\alpha$'s  are smooth functions vanishes near codimension two  faces.
	
	Let $\Gamma_{ij}^k$ be the Christoffel symbols of the Levi-Civita connection on $TM$, that is, 
	$$\nabla^M_{e_i}e_j=\sum_{k=1}^n\Gamma_{ij}^ke_k.$$
	Let $\{X_1,\cdots,X_n\}$ be a  orthonormal basis of $TN$ over $U_\alpha$. Then we may write  
	$$f_*X_i=\sum_{j=1}^n x^j_ie_j,$$
	where $x_i^j$'s are bounded functions over $M$ (since $f$ is Lipschitz) and  smooth in the interior of $M$ (since $f$ is smooth away from codimension two faces). It follows that 
	$$\widetilde \nabla_{X_i}s_{\alpha}=\sum_{j=1}^nX_i(s_\alpha^j) \cdot e_j+\sum_{j,k=1}^n x_i^j\Gamma_{ji}^k e_k.$$
	Since $x_i^j$ and $\Gamma_{ji}^k$ are uniformly bounded over $U_\alpha$, it is  not difficult to see that the  Sobolev $H^1$-norm from line \eqref{eq:sobolevh1} is equivalent to the following norm: 
	$$\|s_\alpha\|_{new}^2\coloneqq \sum_{i=1}^n\|s_\alpha^i\|^2+\sum_{i=1}^n\|\text{grad}(s_\alpha^i)\|^2,$$
	where the latter is independent of the choice of the metric on $M$. Together with  the partition of unity subordinate to $\{U_\alpha\}$, it follows that $H^1(N,f^*TM)$ is independent of the metric on $TM$, up to bounded isomorphisms of Hilbert spaces.
	
	Now we prove the last conclusion: $H^1(N,f^*TM)$ coincides with the usual $H^1$-space if $f$ is smooth. Indeed, if $f$ is smooth, then $f^*TM$ is a smooth vector bundle over the entire $N$. Recall that  removing a subspace of codimension $\geq 2$ does not affect the definition of Sobolev $H^1$-spaces. In particular,   the space of smooth sections that vanish near codimension two faces of $N$ is dense in the usual $H^1$-space, where the usual  $H^1$-space is the completion of all smooth sections over $N$ (that do not necessarily vanish near codimension two faces of $N$). This finishes the proof.
\end{proof}

The proof of Lemma \ref{lemma:H1space} shows that the $H^1$-norm is locally equivalent  to the norm of $H^1$-functions. Therefore the compact embedding theorem and the trace theorem for ordinary $H^1$-space also holds for $H^1(N,f^*TM)$.

Let us summarize the geometric assumption that will be used throughout this paper.
\begin{setup}\label{setup}
	Let $(M, g)$ and $(N, \overbar{g})$ be two oriented compact Riemannian manifolds with polyhedral boundary.  Suppose   $f\colon N \to M$ is a spin map (Definition \ref{def:spin}) and also a polytope  map (Definition \ref{def:polytopeMap}).
\end{setup}


Let us retain the same notation from  Section \ref{sec:smoothboundary}. In particular, $\overbar{\Sc}$ (resp. $\Sc$) is the scalar curvature of $N$ (resp. $M$),  and $\overbar H$ (resp. $H$) is the mean curvature of $\partial N$ (resp. $\partial M$).

Let $\{\overbar F_i\}$  be the collection of  codimension one  faces  of $N$, and denote the intersection $\overbar F_i\cap \overbar F_j$ by $\overbar  F_{ij}$. Similarly,  let $\{F_i\}$ be the collection of  codimension one faces of $M$ and write $F_{ij} = F_i\cap F_j$. We denote $\overbar \theta_{ij}$ and $\theta_{ij}$ to be the dihedral angle functions at $\overbar F_{ij}$ and $F_{ij}$ respectively.  In the following, whenever we write $\overbar F_i$ and $F_i$, we mean that $f$ maps $\overbar F_i$ to $F_i$. The same applies to $F_{ij}$ and $\theta_{ij}$.

Recall that  $S_N\otimes f^*S_M$ stands for the spinor bundle $S_{TN\oplus f^*TM}$ of $TN\oplus f^*TM$ on $N$.  Let $\nabla^{S_N}$ and $\nabla^{S_M}$ be the local spinor connections on the local spinor bundles $S_N$ and $S_M$.  We equip the bundle $S_N\otimes f^*S_M$  with the connection $\nabla=\nabla^{S_N}\otimes 1+1\otimes f^* \nabla^{S_M}$, which is well-defined away from codimension three faces of $N$. If $\varphi$  is a smooth section of $S_N\otimes f^* S_M$ that vanishes near the codimension two faces of $N$, then the exact same proof of Proposition \ref{prop:D^2smooth} shows that inequality \eqref{eq:>=smooth} holds for $\varphi$. To summarize, we have the following proposition. 

\begin{proposition}\label{prop:D^2}
	Assume Geometric Setup \ref{setup}.  Let $D$ be the Dirac operator on $S_N\otimes f^* S_M$.  If both the curvature operator of $M$ and the second fundamental form of $\partial M$ are non-negative, then  we have 
	\begin{equation}\label{eq:>=0corner}
	\begin{aligned}
	\int_N|D\varphi|^2\geq& \int_N |\nabla\varphi|^2 + \int_N\frac{\overbar{\Sc}}{4}|\varphi|^2-\int_N\|\wedge^2 df\|\cdot\frac{f^*\Sc}{4}|\varphi|^2\\
	&+\int_{\partial N}\langle D^\partial\varphi,\varphi\rangle+
	\int_{\partial N}\frac{\overbar H}{2}|\varphi|^2-\int_{\partial N}\|df^\partial\|\cdot\frac{f^*H}{2}|\varphi|^2
	\end{aligned}\end{equation}
	for all smooth sections  $\varphi$ of $S_N\otimes f^* S_M$ that vanish near codimension two faces of $N$. 
\end{proposition}

\begin{remark}
	If we assume $f$ to satisfy some slightly stronger regularity conditions than those listed in Definition \ref{def:polytopeMap}, then a version of Proposition \ref {prop:D^2} holds for all smooth sections of $S_N\otimes f^\ast S_M$ (that do not necessarily vanish near codimension two faces of $N$), which incorporates the dihedral angles. See Proposition \ref{prop:D^2WithAngles} in  Appendix \ref{sec:appendixestimate} for the precise details. In an early version of the paper, Proposition \ref{prop:D^2WithAngles} was used as a key step in the proofs of the main theorems (Theorem \ref{thm:special} and Theorem \ref{thm:extremal-rigid}). However, a new observation from \cite{Wang:2022vf} allows us to complete   the proofs of Theorem \ref{thm:special} and Theorem \ref{thm:extremal-rigid} without the need of Proposition \ref{prop:D^2WithAngles}. For this reason, we have decided to put Proposition \ref{prop:D^2WithAngles} in  the appendix. 
\end{remark}

\section{Index theory on manifolds with polyhedral boundary}\label{sec:indextheory}

While the theory of elliptic boundary conditions for operators on manifolds with smooth boundary is quite well developed,  the corresponding theory for operators on manifolds with polyhedral boundary has been relatively unknown. In this section, we investigate a class of twisted Dirac operators on manifolds with polyhedral boundary and their elliptic boundary conditions. Under suitable assumptions on dihedral angles (cf. Theorem \ref{thm:ess-sa} below),  we show that this class of twisted Dirac operators become essentially self-adjoint  with respect to a natural class of local boundary conditions. We then prove an index theorem for twisted Dirac operators on manifolds with polyhedral boundary in Section \ref{sec:index-poly}.  In Section \ref{sec:proofmain}, we shall use this index theorem to prove our main theorems. 

Throughout this section, let us assume the Geometric Setup \ref{setup} and that the dimensions of  $N$ and $M$ have the same parity. The key part of this section is to analyze Dirac type operators that arise from asymptotically conical metrics.

\subsection{Local boundary conditions on codimension one faces}

In this subsection, we introduce a local  boundary condition on the codimension one faces of $N$  for the Dirac operator $D$ associated to $S_N\otimes f^* S_M \coloneqq  S_{TN\oplus f^\ast TM}$. Since the dimensions of $N$ and $M$ have the same parity, the rank of the bundle  $TN\oplus f^\ast TM$ is even, hence the spinor bundle $S_N\otimes f^* S_M = S_{TN\oplus f^\ast TM}$ carries a natural $\mathbb Z_2$-grading. If both $N$ and $M$ are even dimensional, then $S_{TN\oplus f^\ast TM}$ locally is equal to the tensor product of $S_N$ with $f^\ast S_M$ (which is the reason that we have adopted the notation $S_N\otimes f^* S_M$ in place of $S_{TN\oplus f^\ast TM}$). In this case, let $\overbar\epsilon$ and $\epsilon$ be the grading operators on $S_N$ and $f^*S_M$  respectively. Then the $\mathbb Z_2$-grading on $S_{TN\oplus f^\ast TM}$ is given by $\overbar\epsilon\otimes \epsilon$. All the computation throughout the paper works essentially the same for  both the even dimensional case (where $\dim N$ and $\dim M$ are even) and the odd dimensional case (where $\dim N$ and $\dim M$ are odd). For simplicity, we shall mainly focus on the even dimensional case from now on. 

Let $\{\overbar F_i\}$ be the collection of all  codimension one  faces of $N$. We denote the  unit inner normal vector to $\overbar F_i$ in $N$ by $\overbar e_n$. Let $\overbar\epsilon$ and $\epsilon$ be the grading operators on $S_N$ and $f^*S_M$  respectively.
\begin{definition}\label{def:boundarycondition}
	We say 	a section $\varphi$ of  $S_N\otimes f^* S_M$ over $ N$ satisfies the local boundary condition $B$ if $\varphi|_{\partial N}$  satisfies 
	$$(\overbar \epsilon\otimes\epsilon)(\overbar c(\overbar e_n)\otimes c(e_n))\varphi=-\varphi$$
	on each codimension one face $\overbar F_i$ of $N$, where $\overbar e_n$ is the  unit inner normal vector of $\overbar F_i$ and $ e_n$ is the corresponding unit inner normal vector of $F_i$ in $M$. 
\end{definition}

The following lemma shows that the Dirac operator $D$ associated to $S_N\otimes f^* S_M$ is symmetric  with respect to the above local boundary condition $B$.

\begin{lemma}\label{lemma:formallysa}
	The operator $D$ is symmetric with respect to the above boundary condition $B$. More precisely, if $\varphi$ and  $\psi$ are smooth sections of $S_N\otimes f^* S_M$ over $N$ satisfying the boundary condition $B$, then $\langle D\varphi,\psi\rangle=\langle \varphi,D\psi\rangle$.
\end{lemma}
\begin{proof}
	By the Stokes formula, we have
	$$\int_N\langle D\varphi,\psi\rangle-\int_N \langle \varphi,D\psi\rangle=\int_{\partial N}\langle\varphi,\overbar c(\overbar e_n)\psi\rangle=\sum_i \int_{\overbar F_i}\langle\varphi,\overbar c(\overbar e_n)\psi\rangle.$$
	Due to the boundary condition $B$, we have
	$$\varphi=-(\overbar \epsilon\otimes\epsilon)(\overbar c(\overbar e_n)\otimes c(e_n))\varphi,~\psi=-(\overbar \epsilon\otimes\epsilon)(\overbar c(\overbar e_n)\otimes c(e_n))\psi.$$
	Note that $\overbar c(\overbar e_n)$ commutes with $\overbar c(\overbar e_n)\otimes c(e_n)$, but anti-commutes with $\overbar \epsilon\otimes\epsilon$. It follows that 
	$$\overbar c(\overbar e_n)\psi=(\overbar \epsilon\otimes\epsilon)(\overbar c(\overbar e_n\otimes c(e_n))(\overbar c(\overbar e_n)\psi).$$
	Hence $\overbar c(\overbar e_n)\psi$ lies in $\ker(1-(\overbar \epsilon\otimes\epsilon)(\overbar c(\overbar e_n)\otimes c(e_n)))$, which is orthogonal to  $B = \ker(1+(\overbar \epsilon\otimes\epsilon)(\overbar c(\overbar e_n)\otimes c(e_n)))$. This finishes the proof.
\end{proof}
Recall that we define the boundary Dirac operator $D^\partial$ on $S_N\otimes f^* S_M$ over $\partial N$ by
\begin{equation*}
	D^{\partial}\coloneqq \sum_{\lambda}\overbar c_\partial(\overbar e_\lambda)\nabla^{\partial}_{\overbar e_\lambda}, 
\end{equation*}
where $\nabla^{\partial}$ is given by
\begin{equation*}
	\nabla^\partial=\nabla-\frac 1 2 \sum_{\mu}\langle\prescript{N}{}\nabla \overbar e_n,\overbar e_\mu \rangle_N\overbar c(\overbar e_n)\overbar c(\overbar e_\mu)\otimes 1-\frac 1 2\sum_{\mu}1\otimes \langle\prescript{M}{}\nabla  e_n, e_\mu \rangle_M c( e_n) c(e_\mu).
\end{equation*}
Here $\prescript{N}{}\nabla$ and $\prescript{M}{}\nabla$ denote the Levi--Civita connections on $N$ and $M$ respectively. We will need the following lemma in Section \ref{sec:proofmain}.
\begin{lemma}\label{lemma:D^partialvanishes}
	If $\varphi$ satisfies the boundary condition $B$, then $\langle D^\partial\varphi,\varphi\rangle=0$ on the boundary $\partial N$.
\end{lemma}
\begin{proof}
	For brevity, let us write $\gamma=\overbar c(\overbar e_n)\otimes c(e_n)$. We first show that $\gamma$ is parallel with respect to the connection $\nabla^\partial$. 
	
	Note that
	\begin{align*}
		\nabla^\partial\gamma-\gamma\nabla^\partial & =
		(\nabla\gamma-\gamma\nabla)-(\nabla-\nabla^\partial)\gamma+\gamma(\nabla-\nabla^\partial)\\
		&=\nabla\gamma-\gamma\nabla+2\gamma(\nabla-\nabla^\partial),
	\end{align*}
	since $\gamma$ anti-commutes with $(\nabla-\nabla^\partial)$. By the definition of $\nabla^\partial$, we have
	\begin{align*}
		& 2\gamma(\nabla_{\overbar e_\lambda}-\nabla_{\overbar e_\lambda}^\partial)\\
		=&
		-\sum_{\mu}\langle\prescript{N}{}\nabla_{\overbar e_\lambda} \overbar e_n,\overbar e_\mu \rangle_N\overbar c(\overbar e_\mu)\otimes c(e_n)-\sum_{\mu}\overbar c(\overbar e_n)\otimes \langle\prescript{M}{}\nabla_{f_*\overbar e_\lambda}  e_n, e_\mu \rangle_Mc(e_\mu)\\
		=&-\sum_\mu \overbar c(\prescript{N}{}\nabla_{\overbar e_\lambda} \overbar e_n)\otimes c(e_n)-
		\sum_\mu\overbar c(\overbar e_n)\otimes c(\prescript{M}{}\nabla_{f_*\overbar e_\lambda}  e_n)\\
		=&-(\nabla_{\overbar e_\lambda}\gamma-\gamma\nabla_{\overbar e_\lambda}),
	\end{align*}
	where the last identity follows from the Leibniz  rule of the spinor connection $\nabla$. To summarize, we have shown that $\nabla^\partial \gamma - \gamma \nabla^\partial =0$, that is, $\gamma$ is parallel with respect to $\nabla^\partial$. 
	
	It is easy to see that the grading operator $\overbar\epsilon\otimes\epsilon$ also commutes with $\nabla^\partial$.
	Note that $\overbar c_\partial(\overbar e_\lambda)\coloneqq \overbar c(\overbar e_n)\overbar c(\overbar e_\lambda)$ commutes with $\overbar\epsilon$, but anti-commutes with $\overbar c(\overbar e_n)$. 
	Therefore, if $\varphi$ satisfies the boundary condition $B$, that is, 
	$$\varphi=-(\overbar \epsilon\otimes\epsilon)(\overbar c(\overbar e_n)\otimes c(e_n))\varphi,$$
	then we have
	$$D^\partial\varphi=(\overbar \epsilon\otimes\epsilon)(\overbar c(\overbar e_n)\otimes c(e_n))D^\partial\varphi.$$
	It follows that $D^\partial\varphi$ lies in $\ker(1-(\overbar \epsilon\otimes\epsilon)(\overbar c(\overbar e_n)\otimes c(e_n)))$, which is orthogonal to   $B = \ker(1+(\overbar \epsilon\otimes\epsilon)(\overbar c(\overbar e_n)\otimes c(e_n)))$. This finishes the proof.
\end{proof}

To prepare for the next proposition, let us introduce the following definition. 
\begin{definition}\label{def:smooth,H1} \mbox{}
	\begin{enumerate}
		\item A section of $S_N\otimes f^*S_M$  over $N$ is called smooth if it is smooth in the interior, and its partial derivatives of any order are uniformly bounded.
		\item Let $C^\infty_0(N,S_N\otimes f^*S_M;B)$ be the collection of smooth sections that satisfy the boundary condition $B$ at each codimension one face and vanish near all faces with codimension $\geq 2$. 
		\item Let $H^1(N,S_N\otimes f^*S_M;B)$ be the completion of  $C^\infty_0(N,S_N\otimes f^*S_M;B)$ with respect to  the $H^1$-norm
		$$\|\varphi\|_1\coloneqq (\|\varphi\|^2+\|\nabla\varphi\|^2)^{1/2}.$$
	\end{enumerate}
\end{definition}

The bundle $S_N\otimes f^*S_M$ as well as its connection are constructed from $TN$ and $f^*TM$. By Lemma \ref{lemma:H1space}, the $H^1$-space is well-defined and  satisfies the usual properties of Sobolev $H^1$-spaces such as the Rellich lemma and the Sobolev trace theorem.
\begin{proposition}\label{prop:norm-equivalent}
	Assume Geometry Setup \ref{setup}. Let $D$ be the Dirac operator on $S_N\otimes f^*S_M$ and $B$ a local boundary condition given in Definition $\ref{def:boundarycondition}$. Then for sections in $H^1(N,S_N\otimes f^*S_M;B)$, the $H^1$-norm is equivalent to the following norm
	$$\|\varphi\|_D\coloneqq (\|\varphi\|^2+\|D\varphi\|^2)^{1/2}$$
\end{proposition}
\begin{proof}
	Clearly, there exists a constant $C >0$ such that $\|\varphi\|_D \leq C\|\varphi\|_1$ for all $\varphi\in H^1(N,S_N\otimes f^*S_M;B)$. Hence to prove the proposition, it suffices to show the reversed inequality for smooth sections $\varphi\in C^\infty_0(N,S_N\otimes f^*S_M;B)$.
	
	The same proof of Proposition \ref{prop:D^2} implies that  there exist $C_1>$ and $C_2>0$ such that
	\begin{equation}\label{eq:D-norm>=}
		\|D\varphi\|^2\geq \|\nabla\varphi\|^2-C_1\|\varphi\|^2-C_2\int_{\partial N}|\varphi|^2
	\end{equation}
	for all $\varphi\in C^\infty_0(N,S_N\otimes f^*S_M;B)$. We remark that the above inequality does not require the curvature operator of $M$ or the second fundamental form of $\partial M$ to be non-negative. 
	
	Let $H^{\frac{3}{4}}(N,S_N\otimes f^* S_M)$ be the $H^{\frac{3}{4}}$ Sobolev space over $N$ with norm $\|\cdot\|_{\frac{3}{4}}$, and $H^{\frac{1}{4}}(\partial N, S_N\otimes f^* S_M)$  the $H^{\frac{1}{4}}$ Sobolev space 
	over $\partial N$ with norm $\|\cdot\|_{\frac{1}{4}}$.  By  the Sobolev trace theorem,
	there exists a constant $C_3>0$ such that 
	\[  \left\|\varphi|_{\partial N}\right\|_{\frac{1}{4}} \leq C_3 \|\varphi\|_{\frac{3}{4}}  \]
	for all $\varphi\in H^{\frac{3}{4}}(N,S_N\otimes f^* S_M)$. In particular, it follows that 
	$$\int_{\partial N}|\varphi|^2\leq C_3^2 \big\|\varphi\big\|_{\frac{3}{4}}^2$$
	for all $\varphi\in H^{\frac{3}{4}}(N,S_N\otimes f^* S_M)$. 
	Furthermore, it follows from the interpolation theorem that  there exists $C_4>0$ such that   
	$$\|\varphi\|_{\frac{3}{4}}\leq C_4\|\varphi\|_1^{3/4}\|\varphi\|^{1/4}\leq C_4 \Big(\frac{3}{4} \varepsilon^{4/3}  \|\varphi\|_1+\frac{1}{4} \varepsilon^{-4}\|\varphi\|\Big)$$
	for all $\varepsilon>0$. For a sufficiently small $\varepsilon$,  we see that there exists $C_5>0$ such that 
	\begin{equation}
		\|\varphi\|_D\geq C_5\|\varphi\|_1
	\end{equation}
	for all   $\varphi\in H^1(N,S_N\otimes f^*S_M;B)$. This finishes the proof. 
\end{proof}

\begin{remark}
	Note that Proposition \ref{prop:norm-equivalent} is only a statement about the equivalence of $H^1$-norm and  the graph norm of $D$ on elements within $H^1(N,S_N\otimes f^*S_M;B)$. It does \emph{not} imply that the maximal domain of $D$ is $H^1(N,S_N\otimes f^*S_M;B)$ in general. In Section \ref{sec:conic} below, we shall prove that, under extra assumptions on the dihedral angles, the operator $D$ with the boundary condition $B$ becomes essentially self-adjoint (cf. Theorem \ref{thm:ess-sa}). In particular, under these extra assumptions, the maximal domain of $D_B$ equals $H^1(N,S_N\otimes f^*S_M;B)$. 
\end{remark}

\subsection{Essential self-adjointness of twisted Dirac  operators on model conical  spaces}\label{sec:conic}
The main goal of this subsection and the next subsection is to show the twisted Dirac operator $D$ acting on $S_N\otimes f^*S_M$ over $N$ subject to a suitable local boundary condition is essentially self-adjoint. In order to make our presentation more transparent,  we first consider the special case of twisted Dirac operators on certain model cones. We will then deal with the general case  in the next subsection.

Let $D$ be the Dirac operator acting on $S_N\otimes f^*S_M$ over $N$ subject to the local boundary condition $B$ as in Definition \ref{def:boundarycondition}. From Lemma \ref{lemma:formallysa} and Proposition \ref{prop:norm-equivalent}, $D_B$ is a closable unbounded symmetric operator, and the domain of the minimal extension (i.e.,  the closure of $D_B$) is $H^1(N,S_N\otimes f^*S_M;B)$. The maximal extension of $D_B$ has domain 
$$\dom_{\max }(D_B)=\{f\in L^2(N,S_N\otimes f^*S_M): Df\in L^2(N,S_N\otimes f^*S_M) \},$$
where $Df$ is defined in the weak sense via pairing with sections in  $C^\infty_0(N,S_N\otimes f^*S_M;B)$. See  Definition \ref{def:smooth,H1} for the definition of  $C^\infty_0(N,S_N\otimes f^*S_M;B)$. The operator $D_B$ is  essentially self-adjoint if $\dom_{\max }(D_B)=H^1(N,S_N\otimes f^*S_M;B)$.

To determine whether $D_B$ is essentially self-adjoint, we can localize our computation on small neighborhoods of points in faces of various codimenions. Indeed, if there exists a non-zero element $\varphi$ in $\dom_{\max }(D_B)$ but not in $H^1(N,S_N\otimes f^*S_M;B)$, then, by a smooth partition of unity,  there exists a small neighborhood of a point in $\partial N$ on which the restriction of $\varphi$  is not in $H^1$. 

Suppose $x$ is a point in the interior of a codimension $k$ face $\overbar F_\lambda$ of $N$, then a small neighborhood of $x$ is homeomorphic to  $W\times \fiber$, where $W$ is a ball in $\R^{n-k}$ and  $\fiber$ is a polyhedral corner in $\R^k$, that is,  $\fiber$ is the closed subset in $\R^k$ enclosed by hyperplanes in $\R^k$ passing through the origin.
More precisely,   let $W$ be a small neighborhood of $x$ in $\overbar F_\lambda$ and $\mathcal N(W)$  the normal bundle of $W$ in $N$.  Let $\mathcal N_\varepsilon(W)$ be the subspace of $\mathcal N(W)$ where each fiber is the $\varepsilon$-ball centered at $0$.   The exponential map $\exp\colon \mathcal N_\varepsilon(W) \to N$ is injective for some sufficiently small $\varepsilon >0$. To be precise, the exponential map $\exp\colon \mathcal N_\varepsilon(W) \to N$  is only partially defined, since we are near a codimension $k$ face. On  each fiber of the normal $\varepsilon$-ball bundle $\mathcal N_\varepsilon(W)$, the exponential map is only defined on the intersection of the $\varepsilon$ ball with some asymptotically polyhedral corner of $\mathbb R^k$. In particular, the image $\exp(\mathcal N_\varepsilon(W) )$  is an open neighborhood of $x$ that can be viewed as a fiber bundle  $W\times \fiber$ over $W$. For each $y\in W$, the fiber $\fiber_y$ over $y$ is precisely $\exp_y(\mathcal N_{\varepsilon}(W)_y)$, where $\mathcal N^{\varepsilon}(W)_y$ is the fiber of the normal bundle $\mathcal N_{\varepsilon}(W)$ at $y$.  Note that the (fiberwise) Riemannian metric $g^\fiber$ on  $\fiber_y \cong \fiber$ is asymptotically conical, that is,
\begin{equation}
	g^\fiber=dr^2+r^2 g^{\link}+o(r^2),
\end{equation}
where $r$ is the radial variable of $\fiber$, and $\link = \fiber \cap \mathbb S^{k-1}$ is the link of $\fiber$ with its Riemannian metric denoted by $g^{\link}$. Similarly, there exists a small open  neighborhood of $f(x)$ in $M$ that carries a similar fiber bundle structure, denoted by $V\times \mathbb G$, where $V$ is an open neighborhood of $f(x)$ in  $F_\lambda$, where $F_\lambda $ is the corresponding codimension $k$ face in $M$.

With the above fiber bundle structures, the vector bundle $TN\oplus f^*TM$ near $x$ has the following orthogonal decomposition
$$TN\oplus f^*TM=T\fiber\oplus T\overbar F_\lambda\oplus f^*(T\mathbb G)\oplus TF_\lambda, $$
where $T\fiber$ (resp.  $T\mathbb G$) is the vertical bundle (consisting of vectors tangent to the fiber) of the fiber bundle $W\times \fiber$ (resp. $V
\times \mathbb G$), and    $T\overbar F_\lambda$  (resp. $T F_\lambda$) is the orthogonal complement of  $T\fiber$ in $TN$ (resp. $T\mathbb G$ in $TM$). We fix a smooth orientation preserving bundle isometry between $T\fiber$   and $f^*(T\mathbb G)$ in a small neighborhood of $x$.  We emphasize that this bundle isometry is \emph{not} induced by the map $f$. The choice of such a bundle isometry is certainly not unique, but any such choice will work for our discussion below. Now locally the spinor bundle decomposes as 
\begin{equation}\label{eq:spinordecomp}
	S_N\otimes f^*S_M=S_{T\fiber\oplus f^\ast(T\mathbb G)}\otimes S_{T\overbar F_\lambda\oplus f^*TF_\lambda}.
\end{equation}

\begin{remark}\label{remark:innernormal}
	It should be noted that, as the geometry of the fiber $\fiber_x$ may change as $x$ varies in $W$,  the fiberwise unit inner normal vectors of faces of each fiber $\fiber_x$ are  \emph{different} from the corresponding unit inner normal vectors of faces of the total space $W\times \fiber$ in general. On the other hand, at the vertex of each $\fiber_x$ (that is, at the point $(x, 0)\in W\times\{0\} \subset W\times \fiber$), the fiberwise unit inner normal vectors  coincide with the corresponding unit inner normal vectors of the total space $W\times \fiber$.  Therefore, instead of the vertical subbundle $T\fiber$ above, we may choose a smooth subbundle $\mathcal V$ of $T(W\times \fiber)$ such that
	\begin{enumerate}
		\item at any point $z$ in each codimension one face $\overbar F$ of the total space $W\times \fiber$, the vector space $\mathcal V_z$ at $z$ contains the unit inner normal vector of $\overbar F$ at $z$, and 
		\item $\mathcal V$ is isomorphic to $T\fiber$ via a smooth bundle isometry that is asymptotically the identity map as $r \to 0$, where $r$ is the radial variable of the fiber $\fiber$.  In particular, $\mathcal V$ coincides with $T\fiber$ at $W\times \{0\}$.
	\end{enumerate}	
	Similar remarks also apply to $V\times \mathbb G$ in $M$. 
\end{remark}
We shall revisit this observation from Remark \ref{remark:innernormal} in the proof of Theorem \ref{thm:ess-sa} in  Section \ref{sec:ess-selfadj-general}. For the moment, we first consider the essential self-adjointness of the Dirac operator $D_B$ along a single fiber $\fiber$, as this is the model case which the general case reduces to. Throughout this subsection, we shall consider the model case where both $\fiber$ and $\mathbb G$ are closed subsets in $\R^k$ enclosed by some hyperplanes in $\R^k$ passing through the origin. In particular, $\fiber$ and $\mathbb G$,  and all of their faces are flat. In this case, we shall identify both of the tangent spaces $T\fiber$ and $T\mathbb G$ with the trivial $\mathbb R^k$ bundle over $\fiber$ and $\mathbb G$. By \cite[I.3.9]{spingeometry}, the spinor bundle  $S_{T\fiber\oplus f^\ast(T\mathbb G)} = S_{\mathbb R^{2k}}$  is naturally identified with the bundle of forms $\Bigwedge^*\fiber \coloneqq \Bigwedge^*\R^k$, and the Clifford actions $\overbar c$ and $c$  become the usual Clifford actions on forms. Consequently,  the associated Dirac operator $D^\fiber$ on $S_{T\fiber\oplus f^\ast(T\mathbb G)}$  becomes precisely  the de Rham operator $D^\dR_\fiber$ acting on  $\Bigwedge^*\fiber$.
Note that the metric on $\fiber$ is conical, that is,
$$g^\fiber=dr^2+r^2g^\link$$
where $r$ is the radial coordinate of $\fiber$ and $\link = \fiber \cap \mathbb S^{k-1}$ is the link of $\fiber$.

	There is a natural unitary that transforms $D^\dR_\fiber$ into an elliptic operator in terms of the cylindrical metric (i.e., product metric) on $(0, \infty) \times \link$, cf. \cite[Section 5]{LeschTopology}. Denote by $\Omega^\ast (\fiber)$ the space of forms over $\fiber$. Then the unitary $\Psi$ is given by
	\begin{equation}\label{eq:coordinateChange}
		\Psi=(\Psi_{\mathrm{even}},\Psi_{\mathrm{odd}})\colon 
		C^\infty((0,\infty),\Omega^*\link)\oplus C^\infty((0,\infty),\Omega^*\link)
		\longrightarrow \Omega^\ast(\fiber),
	\end{equation}
	where  
	\begin{equation}\label{eq:transformeven}
		\Psi_{\mathrm{even}}\colon C^\infty((0,\infty),\Omega^* \link)\longrightarrow \Omega^{\mathrm{even}} \fiber,~
		\omega_p\longmapsto\begin{cases}
			r^{p-\frac{k-1}{2}}\omega_p,&\text{ if $p$ is even}\\
			r^{p-\frac{k-1}{2}}\omega_p\wedge dr,&\text{ if $p$ is odd}
		\end{cases}
	\end{equation}
	and
	\begin{equation}\label{eq:transformodd}
		\Psi_{\mathrm{odd}}\colon C^\infty((0,\infty),\Omega^* \link)\longrightarrow \Omega^{\mathrm{odd}} \fiber,~
		\omega_p\longmapsto\begin{cases}
			r^{p-\frac{k-1}{2}}\omega_p,&\text{ if $p$ is odd}\\
			r^{p-\frac{k-1}{2}}\omega_p\wedge dr,&\text{ if $p$ is even}
		\end{cases}
	\end{equation}
	
	With respect to the even/odd grading of differential forms, we have 
	$$D^\dR_\fiber=\begin{pmatrix}
		&D^{\dR,-}_\fiber\\D^{\dR,+}_\fiber&
	\end{pmatrix},$$
	where $D^{\dR,-}_\fiber\colon \Omega^{\mathrm{odd}} \mathbb F \to \Omega^{\mathrm{even}} \mathbb F$ and 
	$D^{\dR,+}_\fiber\colon \Omega^{\mathrm{even}} \mathbb F \to \Omega^{\mathrm{odd}} \mathbb F$. Let us define
	\begin{equation}\label{eq:linkOperator}
		P\coloneqq \begin{pmatrix}
			c_0&d^*&&&\\
			d&c_1&&&\\
			&&\ddots&&\\
			&&& c_{k-2}&d^*\\
			&&& d&c_{k-1}
		\end{pmatrix}
	\end{equation}
	where $d$ is the de Rham differential on $\Omega^\ast \link$, $d^\ast$ is the adjoint of $d$, and  \[ c_p=(-1)^p\big(p-\frac{k-1}{2}\big). \] A straightforward computation shows that (cf. \cite[Section 5]{BruningSeeley} \cite[Proposition 5.3]{LeschTopology}) 
	\begin{equation}\label{eq:fiberdeRham1}
		\Psi_{\mathrm{odd}}^{-1}D^{\dR,+}_\fiber\Psi_{\mathrm{even}}=\frac{\partial}{\partial r}+\frac{1}{r}P\colon C^\infty((0,\infty),\Omega^* \link)\to C^\infty((0,\infty),\Omega^* \link),
	\end{equation}
	and
	\begin{equation}\label{eq:fiberdeRham2}
		\Psi_{\mathrm{even}}^{-1}D^{\dR,-}_\fiber\Psi_{\mathrm{odd}}=-\frac{\partial}{\partial r}+\frac{1}{r}P\colon C^\infty((0,\infty),\Omega^* \link)\to C^\infty((0,\infty),\Omega^* \link).
	\end{equation}
Equivalently,
\begin{equation}\label{eq:deRhamAfterConj}
	\Psi^*D^{\dR}\Psi=\begin{pmatrix}
		0&-\frac{\partial}{\partial r}\\
		\frac{\partial}{\partial r}&0
	\end{pmatrix}+\frac 1 r\begin{pmatrix}
	0&P\\P&0
\end{pmatrix}.
\end{equation}

	Now let us review the boundary condition. By our geometric assumptions, $\fiber$ and $\mathbb G$ have the same combinatorial type, that is, there is a one-to-one correspondence between the codimension one faces of $\fiber$ and those of $\mathbb G$. The boundary condition $B$ on $\Bigwedge^*\fiber$ over each face $\overbar F_i$ of $\fiber$ is given by
	$$\mathscr E(\overbar c(u_i)\otimes c(v_i))\varphi=-\varphi,$$
	where $u_i$ is the unit inner normal vector of $\overbar F_i$ and $v_i$ is the unit inner normal vector of the corresponding face $F_i$ of $\mathbb G$, cf. Definition \ref{def:boundarycondition}. Here $\mathscr E$ is the \mbox{even-odd} grading operator on $\Bigwedge^*\fiber$, and $\overbar c$ and $c$ are the left and right Clifford actions on $\Bigwedge^*\fiber$  given by (cf. \cite[I.3.9]{spingeometry})
	\begin{equation}\label{eq:cliffordonforms}
		\overbar c(u)w=u\wedge w - u \lrcorner w \textup{ and } c(v)w =(-1)^{\deg w}(v\wedge w +v \lrcorner w),
	\end{equation}
	where $\lrcorner$ is the contraction operator. 
	
	Note that the boundary condition $B$ respects the even-odd grading on differential forms. Therefore, under the unitary transform $\Psi$, the boundary condition $B$ gives rise to a boundary condition on the bundle of forms $\Bigwedge^*\link$ over the link $\link$. The link $\link$, which is a subspace of $\mathbb S^{k-1}$, is also a manifold with polyhedral boundary  of dimension $k-1$. The  operator $P$  in line \eqref{eq:linkOperator} is defined along the link $\link$ and  only differs from the de Rham operator of $\link$ by a bounded endomorphism.

	\begin{figure}
		\begin{tikzpicture}[scale=1.3]
			\draw[very thick,black] (0,0) -- (3,0);
			\draw[very thick,black] (0,0) -- ({3*sin(30)},{3*cos(30)});
			\draw[very thick,blue] plot[domain=0:60] ({2*cos(\x)},{2*sin(\x)});
			\filldraw ({2.5*cos(30)},{2.5*sin(30)}) node {$\mathbb L$};
			\draw [very thick,-stealth,red] (2,0) --  (2,0.7);
			\draw [very thick,-stealth,red] ({2*sin(30)},{2*cos(30)}) --  ({2*sin(30)+0.7*cos(30)},{2*cos(30)-0.7*sin(30)});
		\end{tikzpicture}
		\caption{Local boundary condition for the operator $P$ along each link.}
		\label{fig:link}
	\end{figure}
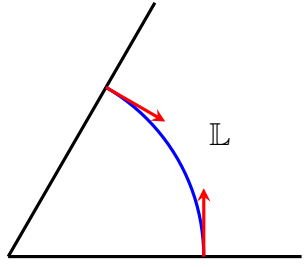

	Let us denote by $D^\dR_{\fiber,B}$ the de Rham operator of $\fiber$, which acts on differential forms that are supported away from the vertex  of $\fiber$, subject to the local boundary condition $B$. Similarly, let $P_B$ be the operator $P$ subject to the boundary condition induced by $B$, cf. Figure \ref{fig:link}.
	The following lemma characterizes when $D^\dR_{\fiber,B}$ is essentially self-adjoint in terms of the spectrum of the operator $P_B$ on the link.  
	\begin{lemma}[{cf. \cite[Theorem 3.1]{BruningSeeley}}]\label{lemm:>=1/2}
		Assume that  $P_B$ is essentially self-adjoint.  Then  $D^\dR_{\fiber,B}$ is essentially self-adjoint if and only if $|P_B|\geq 1/2$.
	\end{lemma}
	\begin{proof}
		Set $E_\pm(D^\dR_{\fiber,B} )\coloneqq \ker(D^\dR_{\fiber, B}\mp i)$. The von Neumann Theorem states that $D^\dR_{\fiber,B}$ is essentially self-adjoint if and only if 
		\[ E_+(D^\dR_{\fiber,B} )=E_-(D^\dR_{\fiber,B} )=0. \]
		By assumption, the $L^2$-space of differential forms on the link $\link$ admits an orthonormal basis $\{\phi_\lambda\}$, where each $\phi_\lambda$ is the eigenvector of $P_B$ with eigenvalue $\lambda$.
		
		Suppose that $\varphi=\varphi_{\text{even}}\oplus \varphi_{odd}$ lies in the kernel of $D^\dR_{\fiber, B}- i$. Let us denote  $\psi_0=\Psi_{\mathrm{even}}^{-1}(\varphi_{\text{even}})$ and 
		$\psi_1=\Psi_{\mathrm{odd}}^{-1}(\varphi_{\text{odd}})$. We have
		\begin{equation}\label{eq:deficiency}
			\Big(\begin{psmallmatrix}
				0&-\frac{\partial}{\partial r}\\
				\frac{\partial}{\partial r}&0
			\end{psmallmatrix}+\frac{1}{r}\begin{pmatrix}
				0&P\\P&0
			\end{pmatrix}- i\Big)\begin{pmatrix}\psi_0\\\psi_1
			\end{pmatrix}=0.
		\end{equation}
	If we write
		$$\psi_0=\sum_\lambda\psi_{0,\lambda}(r)\phi_\lambda \textup{ and } \psi_1=\sum_\lambda\psi_{1,\lambda}(r)\phi_\lambda, $$
then Equation \ref{eq:deficiency} splits into a family of differential equations according to the eigenvectors $\{\phi_\lambda\}$. That is, for each $\lambda$, we have the following system of ordinary differential equations 
		\begin{equation}\label{eq:psi0,1def}
			\begin{cases}
				\displaystyle	\frac{\partial}{\partial r} \psi_{0,\lambda}+\frac \lambda r\psi_{0,\lambda}=i\psi_{1,\lambda},\\
				\displaystyle	-\frac{\partial}{\partial r}\psi_{1,\lambda}+\frac \lambda r\psi_{1,\lambda}=i\psi_{0,\lambda}.
			\end{cases}
		\end{equation}
		It follows that 
		$$\begin{cases}
			\displaystyle -\frac{\partial^2}{\partial r^2}\psi_{0,\lambda}-\frac{\lambda}{r^2}\psi_{0,\lambda}+\frac{\lambda^2}{r^2}\psi_{0,\lambda}+\psi_{0,\lambda}=0, 	\vspace{.3cm}\\
			\displaystyle i\psi_{1,\lambda}=\frac{\partial}{\partial r}\psi_{0,\lambda}+\frac{\lambda} {r}\psi_{0,\lambda}, 
		\end{cases}$$
		where the solution to the first differential equation consists of modified Bessel functions. More precisely, we have
		\begin{equation}\label{eq:solutionKI}
			\begin{cases}
				\psi_{0,\lambda}=c_1\sqrt r\cdot K_{\lambda+ 1/2}(r)\phi_\lambda+c_2\sqrt r\cdot I_{\lambda+ 1/2}(r)\phi_\lambda\\
				\psi_{1,\lambda}=-i(c_1\sqrt r\cdot K_{\lambda- 1/2}(r)\phi_\lambda+c_2\sqrt r\cdot K_{\lambda- 1/2}(r)\phi_\lambda)
			\end{cases},
		\end{equation}
		where $I_\nu$ and $K_\nu$ are modified Bessel functions of the first and the second kind, respectively. It follows that $\psi_{0,\lambda}\oplus \psi_{1,\lambda}$ is  an $L^2$ solution if and only if we have  $-1/2<\lambda<1/2$ (cf. \cite[Lemma 4.2]{LeschTopology}). This finishes the proof. 
	\end{proof}

	\begin{remark}
		In general, if $\dim E_+(D^\dR_{\fiber, B})=\dim E_-(D^\dR_{\fiber, B})$ (but not necessarily zero), then $D^\dR_{\fiber, B}$ admits self-adjoint extensions. The self-adjoint extensions of $D^\dR_{\fiber, B}$ are in one-to-one correspondence to isometries between $E_+(D^\dR_{\fiber, B})$ to $E_-(D^\dR_{\fiber, B})$. More precisely, the domain of a self-adjoint extension of $D^\dR_{\fiber, B}$ has the following form
		\begin{equation}\label{eq:bdf}
			\mathrm{dom}((D^\dR_{\fiber, B})_{\min}) +  \{x-\Phi x: x\in E_+(D^\dR_{\fiber, B}) \},
		\end{equation}
		where   $\Phi$ is a unitary operator from $E_+(D^\dR_{\fiber, B})$ to $E_-(D^\dR_{\fiber, B})$ and   $(D^\dR_{\fiber, B})_{\min}$ is the closure of $D^\dR_{\fiber, B}$. 
	\end{remark}

	Now we investigate the essential self-adjointness of $D^\dR_{\fiber, B}$  on $\fiber$  subject to the local boundary condition $B$ from Definition \ref{def:boundarycondition}.  Let us first consider the case where $\dim \fiber=2$.

	\begin{lemma}\label{lemma:essensa-jumpanglewithf}
		Let $\fiber $ and $\mathbb G$ be two sectors in $\R^2$. Let $B$ be the boundary condition on $\Bigwedge^*\fiber = \Bigwedge^\ast \mathbb R^2$ over each edge given by
		$$\mathscr E(\overbar c(u_i)\otimes c(v_i))\varphi=-\varphi$$
		for $i=1,2$, where $\mathscr E$ is the even-odd grading operator on $\Bigwedge^\ast \mathbb R^2$, $u_i$'s are the inner normal vectors of $\fiber $ and $v_i$'s are the inner normal vectors of $\mathbb G $, cf. Figure \ref{fig:SNfSM}.	
		Let $D^\dR_{\fiber,B}$ be the de Rham operator acting on $\Bigwedge^*\fiber$ with the boundary condition $B$. Suppose both the dihedral angles $\alpha$ of $\fiber$ and $\beta$ of $\mathbb G$ are less than or equal to $\pi$. Then $D^\dR_{\fiber,B}$ is essentially self-adjoint if and only if $\alpha\leq\beta$.
	\end{lemma}
	\begin{proof}
		
		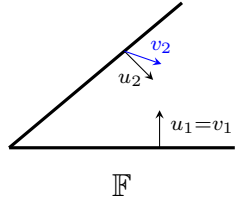
\begin{figure}
			\begin{tikzpicture}[scale=1]
				\draw[very thick,black] (0,0) -- (3,0);
				\draw[very thick,black] (0,0) -- ({3*sin(50)},{3*cos(50)});
				\draw [-stealth] (2,0) --  (2,0.5) node[anchor=north west] {$\scriptstyle u_1=v_1$};
				\draw [-stealth] ({2*sin(50)},{2*cos(50)}) --  ({2*sin(50)+0.5*cos(40)},{2*cos(50)-0.5*sin(50)}) 
				node[anchor=east] {$\scriptstyle u_2$};
				\draw [blue,-stealth] ({2*sin(50)},{2*cos(50)}) --  ({2*sin(50)+0.5*cos(10)},{2*cos(50)-0.5*sin(20)})
				node[anchor=south] {$\scriptstyle v_2$};
				\filldraw (1.5,-0.5) node {$\fiber$};
			\end{tikzpicture}
			\caption{The boundary condition at the two edge of $\fiber$.}
			\label{fig:SNfSM}
		\end{figure}
		By applying a rotation on $\mathbb G$ if necessary, we may assume $u_1=v_1$.  Then the vector $v_2$ differs from $u_2$ by an angle $(\beta - \alpha)$  counterclockwise. For brevity, let us write  $\delta=\beta-\alpha$. More precisely, if we choose unit vectors $u_2^\perp$ and $v_2^\perp$ such that  $u_2^\perp$ is orthogonal to $u_2$,  $v_2^\perp$ is orthogonal to $v_2$, and $u_2\wedge u_2^\perp=v_2\wedge v_2^\perp$ is the volume form of $\mathbb R^2$, then we have
		\begin{equation}
			\begin{pmatrix}
				v_2\\ v_2^\perp
			\end{pmatrix}=\begin{pmatrix}
				\cos\delta &\sin\delta\\-\sin\delta&\cos\delta
			\end{pmatrix}\begin{pmatrix}
				u_2\\ u_2^\perp
			\end{pmatrix} 
		\end{equation}
		
		On the bottom edge $\overbar{F}_1$ of $\fiber$  where $u_1=v_1$ is the unit inner normal vector, a direct computation shows that the boundary condition is the usual absolute boundary condition, that is, if we decompose a differential form as 
		\[ w = w_1 + w_2 dx \] 
		where $dx$ is the differential of the normal direction and $w_j$ are tangential differential forms, then $w$ satisfies the absolute boundary condition at this edge if $w_2$ vanishes. 
	
		On the other edge $\overbar F_2$ of $\fiber$, we have 
		$$\epsilon(\overbar c(u_2)\otimes c(v_2))w=-w.$$
		For a given differential form $w$, we have the following decomposition: 
		$$w=\varphi_1+\varphi_2 u_2+\varphi_3 u_2^\perp+\varphi_4 u_2\wedge u_2^\perp,$$
		and
		$$\epsilon(\overbar c(u_2)\otimes c(v_2))w =
		\psi_1+\psi_2 u_2+\psi_3 u_2^\perp+\psi_4u_2\wedge u_2^\perp. 
		$$
		A direct computation shows that
		\begin{equation*}
			\begin{pmatrix}\psi_1 \\ \psi_4 \\ \psi_2 \\ \psi_3
			\end{pmatrix}=\begin{pmatrix}
				-\cos\delta&\sin\delta&&\\ \sin\delta &\cos\delta &&\\
				&&\cos\delta &\sin\delta\\ &&\sin\delta&-\cos\delta
			\end{pmatrix}\begin{pmatrix}\varphi_1 \\ \varphi_4 \\ \varphi_2 \\ \varphi_3
			\end{pmatrix}
		\end{equation*}
		Therefore, if $\epsilon(\overbar c(u_2)\otimes c(v_2))w= - w$, then we have
		\begin{equation}\label{eq:nbcondition}
			\varphi_1\sin\frac{\delta}{2}+\varphi_4\cos\frac{\delta}{2}=0,\text{ and }
			\varphi_2\sin\frac{\delta}{2}+\varphi_3\cos\frac{\delta}{2}=0
		\end{equation}
		at the edge $\overbar F_2$. 
		Note that this  boundary condition does \emph{not} mix even and odd degree differential forms.
		
		Under the unitaries given in line \eqref{eq:transformeven} and \eqref{eq:transformodd}, the de Rham operator on $\fiber$ is conjugate to the operator 
		\begin{equation}\label{eq:deRhamconj}
			\begin{pmatrix}
				0&-\frac{\partial}{\partial r}\\ \frac{\partial}{\partial r}&0
			\end{pmatrix}+\frac 1 r\begin{pmatrix}
				0&P\\P&0
			\end{pmatrix}
		\end{equation}
		where 
		$$P=\begin{pmatrix}
			-1/2&-\frac{\partial}{\partial \theta}\\\frac{\partial}{\partial \theta}&-1/2
		\end{pmatrix}$$
		as in line \eqref{eq:linkOperator}.
		
		Let $\phi(\theta)=\phi_0(\theta)+\phi_1(\theta) d\theta$ be a differential form on the link $\link$ of $\fiber$. Under the same conjugation above, the boundary condition $B$  becomes  the following boundary condition: $\phi_1(0)=0$, and 
		\begin{equation}\label{eq:newboundary}
			-\phi_0(\alpha)\sin\frac{\delta}{2}+\phi_1(\alpha)\cos\frac{\delta}{2}=0.
		\end{equation}
		Furthermore, the explicit formula in line \eqref{eq:nbcondition} shows that  the boundary conditions for the two copies of $P$ (appearing in the matrix from line \eqref{eq:deRhamconj}) coincide.   It is easy to see that the operator $P$ with this boundary condition becomes an essentially self-adjoint  Fredholm operator, which will be denoted by $P_B$. 
		
		Let $D_\theta^{\dR}=\begin{pmatrix}
			0&-\frac{\partial}{\partial \theta}\\\frac{\partial}{\partial \theta}&0
		\end{pmatrix}$ be the de Rham operator on the link. If $D_\theta^\dR\phi=\lambda\phi$, then
		\begin{equation}
			-\phi_1'=\lambda \phi_0,\text{ and }\phi_0'=\lambda \phi_1.
		\end{equation}
		Hence $\phi_1''=-\lambda^2\phi_1$. By  the boundary condition $\phi_1(0)=0$, we see that  $\phi_1(\theta)=\rho\cdot \sin(\lambda\theta)$ for some constant $\rho$. It follows that $\phi_0(\theta)=- \rho\cdot \cos(\lambda\theta)$. The boundary condition at $\theta=\alpha$ implies that
		\begin{equation}
			\sin(\lambda\alpha)\cos\frac{\delta}{2}+\cos(\lambda\alpha)\sin\frac{\delta}{2}=0,
		\end{equation}
		that is, $\sin(\lambda\alpha+\delta/2)=0$. Therefore the spectrum of the operator $D_\theta^\dR$ with  this mixed boundary condition is 
		$$\Big\{-\frac{\delta}{2\alpha}+\frac{k\pi}{\alpha}\Big\}_{k\in\Z}. $$
		Hence the spectrum of $P_B=-1/2+D_\theta^\dR$ with this mixed boundary condition is given by
		$$ \Big\{-\frac{\beta}{2\alpha}+\frac{k\pi}{\alpha}\Big\}_{k\in\Z}.$$
		Note that when $k=1$, we always have
		$$-\frac{\beta}{2\alpha}+\frac{\pi}{\alpha}=\frac{2\pi-\beta}{2\alpha}\geq \frac 1 2,$$
		since we have assumed that $\alpha+\beta\leq2\pi$. Moreover, since by assumption we have $\alpha\leq \pi$, it follows that  $|P_B|\geq 1/2$ if and only if $-\frac{\beta}{2\alpha}\leq -\frac{1}{2}$, that is,  $ \alpha\leq\beta$. By Lemma \ref{lemm:>=1/2}, this finishes the proof.
	\end{proof}

In Lemma \ref{lemma:essensa-jumpanglewithf} above,  the only geometric input from the sector  $\mathbb G$ is that its dihedral angle enters into the definition of the local boundary condition. More importantly, the proof of  Lemma \ref{lemma:essensa-jumpanglewithf} in fact shows that $|P_B|\geq 1/2$ as long as we have $\alpha \leq \pi$, $\alpha\leq \beta$ and $\alpha + \beta\leq 2\pi$. In particular, the angle $\beta$ does  \emph{not} necessarily  need to be convex. In other words, even when   $\beta >\pi$,  we still have $|P_B|\geq 1/2$ as long as we have $\alpha \leq \pi$, $\alpha\leq \beta$ and $\alpha + \beta\leq 2\pi$. A simple computation shows that the condition 
\[ \alpha \leq \pi, \alpha\leq \beta \textup{ and } \alpha + \beta\leq 2\pi \]
is equivalent to 
\[ \alpha \leq \pi \textup{ and } 	\langle v_1,v_2\rangle\geq \langle u_1,u_2\rangle \]
where $u_i$'s are the inner normal vectors of $\fiber $ and $v_i$'s are the inner normal vectors of $\mathbb G $. The above discussion will play an important role in our computation of the Fredholm index of the twisted Dirac operator $D_B$ associated to $S_N\otimes f^\ast S_M$ over $N$ in Theorem \ref{thm:index-poly}. Let us summarize the above discussion in the following lemma. 

\begin{lemma}\label{lemma:ess-sa-innerproductcomparison}
	Let $\fiber $ be a sector in $\R^2$. Suppose $u_1$ and $u_2$ are the inner normal vectors of $\fiber$, and $v_1$ and $v_2$ are two unit vectors in $\mathbb R^2$.   Let $B$ be the boundary condition on $\Bigwedge^*\fiber = \Bigwedge^\ast \mathbb R^2$ over each edge given by
	$$\mathscr E(\overbar c(u_i)\otimes c(v_i))\varphi=-\varphi$$
	for $i=1,2$, where $\mathscr E$ is the even-odd grading operator on $\Bigwedge^\ast \mathbb R^2$.
	Let $D^\dR_{\fiber,B}$ be the de Rham operator acting on $\Bigwedge^*\fiber$ with the boundary condition $B$. Suppose the dihedral angle $\alpha$ of $\fiber$ is $\leq \pi$. Then  $D^\dR_{\fiber,B}$ is essentially self-adjoint if and only if 
 \[			\langle v_1,v_2\rangle\geq \langle u_1,u_2\rangle. \]
 Moreover, $|P_B| > 1/2$ if and only if $\langle v_1,v_2\rangle > \langle u_1,u_2\rangle$, where $P_B$ is the operator along the link as in the proof of Lemma \ref{lemma:essensa-jumpanglewithf}.  
\end{lemma}

	By a similar argument, we also obtain the following lemma on the essential  self-adjointness of twisted Dirac operator subject to some mixed boundary conditions. This lemma will be important  for  computing the Fredholm index of $D_B$, cf. Theorem \ref{thm:index}.
	\begin{lemma}\label{lemma:essensa-jumpanglewithfmixed}
		Assume the same notation as in Lemma \ref{lemma:essensa-jumpanglewithf}.  Let $\widehat B$ be the local boundary condition\footnote{In other words, the boundary condition $\widehat B$ coincides with the usual boundary condition $B$ on one edge, and is the orthogonal complement of $B$ on the other edge. } on $\Bigwedge^\ast \fiber$ given by  
		$$\mathscr E(\overbar c(u_1)\otimes c(v_1))\varphi=-\varphi$$
		and  
		$$\mathscr E(\overbar c(u_2)\otimes c(v_2))\varphi= \varphi.$$
		Suppose both the dihedral angles $\alpha$ of $\fiber$ and $\beta$ of $\mathbb G$ are less than or equal to $\pi$.	Then the de Rham operator $D^\dR_{\fiber}$ on $\Bigwedge^*\fiber$ subject to the boundary condition $\widehat B$ is essentially self-adjoint if and only if $\alpha+\beta\leq\pi$.
	\end{lemma}
	\begin{proof}
		By the same discussion from the beginning of the proof of Lemma \ref{lemma:essensa-jumpanglewithf}, we may assume that $v_1=u_1$.
		
		Under the unitaries given in line \eqref{eq:transformeven} and \eqref{eq:transformodd}, the de Rham operator on $\fiber$ is conjugate to the operator 
		\begin{equation}
			\begin{pmatrix}
				0&-\frac{\partial}{\partial r}\\ \frac{\partial}{\partial r}&0
			\end{pmatrix}+\frac 1 r\begin{pmatrix}
				0&P\\P&0
			\end{pmatrix}
		\end{equation}
		where
		$$P=\begin{pmatrix}
			-1/2&-\frac{\partial}{\partial \theta}\\\frac{\partial}{\partial \theta}&-1/2
		\end{pmatrix}.$$
		Let $\phi(\theta)=\phi_0(\theta)+\phi_1(\theta) d\theta$ be a differential form on the link $\link$ of $\fiber$. Under the same conjugation above, the boundary condition $\widehat B$  becomes  the following boundary condition: $\phi_1(0)=0$, and 
		\begin{equation}
			-\phi_0(\alpha)\sin\frac{\pi+\delta}{2}+\phi_1(\alpha)\cos\frac{\pi+\delta}{2}=0,
		\end{equation}
		that is,
		\begin{equation}
			\phi_0(\alpha)\cos\frac{\delta}{2}+\phi_1(\alpha)\sin\frac{\delta}{2}=0.
		\end{equation}
		
		Let $D_\theta^{\dR}=\begin{pmatrix}
			0&-\frac{\partial}{\partial \theta}\\\frac{\partial}{\partial \theta}&0
		\end{pmatrix}$ be the de Rham operator on the link $\link$. If $D_\theta^\dR\phi=\lambda\phi$, then
		\begin{equation}
			-\phi_1'=\lambda \phi_0,\text{ and }\phi_0'=\lambda \phi_1.
		\end{equation}
		Hence $\phi_1''=-\lambda^2\phi_1$. By  the boundary condition $\phi_1(0)=0$, we see that  $\phi_1(\theta)=\rho\cdot \sin(\lambda\theta)$ for some constant $\rho$. It follows that $\phi_0(\theta)=- \rho\cdot \cos(\lambda\theta)$. The boundary condition at $\theta=\alpha$ implies that
		\begin{equation}
			\cos(\lambda\alpha)\cos\frac{\delta}{2}-\sin(\lambda\alpha)\sin\frac{\delta}{2}=0,
		\end{equation}
		that is, $\cos(\lambda\alpha+\delta/2)=0$. Therefore the spectrum of the operator $D_\theta^\dR$ with  the mixed boundary condition $\widehat B$ is 
		$$\Big\{-\frac{\delta}{2\alpha}+\frac{\pi}{2\alpha}+\frac{k\pi}{\alpha}\Big\}_{k\in\Z}. $$
		Hence the spectrum of $P_B=-1/2+D_\theta^\dR$ with the mixed boundary condition $\widehat B$ is given by
		$$ \Big\{-\frac{\beta}{2\alpha}+\frac{\pi}{2\alpha}+\frac{k\pi}{\alpha}\Big\}_{k\in\Z}.$$
		
		Note that, for $k=-1$, we have
		$$-\frac{\beta}{2\alpha}+\frac{\pi}{2\alpha} +  \frac{k\pi}{\alpha}=-\frac{\pi+\beta}{2\alpha}\leq -\frac 1 2$$
		since we have assumed that $\alpha\leq\pi$.  Similarly,  for $k=0$, we also have
		$$-\frac{\beta}{2\alpha}+\frac{\pi}{2\alpha} +  \frac{k\pi}{\alpha}=\frac{\pi-\beta}{2\alpha}\geq  -\frac 1 2, \textup{ since } \beta\leq  \pi.$$
		Therefore $D$ is essentially self-adjoint if and only if
		$$\frac{\pi-\beta}{2\alpha}\geq\frac 1 2,$$
		that is, $\alpha+\beta\leq\pi$. By Lemma \ref{lemm:>=1/2}, this finishes the proof.
	\end{proof}
	\begin{remark}\label{remark:mixedBdyConditionInNormalVectors}
		We point out that the condition on dihedral angles in Lemma \ref{lemma:essensa-jumpanglewithfmixed} does \emph{not} require the comparison between $\alpha$ and $\beta$. In other words, for the mixed boundary condition $\widehat B$, as long as $\alpha+ \beta \leq \pi$, the Dirac operator $D^\fiber_{\widehat B}$   is still  essentially self-adjoint even if $\alpha>\beta$.  
	\end{remark}

Before we move to the higher dimensional case, we prove one more technical lemma for the two dimensional case, which will be useful for the gluing formula we shall consider in Section \ref{sec:gluing}. 
\begin{lemma}\label{lemma:essensa-Bs}
	Let $\fiber_r=\{(x,y):x\geq 0,y\geq 0\}$ the first quadrant of $\mathbb R^2$ and $\fiber_l=\{(x,y):x\leq 0,y\geq 0\}$ the second quadrant of $\mathbb R^2$. Suppose  $Q_x$ and $Q_y$ are the orthogonal projections from $\Bigwedge^*\R^2$ to the subspaces generated by $\{dx,dx\wedge dy\}$ and $\{dy,dx\wedge dy\}$, respectively. Let $Q_x^\perp = 1 - Q_x$ and $Q_y^\perp = 1- Q_y$ on $\Bigwedge^* \mathbb R^2$. Let $B_s$, $s\in[0,1]$,  be the following boundary condition on smooth sections $(\varphi_r, \varphi_l)$ of  $\Bigwedge^*\R^2$ over $\fiber_r\coprod\fiber_l$: 
	$$\begin{cases}
		Q_y(\varphi_r(x,0))=0,~Q_y(\varphi_l(x,0))=0,\\Q_x(\varphi_r(0,y))=sQ_x(\varphi_l(0,y)),~sQ_x^\perp(\varphi_r(0,y))=Q_x^\perp(\varphi_l(0,y)).
	\end{cases}$$
	Then the de Rham operator $D^\dR$ acting on $\Bigwedge^*\R^2$ over $\fiber_r\coprod\fiber_l$ with boundary condition $B_s$ is essentially self-adjoint near the origin of $\fiber_r\coprod\fiber_l$.
\end{lemma}
\begin{proof}
	Under the unitaries given in line \eqref{eq:transformeven} and \eqref{eq:transformodd}, the de Rham operator on $\fiber_r\coprod\fiber_l$ is conjugate to the operator 
	\begin{equation}
		\begin{pmatrix}
			0&-\frac{\partial}{\partial r}\\ \frac{\partial}{\partial r}&0
		\end{pmatrix}+\frac 1 r\begin{pmatrix}
			0&P\\P&0
		\end{pmatrix}
	\end{equation}
	where
	$$P=\begin{pmatrix}
		-1/2&-\frac{\partial}{\partial \theta}\\\frac{\partial}{\partial \theta}&-1/2
	\end{pmatrix}.$$
	Let $\phi_r(\theta)=\phi_{r,0}(\theta)+\phi_{r,1}(\theta) d\theta$ be a differential form on the interval $[0,\pi/2]$, and $\phi_l(\theta)=\phi_{l,0}(\theta)+\phi_{l,1}(\theta) d\theta$ a differential form on $[\pi/2,\pi]$. Under the conjugation above, the boundary condition $B_s$ becomes the following boundary condition:
	$$\begin{cases}
		\phi_{r,1}(0)=0,~\phi_{l,1}(\pi)=0,\\
		\phi_{r,1}(\pi/2)=s\cdot \phi_  {l,1}(\pi/2),~s\cdot \phi_{r,0}(\pi/2)=\phi_{l,0}(\pi/2).
	\end{cases}$$
	Note that when $s=0$, the spectrum of the operator $P$ has already been computed in Lemma \ref{lemma:essensa-jumpanglewithf} and Lemma \ref{lemma:essensa-jumpanglewithfmixed}. In particular, we have $|P|\geq 1/2$  with respect to the boundary condition $B_0$. 
	
	Now we assume that $0<s\leq 1$. By the boundary conditions at $\theta = 0$ and $\theta = \pi$, if we have
	$$P\phi_r=(\lambda-\frac 1 2)\phi_r \textup{ and } P\phi_l=(\lambda-\frac 1 2)\phi_l, $$
	then
	$$\begin{cases}
\phi_r(\theta)=a(-\cos(\lambda\theta)+\sin(\lambda\theta)d\theta) \vspace{0.2cm}\\
\phi_l(\theta)=b(-\cos(\lambda(\theta-\pi))+\sin(\lambda(\theta-\pi))d\theta)
	\end{cases} 
	$$
	with $a^2+b^2=1$.
	The boundary condition at $\theta=\pi/2$ yields
	$$a\sin(\frac{\lambda\pi}{2})=-s\cdot b\cdot \sin(\frac{\lambda\pi}{2}),~s\cdot a\cdot \cos(\frac{\lambda\pi}{2})=b\cos(\frac{\lambda\pi}{2}).$$
	Note that if both $\sin(\lambda\pi/2)$ and $\cos(\lambda\pi/2)$ are non-zero, then we obtain
	$$a=-sb\text{ and } sa=b$$
	which is impossible. Therefore either $\sin(\lambda\pi/2) = 0$ or $\cos(\lambda\pi/2) = 0$, that is, $\lambda\in\Z$. Thus the spectrum of $P$ is 
	\[ \Big\{k- \frac{1}{2} : k\in\Z \Big\}. \] By Lemma \ref{lemm:>=1/2}, this finishes the proof.
\end{proof}

	Now we turn to the higher dimensional case. 
	\begin{lemma}\label{lemma:essensa-higherdim}
		Let $\fiber$ and $\mathbb G$ be two convex polyhedral corners in  $\R^k$ that are enclosed by $\ell$ hyperplanes through the origin respectively, where $\ell \geq k$.
		Let $B$ be the boundary condition on $\Bigwedge^*\fiber$ over each codimension one face $\overbar F_i$ given by
		$$\mathscr E(\overbar c(u_i)\otimes c(v_i))\varphi=-\varphi,$$
		where $\mathscr E$ is the even-odd grading operator on $\Bigwedge^\ast \mathbb R^k$, $u_i$'s are the inner normal vectors of $\fiber $ and $v_i$'s are the inner normal vectors of $\mathbb G $.
		Let $D^\dR_{\fiber,B}$ be the de Rham operator acting on $\Bigwedge^*\fiber$ with the boundary condition $B$.  If  the dihedral angles $\alpha_{ij}$ of $\fiber$ are less than or equal to the corresponding dihedral angles  $\beta_{ij}$ of $\mathbb G$, then $D_{\fiber,B}^\dR$ is essentially self-adjoint.
	\end{lemma}
	\begin{figure}
		\begin{tikzpicture}[scale=.7]
			\draw[very thick, black] (0,0) -- (0,5);
			\draw[very thick, black] (0,0) -- (5,0);
			\draw[very thick, black] (0,0) -- (-2,-2);
			\draw[thick,black] plot[domain=0:90] ({4*cos(\x)}, {4*sin(\x)});
			\draw[thick,black] plot[domain=90:0] ({-1.4*cos(\x)}, {4*sin(\x)-1.4*cos(\x)});
			\draw[thick,black] plot[domain=90:0] ({4*sin(\x)-1.4*cos(\x)},{-1.4*cos(\x)});
			\fill[gray!20] plot[domain=0:90] ({4*cos(\x)}, {4*sin(\x)}) --
			plot[domain=90:0] ({4*sin(\x)-1.4*cos(\x)},{-1.4*cos(\x)}) --
			plot[domain=0:90] ({-1.4*cos(\x)}, {4*sin(\x)-1.4*cos(\x)});
			\draw[dashed] (0,0) -- (-1.4,-1.4);
			\draw[dashed] (0,0) -- (4,0);
			\draw[dashed] (0,0) -- (0,4);
			\draw[thick,black] plot[domain=0:90]
			({(4*cos(80))*cos(\x) +(-1.4*cos(75))*sin(\x)},
			{(4*sin(80)-4)*cos(\x)+ (4*sin(75)-1.4*cos(75)-4)*sin(\x) + 4});
			\draw[thick,black] plot[domain=0:90]
			({(4*sin(80)-4)*cos(\x)+ (4*sin(75)-1.4*cos(75)-4)*sin(\x) + 4}
			,{(4*cos(80))*cos(\x) +(-1.4*cos(75))*sin(\x)});
			\draw[thick,black] plot[domain=0:90]
			({(4*sin(10)-1.4*cos(10)+1.4)*cos(\x) + (-1.4*cos(10)+1.4)*sin(\x) -1.4},
			{(-1.4*cos(10)+1.4)*cos(\x) + (4*sin(10)-1.4*cos(10)+1.4)*sin(\x) - 1.4});
			\filldraw (1,1) node {$\mathbb L$};
		\end{tikzpicture}
		\caption{A two dimensional link of a three dimensional cone.}
		\label{fig:link2}
	\end{figure}
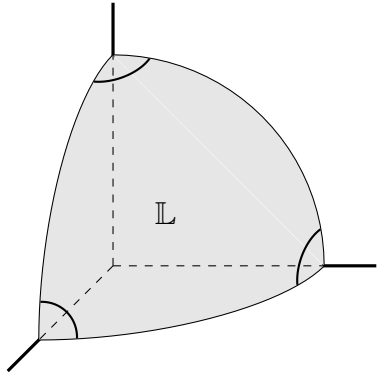
	\begin{proof}
		We prove this by induction on $n$. The case where $n=2$ has been proved in Lemma \ref{lemma:essensa-jumpanglewithf}. Assume that the lemma holds for dimensions up to $(n-1)$.  
		
		Under the unitaries given in line \eqref{eq:transformeven} and \eqref{eq:transformodd}, the de Rham operator on $\fiber$ is conjugate to the operator 
		\begin{equation*}
			\begin{pmatrix}
				0&-\frac{\partial}{\partial r}\\ \frac{\partial}{\partial r}&0
			\end{pmatrix}+\frac 1 r\begin{pmatrix}
				0&P\\P&0
			\end{pmatrix}.
		\end{equation*}
		Recall again that the boundary condition $B$ (as given in Definition \ref{def:boundarycondition}) for codimension one  faces of $\fiber$ does not mix the even and odd degree differential forms. In particular, the boundary condition $B$ induces a boundary condition, still denoted by $B$, for the operator $P$ on the link $\link$ of $\fiber$. 
		
		The link $\link$ of $\fiber$ is a spherical polyhedron in the standard  unit sphere $\mathbb S^{n-1}$, which in particular is itself a manifold with polyhedral boundary. 
		\begin{claim}\label{claim:essa}
			Under the assumption  $\alpha_{ij}\leq \beta_{ij}\leq \pi$, the operator $P$ on the link with the boundary condition $B$ is essentially self-adjoint. 
		\end{claim} 
		We shall prove both the claim and the lemma by inductively alternating between the two statements. More precisely, we shall first prove Claim \ref{claim:essa} for the case of $\dim \fiber =3$, which will in turn be used to prove Lemma \ref{lemma:essensa-higherdim} for the case of $\dim \fiber =3$. We will then use this proven case of Lemma \ref{lemma:essensa-higherdim} to prove Claim \ref{claim:essa} for the case of $\dim \fiber =4$, and so on.

		Let us first prove the above claim in the case where $\dim \fiber =3$. By the explicit formula of the operator $P$ from line \eqref{eq:linkOperator}, $P$ with the boundary condition $B$ is essentially self-adjoint if and only if the de Rham operator $D_{\link}^\dR$ on the link with boundary condition $B$ is essentially self-adjoint, since $P$ and $D_{\link}^\dR$ only differ by a bounded zeroth order term.  Now $\fiber$ is a convex polyhedron in $\R^3$ enclosed by $\ell$ planes $\overbar F_1, \overbar F_2, \cdots, \overbar F_\ell$ that go through the origin.
		The link $\link$ of $\fiber$ is a convex spherical polygon in $\mathbb S^2$, and the dihedral angles of $\link$ coincide with those of $\fiber$. We first show that the de Rham operator $D^\dR_\link$ is essentially self-adjoint with respect to the restriction of $B$ to $\link$. Near an edge $\overbar F_i\cap \overbar F_j$, by rotating $\mathbb G$  if necessary,\footnote{Note that the operator $P$ commutes with rotations on $\mathbb G$. Here is a conceptual explanation. It is clear that the Dirac operator $D_\fiber^{\dR}$ on $\fiber$ commutes with rotations on $\mathbb G$. Furthermore, the unitaries in line \eqref{eq:transformeven} and \eqref{eq:transformodd} (for changing from conic coordinates to cylindrical coordinates) commute with rotations on $\mathbb G$. Since we obtain the operator $P$ by writing the operator $D_\fiber^{\dR}$ in terms of cylindrical coordinates, it follows that $P$ commutes with rotations of $\mathbb G$. We also refer the reader to an explicit computation done at the end of the proof of Lemma \ref{lemma:essensa-higherdim} for an alternative justification. } we may assume without loss of generality that the corresponding edge $F_i\cap F_j$ of $\mathbb G$  coincides with $\overbar F_i\cap \overbar F_j$  and the face $F_i$ coincides with $\overbar F_i$. Of course, the face $F_j$ may still be different from $\overbar F_j$ in general. In any case, the vectors $u_i$, $u_j$, $v_i$ and  $v_j$ are all  orthogonal to $\overbar F_i\cap \overbar F_j$. Now the restriction of the boundary condition $B$ to $\link$ near $\overbar F_i\cap \overbar F_j$ is also of the form as given in Definition \ref{def:boundarycondition}. The metric of $\link$ near $\overbar F_i\cap \overbar F_j$ is asymptotically conical, where the criterion for essential self-adjointness is the same as the standard conical case, cf. Lemma \ref{lemma:asymptoticConical} below. By repeating the same argument near each edge $\overbar F_i\cap \overbar F_j$ of $\fiber$,   this proves that $D^\dR_{\link,B}$, hence $P_B$, is essentially self-adjoint in the case where $\dim \fiber =3$.

		Now we shall prove Lemma \ref{lemma:essensa-higherdim} for the case where $\dim \fiber = 3$. Since we have already shown the operator $P_B$ on the link is essentially self-adjoint in this case,  it suffices to show that the spectrum of $P_B$  satisfies $|P_B|\geq 1/2$ (cf. Lemma \ref{lemm:>=1/2}). 
		\begin{claim}\label{claim:>=1/2}
			If the operator $P_B$ acting on the differential forms $\Bigwedge^*\mathbb L$ over $\mathbb L$ is essentially self-adjoint and $n = \dim \fiber\geq 3$, then $|P_B|\geq 1/2$. 
		\end{claim}
		
		In fact, we shall prove a stronger lower bound of the spectral gap. More precisely, under the same assumption, we will show that 
		\begin{equation}\label{eq:spectralgap}
			|P_B| \geq\frac{\sqrt{(n-1)(n-2)}}{2},
		\end{equation} 
		where again $\dim \fiber = n$ and thus $\dim \link = n-1$. Before we prove Claim \ref{claim:>=1/2} above, let us fix some notation. Recall that the operator $P$ on the link $\link$ is equal to the sum of the de Rham operator $D = D^{\dR}_\link$ acting on $\Bigwedge^*\mathbb L$ over the link $\link$ and  a bounded operator that is diagonal with respect to the $\Z$-grading of differential forms with diagonal entries $(-1)^p(p-\frac{n-1}{2})$, cf. line \eqref{eq:linkOperator}. We will still denote by $\overbar c$ and $c$ the  left and right Clifford multiplications  on  $\Bigwedge^*\mathbb L$. A direct computation shows that 
		\begin{equation}
			P=D+\frac 1 2 \sum_k \overbar c(e_k)\otimes c(e_k),
		\end{equation}
		where $\{e_k\}$ is a local orthonormal basis for $T\link$. We define a new connection $\widehat\nabla$ on $\Bigwedge^*\mathbb L$ by setting 
		\begin{equation}\label{eq:modconnection}
			\widehat\nabla_X\coloneqq \nabla_X+\frac 1 2 c(X),
		\end{equation}
		where $X$ is any tangent vector of $\link$ and $c(X)$ is the right Clifford multiplication of $X$ on $\Bigwedge^*\mathbb L$. In particular, we have 
		$$P=\sum_{k} \overbar c(e_k)\widehat\nabla_{e_k}.$$ If $\dim\link$ is even,  then we have the natural identification  $\Bigwedge^*\link \cong S_\link\otimes S^\ast_\link$. In this case,  $\widehat\nabla$ can be identified with a tensor product connection on $S_\link\otimes S^\ast_\link$, where the first component is equipped with the spinor connection and the second component is equipped with a  flat connection. This gives a more conceptual explanation for some of the computation below. In any case,  the computation below is done explicitly, and in particular does \emph{not} require the above conceptual explanation.  
		
		Now let us derive a Lichnerowicz formula for  $P^2$. For a given point $x\in \link$, we choose a local orthonormal frame $\{e_i\}$ of $T\link$ such that $[e_i,e_j]=0$ at $x$. Note that the Clifford multiplications $\overbar c$ and $c$ commute with each other, so we have
		$$P^2=-\sum_{i}\widehat\nabla_{e_i}\widehat\nabla_{e_i}+\sum_{i<j}\overbar c(e_i)\overbar c(e_j)\widehat R_{e_i,e_j},$$
		where $\widehat R$ is the curvature operator for $\widehat\nabla$, that is,
		$$\widehat R_{e_i,e_j}=\widehat \nabla_{e_i}\widehat\nabla_{e_j}-\widehat \nabla_{e_j}\widehat\nabla_{e_i}.$$
		Since the Levi-Civita connection on $\link$ is torsion free, we obtain that
		\begin{align*}
			\widehat R_{e_i,e_j}=&(\nabla_{e_i}+\frac 1 2 c(e_i))(\nabla_{e_j}+\frac 1 2 c(e_j))-(\nabla_{e_j}+\frac 1 2 c(e_j))(\nabla_{e_i}+\frac 1 2 c(e_i))\\
			=&\nabla_{e_i}\nabla_{e_j}-\nabla_{e_j}\nabla_{e_i}+\frac 1 2 c(e_i)c(e_j).
		\end{align*}
		Note that  $\link$ has constant sectional curvature $1$, in particular has  scalar curvature $(n-1)(n-2)$. From line \eqref{eq:D^2smooth},  we see that
		$$\sum_{i<j}\overbar c(e_i)\overbar c(e_j)(\nabla_{e_i}\nabla_{e_j}-\nabla_{e_j}\nabla_{e_i})=\frac{(n-1)(n-2)}{4}-\frac 1 2\sum_{i<j} \overbar c(e_i)\overbar c(e_j)\otimes c(e_i)c(e_j).$$
		Therefore, we have the following Lichnerowicz formula for $P^2$: 
		\begin{align*}
			P^2=&\widehat\nabla^*\widehat\nabla+\sum_{i<j}\overbar c(e_i)\overbar c(e_j)\widehat R_{e_i,e_j}\\
			=&\widehat\nabla^*\widehat\nabla+\frac{(n-1)(n-2)}{4}-\frac 1 2 \sum_{i<j}\overbar c(e_i)\overbar c(e_j)\otimes c(e_i)c(e_j)  \\
			& \hspace{4cm} +\frac 1 2 \sum_{i<j}\overbar c(e_i)\overbar c(e_j)\otimes c(e_i)c(e_j)\\
			=&\widehat\nabla^*\widehat\nabla+\frac{(n-1)(n-2)}{4}
		\end{align*}
		
		Recall that on each codimension one face $\link_i = \overbar  F_i\cap \link $ of $\link$, the boundary condition $B$ is given by  
		$$\ker \big(\mathscr E(\overbar c(u_i)\otimes c(v_i))+ 1\big),$$
		where $u_i$ is the inner unit normal vector field of $\overbar F_i$ and $v_i$ is the unit inner normal vector field of the corresponding codimension one face $F_i$ of $\mathbb G$. Observe that, in the current setup of the lemma, the vector fields  $u_i$ and $v_i$ are  constant along $\overbar F_i$.

		Let $\varphi$ be a differential form in $C_0^\infty(\mathbb L,\Bigwedge^*\mathbb L;B)$. 
		By a similar computation as in Proposition \ref{prop:D^2}, we have
		\begin{equation}\label{eq:deRhamestimate}
			\begin{split}
				&\int_{\mathbb L}\langle P \varphi,  P \varphi\rangle \\
				=&
				\int_{\mathbb L}\langle P^2\varphi,\varphi\rangle+\sum_i\int_{\link_i}\langle\overbar c(u_i)P\varphi,\varphi\rangle\\ 
				=&\int_{\mathbb L}\langle \Big(\widehat\nabla^*\widehat\nabla+\frac{(n-1)(n-2)}{4}\Big)\varphi,\varphi\rangle+\sum_i\int_{\link_i}\langle\overbar c(u_i)P\varphi,\varphi\rangle\\
				=&\int_{\mathbb L}\langle \widehat\nabla\varphi,\widehat\nabla\varphi\rangle+ \int_{\mathbb L}
				\frac{(n-1)(n-2)}{4}|\varphi|^2+\sum_{i,j}\int_{\link_i}\langle \overbar c(u_i)\overbar c(e^j_i)\widehat\nabla_{e^j_i}\varphi,\varphi\rangle,
			\end{split}
		\end{equation}
		where in the last summation $\{e_i^j\}$ is a local orthonormal basis of tangent vectors of $\link_i$. Observe that the dihedral angles do not appear in Equation \eqref{eq:deRhamestimate}, since the element $\varphi$ vanishes near the codimension two faces of $\link$. 
		
		Set
		\begin{equation}\label{eq:hatDpartial}
			\widehat D^\partial_i=\overbar c(u_i)\sum_j \overbar c(e^j_i)\widehat\nabla_{e^j_i}=\overbar c(u_i)\sum_j \overbar c(e^j_i)\nabla_{e^j_i}+\overbar c(u_i)\sum_j \overbar c(e^j_i)\otimes c(e_i^j).
		\end{equation}
		To show that $|P|\geq \sqrt{(n-1)(n-2)}/2$, it suffices to show that the term
		\[ \langle \overbar c(u_i)\overbar c(e^j_i)\widehat\nabla_{e^j_i}\varphi,\varphi\rangle = \langle \overbar c(u_i)\widehat D^\partial_i\varphi,\varphi\rangle \]
		in line \eqref{eq:deRhamestimate} vanishes if $\varphi$ satisfies the boundary condition. Equivalently, it suffices to show that the operator $\widehat D^\partial_i$ maps the boundary condition $B$ to its orthogonal complement. In the following we shall give two different arguments  to verify this. 
		
		Here is the first argument, which is somewhat more conceptual. We shall verify 
		$\langle \overbar c(u_i)\widehat D^\partial_i\varphi,\varphi\rangle$
		vanishes along each face $\link_i$ of $\link$ one at a time (\emph{not} simultaneously) if $\varphi$ satisfies the boundary condition. Recall that the boundary condition $B$ is given by  \[ B = \ker ( \mathscr E(\overbar c(u_i)\otimes c(v_i) ) + 1). \] After the change of coordinates as in line \eqref{eq:transformeven} and \eqref{eq:transformodd}, the boundary condition $B$ becomes a boundary condition on $\Bigwedge^*\link$, which we still denote by $B$. We shall perform the computation at each codimension one face of $\link$ one at a time. Let us fix a face $\link_i$ of  $\link$.  By  applying a rotation to $\mathbb G$ if necessary, we may assume that $u_i=v_i$ on $\link_i$. We emphasize that we only require the equality  $u_i=v_i$ to hold on the given face $\link_i$. Now the boundary condition $B$ on $\link_i$  is the sub-bundle consisting of differential forms that only contain tangential directions along $\link_i$. In other words,   we have $B=\ker(\mathscr E(\overbar c(u_i)\otimes c(u_i))+1)$ on $\link_i$, where $\mathscr E$ is the even-odd grading operator on $\Bigwedge^*\link$.
		
		Note that $u_i$ is constant along $\link_i$. In particular,  the  term  $\overbar c(u_i)\sum_j \overbar c(e^j_i)\nabla_{e^j_i}$ in line \eqref{eq:hatDpartial} commutes with $\mathscr E$ and $c(u_i)$, and anti-commutes with $\overbar c(u_i)$. Similarly, the term 
		$\overbar c(u_i)\sum_j \overbar c(e^j_i)\otimes c(e_i^j)$ in line \eqref{eq:hatDpartial} anti-commutes with $\mathscr E$, $\overbar c(u_i)$ and $c(u_i)$. In conclusion,  $\widehat D^\partial_i$ anti-commutes with $\mathscr E(\overbar c(u_i)\otimes c(u_i))$, hence maps the boundary condition $B$ to its orthogonal complement along the face $\link_i$. This implies that 	$\langle \overbar c(u_i)\widehat D^\partial_i\varphi,\varphi\rangle$
		vanishes along  $\link_i$.  By  performing the same argument at each face of $\link$, we see that $\langle \overbar c(u_i)\widehat D^\partial_i\varphi,\varphi\rangle$
		vanishes along each face $\link_i$ if $\varphi$ satisfies the boundary condition $B$. This finishes the proof of  Claim \ref{claim:>=1/2}.
		
		Now we shall give an alternative argument to show that $\langle \overbar c(u_i)\widehat D^\partial_i\varphi,\varphi\rangle=0$ if $\varphi$ satisfies the boundary condition $B$, via a direct computation that in fact does \emph{not} require any rotation of  $\mathbb G$. Note that in general the normal vector $v_i$ of the corresponding face $F_i$ of $\mathbb G$ may not be tangential to the link $\link$ of $\fiber$, and we have the following decomposition
		\begin{equation}\label{eq:decompositionOfVi}
			v_i=\alpha_i\frac{\partial}{\partial r}+\beta_i \widehat{v}_i,
		\end{equation}
		along $\link_i$, where $\link_i$ is a codimension one face of $\link$ as above and  $\widehat{v}_i$ is orthogonal to $\frac{\partial}{\partial r}$. A direct computation shows that the boundary condition $B$ on $\link_i$ is the following
		\begin{equation}
			\ker\big( \mathscr E \overbar c(\overbar u_i)\otimes (\alpha_i+\beta_i c(\widehat{v}_i))+1\big),
		\end{equation}
		where $\mathscr E$ again is the even-odd grading operator on $\Bigwedge^*\link$. We remark that the boundary condition $B$  may not preserve the even-odd $\Z_2$-grading on $\Bigwedge^*\link$.
		
		In any case, let us write  $\gamma_i=\mathscr E \overbar c(\overbar u_i)\otimes (\alpha_i+\beta_i c(\widehat{v}_i))$ for short.
		Note that the Clifford multiplications $\overbar c$ and $c$ commute with each other. Therefore, to show that  the operator $\widehat D^\partial_i$ maps the boundary condition $B$ to its orthogonal complement,  it suffices to show that  $\gamma_i$ commutes with the connection $\widehat\nabla$ on $\link_i$.

		For  simplicity, we omit the subscript $i$ in $\gamma_i, \link_i$ and so on.  Although $v$ is a constant vector field on $\fiber$, the functions $\alpha$ and $\beta$, and the vector filed $\widehat v$ in line \eqref{eq:decompositionOfVi} may not be constant in general.  Consider the flat connection  $\overbar \nabla$  on the tangent bundle of $\fiber$. Let $\nabla^\link$ be the Levi--Civita connection on the tangent bundle of $\link$. Then we have 
		$$\begin{cases}\displaystyle
			\overbar\nabla_X Y=\nabla^\link_X Y+\langle X,Y\rangle\frac{\partial}{\partial r}\\
			\displaystyle\overbar\nabla_X \frac{\partial}{\partial r}=-X,
		\end{cases}$$
		for tangent vectors $X$ and $Y$ along $\link$. 
		By applying $\overbar\nabla$ to both sides of line \eqref{eq:decompositionOfVi}, we obtain
		$$0=\overbar\nabla_X v =X(\alpha )\frac{\partial}{\partial r}-\alpha  X+X(\beta ) \widehat v +\beta \nabla^\link_X\widehat v +\beta \langle X,\widehat v \rangle\frac{\partial}{\partial r},$$
		since $v $ is a constant vector field on $\fiber$. By regrouping the terms according to whether they are tangential or orthogonal to $\link$, we obtain that
		\begin{equation}\label{eq:alphaBeta}
			\begin{cases}
				X(\alpha )=-\beta \langle X,\widehat v \rangle\\
				\alpha  X=X(\beta)\widehat v +\beta\nabla^\link_X\widehat v 
			\end{cases}
		\end{equation}
		Let us denote $\gamma =\mathscr E \overbar c(\overbar u )\otimes (\alpha +\beta  c(\widehat{v }))$. We need to show that $[\widehat\nabla,\gamma]=0$, where $\widehat \nabla$ is the connection from line \eqref{eq:modconnection}.  Note that
		\begin{align*}
			[\nabla_X,\gamma]=& \mathscr E\overbar c(u) \otimes\big(X(\alpha)+X(\beta)c(\widehat v)+\beta c(\nabla^\link_X\widehat v)\big)\\
			=&\mathscr E\overbar c(u) \otimes\big( -\beta\langle X,\widehat v\rangle+\alpha c(X)\big),
		\end{align*}
		and
		\begin{align*}
			\frac 1 2[1\otimes c(X),\gamma]=&-\frac 1 2\mathscr E\overbar c(u)\otimes \big(
			c(X)(\alpha +\beta  c(\widehat{v }))+(\alpha +\beta  c(\widehat{v }))c(X)\big)\\
			=&-\frac 1 2\mathscr E\overbar c(u)\otimes \big(2\alpha c(X)-2\beta \langle X,\widehat v\rangle\big)\\
			=&-\mathscr E\overbar c(u)\otimes \big( -\beta\langle X,\widehat v\rangle+\alpha c(X)\big).
		\end{align*}
		This shows that $[\widehat\nabla,\gamma]=0$. It follows that $\langle \overbar c(u_i)\widehat D^\partial_i\varphi,\varphi\rangle=0$ if $\varphi$ satisfies the boundary condition $B$.  This completes the alternative proof of Claim \ref{claim:>=1/2}.
		
		Now we complete the proofs of Claim \ref{claim:essa} and Lemma \ref{lemma:essensa-higherdim} by inductively alternating between Claim \ref{claim:essa} and Lemma \ref{lemma:essensa-higherdim}. 
	\end{proof}

	\subsection{Essential self-adjointness of twisted Dirac operators}\label{sec:ess-selfadj-general}
	In this and the next subsections, we prove the Dirac operator $D_B$ on $S_N\otimes f^\ast S_M$  is essentially self-adjoint, under appropriate conditions on dihedral angles. More precisely, we have  the following theorem.
	\begin{theorem}\label{thm:ess-sa}
		Assume the geometric setup $\ref{setup}$. 	Let $D_B$ be the Dirac operator on $S_N\otimes f^*S_M$ subject to the local  condition $B$ given in Definition \ref{def:boundarycondition}. In particular,  $D_B$ acts on  the space of sections  $C^\infty_0(N,S_N\otimes f^*S_M;B)$ given in Definition $\ref{def:smooth,H1}$. For each pair of  codimension one faces $\overbar F_i, \overbar F_j$ of $N$, assume  the dihedral angles $\theta_{ij}(\overbar{g})$ of $N$ and $\theta_{ij}(g)$ of $M$ satisfy either  
		\begin{equation}\label{eq:dihedralstrict}
			0<\theta_{ij}(\overbar{g})_z < \theta_{ij}(g)_{f(z)} < \pi, \textup{ for all }   z \in \overbar F_i\cap \overbar{F}_j
		\end{equation} 
	or 
	\begin{equation}\label{eq:dihedralequal}
		0<\theta_{ij}(\overbar{g})_z =  \theta_{ij}(g)_{f(z)} < \pi, \textup{ for all }   z \in \overbar F_i\cap \overbar{F}_j. 
	\end{equation} 
 Then   $D_B$ is essentially  self-adjoint.  Furthermore, its self-adjoint extension $\overbar{D}_B$ is Fredholm  with domain   $H^1(N,S_N\otimes f^*S_M;B)$.
	\end{theorem}
	
	First, we need the following lemma, which is roughly speaking a generalization of   Lemma \ref{lemm:>=1/2} to the case of asymptotically conical metrics. The proof essentially follows from the analysis in \cite[Section 3]{BruningSeeley}. For the convenience of the reader, we sketch a proof here. 
	\begin{lemma}\label{lemma:asymptoticConical}
		Let $\fiber$ and $\mathbb G$ be two convex  polyhedral corners in  $\R^k$ that are enclosed by $\ell$ hyperplanes through the origin,  respectively.
		Let $B$ be the boundary condition on $\Bigwedge^*\fiber$ over each codimension one face $\overbar F_i$ given by
		$$\mathscr E(\overbar c(u_i)\otimes c(v_i))\varphi=-\varphi,$$
		where $\mathscr E$ is the even-odd grading operator on $\Bigwedge^\ast \mathbb R^k$, $u_i$'s are the inner normal vectors of $\fiber $ and $v_i$'s are the inner normal vectors of $\mathbb G $. Consider the differential operator $ \mathcal D$ given in cylindrical coordinates by 
		$$\begin{pmatrix}
			0&-\frac{\partial}{\partial r}\\ \frac{\partial}{\partial r}&0
		\end{pmatrix}+\frac 1 r\begin{pmatrix}
			0&P\\P&0
		\end{pmatrix}+\begin{pmatrix}
			0&A_r\\A_r&0
		\end{pmatrix},$$
		where $A_r$ is a family of symmetric first order differential operators.
		Suppose that 
		\begin{enumerate}[label=$(\arabic*)$]
			\item both the operators $P$ and $P+rA_r$ (along each link) are essentially self-adjoint with respect to the boundary condition $B$ and have the same   domain $H^1(\link_r,E;B)$,
			\item $|P|\geq 1/2$,
			\item  $P^{-1}A_r$ and $A_rP^{-1}$ are bounded operators,  and their operator norms are uniformly bounded by a constant that is  independent of $r$,
		\end{enumerate} 
		then $\mathcal D$ is  essentially self-adjoint  with respect to the boundary condition $B$  and its domain is  $H^1(\fiber, \Bigwedge^*\R^k; B)$.
	\end{lemma}
	\begin{proof}
		By  Lemma \ref{lemm:>=1/2}, the operator 
		$$ \slashed{D} =\begin{pmatrix}
			0&-\frac{\partial}{\partial r}\\ \frac{\partial}{\partial r}&0
		\end{pmatrix}+\frac 1 r\begin{pmatrix}
			0&P\\P&0
		\end{pmatrix}.$$
		is essentially self-adjoint, since $|P|\geq 1/2$. Moreover, the domain of $\slashed D$ is precisely $H^1(\fiber, \Bigwedge^*\R^k; B)$. 
		
		It suffices to show that the maximal domain of $\mathcal D$ is contained in the maximal domain of $\slashed D$. Let $\link$ be the link of $\fiber $. Then it  suffices to show that if $u\in L^2([0,1]\times \link,\Bigwedge^*\link)$ and
		$$h\coloneqq \frac{\partial}{\partial r} u+\frac 1 r Pu+A_ru\in L^2([0,1]\times\link,\Bigwedge^*\link)$$
		where  $h$ is defined in the weak sense, then $u$ lies in the (maximal) domain of $$\frac{\partial}{\partial r}+\frac 1 r P.$$
		
		Let $\{\psi_j \}_{j\in\N^+}$ be a partition of unity subordinate to  the open cover 
		\[ \{(2^{-j-1},2^{-j+1}) \}_{j\in\N^+}\] of $(0,1)$. Note that each $\psi_j$ is supported away from the origin. From the condition (1), we see that $\psi_ju$ lies in $H^1((2^{-j-1},2^{-j+1})\times \link,\Bigwedge^*\link)$. Therefore, for any $\varepsilon>0$, there exists a smooth section $u_j$ supported on $(2^{-j-1},2^{-j+1})\times \link$ such that
		$$\|\psi_ju - u_j\|_{\mathcal D}\leq \frac{\varepsilon}{2^j},$$
		where $\|\cdot\|_\mathcal D$ is the graph norm of $\mathcal D$. Since the partition of unity is locally finite, the summation $\sum_{j=1}^\infty u_j$ is well-defined and furthermore we have
		$$\big\|\sum_{j=1}^\infty u_j-u\big\|_{\mathcal D}\leq \varepsilon.$$
		Therefore, by replacing $u$ with $\sum_{j=1}^\infty u_j$, we may assume that $u$ is  smooth on $(0,1]\times\link$. Note that in general  $u$ is not smooth  at $\{0\}\times \link$, and we do not have much control of how $u$ behaves at $\{0\}\times \link$.  In any case,  by multiplying a smooth cut-off function, we assume without loss of generality that $u$ is supported on $[0,\delta)\times\link$ for some sufficiently small $\delta>0$.

		Since $P_{B}$ is essentially self-adjoint, there is an orthonormal basis $\{\varphi_\lambda \}_{\lambda\in\Lambda}$ of $L^2(\link,\Bigwedge^*\link)$ such that $\varphi_\lambda\in H^1(\link,\Bigwedge^*\link;B)$ and $P\varphi_{\lambda}=\lambda\varphi_\lambda$.  	Note that  we have $|\lambda|\geq1/2$ by assumption.
		Since $P^{-1}A_r$ is bounded and its operator norm is uniformly bounded (independent of $r$), there exists $C>0$ such that 
		$$\|A_r^\ast \varphi_\lambda\|\le C|\lambda|$$
		for all $\varphi_\lambda$. 
		
		Denote by $u(r)$ (resp. $h(r)$) the restriction of $u$ (resp. $h$) on $\{r\}\times \link$. We define
		\[ g(r) \coloneqq h(r) - A_ru(r). \]
		Let us set 
		$$u_\lambda(r)\coloneqq \langle u(r), \varphi_\lambda\rangle,  \hspace{.2cm} h_\lambda(r)\coloneqq \langle h(r),\varphi_\lambda\rangle \textup{ and } g_\lambda(r)\coloneqq \big\langle g(r),\varphi_\lambda\big\rangle. $$
		Since both $u_\lambda$ and $h_\lambda$ are in $L^2[0,1]$, it follows that $g_\lambda \in L^2[0,1]$. Moreover, by construction, we have 
		$$g_\lambda=\frac{d}{dr}u_\lambda+\frac\lambda r u_\lambda$$
		in the weak sense. In particular, we have 
		$$u_\lambda(r)=  \int_1^r \Big(\frac t r\Big)^\lambda g_\lambda (t)dt, $$
		since $u_\lambda$ vanishes if $r>\delta$. In the case where  $\lambda\geq 1/2$, we also have
		$$u_\lambda(r)=\int_0^r \Big(\frac t r\Big)^\lambda  g_\lambda(t)dt.$$
		This leads us to define the following operator: 	$$T_\lambda(\xi)=\begin{cases}
			\displaystyle \int_1^r \Big(\frac t r\Big)^\lambda \xi (t)dt & \textup{ if } \lambda\leq -1/2,  \vspace{.4cm}\\
			\displaystyle\int_0^r \Big(\frac t r\Big)^\lambda \xi(t)dt & \textup{ if } \lambda\geq 1/2.
		\end{cases} \textup{ for all $\xi\in L^2[0,\delta]$}. $$
		Moreover, we define 
		$$T(\zeta)\coloneqq \sum_{\lambda\in\Lambda}T_\lambda(\langle \zeta,\varphi_\lambda\rangle)\cdot\varphi_\lambda \textup{ for all $\zeta\in L^2([0,\delta]\times \link,\Bigwedge^*\link )$}. $$
		By the same proofs of  \cite[Lemma 2.2]{BruningSeeley} and  \cite[Lemma 2.3]{BruningSeeley}, we see that  $T$ maps $L^2([0,\delta]\times \link,\Bigwedge^*\link)$ to the domain of $\frac{\partial}{\partial r}+\frac 1 r P$. Strictly speaking, \cite[Lemma 2.2]{BruningSeeley} and  \cite[Lemma 2.3]{BruningSeeley} were only stated for the case where the link $\link$ is a closed manifold. However, since the proofs of \cite[Lemma 2.2]{BruningSeeley} and  \cite[Lemma 2.3]{BruningSeeley} were carried out by  spectral computation, the same proofs also apply in our current setting where $\link$ is a manifold with polyhedral boundary. More precisely, an element $v\in L^2([0, \delta]\times \link, \Bigwedge^*\link)$ lies in the domain of $\frac{\partial}{\partial r}+\frac 1 r P$ if
		$$\sum_\lambda \big\| \frac{\lambda}{r} \langle v,\varphi_\lambda\rangle \big\|^2_{L^2[0,1]}<\infty,\text{ and }\sum_\lambda \big\|\frac{d}{dr}\langle v,\varphi_\lambda\rangle\big\|^2_{L^2[0,1]}<\infty.$$
		In particular,  whether an element $v$ lies in the domain of $\frac{\partial}{\partial r}+\frac 1 r P$  is determined by  $\langle v ,\varphi_\lambda\rangle$ as  functions over $[0, 1]$. Therefore we can apply the same proofs of \cite[Lemma 2.2 \& Lemma 2.3]{BruningSeeley} and Schur's test to show that $T$ maps $L^2([0,\delta]\times \link,\Bigwedge^*\link)$ to the domain of  $\frac{\partial}{\partial r}+\frac 1 r P$.
		
		Let $\mathcal A$ be the (unbounded) operator on $L^2([0,\delta]\times \link,\Bigwedge^*\link)$ defined by setting
		\[  \mathcal A(\xi)_r = A_r(\xi_r)  \] 
		for $\xi$ in the domain of $\frac{\partial}{\partial r}+\frac 1 r P$,  where $\xi_r$ (resp. $\mathcal A(\xi)_r$)  is the restriction of $\xi$  (resp. $\mathcal A(\xi)$)  on $\link_r = \{r\}\times \link$. A straightforward computation shows that  $\mathcal AT$ and $T\mathcal A$ are bounded operators such that  the operator norms $\|\mathcal A T\|$ and $\|T \mathcal A\|$ are $\leq C(\delta)$, where $C(\delta)$ is a positive number that goes to zero as $\delta\to0$,  cf.  \cite[Lemma 2.2]{BruningSeeley}.
		
		Since by construction $h_\lambda(r)=g_\lambda(r)+\langle A_r u(r),\varphi_\lambda\rangle$ and $Tg_\lambda = u_\lambda$, we see that 
		$$Th=u+T \mathcal A u.$$
		By choosing a sufficiently small $\delta$, we can assume without loss of generality that $C(\delta)<1$. It follows that  $(1+T \mathcal A)$ is invertible  bounded operator in this case and 
		$$(1+T \mathcal A)^{-1}=\sum_{j=0}^\infty (-1)^j (T \mathcal A)^j,$$
		which implies that 
		$$u=(1+T \mathcal A)^{-1}Th=\sum_{j=0}^\infty (-1)^j (T \mathcal A)^jTh=T\Big(\sum_{j=0}^\infty (-1)^j ( \mathcal AT)^jh\Big)$$
		Note that $\sum_{j=0}^\infty (-1)^j ( \mathcal AT)^jh$ lies in $L^2([0,\delta]\times \link,\Bigwedge^*\link )$,  since $h$ does. We have already observed that $T$ maps $L^2([0,\delta]\times \link,\Bigwedge^*\link)$ to the domain of $\frac{\partial}{\partial r}+\frac 1 r P$.   Therefore, we conclude that $u$ lies in the maximal domain of $\frac{\partial}{\partial r}+\frac 1 r P$. To summarize, we have shown that the maximal domain of $\mathcal D$ is contained in the maximal domain of $\slashed D$.   This finishes the proof.	
	\end{proof}

	Now we prove a few key lemmas that will allow us to reduce the verification of the essential self-adjointness of $D_B$ in the general case to a standard product case. In order to deal with a smooth family of fiberwise Dirac-type operators, a first step is to find a somewhat natural way to identify the $L^2$ spaces of sections along different fibers while keeping track of the boundary condition on each fiber. One of the main nuances is that the dihedral angles may vary from fiber to fiber, and futhermore the dihedral angles may also vary along a codimension two face within the same fiber. 
	The follow lemma provides  a key ingredient that circumvents such a nuance and enables us to reduce the computation to a standard product case.  
	
	\begin{lemma}\label{lemma:smoothEquivalence}
		Let $\fiber $ and $\mathbb G$ be two sectors in $\R^2$. Let $B$ be the boundary condition on $\Bigwedge^*\fiber = \Bigwedge^\ast \mathbb R^2$ over each edge given by
		$$\mathscr E(\overbar c(u_i)\otimes c(v_i))\varphi=-\varphi$$
		for $i=1,2$, where $\mathscr E$ is the even-odd grading operator on $\Bigwedge^\ast \mathbb R^2$, $u_i$'s are the inner normal vectors of $\fiber $ and $v_i$'s are the inner normal vectors of $\mathbb G $, cf. Definition \ref{def:boundarycondition}. Similarly, let  $\fiber' $ and $\mathbb G'$ be two sectors in $\R^2$, and $B'$  the corresponding boundary condition on $\Bigwedge^*\fiber' = \Bigwedge^\ast \mathbb R^2$ as given in Definition \ref{def:boundarycondition}. Suppose $\alpha, \beta, \alpha'$ and $\beta'$ are the dihedral angels of $\fiber,  \mathbb G,  \fiber'$ and $\mathbb G'$ respectively. If   $0<\alpha<\beta<\pi$ and $0<\alpha'<\beta'<\pi$, 
		then there exists an invertible linear map $T\colon\Bigwedge^\ast \R^2\to \Bigwedge^\ast \R^2$ that maps $B$ to $B'$ and preserves the even-odd grading on $\Bigwedge^\ast \R^2$. Moreover, as a matrix-valued function of the parameters $(\alpha,\beta,\alpha',\beta')$, $T$  is smooth as we vary  $\alpha,\beta,\alpha'$ and $\beta'$, as long as  the condition $0<\alpha<\beta<\pi$ and $0<\alpha'<\beta'<\pi$ are satisfied throughout the variation of parameters.
	\end{lemma}
	\begin{proof}	
		Let $\{e_1,e_2\}$ be the standard orthonormal basis of $\R^2$. We identify $\Bigwedge^\ast \R^2$ over $\fiber$ and $\fiber'$ with the trivial bundle spanned by $1,e_1\wedge e_2,e_1,e_2$.  	Without loss of generality, we assume that $\fiber$ and $\fiber'$ both lie in $\{ (x,y):y\geq 0\}$ and have a common edge $\overbar F_1$ given by $\{(x,y):y=0,x\geq 0 \}$. Moreover, by rotating $\mathbb G$ and $\mathbb G'$ if necessary,  we may assume both the boundary conditions $B$ and $B'$ on this edge are given by
		$$\mathscr E(\overbar c(e_2)\otimes c(e_2))\varphi=-\varphi,$$
		where $\mathscr E$ is the even-odd grading operator on $\Bigwedge^\ast \mathbb R^2$. Equivalently, this boundary condition says that the differential form $\varphi$ lies in the linear span of $1$ and $e_1$ along the edge $\overbar F_1$.

		The boundary condition $B$ on the other edge $\overbar F_2$ of $\fiber$ is given by
		$$\mathscr E(\overbar c(u_\alpha)\otimes c(u_\beta))\varphi=-\varphi,$$
		where $u_\alpha=(\sin\alpha) e_1-(\cos\alpha) e_2$ and $u_\beta=(\sin\beta) e_1-(\cos\beta) e_2$. A direct computation shows that the linear map $\mathscr E(\overbar c(u_\alpha)\otimes c(u_\beta))\colon \Bigwedge^\ast \R^2 \to \Bigwedge^\ast \R^2 $ corresponds to the following matrix 
		$$\begin{pmatrix}
			-\cos(\alpha-\beta)&-\sin(\alpha-\beta)&0&0\\
			-\sin(\alpha-\beta)&\cos(\alpha-\beta)&0&0\\
			0&0&-\cos(\alpha+\beta)&-\sin(\alpha+\beta)\\
			0&0&-\sin(\alpha+\beta)&\cos(\alpha+\beta)
		\end{pmatrix}$$
		with respect to the basis $\{1,e_1\wedge e_2,e_1,e_2\}$ of $\Bigwedge^\ast \R^2$. 	Therefore $\varphi$ satisfies the boundary condition  $B$ on the edge $\overbar F_2$ if and only if $\varphi|_{\overbar F_2}$ lies in the linear span  
		$$\text{span}\{(1+\cos(\alpha-\beta))+\sin(\alpha-\beta)e_1\wedge e_2,~
		\sin(\alpha+\beta)e_1+(1-\cos(\alpha+\beta))e_2 \}.$$ The vectors $w_1 = 1,  w_2 = (1+\cos(\alpha-\beta))+\sin(\alpha-\beta)e_1\wedge e_2, w_3 = e_1$ and $w_4 =  \sin(\alpha+\beta)e_1+(1-\cos(\alpha+\beta))e_2$  are  linearly independent in $\Bigwedge^\ast \R^2$, since $0<\alpha<\beta<\pi$.

		Similarly, $\varphi$ satisfies the boundary condition $B'$ on the other edge $\overbar F'_2$ of $\fiber'$ if and only if
		$\varphi|_{\overbar F'_2}$ lies in the linear span  $$\text{span}\{(1+\cos(\alpha'-\beta'))+\sin(\alpha'-\beta')e_1\wedge e_2,~
		\sin(\alpha'+\beta')e_1+(1-\cos(\alpha'+\beta'))e_2 \}.$$ The vectors $w'_1 = 1,  w'_2 = (1+\cos(\alpha'-\beta'))+\sin(\alpha'-\beta')e_1\wedge e_2, w'_3 = e_1$ and $w'_4 =  \sin(\alpha'+\beta')e_1+(1-\cos(\alpha'+\beta'))e_2$  are  linearly independent in $\Bigwedge^\ast \R^2$, since $0<\alpha'<\beta'<\pi$. 
		
		If we define the linear map $T\colon\Bigwedge^\ast \R^2\to \Bigwedge^\ast \R^2$ by setting $T(w_j) = w'_j$, then $T$ is an invertible linear map that preserves the even-odd grading of $\Bigwedge^\ast \R^2$. It is clear from the explicit expressions of $w_j$ and $w'_j$ above that $T$ is smooth with respect to the parameters $\alpha,\beta,\alpha'$ and $\beta'$, as long as  the condition $0<\alpha<\beta<\pi$ and $0<\alpha'<\beta'<\pi$ are satisfied throughout the variation of these parameters.
	\end{proof}

	\begin{remark}\label{remark:smoothEqual}
		The same proof of Lemma \ref{lemma:smoothEquivalence} also applies to the case where  $0<\alpha=\beta<\pi$ and $0<\alpha'=\beta'<\pi$. In this case, the boundary condition $B$  at the edge $\overbar F_2$ of $\fiber$  is spanned by $$w_1 = 1,  w_2 = e_1\wedge e_2, w_3 = e_1, w_4 =  \sin(2\alpha)e_1+(1-\cos(2\alpha))e_2,$$
		and the boundary condition $B'$  at the edge $\overbar F'_2$ of $\fiber'$
		is spanned by $$w_1' = 1,  w_2' = e_1\wedge e_2, w_3' = e_1, w_4' =  \sin(2\alpha')e_1+(1-\cos(2\alpha'))e_2.$$
We define $T\colon \Bigwedge^\ast \R^2\to \Bigwedge^\ast \R^2$ to be the linear map that  maps $w_i$ to $w_i'$ for $i=1,2,3,4$. 
		We emphasize that however such a linear map $T\colon\Bigwedge^\ast \R^2\to \Bigwedge^\ast \R^2$  does \emph{not} exist if $\alpha<\beta$ but $\alpha'=\beta'$. In other words, we can vary smoothly the dihedral angles $\alpha$ and $\beta$ as we wish,   as long as either the  inequality $0<\alpha < \beta<\pi$ is satisfied throughout the variation, or   the inequality $0<\alpha = \beta<\pi$ is satisfied throughout the variation. 
	\end{remark}

	We also need some technical lemmas concerning the boundedness of  certain multiplication operators  on Sobolev spaces. In particular, these lemmas will be useful when proving the essential self-adjointness of $D_B$ near faces of codimension $\geq 3$. 		
			\begin{lemma}\label{lemma:sobolevEmbedding}
				Let $n\geq 3$. Suppose $r$ is the radial variable of $\R^n$. Then multiplying by $r^{-1}$ defines a bounded linear operator
				$$H^1(\R^n)\to L^2(\R^n),~\varphi\mapsto r^{-1} \varphi,$$
				where $H^1(\R^n)$ is the Sobolev $H^1$-space on $\R^n$. 
			\end{lemma}
			\begin{proof}
				We parameterize $\R^n$ by $(\theta,r)$ with $\theta\in \mathbb S^{n-1}$. Let $C^\infty_{c, 0}(\R^n) $ be  the subspace of compactly supported smooth functions on $\mathbb R^n$ that vanish at the origin. We know that $C^\infty_{c, 0}(\R^n) $ is dense in $H^1(\R^n)$. Hence to prove the lemma, it suffices to show that there exists $C>0$ such that 
				$$\big\|\frac 1 r\varphi\big\|_{L^2}\leq C\cdot\|\varphi\|_{H^1}$$
				for all $\varphi \in C^\infty_{c, 0}(\R^n) $.
				
				In polar coordinates, we have 
				\[ \big\|\frac 1 r\varphi\big\|_{L^2} = \Big(\int_{\mathbb S^{n-1}}\int_0^{+\infty}\varphi^2(\theta,r) r^{n-3}drd\theta\Big)^{1/2} \] For any fixed $\theta\in \mathbb S^{n-1}$, integration by parts implies that 
				\begin{align*}
					\int_0^{+\infty}\varphi^2(\theta,r) r^{n-3}dr=&-\frac{2}{n-2}\int_0^{+\infty}\varphi(\theta,r)\frac{\partial}{\partial r}\varphi(\theta,r)\cdot r^{n-2}dr\\
					=&-\frac{2}{n-2}\int_0^{+\infty}r^{(n-3)/2}\varphi(\theta,r)\cdot r^{(n-1)/2}\frac{\partial}{\partial r}\varphi(\theta,r)dr\\
					\leq &\frac{2}{(n-2)}\Big(\int_0^{+\infty}\varphi^2(\theta,r) r^{n-3}dr\Big)^{\frac 1 2}\Big(\int_0^{+\infty}\Big[\frac{\partial}{\partial r}\varphi(\theta,r)\Big]^2 r^{n-1}dr\Big)^{\frac 1 2}.
				\end{align*}
				It follows that 
				$$\int_0^{+\infty}\varphi^2(\theta,r) r^{n-3}dr\leq \frac{4}{(n-2)^2}\int_0^{+\infty}\Big[\frac{\partial}{\partial r}\varphi(\theta,r)\Big]^2 r^{n-1}dr$$
				Now by integrating both sides of the above inequality  over $\mathbb S^{n-1}$, we obtain that
				$$\big\|\frac 1 r\varphi\big\|^2_{L^2}\leq  \frac{4}{(n-2)^2}\int_{\mathbb S^{n-1}}\int_0^{+\infty}\Big[\frac{\partial}{\partial r}\varphi(\theta,r)\Big]^2 r^{n-1}dr d\theta. $$
				By the definition of Sobolev $H^1$-norm, there exists a constant $C>0$ such that 
				\[\frac{4}{(n-2)^2}\int_{\mathbb S^{n-1}}\int_0^{+\infty}\Big[\frac{\partial}{\partial r}\varphi(\theta,r)\Big]^2 r^{n-1}dr d\theta \leq  C \|\varphi\|^2_{H^1} \]
				for all $\varphi \in C^\infty_{c, 0}(\R^n) $. This finishes the proof. 
			\end{proof}

We will also need the following iterated version of Lemma \ref{lemma:sobolevEmbedding}.
\begin{lemma}\label{lemma:iteratedSobolev}
	Let $n\geq 3$. Set
	$$\Sigma_k=\{(x_1,x_2,\ldots,x_n)\in\R^n:x_1=x_2=\cdots=x_k=0\}.$$
	Then multiplying by the function 
	\[ \prod_{k=3}^n\frac{1}{\dist(x,\Sigma_k)}  \]
	defines a bounded linear operator from $H^1(\R^n)$ to $L^2(\R^n)$, where $\dist(x,\Sigma_k)$ is the distance from $x$ to $\Sigma_k$. 
\end{lemma}
\begin{proof}
	Let $\mathcal M_k$ stand for multiplying by the function 
	\[ \frac{1}{\dist(x,\Sigma_k)}. \]  Let $\Bigwedge^*\R^n$ be the trivial bundle of differential forms over $\R^n$. It suffices to show the operator $\mathcal M_n\mathcal M_{n-1}\cdots \mathcal M_{3}$ defines a bounded linear operator from $H^1(\R^n,\Bigwedge^*\R^n)$ to $L^2(\R^n,\Bigwedge^*\R^n)$. The reason for working with the space of differential forms instead of functions is only a matter of convenience. 
	 
	 We will prove the lemma by induction on $n$. If $n=3$, it follows immediately from  Lemma \ref{lemma:sobolevEmbedding} that $\mathcal M_3\colon H^1(\R^3,\Bigwedge^*\R^3) \to L^2(\R^3,\Bigwedge^*\R^3)$ is bounded.  For a general $n$, let $D^\dR$ be the standard de Rham operator $d+d^*$ acting on $\Bigwedge^*\R^n$. Recall that there is isometry
	 	\begin{equation*}
	 	\Psi=(\Psi_{\mathrm{even}},\Psi_{\mathrm{odd}})\colon 
	 	C^\infty((0,\infty),\Omega^*\mathbb S^{n-1})\oplus C^\infty((0,\infty),\Omega^*\mathbb S^{n-1})
	 	\longrightarrow \Omega^\ast(\R^n),
	 \end{equation*}
	 where  
	 \begin{equation*}
	 	\Psi_{\mathrm{even}}\colon C^\infty((0,\infty),\Omega^*\mathbb S^{n-1})\longrightarrow \Omega^{\mathrm{even}} \R^n,~
	 	\omega_p\longmapsto\begin{cases}
	 		r^{p-\frac{n-1}{2}}\omega_p,&\text{ if $p$ is even}\\
	 		r^{p-\frac{n-1}{2}}\omega_p\wedge dr,&\text{ if $p$ is odd}
	 	\end{cases}
	 \end{equation*}
 	and
 \begin{equation*}
 	\Psi_{\mathrm{odd}}\colon C^\infty((0,\infty),\Omega^*\mathbb S^{n-1})\longrightarrow \Omega^{\mathrm{odd}} \R^n,~
 	\omega_p\longmapsto\begin{cases}
 		r^{p-\frac{n-1}{2}}\omega_p,&\text{ if $p$ is odd}\\
 		r^{p-\frac{n-1}{2}}\omega_p\wedge dr,&\text{ if $p$ is even}
 	\end{cases}
 \end{equation*}
	 as in line \eqref{eq:coordinateChange}, where $\mathbb S^{n-1}$ is the unit $(n-1)$-sphere. After conjugating with $\Psi$,  the de Rham operator  becomes 
	 \begin{equation}\label{eq:DdR}
	 		\Psi^*D^{\dR}\Psi=\begin{pmatrix}
	 	0&-\frac{\partial}{\partial r}\\
	 	\frac{\partial}{\partial r}&0
	 \end{pmatrix}+\frac 1 r\begin{pmatrix}
	 	0&P\\P&0
	 \end{pmatrix}.
	 \end{equation}
	 as in line \eqref{eq:deRhamAfterConj}, where $P$ is equal to the de Rham operator on $\mathbb S^{n-1}$ up to a bounded endomorphism (cf. line \eqref{eq:linkOperator}).
	 
	Since $C_c^\infty(\R^n-\{0\},\Bigwedge^*\R^n)$ is dense in $H^1$, it suffices to show that there is $C>0$ such that
	 $$\|\mathcal M_n\mathcal M_{n-1}\cdots \mathcal M_{3}\varphi\|_{L^2}\leq C\|\varphi\|_{H^1}$$
	 for all $\varphi\in C_c^\infty(\R^n-\{0\},\Bigwedge^*\R^n)$.  
	 
	 We first show that there is $C_1>0$ such that
	 \begin{equation}\label{eq:1/rP}
	 	\Big\|\frac 1 r\begin{pmatrix}
	 		0&P\\P&0
	 	\end{pmatrix}\Psi^*\varphi\Big\|_{L^2}
	 	\leq C_1\|\varphi\|_{H^1}.
	 \end{equation} 
	By definition of the Sobolev $H^1$-space, the $L^2$-norm of $\frac{\partial}{\partial r}\varphi$ is bounded by the $H^1$-norm of $\varphi$. Note that 
	 $$\Psi^*\frac{\partial}{\partial r}\Psi(\omega_p)=\frac{\partial}{\partial r}\omega_p+\frac{1}{r}(p-\frac{n-1}{2})\omega_p,$$
	 for any degree $p$ form  $\omega_p$ in $C^\infty((0,\infty),\Omega^*\mathbb S^{n-1})$.
	  Since multiplication by $1/r$ defines a bounded linear map from $H^1$ to $L^2$ by Lemma \ref{lemma:sobolevEmbedding}, once we write  $\Psi^\ast \varphi$ as a sum of differential forms of homogeneous degrees and apply the above equation, we see that  there is $C_2>0$ such that
	 \begin{equation}\label{eq:d/dr}
	 	 \|\frac{\partial}{\partial r}(\Psi^\ast \varphi)\|_{L^2}\leq C_2\|\varphi\|_{H^1},
	 \end{equation}
	   As the $L^2$-norm $D^\dR\varphi$ is bounded by the $H^1$-norm of $\varphi$, we obtain \eqref{eq:1/rP} from line \eqref{eq:DdR} and \eqref{eq:d/dr}. By definition of the $L^2$-norm in the cylindrical coordinates, we have	   
	 \begin{equation}\label{eq:1/rPP}
		\int_0^\infty \frac{1}{r^2}\big\|\begin{psmallmatrix}
			0&P\\P&0
		\end{psmallmatrix}(\Psi^*\varphi (r))\big\|^2_{L^2}\, dr = \Big\|\frac 1 r\begin{pmatrix}
	 	0&P\\P&0
	 \end{pmatrix}\Psi^*\varphi\Big\|_{L^2}^2\leq C_1^2\|\varphi\|_{H^1}^2
	 \end{equation}

	 Let $\mathbb S^{n-1}_k=\mathbb S^{n-1}\cap \Sigma_k$ for $k<n$. Then the operator $\Psi^*\mathcal M_k\Psi$ acts on the space $C^\infty((0,\infty),\Omega^*\mathbb S^{n-1}) \oplus C^\infty((0,\infty),\Omega^*\mathbb S^{n-1})$ as multiplying by the function\footnote{Strictly speaking, $\dist(x,\mathbb S^{n-1}_k)$ is the distance measured in $\R^n$ instead of $\mathbb S^{n-1}$. But the function $\dist_{\mathbb S^{n-1}}(- ,\mathbb S^{n-1}_k)$ is Lipschitz equivalent to $\dist(- ,\mathbb S^{n-1}_k)$. Hence it suffices to work with $\dist(x,\mathbb S^{n-1}_k)$ instead of $\dist_{\mathbb S^{n-1}}(x,\mathbb S^{n-1}_k)$ in the above induction argument. }  $1/\dist(x,\mathbb S^{n-1}_k)$.  For brevity, we will still denote $\Psi^*\mathcal M_k\Psi$ by $\widehat {\mathcal M}_k$. By the induction hypothesis, there is $C_3>0$ such that 
	 $$\|\mathcal M_{n-1}\cdots \mathcal M_{3}\psi\|_{L^2} \leq C_3\|\psi\|_{H^1}$$
	 for all $\psi\in H^1(\mathbb S^{n-1},\Bigwedge^*\mathbb S^{n-1}). $
	 By the   G\r{a}rding's inequality  for the operator $P$ on $\mathbb S^{n-1}$  (cf. Proposition \ref{prop:norm-equivalent}),  there is $C_4>0$ such that
	 $$\big\|\psi\big\|_{H^1}^2\leq C_4(\big\|\psi\big\|^2_{L^2}+\big\|P\psi\big\|^2_{L^2})$$
	 for all $\psi\in H^1(\mathbb S^{n-1},\Bigwedge^*\mathbb S^{n-1})$. It follows that there exists $C>0$ such that 
	 \begin{align*}
	 	&\big\|\mathcal M_n\mathcal M_{n-1}\cdots \mathcal M_{3}\varphi\big\|_{L^2}^2=\Big\|\frac 1 r\widehat{\mathcal M}_{n-1}\cdots \widehat{\mathcal M}_{3}(\Psi^*\varphi)\Big\|_{L^2}^2\\
	 	=&\int_0^\infty \frac{1}{r^2}\big\|\widehat{\mathcal M}_{n-1}\cdots \widehat{\mathcal M}_{3}(\Psi^*\varphi(r))\big\|_{L^2}^2\, dr\\
	 	\leq &C^2_3C_4\int_0^\infty \Big(\frac{1}{r^2} \big\|\Psi^*\varphi(r)\big\|^2_{L^2}+\frac{1}{r^2}\big\|\begin{psmallmatrix}
	 		0&P\\P&0
	 	\end{psmallmatrix}(\Psi^*\varphi(r))\big\|^2_{L^2}\Big)dr\\
= &C^2_3C_4\int_0^\infty \Big(\frac{1}{r^2}\big\|\varphi(r)\big\|^2_{L^2}+\frac{1}{r^2}\|\begin{psmallmatrix}
			0&P\\P&0
		\end{psmallmatrix}(\Psi^*\varphi(r))\|^2_{L^2}\Big)dr\\
	 	\leq & C\|\varphi\|^2_{H^1},
	 \end{align*}
for all $\varphi\in C_c^\infty(\R^n-\{0\},\Bigwedge^*\R^n)$,   where the last line follows from  Lemma \ref{lemma:sobolevEmbedding} and line \eqref{eq:1/rPP}. This finishes the proof.
\end{proof}

			Now we are ready to prove Theorem \ref{thm:ess-sa}: the essential self-adjointness of the twisted Dirac operator $D_B$ acting on $S_N\otimes f^\ast S_M$ under appropriate comparison conditions on dihedral angles. 
			
\begin{proof}[Proof of Theorem \ref{thm:ess-sa}]
	By Proposition \ref{prop:norm-equivalent}, $D$ is a closable symmetric operator and its minimal domain (i.e., the domain of its closure) is  $H^1(N,S_N\otimes f^*S_M;B)$. By the discussion at the beginning of Section \ref{sec:conic}, to determine whether $D_B$ is essentially self-adjoint, we can localize our computation on small neighborhoods of points in the interiors of faces of codimensions $\geq 1$. Our general strategy is to prove the theorem by induction on the codimensions of singular strata. 
				
   The essential self-adjointness of $D_B$ near codimension one  but away codimension two faces is classical and standard. So let us only consider the essential self-adjointness of $D_B$ near faces of codimension two or higher. 
				
	Suppose $x$ is a point in the interior of a codimension two face $\overbar F$ of $N$. Near $x$, by applying a smooth cut-off function, we may localize the verification of the essential-adjointness of $D_B$ in a sufficiently small neighborhood of $x$. By the definition of manifolds with polyhedral boundary, the tangent cone at $x$ is $\R^{n-2} \times \fiber$, where $n=\dim N$ and $\fiber$ is a flat Euclidean polyhedral corner in $\R^2$, that is,  $\fiber$ is the region enclosed by two straight rays passing through the origin of $\R^2$. Similarly, the tangent cone at $f(x)$ is $\R^{m-2} \times \mathbb G$, where $m=\dim M$ and $\mathbb G$ is a flat Euclidean polyhedral corner in $\R^2$. Let $W$ be a small neighborhood of $x$ in the codimension two face $\overbar F$. Let $\mathcal N(W)$ be the normal bundle of $W$ in $N$.  We denote by $\mathcal N_\delta(W)$   the subspace of $\mathcal N(W)$ where each fiber is the $\delta$-ball of the normal vector space. The image of $\mathcal N_\delta(W)$ under the exponential map in $N$ gives a small neighborhood of $x$ in $N$, which we denote by $Z$.   Note that $Z$ inherits a natural fiber bundle structure from that of the normal bundle $\mathcal N(W)$, where each fiber of the projection map $Z\to W$ is a two dimensional asymptotically Euclidean  polyhedral corner. Let  $T_vZ \subset TZ$ be the vertical subbundle that  consists of tangent vectors that are tangential to the fiber of the fiber bundle $Z\to W$. As pointed out in Remark \ref{remark:innernormal}, the inner normal vectors of the codimension one faces of $Z$ may not lie in the vertical bundle $T_vZ$ in general. On the other hand, at each point $y$ in the base space  $W$,  the inner normal vectors of the codimension one faces of $Z$  at $y$ actually lie in the vertical bundle $(T_vZ)_y$ at $y$. In particular, we may choose a smooth two dimensional smooth subbundle $\mathcal V \subset TZ$ such that 
	\begin{enumerate}[label=(2-\alph*)]
		\item at any point $z$ in a codimension one face of $Z$, the vector space $\mathcal V_z$ at $z$ contains the unit inner normal vector of that codimension one face at $z$, and 
		\item $\mathcal V$ is isomorphic to $T_vZ$ via a smooth bundle isometry that is asymptotically the identity map as $r \to 0$, where $r$ is the radial variable of each fiber of the fiber bundle $Z\to W$.  In particular, $\mathcal V$ coincides with $T_vZ$ over $W = W\times \{0\}\subset Z$.
	\end{enumerate}	
	By the definition of polytope maps (cf. Definition \ref{def:polytopeMap}), $f(x)$ is also an interior point of a codimension two face $F$ of $M$. Let $U$ be a small neighborhood of $f(x)$ in $F$. Similarly, the normal bundle $\mathcal N(U)$ of $U$ in $M$ induces a fiber bundle structure on a small neighborhood $Y$ of $f(x)$ in $M$. We may also choose a smooth two dimensional subbundle $\widehat{\mathcal V}\subset TY$ with properties that are analogous to those of $\mathcal V$ above. In particular, we have the following decomposition 
	\[  (S_N\otimes f^\ast S_M)|_Z = S_{\mathcal V\oplus f^\ast \widehat{\mathcal V} } \otimes S_{\mathcal V^\perp \oplus f^\ast (\widehat{\mathcal V})^\perp  }. \]

	Now consider the tangent cone $\mathbb R^{n-2}\times \fiber$ at $x$ and the tangent cone  $\mathbb R^{m-2}\times \mathbb G$ at $f(x)$ above.  The associated spinor bundle $S_{T\mathbb F \oplus T\mathbb G} \otimes S_{T\mathbb R^{n-2}\oplus f^\ast T\mathbb R^{m-2}}$ over $\mathbb R^{n-2}\times \fiber$ carries a natural local boundary condition $B_o$ determined by the unit inner normal vectors of codimension one faces of $\fiber$ and $\mathbb G$ (cf. Definition \ref{def:boundarycondition}). Recall that  the dihedral angles of $N$  and the corresponding dihedral angles of $M$ satisfy either the strict comparison inequality \eqref{eq:dihedralstrict} throughout the given codimension two face $\overbar F$, or  the equality  \eqref{eq:dihedralequal} throughout the given codimension two face $\overbar F$. It follows from Lemma \ref{lemma:smoothEquivalence} and Remark \ref{remark:smoothEqual} that  there exists a smooth invertible bundle map
 \begin{equation}\label{cd:bundlemap}
	\begin{CD}
	S_N\otimes f^\ast S_M = S_{\mathcal V\oplus f^\ast \widehat{\mathcal V} } \otimes       S_{\mathcal V^\perp \oplus f^\ast (\widehat{\mathcal V})^\perp  }  @>\varPhi>>  S_{T\mathbb F \oplus T\mathbb G} \otimes S_{T\mathbb R^{n-2}\oplus f^\ast T\mathbb R^{m-2}}  \\
		@VVV  @VVV \\
			Z @>\varPhi|_Z>>   \mathbb R^{n-2}\times \mathbb F
   \end{CD}
\end{equation}
 such that
 \begin{enumerate}[label=(2-\roman*)]
	\item $\varPhi|_Z$ is a diffeomorphism from $Z$ to its image in $\mathbb R^{n-2} \times \fiber$;  
	\item $\varPhi$ maps $S_{\mathcal V\oplus f^\ast \widehat{\mathcal V} }$ to $S_{T\mathbb F \oplus T\mathbb G}$ and maps $S_{\mathcal V^\perp \oplus f^\ast (\widehat{\mathcal V})^\perp  }$ to $S_{T\mathbb R^{n-2}\oplus f^\ast T\mathbb R^{m-2}}$;
	\item $\varPhi$ maps the local boundary condition $B$ on $S_N\otimes f^\ast S_M$ to the local boundary condition $B_o$ on $S_{T\mathbb F \oplus T\mathbb G} \otimes S_{T\mathbb R^{n-2}\oplus f^\ast T\mathbb R^{m-2}} $; 
	\item \label{item:approx} by choosing the neighborhood $Z$ of $x$ to be sufficiently small so that both the base space $W$ and the fiber are sufficiently small,\footnote{that is, if we choose $W$ and  the number $\delta$ in $\mathcal N_\delta(W)$ to be sufficiently small}  $\varPhi$ becomes arbitrarily close to the identity map. More precisely, let $\varepsilon_0$ be some fixed constant that is less than  the minimum of the injective radii of $N$ at $x$ and of $M$ at $f(x)$. Suppose $D$ (resp. $\slashed D$) is the associated Dirac operator on $Z$ (resp. $\mathbb R^{n-2}\times \fiber$) subject to the local boundary condition $B$ (resp. $B_o$). The operator $\varPhi D\varPhi^\ast$ is a first order differential operator defined on $\mathbb B_{\varepsilon_0} \times \fiber_\delta$ that is symmetric with respect to the boundary condition $B_o$, where $ \mathbb B_{\varepsilon_0}$ is the open ball of radius $\varepsilon_0$ centered at the origin of $\mathbb R^{n-2}$ and $\fiber_\delta$ is the intersection of $\fiber$ with the open ball of radius $\delta$ centered at the origin of $\mathbb R^2$. By rescaling the metrics on $N$ and $M$ by $1/\delta$, the tangent cones $\mathbb R^{n-2}\times \fiber$ at $x$ and  $\mathbb R^{m-2}\times \mathbb G$ at $f(x)$ before and after the rescaling  are canonically identical. Let $\varPhi_\delta$ be a smooth bundle map similar to the map $\varPhi$ above but with respect to the new rescaled metrics on $N$ and $M$.  Correspondingly, we denote by $D_\delta$  the associated Dirac operator on $Z$ subject to the local boundary condition $B$ under the new rescaled metrics on $N$ and $M$. When $\delta$ is sufficiently small, each operator $\varPhi_\delta D_\delta \varPhi^\ast_\delta$ is well defined on $\mathbb B_1\times \fiber_1$, where $ \mathbb B_1$ is the open ball of radius $1$ centered at the origin of $\mathbb R^{n-2}$ and $\fiber_1$ is the intersection of $\fiber$ with the open ball of radius $1$ centered at the origin of $\mathbb R^2$.  We may choose $\Phi_\delta$ so that  $\varPhi_\delta D_\delta \varPhi^\ast_\delta$ converges smoothly to $\slashed D$ on $\mathbb B_1\times \fiber_1$ as $\delta \to 0$ in the sense that the coefficients of $\varPhi_\delta D_\delta \varPhi^\ast_\delta$ (as a first order differential operator) converges smoothly and uniformly  over $\mathbb B_1\times \fiber_1$ to the coefficients of $\slashed D$.
 \end{enumerate}   
	 Now  it follows that  
 \[ \mathbf T_\delta = \varPhi_\delta D_\delta \varPhi^\ast_\delta  - \slashed D \] 
 becomes arbitrarily small relative to $\slashed D$ on $\mathbb B_1\times \fiber_1$ as $\delta \to 0$, that is, there exist numbers $a_\delta$ and $b_\delta$ (depending on $\delta$) such that $b_\delta \to 0$ as $\delta\to 0$ and 
 \[ \|\mathbf T_\delta (\varphi)\|^2 \leq a_\delta\|\varphi\|^2 + b_\delta\|\slashed D\varphi\|^2  \]
 for all $\varphi \in C_c^\infty(\mathbb B_1\times \fiber_1, E; B_o)$, where $C_c^\infty(\mathbb B_1\times \fiber_1, E; B_o)$ is the space of all smooth sections that are compactly supported in $B_1\times \fiber_1$ and satisfy the boundary condition $B_o$ at the codimension one faces of $\fiber_1$. It follows from that the Kato-Rellich perturbation theorem\footnote{For the convenience of the reader, we recall the Kato-Rellich perturbation theorem here. Suppose $\mathbf D$ is essentially self-adjoint acting on a Hilbert space $H$. Let $\mathbf T$ be a symmetric operator whose domain contains the domain $\dom(\mathbf D)$ of $\mathbf D$. If there exist numbers $a$ and $b$	with $b<1$ such that for all $u\in \dom(\mathbf D)$, 
	\[ \|\mathbf Tu\|^2\leq a \|u\|^2 + b\|\mathbf Du\|^2. \] Then $\mathbf D+ \mathbf T$ is also essentially self-adjoint. More precisely, let $\overbar {\mathbf D}$ and $\overbar {\mathbf T}$ be the closure of $\mathbf D$ and $\mathbf T$. Then under the above conditions, the domain of the closure $\overbar {\mathbf T}$ of $\mathbf T$ contains the domain $\dom (\overbar {\mathbf D})$ and $\overbar {\mathbf D} + \overbar {\mathbf T}$ is self-adjoint on $\dom (\overbar {\mathbf D})$. } that $\varPhi_\delta D_\delta \varPhi^\ast_\delta$ is (locally) essentially self-adjoint if $\slashed D$ is (locally) essentially self-adjoint. Since $\Phi_\delta$ is a smooth invertible bundle map (which is locally diffeomorphic on the base spaces), it follows that $D_\delta$ is (locally) essentially self-adjoint if and only if $\varPhi_\delta D_\delta \varPhi^\ast_\delta$ is (locally) essentially self-adjoint.\footnote{If $\varphi\in L^2(\mathbb B_1\times \mathbb F_1,E)$ and $D_\delta \varphi\in L^2(\mathbb B_1\times \mathbb F_1,E)$, then $(\varPhi_\delta^*)^{-1}\varphi$ and $\varPhi_\delta D_\delta \varPhi^*_\delta(\varPhi_\delta^*)^{-1}\varphi$ are also in $L^2(\mathbb B_1\times \mathbb F_1,E)$. Therefore, if $\varPhi_\delta D_\delta \varPhi^*_\delta$ is essentially self-adjoint, which implies that $(\varPhi_\delta^*)^{-1}\varphi\in H^1(\mathbb B_1\times \mathbb F_1,E;B_o)$, then $\varphi$ is also in the $H^1$-space. This shows that if $\varPhi_\delta D_\delta \varPhi^*_\delta$ is essentially self-adjoint, then $D_\delta$ is also essentially self-adjoint.} 
 As $D$ is just a rescaling of $D_\delta$,   we have reduced the verification of essentially self-adjointness of $D_B$ near faces of codimension two (but away from codimension three) to a standard product case $\slashed{D}_{B_o}$, where the latter is indeed essentially self-adjoint by Lemma \ref{lemma:essensa-jumpanglewithf}. As a consequence, we have verified that $D_B$ is essentially self-adjoint near faces of codimension two (but away from codimension three).  Moreover, since $\varPhi_\delta$ is smooth, it preserves the Sobolev $H^1$ spaces. It follows that the domain of $D_B$ (near faces of codimension two) is $H^1(N, S_N\otimes f^\ast S_M; B)$.

    The proof of the essentially self-adjointness of $D_B$ near faces of codimension $\geq 3$ is very similar to the codimension two case, except that the analogue of the bundle map $\varPhi$ (from line \eqref{cd:bundlemap}) will no longer be smooth in general. We will apply Lemma \ref{lemma:sobolevEmbedding} and Lemma \ref{lemma:iteratedSobolev} to take care of this extra technical difficulty. 

	We shall retain the same notation as the codimension two case. Suppose $x$ is a point in the interior of a codimension three face $\overbar F$ of $N$.   By the definition of manifolds with polyhedral boundary, the tangent cone at $x$ is $\R^{n-3} \times \fiber$, where $n=\dim N$ and $\fiber$ is a flat Euclidean polyhedral corner in $\R^3$. Similarly, the tangent cone at $f(x)$ is $\R^{m-3} \times \mathbb G$, where $m=\dim M$ and $\mathbb G$ is a flat Euclidean polyhedral corner in $\R^3$. Let $W$ be a small neighborhood of $x$ in the codimension three face $\overbar F$. Let $\mathcal N(W)$ be the normal bundle of $W$ in $N$.  We denote by $\mathcal N_\delta(W)$   the subspace of $\mathcal N(W)$ where each fiber is the $\varepsilon$-ball of the normal vector space. The image of $\mathcal N_\delta(W)$ under the expoential map in $N$ gives a small neighborhood of $x$ in $N$, which we denote by $Z$.   Note that $Z$ inherits a natural fiber bundle structure from that of the normal bundle $\mathcal N(W)$, where each fiber of the projection map $Z\to W$ is a three dimensional asymptotically Euclidean  polyhedral corner. Let  $T_vZ \subset TZ$ be the vertical subbundle that  consists of tangent vectors that are tangential to the fiber of the fiber bundle $Z\to W$. Similar to the codimension two case above,  we may choose a smooth three dimensional smooth subbundle $\mathcal V \subset TZ$ such that the following hold. 
	\begin{enumerate}[label=(3-\alph*)]
		\item At any point $z$ in a codimension one face of $Z$, the vector space $\mathcal V_z$ at $z$ contains the unit inner normal vector of that codimension one face at $z$. Note that if $z$ happens to lie in the interior of a codimension two face,  the vector space $\mathcal V_z$ at $z$ contains both the unit inner normal vectors of the codimension one faces meeting at $z$. Similarly, if $z$ lies in the interior of a codimension three face (which is $W$ in the current case),  the vector space $\mathcal V_z$ at $z$ contains  the unit inner normal vectors of all  codimension one faces meeting at $z$. 
		\item $\mathcal V$ is isomorphic to $T_vZ$ via a smooth bundle isometry that is asymptotically the identity map as $r \to 0$, where $r$ is the radial variable of each fiber of the fiber bundle $Z\to W$.  In particular, $\mathcal V$ coincides with $T_vZ$ over $W = W\times \{0\}\subset Z$. 
		\item Near any point $z$ in the interior of a codimension two face of $Z$, the bundle $\mathcal V$ itself admits a two dimensional smooth subbundle $\mathcal V_{(2)}$ that satisfies the conditions (2-a) and (2-b) as in the codimension two case.  
	\end{enumerate}	
    A similar construction applies to a small neighborhood $Y$ of $f(x)$ in $M$ to produce a fiber bundle structure of $Y$ over $U$ and  a smooth three dimensional subbundle $\widehat{\mathcal V} \subset TY$ with properties analogous to (3-a), (3-b) and (3-c) above, where $U$ is a small neighborhood of $f(x)$ in a codimension three face $F$ of $M$. In particular, near any point $y$ in the interior of a codimension two face of $Y$, the bundle $\widehat{\mathcal V} $ itself admits a two dimensional smooth subbundle $\widehat{\mathcal V}_{(2)}$ that satisfies the conditions (2-a) and (2-b) as in the codimension two case.  

	Let $S_r(W)_w$  be the sphere of radius $r$ in the normal vector space $\mathcal N(W)_w$ over $w\in W$.  Since $N$ is a manifold with polyhedral boundary, the exponential map $\exp_w \colon \mathcal N(W)_w \to N$ is only defined on a partial region of $\mathcal N(W)_w$. In particular, the image of $S_r(W)_w$ in $N$ under the exponential map $\exp_w$ is a (possibly curved) polygon (with possibly curved boundary). We denote by $Q_r(W)_w$ the preimage in the normal vector space  $\mathcal N(W)_w$ of this polygon with respect to the exponential map  $\exp_w$.   Now rescale the set $Q_r(W)_w \subset \mathcal N(W)_w$ by $r^{-1}$ and denote the resulting set by $r^{-1}Q_r(W)_w$. Suppose $\mathbb R^{n-3}\times \fiber_w$ is the tangent cone (as a subset of $T_wZ$) at $w$  and $\link(\fiber_w)$ is the link of $\fiber_w$ at radius $1$. Then 
	 $r^{-1}Q_r(W)_w$ converges to $\link(\fiber_w)$, as $r\to 0$. The collection $\{Q_r(W)_w\}_{w\in W}$ forms a fiber bundle over $W$, which will be denoted by $Q_r(W)$. The collection $\{\link(\fiber_w)\}_{w\in W}$ also forms a fiber bundle over $W$, which will be denoted by $\link(\fiber)(W)$. 
	 
	 Consider the tangent cones $\mathbb R^{n-3}\times \fiber$ at $x$. It follows from the above discussion that there exists a polytope map (cf. Definition \ref{def:polytopeMap}) $\psi\colon Z\to \mathbb R^{n-3}\times \fiber$ such that the following hold. 
	 \begin{enumerate}[label=($\alpha$\arabic*)]
		\item $\psi$ is a fiber bundle map fits into the following diagram of fiber bundles 
		\[ \begin{CD}
		   Z @>\psi>>   \mathbb R^{n-3}\times \fiber \\
		   @VVV  @VVV \\
		   W @>\psi|_W>>  \mathbb R^{n-3} 
	   \end{CD}\]
	   where $\psi|_W$ is a diffeomorphism\footnote{For example, choose $\psi|_W$ to be the inverse of the exponential map $\exp_x\colon T_xW = \mathbb R^{n-3} \to W$, as long as $W$ is sufficiently small.} from $W$ to its image in $\mathbb R^{n-3}$.  
	   \item  For each $r>0$, $\psi$ is a smooth diffeomorphism from $Q_r(W)$ to $\mathbb R^{n-3}\times \link(\fiber)_r$, where $\link(\fiber)_r$ is the link of $\fiber$ at radius $r$. Here we have identified $Z$ as a subspace of $\mathcal N(W)$ under the exponential map, hence there is no confusion to view $Q_r(W)$ as part of $Z$.  
	   \item $\psi$ is \emph{fiberwise asymptotically conical} in the following sense. For each $r$, we view  $\psi \colon Q_r(W) \to \mathbb R^{n-3}\times \link(\fiber)_r$ as a map 
		\[ \psi \colon  r^{-1} Q_r(W) \to \mathbb R^{n-3}\times r^{-1}\link(\fiber)_r = \mathbb R^{n-3}\times \link(\fiber) \]
	     which will be denoted by $\psi_r$. Recall that $r^{-1} Q_r(W)_w$ converges to $\link(\fiber_w)$ as $r\to 0$. We say $\psi$ is fiberwise asymptotically conical if $\psi_r$ uniformly converges to a smooth fiber bundle map $\link(\fiber)(W) \to \mathbb R^{n-3}\times \link(\fiber)$.  
 	 \end{enumerate}
	Similar constructions clearly also apply to the neighborhood $Y$ of $f(x)$ in $M$ and the tangent cone $\mathbb R^{m-3}\times \mathbb G$ at $f(x)$. 
	
    Consider the associated spinor bundle $S_{T\mathbb F \oplus T\mathbb G} \otimes S_{T\mathbb R^{n-3}\oplus f^\ast T\mathbb R^{m-3}}$ over $\mathbb R^{n-3}\times \fiber$ with a natural local boundary condition $B_o$ determined by the unit inner normal vectors of codimension one faces of $\fiber$ and $\mathbb G$ (cf. Definition \ref{def:boundarycondition}).  Recall that  the dihedral angles of $N$  and the corresponding dihedral angles of $M$ satisfy either the strict comparison inequality \eqref{eq:dihedralstrict} throughout each given codimension two face, or  the equality  \eqref{eq:dihedralequal} throughout each given codimension two face. It follows from Lemma \ref{lemma:smoothEquivalence}  and Remark \ref{remark:smoothEqual} that there exists an invertible bundle map 
	\begin{equation}\label{cd:bundlemap3}
		\begin{CD}
		S_N\otimes f^\ast S_M = S_{\mathcal V\oplus f^\ast \widehat{\mathcal V} } \otimes       S_{\mathcal V^\perp \oplus f^\ast (\widehat{\mathcal V})^\perp  }  @>\Psi>>  S_{T\fiber\oplus T\mathbb G} \otimes S_{T\mathbb R^{n-3}\oplus f^\ast T\mathbb R^{m-3}}  \\
			@VVV  @VVV \\
				Z- W @>\psi>>   \mathbb R^{n-3}\times \interior{\fiber}
	   \end{CD}
	\end{equation}
	  such that 
	 \begin{enumerate}[label=(3-\roman*)]
		\item $\psi\colon (Z- W) \to \mathbb R^{n-3}\times \interior{\fiber}$ is the restriction of the map $\psi \colon Z\to  \mathbb R^{n-3}\times \fiber$ above, 
		\item $\Psi$ maps $S_{\mathcal V\oplus f^\ast \widehat{\mathcal V} }$ to $S_{T\mathbb F \oplus T\mathbb G}$ and maps $S_{\mathcal V^\perp \oplus f^\ast (\widehat{\mathcal V})^\perp  }$ to $S_{T\mathbb R^{n-3}\oplus f^\ast T\mathbb R^{m-3}} $;
		\item along each $Q_r(W)$, $\Psi$ is a smooth invertible bundle morphism that maps $S_{\mathcal V_{(2)}\oplus f^\ast \widehat{\mathcal V}_{(2)} }$ to $S_{T\link(\fiber) \oplus T\link(\mathbb G)}$ and  $S_{\mathcal V_{(2)}^\perp \oplus f^\ast (\widehat{\mathcal V}_{(2)}^\perp)  }$ to $S_{T\link(\fiber)^\perp \oplus f^\ast T\link(\mathbb G)^\perp } $, where $T\link(\fiber)$ and $T\link(\mathbb G)$ are the tangent bundle of $\link(\fiber)$ and $\link(\mathbb G)$;
		\item $\Psi$ is fiberwise asymptotically conical\footnote{$\Psi$ can be chosen to be fiberwise asymptotically conical because all the relevant geometric data and constructions are fiberwise asymptotically conical. We should point out that $\Psi$ is usually not well-defined at the base space $W = W\times \{0\}\subset Z$ when the three dimensional polyhedral corner  $\fiber$ has  more than three faces. However, this does not affect our asymptotically conical type analysis. } in the sense of Condition ($\alpha$3) above; 
		\item $\Psi$ maps the local boundary condition $B$ on $S_N\otimes f^\ast S_M$ to the local boundary condition $B_o$ on $S_{T\mathbb F \oplus T\mathbb G} \otimes S_{T\mathbb R^{n-3}\oplus f^\ast T\mathbb R^{m-3}}$; 
		\item similar to Property \ref{item:approx} of the map $\varPhi$ in the codimension two case, by choosing the neighborhood $Z$ of $x$ to be sufficiently small,   $\Psi$ becomes arbitrarily close to the identity map in the following sense. By rescaling the metrics on $N$ and $M$ by $1/\delta$, the tangent cone $\mathbb R^{n-3}\times \fiber$ at $x$ and  the tangent cone $\mathbb R^{m-3}\times \mathbb G$ at $f(x)$ before and after the rescaling are canonically identical. Let $\Psi_\delta$ be an invertible bundle map similar to the map $\Psi$ above but with respect to the new rescaled metrics on $N$ and $M$.  Correspondingly, we denote by $D_\delta$  the associated Dirac operator on $Z$ subject to the local boundary condition $B$ under the new rescaled metrics on $N$ and $M$. When $\delta$ is sufficiently small, each operator $\Psi_\delta D_\delta \Psi^\ast_\delta$ is well defined on $\mathbb B_1\times \interior {\fiber}_1$, where $ \mathbb B_1$ is the open ball of radius $1$ centered at the origin of $\mathbb R^{n-3}$ and $\interior{\fiber}_1$ is the intersection of $\interior{\fiber}$ with the open ball of radius $1$ centered at the origin of $\mathbb R^3$. Recall that $\interior{\fiber}$ stands for $\fiber$ with its vertex removed. Here is a key difference that distinguishes the case of codimension three  from the codimension two case. Since the map $\Psi_\delta$ is fiberwsie asymptotically conical but generally only smooth along the (asymptotic) links, the commutator $[D_\delta, \varPhi^\ast_\delta]$ is a bundle endomorphism  of the form $r^{-1} \mathcal A_\delta$, where $r$ is the radial variable of the fiber $\interior{\fiber}$ and $\mathcal A_\delta$ is a uniformly bounded bundle endomorphism.  Let $\slashed D$ be the associated Dirac operator on $\R^{n-3} \times \fiber$. We may choose $\Psi_\delta$ so that  $\Psi_\delta \Psi^\ast_\delta D_\delta $ converges to $\slashed D$ on $\mathbb B_1\times \interior{\fiber}_1$ and the supremum norm of $\mathcal A_\delta$ goes to zero, as $\delta \to 0$. 
	 \end{enumerate} 
	 
     Since $\dim \fiber =3$ in the current case,  it follows from Lemma \ref{lemma:sobolevEmbedding} that $r^{-1}\mathcal A_\delta$ defines a bounded linear map $H^1(\mathbb B_1\times \interior {\fiber}_1, E) \to L^2(\mathbb B_1\times \interior {\fiber}_1, E)$, where $E = S_{T\fiber\oplus T\mathbb G} \otimes S_{T\mathbb R^{n-3}\oplus f^\ast T\mathbb R^{m-3}}$. It follows from the same argument in the codimension two case that, as long as $\delta$ is sufficiently small, we may apply the Kato-Rellich perturbation theorem to reduce the verifcation of the essential self-adjointness of $D_B$ near faces of codimension three (but away from codimension four) to a standard product case. This finishes the codimension three case.

	Now the case of codimension $\geq 4$ follows by an exact same argument as the codimension three case. We shall be brief. An analogue of $\Psi$ (as in the codimension three case) is constructed by first defining it near codimension two faces of $Z$, then inductively extending  it in an asymptotically conical way to faces of higher codimensions. The commutator $[D, \Psi^\ast]$ is a bundle endomorphism that is of the iterated form  as the term  $\prod_{k=3}^n\frac{1}{\dist(x,\Sigma_k)}$ from Lemma \ref{lemma:iteratedSobolev}. Finally, we apply Lemma \ref{lemma:iteratedSobolev} (instead of Lemma \ref{lemma:sobolevEmbedding}) together with the Kato-Rellich perturbation theorem to reduce the verifcation of the essential self-adjointness of $D_B$ to a standard product case. This complete the proof of the theorem. 
\end{proof}

			\section{A gluing formula for Fredholm index}\label{sec:gluing}
		In this section, we  prove a gluing formula for the Fredholm index of Dirac operator, which is a key ingredient for computing the Fredholm index of the Dirac operator $D_B$ on $S_N\otimes f^\ast S_M$.
			\begin{definition}\label{def:extension}
				Let $N$ be a manifold with polyhedral boundary and $X$ a codimension one face of $N$. Let $E$ be a Hermitian vector bundle over $N$. Suppose that $B$ is a local boundary condition on $\partial N$ for sections of $E$. We say a section $\varphi$ of $E$ over $X$ satisfies the boundary condition $B_{\partial X}$ on $\partial X$ if $\varphi$ satisfies  the boundary condition $B$ on  $\partial X\subset \partial N$. We call $B_{\partial X}$ the restriction of $B$ to $\partial X$ (cf. Figure \ref{fig:restriction}). Let $\tau$ be the trace map from $H^1(N,E)$ to $H^{1/2}(X,E)$. We say $X$ satisfies the extension property if there exists a bounded linear map
				$$\mathcal E\colon H^{1/2}(X,E;B_{\partial X})\to H^1(N,E; B_{\partial N - X})$$
				such that $\tau \mathcal E=1$, where $B_{\partial N - X}$ stands for the boundary condition $B$ on the codimension one faces in $\partial N$ except $X$.  
			\end{definition}
		
		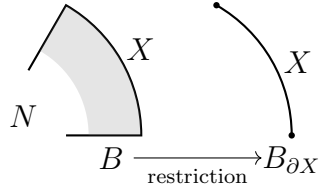
\begin{figure}[h]
			\begin{tikzpicture}[scale=1]
				\draw[thick]  (2+2,0) arc (0:60:2);
				\draw (1.6,-0.3) node {$B$};
				\draw[->] (1.9,-0.3) -- (3.6,-0.3) node [midway, below] {\tiny{restriction}};
				\draw ({2+2.3*cos(25)},{2.3*sin(25)}) node {$X$};
				\fill (2+2,0) circle (1.2pt);
				\fill ({2+2*cos(60)},{2*sin(60)}) circle (1.2pt);
				\draw (2+2,0) node [anchor=north] {$B_{\partial X}$};
				\draw[thick] (1,0) -- (2,0) arc (0:60:2) -- ({1*cos(60)},{1*sin(60)});
				\draw ({2.3*cos(30)},{2.3*sin(30)}) node {$X$};
				\draw ({0.5*cos(30)},{0.5*sin(30)}) node {$N$};
				\fill[gray,opacity=0.2] (1.3,0) -- (2,0) arc (0:60:2) -- ({1.3*cos(60)},{1.3*sin(60)}) arc(60:0:1.3);
			\end{tikzpicture}
		\caption{Restriction of the boundary condition $B$.}
		\label{fig:restriction}
		\end{figure}

			\begin{proposition}\label{prop:gluing}
				Let $N$ be a Riemannian manifold with polyhedral boundary and $E$ a Clifford bundle over $N$. Let $N=N_1\cup_X N_2$ be a partition of $N$, where $N_1$ and $N_2$ are Riemannian manifolds with polyhedral boundary, and $X$ is a common codimension one face of $N_1$ and $N_2$. Suppose that $B$ is a local boundary condition on $\partial N$ for sections of $E$.  Let $E_1$ be a subbundle of $E$ over $X$ and $E_2 = E_1^\perp$  its orthogonal complement. 
				Let $Q$ be the orthogonal projection to $E_1$ and $Q^\perp=1-Q$. For each $s\in [0, 1], $ consider  the following boundary condition $B_s$ for $E$ on the boundary of the disjoint union  $N_1\sqcup  N_2$: 
				\begin{enumerate}[label=$(\roman*)$]
					\item $s Q(\varphi_1)=Q(\varphi_2)$ and $ Q^\perp(\varphi_1)=s Q^\perp(\varphi_2)$ along $X$, 
					\item $\varphi_1$ (resp. $\varphi_2$) satisfies the boundary condition $B$ along all codimension one faces of $N_1$ (resp. $N_2$)  other than $X$, 
				\end{enumerate}
				for all smooth sections $\varphi_1$ and $\varphi_2$ of $E$ over $N_1$ and $N_2$ respectively.   
				Assume that 
				\begin{enumerate}[label=$(\arabic*)$]
					\item the associated Dirac operator $D$ on $E$ over $N_1\sqcup N_2$ subject to the boundary condition $B_s$ is an essentially self-adjoint Fredholm operator with domain $H^1(N_1\sqcup N_2,E;B_s)$, 
					\item $X$ satisfies the extension property with respect to both $N_1$ and $N_2$ in the sense of Definition \ref{def:extension},
				\end{enumerate}
				then the Fredholm index of $D_{B_s}$ is constant for $s\in[0,1]$.
			\end{proposition}
			\begin{remark}
				Note that when $s=0$, the boundary condition $B_0$ is precisely the boundary condition $B$ on $N$. When $s=1$, $B_1$ consists of two local boundary conditions $\mathbf B_1$ and $\mathbf B_2$ for $N_1$ and $N_2$ separately. More precisely, $\mathbf B_1$ (resp. $\mathbf B_2$) is the boundary condition for $N_1$ (resp. $N_2$) that coincides with $B$ along codimension one faces of $N_1$ (resp. $N_2$) other than the face $X$ and is determined by\footnote{We  say a section $\psi$ of $E$ over $N_i$ satisfies  the boundary condition on $X$ determined by the subbundle $E_i$ if the restriction of $\psi$ on $X$ is a section of $E_i$. } the subbundle $E_1$ (resp. $E_2$) along the face $X$. In particular, it follows from  Proposition \ref{prop:gluing} that 
				$$\ind(D^N_B)=\ind(D^{N_1}_{\mathbf B_1})+\ind(D^{N_2}_{\mathbf B_2}).$$
			\end{remark}
			\begin{proof}[Proof of Proposition \ref{prop:gluing}]
				
				Let $\tau$ be the Sobolev trace map from $H^1(N,E)$ to $X$.	By assumption, for each $s\in [0, 1]$, we have a Fredholm operator
				$$D_{B_s} \colon H^1(N_1\sqcup N_2,E;B_s)\to L^2(N_1\sqcup  N_2,E).$$
				Since $X$ satisfies the extension property with respect to both $N_1$ and $N_2$ in the sense of Definition \ref{def:extension},  there exists a bounded linear map 
				$$\mathcal E_i\colon H^{1/2}(X,E;B_{\partial X})\to H^1(N_i,E; B)$$
				such that $\tau \mathcal E_i = 1$, where $B_{\partial X}$ is the restriction of the boundary condition $B$ on $\partial X$ (cf. Definition \ref{def:extension}). For $s, t\in [0,1 ]$, consider the following bounded linear map
				\begin{align*}
					T_{s,t}\colon H^1(N_1\sqcup  N_2,E;B_s)&\to H^1(N_1\sqcup  N_2,E;B_t)\\
					(\psi_1,\psi_2)&\mapsto (\psi_1+(t-s)\mathcal E_1Q^\perp\tau\psi_2, \psi_2+(t-s)\mathcal E_2Q\tau \psi_1 ). 
				\end{align*}
				Note that $T_{s,t}$ is invertible with its inverse given by $T_{t,s}$. Moreover, $T_{s, t}$ is continuous in both $s$ and $t$ with respect to the operator norm. In particular, the following operators 
				\[  H^1(N_1\sqcup N_2,E;B_1)\xrightarrow{\ T_{1, s} \ }  H^1(N_1\sqcup N_2,E;B_s)\xrightarrow{\ D_{B_s} \ } L^2(N_1\sqcup N_2,E) \] 
				is continuous with respect to the operator norm. By the homotopy invariance of Fredholm index, the index of $D_{B^s}\circ T_{1, s}$ is constant. Since $T_{1, s}$ is invertible, we have
				\[ \ind(D_{B_s}) = \ind(D_{B_s}\circ T_{1, s}). \] 
				To summarize, we have shown that $\ind(D_{B_s})$ remains constant for  $s\in [0, 1]$.  This finishes the proof.	
			\end{proof}
			\begin{remark}\label{rk:gluingBs}
				Here is a prototypical situation to which we shall apply Proposition \ref{prop:gluing}. 	Let $f\colon N\to M$ be a spin polytope map between manifolds with polyhedral boundary, and $B$ the boundary condition given in Definition \ref{def:boundarycondition}. Suppose that $N$ decomposes into two submanifolds (with polyhedral boundary) along a hypersurface $X$  such that\footnote{For example, $X$ is composed of fiberwise links of a given codimension $k$ face of $N$.} (cf. Figure \ref{fig:gluing})
				\begin{enumerate}[label=(\roman*)]
					\item $X$ is orthogonal to all codimension one faces of $N$  that intersect $X$, 
					\item $X$ separates $N$ into $N_1$ and $N_2$.  
				\end{enumerate}
				Moreover, assume the image $f(X)$ is a hypersurface in $M$ that  decomposes $M$ into $M_1\cup_{f(X)} M_2$ and satisfies the same properties in $M$. Let $E_1$ be the subbundle of $S_N\otimes f^*S_M$ such that its sections $\varphi$ satisfy the condition 
				$$(\overbar \epsilon\otimes\epsilon)(\overbar c(\overbar e_n)\otimes c(e_n))\varphi=-\varphi  \textup{  over } X, $$
				where $\overbar e_n$ and $e_n$ are unit inner normal vectors of $X$ and $f(X)$ (with respect to $N_1$ and $M_1$). As long as $f\colon N \to M$ satisfies the assumptions of Theorem \ref{thm:ess-sa}, then the above decompositions of $N$ and $M$ satisfy all the conditions of Proposition \ref{prop:gluing}. In particular, it follows from Lemma \ref{lemma:essensa-Bs} for codimension two faces near $X$ and Claim \ref{claim:>=1/2} for higher codimensional faces near $X$ that  Condition (1) of Proposition \ref{prop:gluing} holds for the above decompositions. 
			\end{remark}
			\begin{figure}[h]
			\begin{tikzpicture}[scale=1]
				\draw[thick] (0,0) -- (3,0);
				\draw (1.5,-0.3) node {$B$};
				\draw (0.8,2) node {$B$};
				\draw[thick] (0,0) -- ({3*cos(60)},{3*sin(60)});
				\draw[smooth,domain=3.3:3.6, thick] plot({\x},{0.05*sin(3000*\x)+1.5});
				\draw[->, thick] (3.6,1.5) -- (3.8,1.5);
				\draw[thick] ({3*cos(60)+4.5},{3*sin(60)}) -- ({2*cos(60)+4.5},{2*sin(60)}) arc (60:0:2) -- ({7.5},0);
				\draw[thick] ({4.5-0.6*cos(30)},{-0.6*sin(30)}) -- ({6.5-0.6*cos(30)},{-0.6*sin(30)}) arc (0:60:2) -- ({4.5-0.6*cos(30)},{-0.6*sin(30)});
				\draw[very thick]  ({6.5-0.6*cos(30)},{-0.6*sin(30)}) arc (0:60:2);
				\draw[thick]  ({2*cos(60)+4.5},{2*sin(60)}) arc (60:0:2);
				\draw ({1.5*cos(30)},{1.5*sin(30)}) node {$N$};
				\draw ({2.5*cos(30)+4.5},{2.5*sin(30)}) node {$N_{1}$};
				\draw ({0.9*cos(30)+4.5},{0.9*sin(30)}) node {$N_{2}$};
				\draw ({1.7*cos(30)+4.5},{1.7*sin(30)}) node {$X$};
			\end{tikzpicture}
			\caption{The cutting construction and the gluing formula}
			\label{fig:gluing} 
		\end{figure}
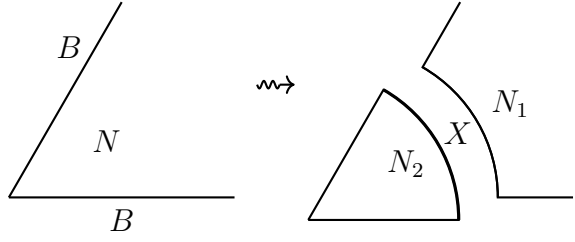
		
			\begin{lemma}\label{lemma:extension}
				Let $X$ be a compact Riemannian manifold with polyhedral boundary  and $N = X\times [0, 1]$ equipped with the direct product metric. Assume that $E$ is a Clifford bundle over $N$ which is flat along $[0,1]$. Let $D$ be the Dirac operator acting on $E$. Let $B$ be a boundary condition on $E$ over $\partial N$. Assume that the restriction $B$ on $\partial X\times [0, 1]$ is constant along $[0,1]$. 
				If
				\begin{enumerate}[label=$(\arabic*)$]
					\item $D$ subject to the boundary condition $B$ is essentially self-adjoint and its  domain is $H^1(N,E;B)$, and
					\item the Dirac operator $D_X$ along $X$ subject to the boundary condition  $B_{\partial X}$ (cf. Definition \ref{def:extension}) is  essentially self-adjoint and its domain is $H^1(X,E;B_{\partial X})$,
				\end{enumerate}
				then $X = X\times \{0\}$ satisfies the extension property in the sense of Definition \ref{def:extension}.
			\end{lemma}
			\begin{proof}
				By assumption, we have
				$$D=\alpha\frac{\partial}{\partial t}+D_X,$$
				where $t$ is the coordinate for $[0,1]$, where $\alpha$ is the Clifford multiplication of $\partial/\partial t$. 
				
				The eigenfunctions $\{\phi_\lambda\colon \phi_\lambda \in H^1(X,E; B_{\partial X}) \}$   of $D_X$ subject to the boundary condition $B_{\partial X}$  form  a complete orthonormal basis of $L^2(X,E)$. We define the Sobolev space
				$H^s(X,E; B_{\partial X})$ for $0\leq s\leq 1$ by the following norm
				$$\|\phi\|_s^2\coloneqq \sum_{\lambda}(1+|\lambda|^{2s})|\langle\phi,\phi_\lambda\rangle|^2$$
				which coincides with the usual Sobolev $H^s$-norm by standard elliptic estimates. 
				
				Let $\chi$ be a smooth cut-off function such that $\chi(t)=1$ if $0\leq t\leq 1/3$ and $\chi(t)=0$ if $t\geq 2/3$. Choose a positive smooth Schwartz function $\Psi$ on $\R$ with $\Psi(0)=1$. We define  
				$$ \mathcal  E(\phi)\coloneqq \chi(t)\cdot \Psi(tD_X)\phi$$
				for each $\phi\in H^{1/2}(X,E; B_{\partial X})$. 
				Note that
				\begin{align*}
					\|\Psi(tD_X)\phi\|^2
					=\sum_{\lambda}\int_0^1 \Psi^2(t\lambda) \cdot  |\langle\phi,\phi_\lambda\rangle|^2dt
					\leq \|\Psi\|^2_{L^\infty(\R)}\cdot \big\|\phi\big\|_{1/2}^2,
				\end{align*}
				\begin{align*}
					\big\|\frac{\partial}{\partial t}\Psi(tD_X)\phi\big\|^2
					=&\sum_{\lambda}\int_0^1(\Psi'(t\lambda))^2\lambda^2 |\langle\phi,\phi_\lambda\rangle|^2dt
					=\sum_{\lambda}|\lambda|\cdot|\langle\phi,\phi_\lambda\rangle|^2\int_0^{|\lambda|}(\Psi'(s))^2ds\\
					\leq &\|\Psi'\|^2_{L^2(\R)}\cdot \big\|\phi\big\|_{1/2}^2,
				\end{align*}
				and
				\begin{align*}
					\|D_X\Psi(tD_X)\phi\|^2
					=&\sum_{\lambda}\int_0^1 \Psi^2(t\lambda) \lambda^2 |\langle\phi,\phi_\lambda\rangle|^2dt
					=\sum_{\lambda}|\lambda|\cdot|\langle\phi,\phi_\lambda\rangle|^2\int_0^{|\lambda|}(\Psi(s))^2ds\\
					\leq &\|\Psi\|^2_{L^2(\R)}\cdot \big\|\phi\big\|_{1/2}^2.
				\end{align*} 
				Moreover, since $\Psi(0)=1$, we have $\tau\mathcal E\phi=\phi,$  where  $\tau$ is the trace map from $H^1(N,E)$ to $H^{1/2}(X,E)$. 
				It follows that $\mathcal E$ defines a bounded linear map 
				\[ \mathcal E \colon H^{1/2}(X,E; B_{\partial X}) \to H^1(N, E; B_{\partial N - X})  \] 
				such that $\tau\mathcal E = 1$.  This finishes the proof.
			\end{proof}

\section{Fredholm index of twisted Dirac operators on manifolds with corners}\label{sec:fredindex}
					
In this section and the next section, we  compute the Fredholm index of the Dirac operator $D_B$ on $S_N\otimes f^\ast S_M$. The computation for the case of manifolds with polyhedral boundary is more subtle than the case of manifolds with corners.  In order to make our presentation more transparent, we shall mainly focus on the case of manifolds with corners in this section, and prove the general case of manifolds with polyhedral boundary in Section \ref{sec:index-poly}.   
					
		\begin{theorem}\label{thm:index}
		Let $(M, g)$ and $(N, \overbar{g})$ be two oriented compact Riemannian manifolds with corners.  Suppose   $f\colon N \to M$ is a spin polytope map (Definition \ref{def:spin} and Definition \ref{def:polytopeMap}).	Let $D_B$ be the Dirac operator on $S_N\otimes f^*S_M$ subject to the local  condition $B$ given in Definition \ref{def:boundarycondition}. In particular,  $D_B$ acts on  the space of sections  $C^\infty_0(N,S_N\otimes f^*S_M;B)$ given in Definition $\ref{def:smooth,H1}$.	
		For each pair of  codimension one faces $\overbar F_i, \overbar F_j$ of $N$, assume  the dihedral angles $\theta_{ij}(\overbar{g})$ of $N$ and $\theta_{ij}(g)$ of $M$ satisfy either 
		\begin{equation}\label{eq:strictdihedral}
			0<\theta_{ij}(\overbar{g})_z < \theta_{ij}(g)_{f(z)} < \pi,  \textup{ for all }   z \in \overbar F_i\cap \overbar{F}_j
		\end{equation}  
	or 
	\begin{equation}\label{eq:equaldihedral}
		0<\theta_{ij}(\overbar{g})_z =  \theta_{ij}(g)_{f(z)} < \pi, \textup{ for all }   z \in \overbar F_i\cap \overbar{F}_j. 
	\end{equation}   Let $\overbar{D}_B$ be the unique self-adjoint extension of $D_B$ as in Theorem $\ref{thm:ess-sa}$. Then the Fredholm index of $\overbar{D}_B$ is 
							\[ \ind(\overbar{D}_B)=\deg_{\widehat A}(f)\cdot \chi(M),\]
							where $\deg_{\widehat A}(f)$ is $\widehat A$-degree of $f$ and $\chi(M)$ is Euler characteristic of $M$. 
						\end{theorem}

 As a preparation, we first prove several technical lemmas that will be crucial in the proof of Theorem \ref{thm:index}. Let $(N,\overbar g)$ and $(M,g)$ be Riemannian manifolds with corners.
Let $\overbar g_n$ (resp. $g_n$) be a sequence of Riemannian metrics on $N$ (resp. $M$) such that $\overbar g_n\to \overbar g$ (resp. $g_n\to g$) in $C^2$-norm.	
Let $B_n$ be a mixed  local boundary condition of $(S_N\otimes f^*S_M, \overbar g_n\otimes g_n)$  as follows:   on each codimension one face of $N$, $B_n$  is equal to either the boundary condition given in  Definition \ref{def:boundarycondition} or  the orthogonal complement of the boundary condition given in  Definition \ref{def:boundarycondition}.

\begin{definition}
	We say a sequence of boundary conditions $\{B_n\}$ above is of the same mixed type if on each codimension one face $\overbar F$ of $N$, either $B_n$ is always equal to the boundary condition given in  Definition \ref{def:boundarycondition} for all $n$, or $B_n$ is always equal to the orthogonal complement of the boundary condition given in  Definition \ref{def:boundarycondition} for all $n$. 
\end{definition} 

\begin{lemma}\label{lemma:solutionapproximate}
	Let $(N,\overbar g)$ and $(M,g)$ be even dimensional Riemannian manifolds with corners and $f\colon N\to M$  a  spin corner map. Let $\overbar g_n$ (resp. $g_n$) be a sequence of Riemannian metrics on $N$ (resp. $M$) such that $\overbar g_n\to \overbar g$ (resp. $g_n\to g$) in $C^2$-norm. Suppose $\{B_n\}$ is a sequence of local boundary conditions of $(S_N\otimes f^*S_M, \overbar g_n\otimes f^\ast g_n)$  of the same mixed type such that $B_n$ converges to a   local boundary condition $B$ of $(S_N\otimes f^*S_M, \overbar g\otimes g)$  in $C^1$-norm  on each codimension one face of $N$. 	Let $D_n$ be the Dirac operator on $(S_N\otimes f^*S_M, \overbar g_n\otimes g_n)$ over $N$,  and $D$ the Dirac operator on $(S_N\otimes f^*S_M, \overbar g\otimes g)$ over $N$. Let $R$ and $A$ be the bounded  bundle endomorphisms of $S_N\otimes f^*S_M$  appearing in the following Bochner-type  formula for $D$ (cf. Section \ref{sec:boundaryconditions}): 
	\begin{equation}\label{eq:homogenousIneq}
		\int_N|D\psi|^2=\int_N |\nabla\psi|^2+\int_N \langle R\psi,\psi\rangle+\int_{\partial N}\langle A\psi,\psi\rangle,
	\end{equation}
	for all $\psi\in C_0^\infty(N,S_N\otimes f^*S_M;B)$, cf. Definition \ref{def:smooth,H1}. If we have 
	\begin{enumerate}[label=$(\arabic*)$]
		\item both $R$ and $A$ are pointwise nonnegative, and
		\item for each $n$, there exists a nonzero $\varphi_n\in H^1(N,S_N\otimes f^*S_M;B_n)$ such that $D_n\varphi_n=0$,
	\end{enumerate}  
	then there exists a nonzero $\varphi\in H^1(N,S_N\otimes f^*S_M;B)$ such that  $D\varphi=0$ and $\nabla\varphi=0$. 	
\end{lemma}

\begin{proof}
	Let $\widetilde g_n=\overbar g_n\otimes f^*g_n$ be the metric induced by $\overbar g_n$ and $g_n$ on $S_N\otimes f^*S_M$, and similarly $\widetilde g = \overbar g\otimes f^*g$. Without loss of generality, we may assume that the $L^2$-norm of $\varphi_n$ in $L^2(N,(S_N\otimes f^*S_M)_{\widetilde g_n})$ is equal to $1$.
	Let $R_n$ and $A_n$ be the bounded  bundle endomorphisms of $S_N\otimes f^*S_M$  appearing in the following Bochner-type  formula for $D_n$:
	\begin{equation}\label{eq:BochnerDn}		
		\int_N|D_n\psi|^2_{\widetilde g_n} =\int_N |\prescript{n}{}{\nabla}\psi|^2_{\widetilde g_n}+\int_N \langle R_n\psi,\psi\rangle_{\widetilde g_n}+\int_{\partial N}\langle A_n\psi,\psi\rangle_{\widetilde g_n},
	\end{equation}
	for all $\psi\in C_0^\infty(N,S_N\otimes f^*S_M;B_n)$, where $\prescript{n}{}{\nabla}$ is the connection on $S_N\otimes f^*S_M$ induced by the metrics $\overbar g_n$ and $g_n$. By assumption, we have $R_n \to R$ and $A_n \to A$ uniformly, as $n\to \infty$.  It follows that there exist $\varepsilon_n >0$ such that 
	\begin{align*}
		\int_N\big|\prescript{n}{}{\nabla}\varphi_n\big|^2_{\widetilde g_n} & = \int_N\big|D_n\varphi_n\big|^2_{\widetilde g_n} - \int_N \langle R_n\varphi_n,\varphi_n\rangle_{\widetilde g_n} - \int_{\partial N}\langle A_n\varphi_n,\varphi_n\rangle_{\widetilde g_n} \\
		&  \leq  \varepsilon_n\big(\int_N\big|\varphi_n\big|_{\widetilde g_n}^2+
		\int_{\partial N}\big|\varphi_n\big|_{\widetilde g_n}^2
		\big)\\
		&=\varepsilon_n\big(1+\int_{\partial N}\big|\varphi_n\big|_{\widetilde g_n}^2
		\big)
	\end{align*}
	and $\varepsilon_n \to 0$, as $n \to \infty$. Since $\overbar g_n\to \overbar g$ (resp. $g_n\to g$) in $C^2$-norm, the trace theorem on Sobolev spaces $H^1(N,(S_N\otimes f^*S_M)_{\widetilde g_n})$ holds uniformly in $n$. That is,  there exists $C>0$ independent of $n$ such that
	$$\int_{\partial N}\big|\varphi_n\big|_{\widetilde g_n}^2
	\leq C\big\|\varphi_n\big\|_{H^1(N,(S_N\otimes f^*S_M)_{\widetilde g_n})}^2
	=C(1 +\int_N\big|\prescript{n}{}{\nabla}\varphi_n\big|^2_{\widetilde g_n})$$
	for all $n$. It follows that
	$$\int_N\big|\prescript{n}{}{\nabla}\varphi_n\big|^2_{\widetilde g_n}\leq \frac{\varepsilon_n(1+C)}{1-\varepsilon_nC}\to 0, \text{ as }n\to\infty.$$
	Hence $\big\|\varphi_n\big\|_{H^1(N,(S_N\otimes f^*S_M)_{\widetilde g_n})}^2$ is uniformly bounded for all $n$.
	
	The Sobolev space $H^1(N,(S_N\otimes f^*S_M)_{\widetilde g})$ is norm-equivalent  to the Sobolev space $H^1(N,(S_N\otimes f^*S_M)_{\widetilde g_n})$, where the constants of the norm equivalences are uniformly bounded for all $n$. Therefore $\{\varphi_n\}$ is a bounded sequence of sections in $H^1(N,(S_N\otimes f^*S_M)_{\widetilde g})$, hence has a convergent subsequence in $L^2(N,(S_N\otimes f^*S_M)_{\widetilde g})$. Without loss of generality, we may assume that $\{\varphi_n\}$ converges to $\varphi$ in $L^2(N,(S_N\otimes f^*S_M)_{\widetilde g})$. Note that  $L^2(N,(S_N\otimes f^*S_M)_{\widetilde g})$ is norm-equivalent to $L^2(N,(S_N\otimes f^*S_M)_{\widetilde g_n})$ with uniformly bounded constants for the norm equivalences. It follows that $\|\varphi\|_{\widetilde g}=1$. In particular,  $\varphi$ is nonzero.

	Note that the $L^2$-norm of $\prescript{n}{}{\nabla}\varphi_n$ with respect to $\widetilde g$ also converges to zero. Since $\overbar g_n\to \overbar g$ (resp. $g_n\to g$) in $C^2$-norm,   there exist uniformly bounded endomorphisms $T_n$ such that
	$$\prescript{n}{}{\nabla}=\nabla+T_n,$$
	and $T_n\to 0$ uniformly, as $n\to \infty$. Since the $L^2$-norm of $\varphi_n$ with respect to $\widetilde g$ is uniformly bounded, it follows that 
	\[\lim_{n\to \infty} \int_{N}\big|\nabla\varphi_n\big|^2_{\widetilde g} = 0.  \]
	Therefore, $\{\varphi_n\}$ forms a Cauchy sequence in $H^1(N,(S_N\otimes f^*S_M)_{\widetilde g})$. Consequently,  $\varphi_n$ converges  to  $\varphi$ with respect to the $H^1$-norm. In particular, we have 
	\[ \int_{N}|\nabla\varphi|^2 =  \lim_{n\to \infty} \int_{N}|\nabla\varphi_n|^2 = 0,\]
	It follows that $\nabla \varphi =0$, hence $D\varphi = 0$. Now for each codimension one face $\overbar F$ of $N$, it follows from the trace theorem that  the restriction of $\varphi_n$  on $\overbar F$ converges  to the restriction  of $\varphi$ on $\overbar F$. This implies that  $\varphi\in H^1(N,S_N\otimes f^*S_M;B)$, since $B_n$ converges to $B$ on $\overbar F$ in $C^1$-norm.  
	This finishes the proof. 
\end{proof}

For example, line \eqref{eq:homogenousIneq} holds if $f$, $(N,\overbar g)$ and $(M,g)$ satisfy the assumptions of Theorem \ref{thm:extremal-rigid} either with or without the comparison on dihedral angles. This follows from approximating $H^1$-sections by smooth sections vanishing near codimension two singularities.
We emphasize the assumption of the non-negativity in line \eqref{eq:homogenousIneq} is essential for Lemma \ref{lemma:solutionapproximate} to hold. If $D_B$ were to only satisfy the usual G\r{a}rding's inequality as in Proposition \ref{prop:norm-equivalent}, then we would only be able to obtain a convergence in $L^2$ but not in $H^1$ in general. We also emphasize that Lemma \ref{lemma:solutionapproximate} does not require the essential self-adjointness of $D$ or $D_n$.

The following immediate consequence of Lemma \ref{lemma:solutionapproximate} will  be crucial in our proof of Theorem \ref{thm:index}.
\begin{corollary}\label{coro:nosolution}
	Let $(N,\overbar g)$ and $(M,g)$ be even dimensional Riemannian manifolds with corners and $f\colon N\to M$  a  spin corner map. Let $\overbar g_n$ (resp. $g_n$) be a sequence of Riemannian metrics on $N$ (resp. $M$) such that $\overbar g_n\to \overbar g$ (resp. $g_n\to g$) in $C^2$-norm. Suppose $\{B_n\}$ is a sequence of local boundary conditions of $(S_N\otimes f^*S_M, \overbar g_n\otimes f^\ast g_n)$  of the same mixed type such that $B_n$ converges to a   local boundary condition $B$ of $(S_N\otimes f^*S_M, \overbar g\otimes g)$  in $C^1$-norm  on each codimension one face of $N$. 	Let $D_n$ be the Dirac operator on $(S_N\otimes f^*S_M, \overbar g_n\otimes g_n)$ over $N$,  and $D$ the Dirac operator on $(S_N\otimes f^*S_M, \overbar g\otimes g)$ over $N$. Let $R$ and $A$ be the bounded  bundle endomorphisms of $S_N\otimes f^*S_M$  appearing in the following Bochner-type  formula for $D$ (cf. Section \ref{sec:boundaryconditions}): 
	\begin{equation}			
		\int_N|D\psi|^2=\int_N |\nabla\psi|^2+\int_N \langle R\psi,\psi\rangle+\int_{\partial N}\langle A\psi,\psi\rangle,
	\end{equation}
	for all $\psi\in C_0^\infty(N,S_N\otimes f^*S_M;B)$, cf. Definition \ref{def:smooth,H1}. If we have 
	\begin{enumerate}[label=$(\arabic*)$]
		\item both $R$ and $A$ are pointwise nonnegative, and
		\item there does not exist a nonzero element $\varphi\in H^1(N,S_N\otimes f^*S_M;B)$ such that $D\varphi=0$,
	\end{enumerate}   then there does not exist a nonzero element $\varphi_n\in H^1(N,S_N\otimes f^*S_M;B_n)$ such that $D_n\varphi_n=0$,  as long as $n$ is sufficiently large.
\end{corollary}

Now we prove the main theorem of this section (Theorem \ref{thm:index}). 

\begin{proof}[Proof of Theorem \ref{thm:index}]
	
For simplicity, we shall first prove the theorem under the assumption that, for all pairs of adjacent codimension one faces $\overbar F_i, \overbar F_j$ of $N$,   the dihedral angles $\theta_{ij}(\overbar{g})$ of $N$ and $\theta_{ij}(g)$ of $M$ satisfy the strict comparison 
\begin{equation*}
	0<\theta_{ij}(\overbar{g})_z < \theta_{ij}(g)_{f(z)} < \pi,  \textup{ for all }   z \in \overbar F_i\cap \overbar{F}_j.
\end{equation*}  
The proof can be easily adapted to prove the  general case where both the strict comparison \eqref{eq:strictdihedral} and the equality \eqref {eq:equaldihedral} of dihedral angles appear (but on different codimension two faces). 	We divide the proof into two main steps.
	
	\textbf{Step I.} We deform the manifolds $N$ and $M$ together with the map $f\colon N\to M$ to obtain a continuous family of maps $f_t\colon (N_t, \overbar g_t) \to (M_t, g_t)$, $t\in [0, 1]$, such that
	\begin{enumerate}[label=$(\arabic*)$]
		\item  $(N_0, \overbar g_0) = (N, \overbar g)$, $ (M_0, g_0) = (M, g)$ and $f_0 = f$, 
		\item for each $0<t<1$, the dihedral angles of $(N_t, \overbar g_t)$ is strictly less than the corresponding dihedral angles of $(M_t, g_t)$, 
		\item for $t=1$, the metric $\overbar g_1$ on $N_1$ and  the metric $g_1$ on $M_1$ are  of product structure near each codimension $k$ face, for all $k\geq 1$,  and the map $f_1 \colon (N_1, \overbar g_1) \to (M_1, g_1)$ is a fiberwise isometry of some flat convex Euclidean corner in $\mathbb R^k$.  More precisely,  for $\overbar F_\lambda$ be a  codimension $k$ face of $N_1$, let $W_\lambda$ be the open subset of $\overbar F_\lambda$ that is the complement of sufficiently small (closed) neighborhoods of all codimension $k+1$ faces. Then a tubular neighborhood of $W_\lambda$ in $N_1$ is of the form $W_\lambda \times \fiber$, where $\fiber$ is a corner in the standard Euclidean space $\mathbb R^k$ and  $\overbar g_1$ is a direct product metric. The metric $g_1$ on $M_1$ carries a similar structure near the corresponding  codimension $k$ face $F_\lambda$, which we denote by $V_\lambda\times \mathbb G$. Note that we require $W_\lambda \times \fiber$ and $V_\lambda\times \mathbb G$ to have the same fiber $\fiber$. Moreover, the map $f_1$ maps $W_\lambda \times \fiber$ to $V_\lambda\times \mathbb G$ and is an isometry along the fiber $\fiber$. In particular, we note that the corresponding dihedral angles of  $(N_1, \overbar g_1)$ and $  (M_1, g_1)$ are equal everywhere. 
	\end{enumerate} 
	For manifolds with corners, a family of maps $f_t\colon (N_t, \overbar g_t) \to (M_t, g_t)$ with the above properties can be achieved by first smoothly deforming the metric $g$ on $M$ to $g_t$ so that the dihedral angles of $(M, g_t)$
	are strictly greater than those of $g$, and smoothly deforming the metric $\overbar g$ on $N$ to $\overbar g_t$ so that the dihedral angles of $(N, \overbar g_t)$
	are strictly \emph{less} than those of $\overbar g$, for $t\in [0, 1/2]$. To find a family of metrics $\{g_t\}$ satisfying the above conditions, it amounts to choosing a  family of convex corners $\{\mathbb G_t\}_{0\leq t \leq 1}$ in $\mathbb R^k$  with the following properties: 
	\begin{itemize}
		\item $\mathbb G_0 = \mathbb G$,
		\item there exists a  family  of invertible linear maps $L_t\colon \mathbb R^k \to \mathbb R^k$ that is $C^2$-continuous with respect to $t$ and  $L_t$ maps $\mathbb G$ to $ \mathbb G_t$, 
		\item for each $t>0$,  the dihedral angles of $\mathbb G_t$ are strictly larger than the corresponding dihedral angles of $\mathbb G_0$. 
	\end{itemize}
	Since $\mathbb G$ is enclosed by precisely $k$ hyperplanes passing through the origin in $\mathbb R^k$,  it is clear that such a family of $\{\mathbb G_t\}_{0\leq t \leq 1}$ exists by deforming the corresponding  normal vectors of these hyperplanes.

	At $t= 1/2$, we may assume that  $(N, \overbar g_{\frac{1}{2}})$  and $(M, g_{\frac{1}{2}})$ are of product structure near each codimension $k$ face, for all $k\geq 1$,  in the sense of item (3) above. Furthermore, with the same notation as in item (3) above,  we may assume without loss of generality that the fiber $\fiber$ of $W_\lambda\times \fiber$  in $(N, \overbar g_{\frac{1}{2}})$ is contained in the corresponding fiber $\mathbb G$ of $V_\lambda\times \mathbb G$ in  $(M,  g_{\frac{1}{2}})$, where $\fiber$ and $\mathbb G$ are viewed as subsets of the standard Euclidean space $\mathbb R^k$.  Now we further deform the metric  $g_{\frac{1}{2}}$ on $M$ to $g_t$ with $t\in [1/2, 1]$ so that the dihedral angles of $(M, g_t)$
	are strictly greater than those of $(N, \overbar g_{\frac{1}{2}}) $ for $t\in [1/2, 1)$, and the fiber $(\mathbb G, g_{t})$ converges to $(\mathbb F, \overbar g)$ as $t\to 1$. Since we are only changing the metrics while keeping the underlying spaces fixed, we choose $f_t = f$ for all $t\in [0, 1]$. In particular, when $t=1$, the map $f_1$ near each codimension $k$ face is a fiberwise isometry in the sense of item (3) above.

	\textbf{Step II.} Let $D_t$ be the Dirac operator for associated to the data $f_t\colon (N_t, \overbar g_t) \to (M_t, g_t)$ and $B_t$ the corresponding boundary condition as given in Definition \ref{def:boundarycondition}. We will first show that the Fredholm index of $(D_t)_{B_t}$ remains constant for all $t\in [0, 1)$, via a homotopy argument. However, this homotopy argument a priori fails at $t=1$, due to the fact there is \emph{no} smooth map that converts a boundary condition with the strict dihedral angle comparison to a boundary condition with the equality of  dihedral angles  (cf.  Lemma \ref{lemma:smoothEquivalence} and Remark \ref{remark:smoothEqual}). To remedy this, we shall use a cutting-and-pasting argument to directly show that there exists a sufficiently small $\delta>0$ such that the Fredholm index of $(D_t)_{B_t}$ remains constant for all $t$ with $1-\delta\leq t\leq 1$.

	Now let us complete \textbf{Step II}.	Write for short $E_t=S_{N_t}\otimes f_t^*S_{M_t}$. Let $D_t$ be the Dirac operator acting on $E_t$, and $B_t$ the boundary condition of $E_t$ as in Definition \ref{def:boundarycondition}. By Theorem \ref{thm:ess-sa}, $D_t$ with boundary condition $B_t$ is essentially self-adjoint with domain $H^1(N_t,E_t;B_t)$. Let us first  show that $\ind((D_t)_{B_t})$ remains constant for $t\in [0,1)$. 
	
	Assume that $0\leq t_1<t_2<1$. Since the dihedral angles satisfy the  strict comparison inequality  \eqref{eq:strictdihedral} for both $t_1$ and $t_2$, there exists a bounded invertible map (cf. the proof of Theorem \ref{thm:ess-sa} and Lemma \ref{lemma:smoothEquivalence})
	\[ \Psi \colon L^2(N_{t_1}, S_{N_{t_1}}\otimes f^\ast_{t_1} S_{M_{t_1}}) \to  L^2(N_{t_2}, S_{N_{t_2}}\otimes f^\ast_{t_2} S_{M_{t_2}}) \] 
	such that $\Psi$ maps the boundary condition  $B_{t_1}$ to the boundary condition $B_{t_2}$. Recall that by construction $\Psi$ is induced by a bundle map that  may not be defined on faces with codimension $\geq 3$, but the bundle map is smooth near faces of codimension $\leq 2$ and is asymptotically conical near all faces of all codimensions.

	As long as $t_1$ and $t_2$ are sufficiently close to each other, by the same argument as in the proof of Theorem \ref{thm:ess-sa}, it follows from the Kato-Rellich perturbation theorem, Lemma \ref{lemma:sobolevEmbedding} and Lemma \ref{lemma:iteratedSobolev}   that  $(\Psi D_{t_1} \Psi^\ast)_{B_{t_2}} = \Psi (D_{t_1})_{B_{t_1}} \Psi^\ast $ is essentially self-adjoint. Furthermore, $\Psi$ maps $H^1(N_{t_1}, S_{N_{t_1}}\otimes f^\ast_{t_1} S_{M_{t_1}}; B_{t_1})$ to  $H^1(N_{t_2}, S_{N_{t_2}}\otimes f^\ast_{t_2} S_{M_{t_2}}; B_{t_2})$, and
	\[ \dom (\Psi D_{t_1} \Psi^\ast)_{B_{t_2}} = H^1(N_{t_2}, S_{N_{t_2}}\otimes f^\ast_{t_2} S_{M_{t_2}}; B_{t_2}). \]
	In particular,  $(\Psi D_{t_1} \Psi^\ast)_{B_{t_2}}$ and  $(D_{t_2})_{B_{t_2}}$  have the same domain,  and the difference $\Psi D_{t_1} \Psi^\ast - D_{t_2}$ as a linear operator 
	\[(\Psi D_{t_1} \Psi^\ast - D_{t_2}) \colon H^1(N_{t_2}, S_{N_{t_2}}\otimes f^\ast_{t_2} S_{M_{t_2}}; B_{t_2})\to L^2(N_{t_2}, S_{N_{t_2}}\otimes f^\ast_{t_2} S_{M_{t_2}}) \]
	is bounded such that its operator norm goes to zero, as $t_1$ approaches $t_2$. It follows that 
\begin{equation}\label{eq:deformation-equal}
	\ind((D_{t_1})_{B_{t_1}}) = \ind( (\Psi D_{t_1} \Psi^\ast)_{B_{t_2}}) = \ind((D_{t_2})_{B_{t_2}})
\end{equation}
	as long as $t_1$ and $t_2$ are sufficiently close. This proves that   $\ind((D_t)_{B_t})$ remains constant for $t\in [0,1)$.

	Now we shall use a cutting-and-pasting argument to  show that  the Fredholm index of $(D_t)_{B_t}$ remains constant for all $t$ with $1-\delta\leq t\leq 1$, where $\delta>0$ is some sufficiently small constant. For each $t\in [1-\delta, 1]$,  we decompose $N_t$ as follows. First we cut away each corner of $N_t$ with codimension $n$, which is isometric to a corner in $\R^n$, then cut away a neighborhood of each codimension $n-1$ face, which is isometric to the direct product of a corner in $\R^{n-1}$ and a line segment, and so on. Here a corner in $\mathbb R^k$ means the intersection of a small closed ball centered at the origin in $\mathbb R^k$ and the closed subset of $\mathbb R^k$ enclosed by $k$ hyperplanes passing through the origin. We denote such a direct product\footnote{Strictly speaking, $N_{t, \lambda}$ is the direct product of a corner in $\R^k$ and a part of $F_\lambda$ that is away from   the codimension $(k+1)$ faces.} of a codimension $k$ face $F_\lambda$  and a corner in $\mathbb R^k$ by $N_{t,\lambda}$.  In particular,  we have  the following decomposition
	$$N_t= N_{t,0}\cup \bigcup_{k=2}^n\bigcup_{\ |\lambda|=k}N_{t,\lambda},$$
	where $ N_{t,0}$ is the complement of $ \bigcup_{k=2}^n\bigcup_{ |\lambda|=k}N_{t,\lambda}$. See Figure \ref{fig:decomposition} for an illustration of this cutting process  in the case of dimension two.
	We apply a similar decomposition on $M_t$ near its faces of codimension $\geq 2$: 
	$$M_t= M_{t,0}\cup\bigcup_{k=2}^n\bigcup_{\ |\lambda|=k} M_{t,\lambda}$$
	such that $f_t$ maps $N_{t, \lambda}$ to $M_{t,\lambda}$ and $f_t$ is a fiberwise linear map that converges to the fiberwise identity map as $t\to 1$.

	Now we define local boundary conditions  $B_{t,\lambda}$ for $N_{t,\lambda}$ and $B_{t,0}$ for $N_{t,0}$. For each $N_{t, \lambda}$, its codimension one faces can be categorized into the following three types. 
	\begin{itemize}
		\item   The new codimension one face $X_{t,\lambda}$  where the cutting is performed to obtain $N_{t,\lambda}$. Note that $X_{t, \lambda}$ is orthogonal to all the original codimension one faces of  $N_t$. Moreover, $X_{t, \lambda}$ has a direct product structure where each fiber is a spherical simplex. In this case, we define the boundary condition $B_{t, \lambda}$  to be
		$$(\overbar\epsilon\otimes \epsilon)(\overbar c(\overbar e_n)\otimes c(e_n))\varphi=\varphi,$$
		where $\overbar e_n$ and $e_n$ are the unit inner normal vector fields of $X_{t,\lambda}$ and $f_t(X_{t,\lambda})$. Note that this boundary condition is the orthogonal complement of the boundary condition given in Definition \ref{def:boundarycondition}.
		\item A codimension one face that is part of a new codimension one face where the cutting is performed to obtain $N_{t, \eta}$ in the previous inductive steps with the codimension $|\eta| > |\lambda|$. On these codimension one faces, we  define the boundary condition $B_{t, \lambda}$  to be the usual boundary condition given in Definition \ref{def:boundarycondition}, that is, 
		$$(\overbar\epsilon\otimes \epsilon)(\overbar c(\overbar e_n)\otimes c(e_n))\varphi= - \varphi.$$
		\item A codimension one face that is  part of a original codimension one face of $N_t$. On these codimension one faces, we also define the boundary condition $B_{t, \lambda}$  to be the usual boundary condition given in Definition \ref{def:boundarycondition}, that is, 
		$$(\overbar\epsilon\otimes \epsilon)(\overbar c(\overbar e_n)\otimes c(e_n))\varphi= - \varphi.$$
	\end{itemize} 
	In particular, by construction, the boundary condition $B_{t, 0}$ for $N_{t, 0}$ is simply the usual boundary condition from Definition \ref{def:boundarycondition} on each codimension one face of $N_{t, 0}$.

	The dihedral angles of  $N_{t, \lambda}$ and $M_{t, \lambda}$ are either the original dihedral angles of $N_t$ and $M_t$, or $\pi/2$ where the new codimesion one faces appear. Therefore, it follows from the proof of Theorem \ref{thm:ess-sa} and Lemma \ref{lemma:essensa-jumpanglewithfmixed}  that the Dirac operator $D_{t, \lambda}$ on $S_{N_{t,\lambda}}\otimes f_t^\ast S_{M_{t,\lambda}}$ subject to the boundary condition $B_{t, \lambda}$ is  essentially self-adjoint with domain $H^1(N_{t,\lambda}, S_{N_{t,\lambda}}\otimes f_t^\ast S_{M_{t,\lambda}}; B_{t, \lambda})$. 
	
	Recall that each $X_{t, \lambda}$ has a direct product structure where each fiber is a spherical simplex. Near each $X_{t,\lambda}$,   we apply the change of coordinates from line \eqref{eq:coordinateChange} that conjugates the Dirac operator from polar coordinates to cylindrical coordinates:
	$$\begin{pmatrix}
		0&-1\\1&0
	\end{pmatrix}\frac{\partial}{\partial r}+\frac 1 r \begin{pmatrix}
		0&P\\P&0
	\end{pmatrix},$$
	where $r$ is the radial variable. The boundary condition $B_{t, \lambda}$ restricts to a boundary condition $B_{\partial X_{t, \lambda}}$ on  $\partial X_{t, \lambda}$ in the sense of Definition \ref{def:extension}.  It follows from  the proof of Lemma \ref{lemma:essensa-higherdim} (cf. Claim \ref{claim:essa}) that the operator $P$ subject to the boundary condition $B_{\partial X_{t, \lambda}}$  is essentially self-adjoint.  
	By Lemma \ref{lemma:extension}, if we view $X_{t, \lambda}$ as the boundary $X_{t, \lambda}\times \{0\}$ of the direct product $X_{t, \lambda}\times [0,1]$, then $X_{t, \lambda}$ satisfies the extension property in the sense of Definition \ref{def:extension}. On the other hand, since near $X_{t, \lambda}$, the radial variable  $r$ is uniformly bounded away from zero,  it follows that the $H^1$ space on $N_{t, \lambda}$  near $X_{t, \lambda}$ is equivalent to  $H^1$ space on $X_{t, \lambda}\times [0, 1]$ near $X_{t, \lambda}\times \{0\}$ via a smooth map. Therefore, $X_{t, \lambda}$ satisfies the extension property with respect to $N_{t, \lambda}$ in the sense of Definition \ref{def:extension}.  By the gluing formula from Proposition \ref{prop:gluing}, we have
	\begin{equation}\label{eq:sumOfIndex}
		\ind((D_t)_{B_t})=\ind((D_{t,0})_{B_{t, 0}})+\sum_{k=2}^n\sum_{|\lambda|=k}\ind((D_{t,\lambda})_{B_{t, \lambda} }).
	\end{equation}
	
	\begin{claim}\label{claim:vanish}
		There exists $\delta>0$ such that  $\ind((D_{t, \lambda})_{B_{t, \lambda}}) = 0$ for all $\lambda \neq 0$ and all $t\in [1-\delta, 1]$. 
	\end{claim}
	Recall that $f_t\colon N_{t, \lambda} \to M_{t, \lambda}$ is of the form   
	\[ f_t \colon W_\lambda \times \fiber_t  \to V_\lambda \times \mathbb G_t\] 
	such that $f_t$ is also of product type.  Hence we can reduce the computation to a single fiber $f_t\colon \fiber_t \to \mathbb G_t$. By the discussion at the beginning of Section \ref{sec:conic}, the vertical bundle along the fiber  can be canonically identified with the bundle of differential forms $\Bigwedge^\ast \R^k$, where $k$ is the dimension of $\fiber_t$.   We still denote the associated Dirac operator by $D_{t, \lambda}$ and the boundary condition by $B_{t, \lambda}$, since no confusion is likely to arise.  
	
	When $t=1$, $f_1$ is exactly a fiberwise isometry. Therefore  the Dirac operator $D_{1, \lambda}$  with boundary condition $B_{1,\lambda}$ satisfies the non-negativity conditions in line \eqref{eq:homogenousIneq}, cf. Proposition \ref{prop:D^2}. It follows that, if a section $\varphi\in H^1(\fiber_{1}, \Bigwedge^\ast  \R^k ;B_{1,\lambda})$ satisfies $D_{1, \lambda} (\varphi)=0$, then $\varphi$ is parallel and satisfies the boundary condition $B_{1,\lambda}$. In particular,  $\varphi$ is a differential form on $\fiber_1$ with constant coefficients. Now on each flat codimension one face of $\fiber_1$  (that is, except the spherical face $X_{1, \lambda}$), the boundary condition $B_{1, \lambda}$ implies  that $\varphi$ does \emph{not} contain the normal direction of that given face. By going through  all flat codimension one faces, we see that $\varphi$ has to be a degree zero differential form with constant coefficients, that is, a constant scalar. However, the boundary condition  $B_{1,\lambda}$ on the spherical face $X_{1,\lambda}$ is the orthogonal complement of the boundary condition given in Definition \ref{def:boundarycondition}, which implies that $\varphi$ has to contain the normal direction of $X_{t, \lambda}$. This forces $\varphi$ to be zero.  In conclusion, there is no nontrivial  section $\varphi \in H^1(\fiber_{1}, \Bigwedge^\ast  \R^k ;B_{1,\lambda})$ such that $D_{1, \lambda} (\varphi) =0$.  By Corollary \ref{coro:nosolution}, there exists $\delta>0$ such that there is no nontrivial  section $\varphi \in H^1(\fiber_{t}, \Bigwedge^\ast  \R^k ;B_{t,\lambda})$ such that $D_{t, \lambda} (\varphi) =0$ for all $t\in [1-\delta, 1]$. In particular, it follows that 
	\[ \ind((D_{t, \lambda})_{B_{t, \lambda}}) = 0\] for all $t\in [1-\delta, 1]$. This proves the claim. 
	
		\begin{figure}
		\begin{tikzpicture}[scale=1]
			\draw[thick] (0,0) -- (3,0);
			\draw[thick] (0,0) -- ({3*cos(60)},{3*sin(60)});
			\draw[smooth,domain=3.3:3.6, thick] plot({\x},{0.05*sin(3000*\x)+1.5});
			\draw[->, thick] (3.6,1.5) -- (3.8,1.5);
			\draw[thick] ({3*cos(60)+4.5},{3*sin(60)}) -- ({2*cos(60)+4.5},{2*sin(60)}) arc (60:0:2) -- ({7.5},0);
			\draw[thick] ({4.5-0.6*cos(30)},{-0.6*sin(30)}) -- ({6.5-0.6*cos(30)},{-0.6*sin(30)}) arc (0:60:2) -- ({4.5-0.6*cos(30)},{-0.6*sin(30)});
			\draw[ultra thick]  ({6.5-0.6*cos(30)},{-0.6*sin(30)}) arc (0:60:2);
			\draw ({1.5*cos(30)},{1.5*sin(30)}) node {$N_t$};
			\draw ({2.5*cos(30)+4.5},{2.5*sin(30)}) node {$N_{t, 0}$};
			\draw ({0.9*cos(30)+4.5},{0.9*sin(30)}) node {$N_{t, \lambda}$};
			\draw ({1.7*cos(30)+4.5},{1.7*sin(30)}) node {$X$};
		\end{tikzpicture}
		\caption{The cutting construction on $N_t$. The face with the boundary condition $B^\perp$ is highlighted by thick line.}	\label{fig:decomposition}
	\end{figure}
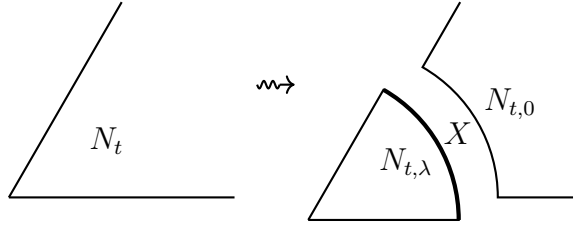
	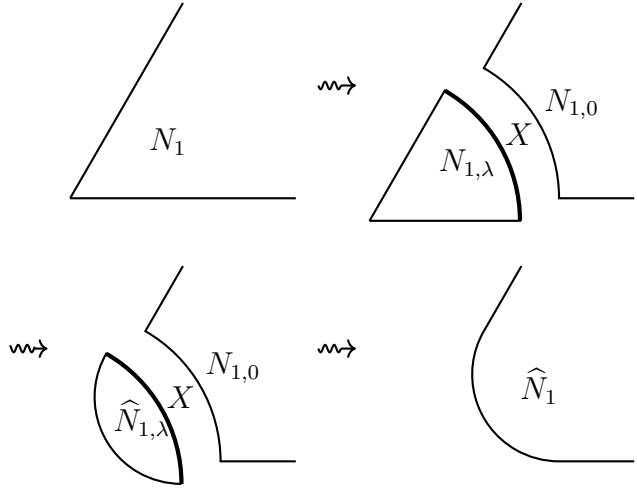
\begin{figure}
		\begin{tikzpicture}[scale=1]
			\draw[thick] (0,0) -- (3,0);
			\draw[thick] (0,0) -- ({3*cos(60)},{3*sin(60)});
			\draw[smooth,domain=3.3:3.6, thick] plot({\x},{0.05*sin(3000*\x)+1.5});
			\draw[->, thick] (3.6,1.5) -- (3.8,1.5);
			\draw[thick] ({3*cos(60)+4.5},{3*sin(60)}) -- ({2*cos(60)+4.5},{2*sin(60)}) arc (60:0:2) -- ({7.5},0);
			\draw[thick] ({4.5-0.6*cos(30)},{-0.6*sin(30)}) -- ({6.5-0.6*cos(30)},{-0.6*sin(30)}) arc (0:60:2) -- ({4.5-0.6*cos(30)},{-0.6*sin(30)});
			\draw[ultra thick]  ({6.5-0.6*cos(30)},{-0.6*sin(30)}) arc (0:60:2);
			\draw[smooth,domain=-0.8:-0.47, thick] plot({\x},{0.05*sin(3000*\x)-2});
			\draw[->,thick] (-0.47,1.5-3.5) -- (-0.3,1.5-3.5);
			\draw[thick] ({3*cos(60)},{3*sin(60)-3.5}) -- ({2*cos(60)},{2*sin(60)-3.5}) arc (60:0:2) -- ({3},-3.5);
			\draw[thick]({2-0.6*cos(30)},{-0.6*sin(30)-3.5}) arc (0:60:2) arc (-210:-90:{2/sqrt(3)});
			\draw[ultra thick]({2-0.6*cos(30)},{-0.6*sin(30)-3.5}) arc (0:60:2);
			\draw[smooth,domain=3.3:3.6, thick] plot({\x},{0.05*sin(3000*\x)-2});
			\draw[->,thick] (3.6,-2) -- (3.8,-2);
			\draw[thick] ({3*cos(60)+4.5},{3*sin(60)-3.5}) -- ({2*cos(60)+4.5},{2*sin(60)-3.5}) arc (-210:-90:{2/sqrt(3)}) -- (7.5,-3.5);
			\draw ({1.5*cos(30)},{1.5*sin(30)}) node {$N_1$};
			\draw ({2.5*cos(30)+4.5},{2.5*sin(30)}) node {$N_{1, 0}$};
			\draw ({0.9*cos(30)+4.5},{0.9*sin(30)}) node {$N_{1, \lambda}$};
			\draw ({1.7*cos(30)+4.5},{1.7*sin(30)}) node {$X$};
			\draw ({2.5*cos(30)},{2.5*sin(30)-3.5}) node {$N_{1, 0}$};
			\draw ({1.1*cos(30)},{1.1*sin(30)-3.5}) node {$\widehat N_{1, \lambda}$};
			\draw ({1.7*cos(30)},{1.7*sin(30)-3.5}) node {$X$};
			\draw ({2*cos(30)+4.5},{2*sin(30)-3.5}) node {$\widehat N_1$};
		\end{tikzpicture}
		\caption{From $N_1$ to $\widehat N_1$ by cutting and pasting. The faces with boundary condition $B^\perp$ are highlighted by thick lines.}	\label{fig:cuttingAndPasting}
	\end{figure}
	
	Recall that $N_{t, 0}$  is the complement of $ \bigcup_{k=2}^n\bigcup_{ |\lambda|=k}N_{t,\lambda}$. In particular, $N_{t, 0}$ is a manifold with corners with all dihedral angles equal to $\pi/2$, for all $t\in [1-\delta, 1]$. Similarly,  $M_{t, 0}$ is also a manifold with corners with all dihedral angles equal to $\pi/2$, for all $t\in [1-\delta, 1]$. It follows from Remark \ref{remark:smoothEqual} and the usual deformation argument (cf. the proof of line \eqref{eq:deformation-equal}) that $\ind((D_{t,0})_{B_{t,0}})$ is constant for all $t\in[1-\delta,1]$. By  line \eqref{eq:sumOfIndex} and Claim \ref{claim:vanish}, we see that $\ind((D_t)_{B_t})$ is constant for all $t\in[1-\delta,1]$. Combined with line \eqref{eq:deformation-equal},  we have shown that $\ind((D_t)_{B_t})$ is constant for all $t\in[0,1]$.

	It remains to  compute the Fredholm index of $(D_1)_{B_1}$. We shall replace each $N_{1, \lambda}$ by a new manifold with corners  $\widehat N_{1,\lambda}$ such that,  when we glue $\widehat N_{1,\lambda}$ back to $N_{1,0}$ inductively, we obtain a manifold with smooth boundary, while not changing the Fredholm index of the Dirac operator. More precisely, in the codimension two case, $N_{1, \lambda} $ is of the form $W_\lambda \times \fiber^2$ where $\fiber^2 $ is a sector in $\R^2$. In this case, we define $\widehat N_{1,\lambda}$  to be $W_\lambda \times \mathcal E_2$, where $\mathcal E_2$ is a manifold with corners that is ``eye-shaped" as shown in Figure \ref{fig:cuttingAndPasting}. In other words, $\mathcal E_2$ is obtained from $\fiber^2$ by smoothing out the vertex. Recall that the faces $X_{1, \eta}$ that were created by the cutting in  the codimension three step above is fiberwise a  two dimensional spherical polygon.    Now we glue all these new $N_{1, \lambda}$'s with $|\lambda| =2$ back to $N_{1, 0}$ along the original cuts $X_{1, \lambda}$. This process also smoothes out the vertices of the above spherical polygons $X_{1, \eta}$. In other words, $X_{1, \eta}$ has been replaced fiberwise by a convex region $\widehat {\mathbb S}_{1, \eta}$  with smooth  boundary in $\mathbb S^2$.  Now $N_{1, \eta}$ with $|\eta|=3$ is of the form $W_\eta\times \fiber^3$, where  $\fiber^3$ is a corner in $\mathbb R^3$. We then define $\widehat N_{1,\eta}$  to be $W_\eta \times \mathcal E_3$, where $\mathcal E_3$ is the convex region in $\mathbb R^3$ is enclosed by  $\widehat {\mathbb S}_{1, \eta}$  and another smooth convex hypersurface that intersect $\widehat {\mathbb S}_{1, \eta}$ at a dihedral angle of $\pi/2$. In other words, $E_3$ is enclosed by two smooth hypersurfaces that intersecting orthogonally  along a common smooth curve.  Again,  $\mathcal E_3$ is an ``eye-shaped" convex region in $\mathbb R^3$. Now we repeat the same construction inductively until we replace all $N_{1, \lambda}$ by $\widehat N_{1, \lambda}$ so that the resulting new space $\widehat N_1$ obtained by gluing $\widehat N_{1, \lambda}$ back to $N_{1, 0}$ is a manifold with smooth boundary: 
	\[\widehat N_1\coloneqq N_{1,0}\cup \bigcup_{k=2}^n\bigcup_{\ |\lambda|=k}\widehat N_{1,\lambda}\]
	See Figure \ref{fig:cuttingAndPasting} for an illustration of the dimension two case. Similarly, we replace $M_{1, \lambda}$ by new manifolds with corners $\widehat M_{1,\lambda}$ such that 
	\[\widehat M_1\coloneqq M_{1,0}\cup \bigcup_{k=2}^n\bigcup_{\ |\lambda|=k}\widehat M_{1,\lambda}\]
	is a manifold with smooth boundary. Without loss of generality, we may assume each  $\widehat M_{1,\lambda}$ is of product structure and its fiber is isometric to the fiber of the corresponding $\widehat N_{1,\lambda}$. We extend the map $f_1\colon N_{1,0} \to M_{1, 0}$ to  a smooth map $\widehat f_1\colon \widehat N_1\to \widehat M_1$ that is  of product structure near all faces. The boundary conditions $\widehat B_{1,\lambda}$ for $\widehat  N_{1,\lambda}$ are defined by an induction process similar to how the boundary $B_{1, \lambda}$ for $N_{1, \lambda}$ was introduced above.  Since the dihedral angles of both $\widehat N_{1,\lambda}$ and $\widehat M_{1,\lambda}$ are  $\pi/2$, it follows from the proof of Theorem \ref{thm:ess-sa} and Lemma \ref{lemma:essensa-jumpanglewithfmixed}  that the Dirac operator $\widehat D_{1, \lambda}$ on $S_{\widehat N_{1,\lambda}}\otimes \widehat f_1^\ast S_{\widehat M_{1,\lambda}}$ subject to the boundary condition $\widehat B_{1, \lambda}$ is  essentially self-adjoint with domain $H^1(\widehat N_{1,\lambda}, S_{\widehat N_{1,\lambda}}\otimes \widehat f_1^\ast S_{\widehat M_{1,\lambda}}; \widehat B_{1, \lambda})$. By applying the gluing formula from Proposition \ref{prop:gluing} again, we have
	\[\ind((\widehat D_1)_{\widehat B_1})=\ind((D_{1,0})_{B_{1, 0}})+\sum_{k=2}^n\sum_{\ |\lambda|=k}\ind((\widehat D_{1,\lambda})_{\widehat B_{1, \lambda}})\]
	
	Recall that the fibers of $\widehat N_{1,\lambda}$ and $\widehat M_{1,\lambda}$ are an eye-shaped convex region in $\mathbb R^{k}$ enclosed by two smooth hypersurfaces that are diffeomorphic to a hemisphere and intersect orthogonally  along their boundaries. We may deform  $\widehat N_{1,\lambda}$, $\widehat M_{1,\lambda}$  and the map $\widehat f_1$ to a standard model where both $\widehat N_{1,\lambda}$ and $\widehat M_{1,\lambda}$ become fiberwise the standard Euclidean unit half-ball in $\mathbb R^k$, and fiberwise $\widehat f_1 \colon \widehat N_{1,\lambda}\to \widehat M_{1,\lambda}$ becomes the identity map, while keeping the dihedral angles always at $\pi/2$ through this deformation.   It follows from Remark \ref{remark:smoothEqual} and the usual deformation argument (cf. the proof of line \eqref{eq:deformation-equal}) that $\ind((\widehat D_{1,\lambda})_{\widehat B_{1, \lambda}}) =0$,  for all $|\lambda| \neq 0$. Therefore,  we have 
	$$\ind((D_1)_{B_1})=\ind((D_{1,0})_{B_{1, 0}})=\ind((\widehat D_1)_{\widehat B_1}).$$
    Since  $\widehat N_1$ and $\widehat M_1$ are manifolds with smooth boundary, it is a classical result that in this case 
	$$\ind(\widehat D_1)=\deg_{\widehat A}(\widehat f_1)\cdot\chi(\widehat M_1)=\deg_{\widehat A}( f)\cdot\chi(M).$$
	This finishes the proof of the theorem under the slightly stronger assumption that the comparison of dihedral angles is strict along all codimension two faces. 
	
	Now suppose the dihedral angles satisfy the equality  condition \eqref {eq:equaldihedral} along a certain codimension two face. Then  we simply keep the dihedral angles to satisfy the equality  condition  \eqref {eq:equaldihedral} along this given codimension two face throughout the construction of  the deformation $f_t\colon N_t \to M_t$ above, which is always achievable for manifolds with corners. The rest of the proof proceeds the same as above. This finishes the proof of the theorem.

\end{proof}

\begin{remark}\label{remark:poly-deform}
A key ingredient in    \textbf{Step I} of  the proof of Theorem \ref{thm:index} is the existence  of a continuous deformation of convex spherical polyhedra while increasing all corresponding dihedral angles. While such a deformation clearly exists in the case of manifolds with corners, it is currently an open question whether such a deformation exists in the case of general manifolds with polyhedral boundary, cf. Gromov's $\angle$-shrinking conjecture \cite[Section 7]{Gromov:2022tr}. In other words, if Gromov's $\angle$-shrinking conjecture holds in its complete generality, then Theorem \ref{thm:index} and its proof would also hold for  manifolds with polyhedral boundary. Note that,  in the case of two dimensional convex spherical polyhedra,  there always exist continuous deformations of convex spherical polyhedra that increase all corresponding dihedral angles as those in \textbf{Step I} of the proof of Theorem \ref{thm:index} (cf. Appendix \ref{sec:appendixdeform} for more details).  As a consequence, the same proof shows that Theorem \ref{thm:index} holds for manifolds with polyhedral boundary whose top codimension of singularities is $\leq 3$ (e.g. all three dimensional manifolds with polyhedral boundary).  
\end{remark}

It follows from the above discussion in Remark \ref{remark:poly-deform} that the proof of   Theorem \ref{thm:index} also applies to some special cases of manifolds with polyhedral boundary in higher dimensions. For example, we have the following proposition.

\begin{proposition}\label{prop:index-special}
	Let $(M, g)$ and $(N, \overbar{g})$ be two oriented compact Riemannian manifolds with polyhedral boundary.  Suppose   $f\colon N \to M$ is a spin polytope map.
For a point $x$ in a codimension $k$ face of $N$, let $\{\overbar v_1, \overbar v_2, \cdots, \overbar v_\ell\}$ be the unit inner normal vectors of all codimension one faces of $N$ that passing through $x$. Let $v_1,  v_2, \cdots,  v_\ell$ be the corresponding unit inner normal vectors passing through $f(x) \in M$.   Assume that\footnote{In contrast to the comparison conditions on dihedral angles in \eqref{eq:strictdihedral} and \eqref{eq:equaldihedral}, the  condition \eqref{eq:all-angle} is a much stronger geometric assumption in the sense that not only the corresponding dihedral angles of $N$ and $M$ coincide, but also all other angles (including angles between intersecting faces that are not adjacent) of $N$ and $M$ coincide. } 
\begin{equation}\label{eq:all-angle}
 \langle \overbar v_i, \overbar v_j\rangle = \langle v_i, v_j\rangle
\end{equation}
for all $1\leq i, j\leq \ell$, all $x$ in any codimension $k$ face of $N$ for all $k$.  
		Let $D_B$ be the Dirac operator on $S_N\otimes f^*S_M$ subject to the local  condition $B$ given in Definition \ref{def:boundarycondition}, and $\overbar{D}_B$  the unique self-adjoint extension of $D_B$ as in Theorem $\ref{thm:ess-sa}$. Then the Fredholm index of $\overbar{D}_B$ is 
	\[ \ind(\overbar{D}_B)=\deg_{\widehat A}(f)\cdot \chi(M),\]
	where $\deg_{\widehat A}(f)$ is $\widehat A$-degree of $f$ and $\chi(M)$ is Euler characteristic of $M$. 
\end{proposition}
\begin{proof}
Condition \eqref{eq:all-angle} ensures that the corresponding  fiberwise (asymptotic) polyhedral corners $\fiber$ and $\mathbb G$ that appear in $N$ and $M$ (as in \textup{Step I} of the proof of Theorem \ref{thm:index}) are in fact isometric.  Therefore, we may deform $N$ and $M$ simultaneously  to a locally direct product case  and directly apply the cutting-and-pasting argument as in the proof of Theorem \ref{thm:index}. This finishes the proof. 
\end{proof}

\section{An index theorem for  manifolds with polyhedral boundary} \label{sec:index-poly}

In this section, we prove our main index theorem for compact manifolds with polyhedral boundary in all dimensions. As pointed out above, a main difficulty of generalizing the method of proof in Section \ref{sec:fredindex}  to the general case of manifolds with polyhedral boundary is that it is  currently still an open question whether we may continuously deform a general convex spherical polyhedron while increasing all of its dihedral angles, cf. Gromov's $\angle$-shrinking conjecture \cite[Section 7]{Gromov:2022tr}.  Roughly speaking, in terms of linear algebra,  the main difficulty of finding such a continuous deformation arises from the fact that there are too many linear relations among a given set of vectors. We do not attempt to directly solve   this deformation issue \emph{geometrically} for general  manifolds with polyhedral boundary. Instead, we consider the associated \emph{linear-algebraic} problem and resolve it  by introducing auxiliary space dimensions. The main theorem of this section is the following. 
\begin{theorem}\label{thm:index-poly}
	Let $(M, g)$ and $(N, \overbar{g})$ be two oriented compact Riemannian manifolds with polyhedral boundary.  Suppose   $f\colon N \to M$ is a spin polytope map.	Let $D_B$ be the Dirac operator on $S_N\otimes f^*S_M = S_{TN\oplus f^\ast TM}$ subject to the local  condition $B$ given in Definition \ref{def:boundarycondition}. 
	For each pair of  codimension one faces $\overbar F_i, \overbar F_j$ of $N$, assume  the dihedral angles $\theta_{ij}(\overbar{g})$ of $N$ and $\theta_{ij}(g)$ of $M$ satisfy
	\begin{equation}\label{eq:strictdihedral-poly}
		0<\theta_{ij}(\overbar{g})_z < \theta_{ij}(g)_{f(z)} < \pi,  \textup{ for all }   z \in \overbar F_i\cap \overbar{F}_j.
	\end{equation}  
  Let $\overbar{D}_B$ be the unique self-adjoint extension of $D_B$ as in Theorem $\ref{thm:ess-sa}$. Then the Fredholm index of $\overbar{D}_B$ is 
	\[ \ind(\overbar{D}_B)=\deg_{\widehat A}(f)\cdot \chi(M),\]
	where $\deg_{\widehat A}(f)$ is $\widehat A$-degree of $f$ and $\chi(M)$ is Euler characteristic of $M$. 
\end{theorem}
The proof of the above theorem will be given in Section \ref{sec:fredholm-index-poly}. We first need some preparation.  
			
\subsection{Essential self-adjointness of the de Rham operator on model conical spaces with respect to more general local boundary conditions}\label{sec:ess-sa-general} 

In this subsection, we shall prove the essential self-adjointness of the de Rham operator on some model conical spaces subject to certain local boundary conditions that are more general than those given in  Definition \ref{def:boundarycondition}. Roughly speaking, while the local boundary conditions in Definition \ref{def:boundarycondition} are induced by the geometry of the polyhedral corners in  $N$ and $M$, we shall deal with more general local boundary conditions that are somewhat more algebraic. In particular, we may encounter local boundary conditions that are \emph{irrelevant} to the geometry of the polyhedral corners of the target space $M$. Such a flexibility will be important in the proof of our index theorem for twisted Dirac operators on compact manifolds with polyhedral boundary.

Much of the same discussion from Section \ref{sec:conic} carries over to the more general case we are dealing with in the current section.  In order to streamline the discussion of this section, we shall first simplify the notation from Section \ref{sec:conic}. In particular,  we shall first review Lemma \ref{lemma:essensa-higherdim} and its proof, and simplify the notation so that our discussion of the more general case becomes more transparent.

\begin{lemma}[cf. Lemma \ref{lemma:essensa-higherdim}] \label{lemma:ess-saNew} 
	Let $\fiber$ and $\mathbb G$ be two convex polyhedral corners in  $\R^n$ that are enclosed by $\ell$ hyperplanes through the origin respectively, where $\ell \geq n$.
	Let $B$ be the boundary condition on $\Bigwedge^*\fiber =  \Bigwedge^*\R^n$ over each codimension one face $\overbar F_i$ given by
	$$\mathscr E(\overbar c(u_i)\otimes c(v_i))\varphi=-\varphi,$$
	where $\mathscr E$ is the even-odd grading operator on $\Bigwedge^\ast \mathbb R^n$, $u_i$'s are the inner normal vectors of $\fiber $ and $v_i$'s are the inner normal vectors of $\mathbb G $.
	Let $D^\dR_{\fiber,B}$ be the de Rham operator acting on $\Bigwedge^*\fiber$ with the boundary condition $B$.  If  the dihedral angles $\alpha_{ij}$ of $\fiber$ are less than or equal to the corresponding dihedral angles  $\beta_{ij}$ of $\mathbb G$, then $D_{\fiber,B}^\dR$ is essentially self-adjoint.
\end{lemma}
\begin{proof}
	The main strategy of the proof is exactly the same as that of Lemma \ref{lemma:essensa-higherdim}. 
	
	The dimension two case (where $n=2$) is proved  exactly the same way as in Lemma \ref{lemma:essensa-jumpanglewithf}. We shall focus on simplifying the notation for the case where $n\geq 3$.  Let $\link$ be the link of $\fiber$ , that is, $\link$ is the intersection of $\fiber$ and the unit sphere of $\R^n$. Let $\R_+  = [0, \infty)$ and $\link \times \R_+$  the direct product space of $\link$ and $\R_+$.  The polar coordinates $(r,\theta)$ of $\R^n$, with $r\in \R_+ = [0, \infty)$ and $\theta\in \sph^{n-1}$, 
	induces an isometry 
	$$\Psi\colon L^2(\link\times\R_+,T(\link\times\R_+))\to L^2(\fiber, T\R^n)$$ given by
	$$\partial_r\mapsto r^{(n-1)/2}\cdot \partial_r \text{ and }v\mapsto r^{(n-1)/2}\cdot \frac 1 r v,$$
	where $r$ is the radial variable of $\fiber$, and $v\in T\link$. The map $\Psi$ induces an isometry
	$$L^2(\link\times\R_+,\Bigwedge^*(\link\times\R_+))\to L^2(\fiber, \Bigwedge^*\R^n),$$
	which we will still denote by $\Psi$,  cf. line \eqref{eq:coordinateChange}.
	
	For simplicity, let us write $D = D^\dR_{\fiber,B}$. Now we consider the operator $\Psi^*D\Psi$ acting on $\Bigwedge^*(\link\times\R_+)$ over $\link\times\R_+$. Similar to line \eqref{eq:deRhamAfterConj}, we have
	$$\Psi^*D\Psi=\overbar c(\partial_r)\partial_r+\sum_{i=1}^{n-1}\overbar c(e_i)\nabla_{e_i}^\link+\sum_{i=1}^{n-1}\overbar c(e_i)\otimes c(\partial_r)c(e_i),$$
	where $\{e_i\}$ is a local orthonormal frame of $T\link$, and $\nabla^\link$ is the canonical connection on $\Bigwedge^*\link $ over $\link$ induced by  the Levi-Civita connection of $\link \subset \sph^{n-1}$.
	
	Let us define 
	\begin{equation}\label{eq:operatoronlink}
		P= \sum_{i=1}^{n-1}\overbar c(e_i)\nabla_{e_i}^\link + \sum_{i=1}^{n-1}\overbar c(e_i)\otimes c(\partial_r)c(e_i).
	\end{equation}
	By Lemma \ref{lemm:>=1/2} and the induction argument in the proof of Lemma \ref{lemma:essensa-higherdim}, in order to prove  the essential self-adjointness of $D$, it suffices to prove the following claim.
	\begin{claim*}
		For any $\varphi\in C^\infty_0(\link,\Bigwedge^*(\link\times\R_+);B)$ (cf. Definition \ref{def:smooth,H1}), we have
		$$\|P\varphi\|\geq \frac 1 2\|\varphi\|.$$
	\end{claim*}
   We define 
	$$\widehat\nabla_X\coloneqq \nabla^\link_X+1\otimes c(\partial_r)c(X)$$
	to be a new connection on the bundle  $\Bigwedge^*(\link\times\R_+)$ over $\link$, where $X\in T\link$. We view  $P$ as a Dirac-type operator associated to this new  connection $\widehat\nabla$. The same computation from Claim \ref{claim:>=1/2} shows that 
	$$P^2=\widehat\nabla^*\widehat\nabla+\frac{(n-1)(n-2)}{4}.$$
	By the Stokes' formula (cf. line \eqref{eq:deRhamestimate}), we obtain
	\begin{equation}
		\begin{split}
			\int_{\link}|P\varphi|^2=\int_\link |\widehat\nabla\varphi|^2+\int_\link\frac{(n-1)(n-2)}{4}|\varphi|^2
			+\sum_{i, j}\int_{\link_j}\langle\overbar c(u_j)\overbar c(e_i^j)\widehat\nabla_{e_i^j}\varphi,\varphi\rangle.
		\end{split}
	\end{equation}
	where $\link_j=\link\cap\overbar F_j$ is a codimension one face of $\link$, $u_j$ is the unit inner normal vector of $\link_j$, and $\{e^j_i\}$ is a local orthonormal basis of tangent vectors of $\link_j$.
	
	Recall that the boundary condition $B$ on $\link_j$ is given by
	$$\mathscr E(\overbar c(u_j)\otimes c(v_j))\varphi=-\varphi.$$
	It is easy to see that $\mathscr E$ commutes with the operator $\overbar c(u_j)\overbar c(e_i^j)\widehat\nabla_{e_i^j}$ on $\link_j$. Since $\link_j$ has zero mean curvature in $\link$, it follows that $\overbar c(u_j)$ commutes with $\widehat\nabla_{e_i^j}$. Therefore $\overbar c(u_j)$ anti-commutes with $\overbar c(u_j)\overbar c(e_i^j)\widehat\nabla_{e_i^j}$.
	
	By assumption, $v_j$ is a constant vector in $\R^n$. Therefore we have
	$$[\nabla_X,1\otimes c(v_j)]=0$$
	where $\nabla_X$ is the standard flat connection on $\Bigwedge^*\R^n$ over $\mathbb R^n$.
By the classical Gauss-Codazzi equations, we have 
	$$\nabla_X=\nabla^\link_X+\frac 1 2 \overbar c(\partial_r)\overbar c(X)\otimes 1+\frac 1 2\otimes c(\partial_r) c(X)=\widehat\nabla_X+\frac 1 2 \overbar c(\partial_r)\overbar c(X)\otimes 1$$
It follows that 
	$$[\widehat\nabla_X,1\otimes c(v_j)]=-\frac 1 2\big[\overbar c(\partial_r)\overbar c(X)\otimes 1,1\otimes c(v_j)\big]=0.$$
	To summarize, we have shown that $\mathscr E(\overbar c(u_j)\otimes c(v_j))$ anti-commutes with the boundary operator $\overbar c(u_j)\overbar c(e_i^j)\widehat\nabla_{e_i^j}$. Therefore, as long as $\varphi$ satisfies the boundary condition $B$,  we have 
	$$\int_{\link}|P\varphi|^2=\int_\link |\widehat\nabla\varphi|^2+\int_\link\frac{(n-1)(n-2)}{4}|\varphi|^2.$$
	It follows that 
	$$\|P\varphi\|\geq\frac{\sqrt{(n-1)(n-2)}}{2}\|\varphi\|> \frac 1 2\|\varphi\|$$
	as $n\geq 3$.
\end{proof}

The computation of Lemma \ref{lemma:ess-saNew} leads us to the following key observation (cf. Lemma \ref{lemma:ess-sa-innerproductcomparison} and the discussion before it) . Although the local boundary condition $B$ in Lemma \ref{lemma:ess-saNew} involves the unit inner normal vectors  $v_j$ of the faces of $\mathbb G$, the computation of Lemma \ref{lemma:ess-saNew} shows that the fact  whether the vectors $v_j$ come from geometry (i.e.,   whether the vectors $v_j$ are the unit inner normal vectors of the faces of some polyhedral corners) is \emph{not} essential at all. The two key properties of the set of vector fields $\{v_j\}$ that were used in the proof of Lemma \ref{lemma:ess-saNew} are the following: 
\begin{enumerate}[label=(\roman*)]
	\item each $v_j$ is flat; more precisely, on the codimension one face  $\overbar F_j$ of $\fiber$, the covariant derivative of $v_j$ along each tangent direction of $\overbar F_j$  is zero  (or more generally  asymptotically zero, that is, the covariant derivatives  of $v_j$ become zero as we approach the singularities of $\fiber$);  
	\item the dihedral angles of $\fiber$ are $\leq \pi$, and 
	$$\langle v_i,v_j\rangle\geq \langle u_i, u_j\rangle$$
	for any two adjacent faces $\overbar F_i$ and $\overbar F_j$, where $u_i, u_j$ are the inner normal vectors of $\overbar F_i$ and $\overbar F_j$. 
\end{enumerate}
In particular, the exact same proof of Lemma  \ref{lemma:ess-saNew} implies the following generalization of Lemma \ref{lemma:ess-saNew}.
\begin{lemma}\label{lemma:ess-saNewOneFiber}
	Let $\fiber$ be a convex polyhedral corner in $\R^n$.  For each face $\overbar F_j$ of $\fiber$,  we denote its inner normal vector by $u_j$. Moreover, for each $\overbar F_j$,   let $v_j$ be a unit vector in $\R^n\oplus\R^k$, where $k$ is some fixed integer. Let $B$ be the local boundary condition on sections $\varphi$ of  $\Bigwedge^*(\R^n\times\R^k)$  given by
	$$\mathscr E(\overbar c(u_j)\otimes c(v_j))\varphi=-\varphi$$
	at each  $\overbar F_j$, 
	where $\mathscr E$ is the even-odd grading operator on $\Bigwedge^*(\R^n\times\R^k)$. Then for $n\geq 3$, the operator (cf. line \eqref{eq:operatoronlink}) 
	$$P= \sum_{i=1}^{n-1}\overbar c(e_i)\nabla_{e_i}^\link + \sum_{i=1}^{n-1}\overbar c(e_i)\otimes c(\partial_r)c(e_i)$$
	acting on $\Bigwedge^*(\link\times\R_+\times\R^k)$ over $\link$ satisfies
	$$\|P\varphi\|\geq\frac{\sqrt{(n-1)(n-2)}}{2}\|\varphi\|$$  for all $\varphi\in C_0^\infty(\link,\Bigwedge^*(\link\times\R_+\times\R^k);B)$. 
	 Furthermore, the de Rham operator acting on the bundle  $\Bigwedge^*(\R^n\times\R^k) =  \Bigwedge^*\R^n \otimes \Bigwedge^* \R^k$ over $\fiber$ subject to the above local  boundary condition $B$ is essentially self-adjoint, provided that
	$$\langle u_i,u_j\rangle\leq \langle v_i,v_j\rangle$$
	for all pairs of adjacent\footnote{Codimension one faces $\overbar F_i$ and $\overbar F_j$ are said to be adjacent if $\overbar F_i\cap \overbar F_j$ is a nonempty codimension two face of $\fiber$} codimension one faces $\overbar F_i$ and $\overbar F_j$ of $\fiber$.
\end{lemma}

Now we are ready to prove the main conclusion of this subsection.
\begin{theorem}\label{thm:ess-saNew}
	Let $N$ be an $n$-dimensional compact submanifold of $\R^n$ with polyhedral boundary and $Z$ a closed manifold. For  each codimension one face $\overbar F_j\times Z$ of $N\times Z$, let $u_j$ be its unit inner normal vector field. Moreover, on each $\overbar F_j\times Z$, let $v_j$ be a smooth unit vector field in $TN\oplus TZ = \R^n\oplus TZ$. Let $B$ be the boundary condition on sections $\varphi$ of $\Bigwedge^*(N\times Z)$  given by
	$$\mathscr E(\overbar c(u_j)\otimes c(v_j))\varphi=-\varphi$$
	at each $\overbar F_j\times Z$. Assume that the dihedral angles of $N$ are $< \pi$. If we have 
	\begin{equation}
		\langle u_i,u_j\rangle< \langle v_i,v_j\rangle
	\end{equation} 
 for each pair of adjacent codimension one faces $\overbar F_i$ and $\overbar F_j$ of $N$, then  the de Rham operator $D^\dR$ of $N\times Z$ (which acts on the bundle $\Bigwedge^*(N\times Z)$ over $N\times Z$) subject to the above boundary condition $B$ is an essentially self-adjoint Fredholm operator with domain $H^1(N\times Z,\Bigwedge^*(N\times Z);B)$. 
\end{theorem}

Theorem \ref{thm:ess-saNew} should be viewed as a generalization of Theorem \ref{thm:ess-sa}. For simplicity, we have  stated Theorem \ref{thm:ess-saNew}  for  the case where $N$ is a submanifold (with polyhedral boundary) in $\R^n$, but it is clear the theorem and its proof naturally extend to  the case where $N$ is a general manifold with polyhedral boundary, as long as the same conditions on dihedral angles are satisfied.  

However, there are two key differences between Theorem  \ref{thm:ess-saNew} and  Theorem \ref{thm:ess-sa}. The first key difference is that  Theorem  \ref{thm:ess-saNew} allows   more general (and more algebraic) local boundary conditions where the vectors $v_j$  do \emph{not} necessarily arise as the unit inner normal vectors of codimension one faces of some manifold with polyhedral boundary. On the other hand, we emphasize that the vectors $u_i$'s are always assumed to  be the inner normal vectors of codimension one faces of the domain space $N\times Z$. 
 Another key difference between Theorem \ref{thm:ess-saNew} and Theorem \ref{thm:ess-sa} is  that  the inner product $\langle v_i,v_j\rangle$ is allowed to be $1$. In other words,  the ``dihedral angle" between $v_i$ and $v_j$ is allowed to be $\pi$. 
 
 Recall that Lemma \ref{lemma:smoothEquivalence} is a key lemma that allows us to identify the local boundary conditions of nearby points. However, in order to apply  Lemma \ref{lemma:smoothEquivalence}, we need to assume the more restrictive assumption that both the dihedral angles of the domain and the target are strictly less than $\pi$. In particular, if  we have 
 $\langle v_i,v_j\rangle = 1$, then Lemma \ref{lemma:smoothEquivalence} does \emph{not} apply. To remedy this, we prove  the following key technical lemma, which should be viewed as an improvement of both  Lemma \ref{lemma:smoothEquivalence} and Lemma \ref{lemma:sobolevEmbedding}. 
 
 \begin{lemma}\label{lemma:1/rEmbedding}
 	Let $\fiber $ be a sector in $\R^2$ with angle $\alpha<\pi$. Let $B$ be the boundary condition on $\Bigwedge^*\fiber = \Bigwedge^\ast \mathbb R^2$ over each edge of $\fiber$ given by
 $$\mathscr E(\overbar c(u_i)\otimes c(v_i))\varphi=-\varphi$$
 for $i=1,2$, where $\mathscr E$ is the even-odd grading operator on $\Bigwedge^\ast \mathbb R^2$, $u_i$'s are the inner normal vectors of $\fiber$ and $v_i$'s are unit vectors.	Let $D^\dR_{\fiber,B}$ be the de Rham operator acting on $\Bigwedge^*\fiber$ with the boundary condition $B$. Suppose that 
 \begin{equation}\label{eq:strictinnerproduct}
 	\langle v_1,v_2\rangle> \langle u_1,u_2\rangle.
 \end{equation}
Then multiplying by $r^{-1}$ defines a bounded linear operator
\[ H^1(\fiber,\Bigwedge^*\R^2;B)\to L^2(\fiber,\Bigwedge^*\R^2)\]
\[  \varphi\mapsto r^{-1} \varphi.\]

 \end{lemma}
 \begin{proof}
 	Under the unitaries given in line \eqref{eq:transformeven} and \eqref{eq:transformodd}, the de Rham operator $D^\dR$ on $\fiber$ becomes 
 	$$ D=\begin{pmatrix}
 		0&-\frac{\partial}{\partial r}\\ \frac{\partial}{\partial r}&0
 	\end{pmatrix}+\frac 1 r\begin{pmatrix}
 		0&P\\P&0
 	\end{pmatrix}$$
 	where $P$ is the induced operator along the link. In particular, $P$ acts on the forms $\Bigwedge^*\link$ over the link $\link=[0,\alpha]$ subject to a  local  boundary condition induced by $B$. By Lemma \ref{lemma:ess-sa-innerproductcomparison} , we have $|P_B|>1/2$. 
 	
 	To prove the lemma, it suffices to show that there is a $C>0$ such that 
 	$$\big\|\frac 1 r \varphi\big\|^2_{L^2}\leq C(\big\|\varphi\big\|^2_{L^2}+\big\|D\varphi\big\|^2_{L^2}).$$
 	for all $\varphi\in \dom(D_B)= H^1(\fiber,\Bigwedge^*\R^2;B)$.  
 	
 	For each $\varphi\in H^1(\fiber,\Bigwedge^*\R^2;B)$, we denote  $\psi = D\varphi$. Let $\{\phi_k\}$ be an orthonormal basis of $L^2(\link,\Bigwedge^*\link)$ such that  $P\phi_k=\lambda_k\phi_k$. Then we have the orthogonal decompositions
 	$$\varphi=\begin{pmatrix}
 		\sum_{k}\varphi_{k,0}(r)\phi_k \vspace{0.2cm}\\
 		\sum_{k}\varphi_{k,1}(r)\phi_k
 	\end{pmatrix}\text{and }\psi=\begin{pmatrix}
 		\sum_{k}\psi_{k,0}(r)\phi_k \vspace{0.2cm}\\
 		\sum_{k}\psi_{k,1}(r)\phi_k
 	\end{pmatrix}, $$
 where $\varphi_{k, i}$ and $\psi_{k, i}$ are functions on $\R_+ = [0, \infty)$ and the top (resp. bottom) entry of a column vector is a $0$-form (resp. $1$-form). 
 	Since $D\varphi=\psi$, we have 
 	$$-\frac{d}{dr}\varphi_{k,1}+\frac{\lambda_k}{r}\varphi_{k,1}=\psi_{k,0} \quad \text{ and } \quad \frac{d}{dr} \varphi_{k,0}+\frac{\lambda_k}{r}\varphi_{k,0}=\psi_{k,1}.$$
It follows from Lemma \ref{lemma:1/rConical} below  that there exists $C>0$ such that 
 	$$\big\|\frac 1 r \varphi_{k,1}\big\|_{L^2}^2\leq C \big(\big\|\varphi_{k,1}\big\|_{L^2}^2+\big\|\psi_{k,0}\big\|_{L^2}^2\big) \text{ and } \big\|\frac 1 r \varphi_{k,0}\big\|_{L^2}^2\leq C \big(\big\|\varphi_{k,0}\big\|_{L^2}^2+\big\|\psi_{k,1}\big\|_{L^2}^2 \big)$$
 	for all $k$. 
 	This finishes the proof.
 \end{proof}
 \begin{lemma}\label{lemma:1/rConical}
 	Assume that $|\lambda|>1/2$. Let $\mathcal L_\lambda$ be the linear subspace of $L^2(\R_+)$ consisting of functions $\varphi \in  L^2(\R_+)$ such that  
 	$$\psi \coloneqq \varphi'+\frac{\lambda}{r}\varphi \textup{ lies in } L^2(\R_+).$$
 	Then there exists $C>0$ (independent of $\lambda$) such that
 	$$\big\|\frac 1 r\varphi \big\|_{L^2}\leq \frac{C}{|\lambda|-1/2}\sqrt{\big\|\varphi\big\|_{L^2}^2+\big\|\psi\big\|_{L^2}^2}$$	
 	for all  $\varphi\in \mathcal L_\lambda$. 
 \end{lemma}
 \begin{proof}
 	Without loss of generality, we may assume that $\varphi$ and $\psi$ are smooth functions on $(0, \infty)$ such that they vanish on $[2/3, \infty)$ (cf. the proof of Lemma \ref{lemma:asymptoticConical}).  Under such a assumption, for a given $\psi\in L^2(\R_+)$, there is  a unique $L^2$ solution $\varphi$ of the differential equation  
 	\[ \varphi'+\frac{\lambda}{r}\varphi=\psi \] 
 	subject to the condition that $\varphi$ vanishes on  $(2/3,+\infty)$. More precisely,  $\varphi$ has the following explicit expression: 
 	$$\varphi(r)=\begin{cases}
 		\displaystyle\int_0^r\Big(\frac s r\Big)^\lambda \psi(s)ds,& \textup{ if } \lambda>0,\vspace{0.3cm}\\
 		-\displaystyle\int_r^1\Big(\frac s r\Big)^\lambda \psi(s)ds,&\textup{ if } \lambda<0.
 	\end{cases}$$
 	Let us define 
 	$$K_{+}(r,s)=\frac 1 r\chi_{+}(r,s)\Big(\frac s r\Big)^\lambda \text{ and }
 	K_{-}(r,s)=\frac 1 r\chi_{-}(r,s)\Big(\frac s r\Big)^\lambda,$$
 	where $\chi_+$ is the characteristic function on $\{(r,s) : 0\leq s\leq r\leq 1\}$ and $\chi_-$ is the characteristic function on $\{(r,s)  : 0\leq r\leq s\leq 1\}$.
 	
 	To prove the lemma,  it suffices to show that the integral operators induced by the integral kernels $K_\pm$ are bounded on $L^2(\R_+)$. We prove this by the Schur test.
 	
 	If $\lambda>0$, then
 	\begin{align*}
 		\int K_+(r,s)ds=&\int_0^r\frac{s^\lambda}{r^{1+\lambda}}ds=\frac{1}{1+\lambda} 
 	\end{align*}
 for all $r\in [0, 1]$,  and 
 \begin{align*}
 		\int K_+(r,s)dr=&\int_s^1\frac{s^\lambda}{r^{1+\lambda}}dr=\frac{s^\lambda}{ \lambda}(s^{-\lambda}-1)=\frac{1-s^\lambda}{\lambda}\leq \frac 1 \lambda 
 	\end{align*}
 	for all $s\in [0, 1]$. Now it follows from  the Schur test that the integral operator induced by the integral kernel $K_+$ is bounded on $L^2(\R_+)$.
 	
 	Similarly, if $\lambda<0$, then
 	\begin{align*}
 		\int K_-(r,s)s^{-1/2}ds=&\int_r^1\frac{r^{|\lambda|-1}}{s^{|\lambda|+1/2}}ds=
 		r^{|\lambda|-1}\frac{1}{|\lambda|-1/2}(r^{-|\lambda|+1/2}-1)\\
 		=&\frac{1}{|\lambda|-1/2}(r^{-1/2}-r^{|\lambda|-1})\\
 		\leq&\frac{1}{|\lambda|-1/2}r^{-1/2}
 	\end{align*}
 for all $r\in [0, 1]$ and 
 	\begin{align*}
 		\int K_-(r,s)r^{-1/2}dr=&\int_0^s\frac{r^{|\lambda|-3/2}}{s^{|\lambda|}}dr=
 		\frac{1}{s^{|\lambda|}}\left(\frac{1}{|\lambda|-1/2}\right)s^{|\lambda|-1/2}\\
 		=&\frac{1}{|\lambda|-1/2}s^{-1/2}
 	\end{align*}
 for all $s\in [0, 1]$. Again, it follows from  the Schur test that the integral operator induced by the integral kernel $K_-$ is bounded on $L^2(\R_+)$. This finishes the proof. 
 \end{proof}

We remark that the strict comparison condition  \eqref{eq:strictinnerproduct} in Lemma \ref{lemma:1/rEmbedding} cannot be dropped. For example, Lemma \ref{lemma:1/rEmbedding} would fail if we assume $\langle v_1, v_2\rangle = \langle u_1, u_2\rangle$.    Indeed, for example, if $v_1=v_2$ and $u_1=u_2$, then the constant function $1$ lies in $H^1(\fiber,\Bigwedge^*\R^2;B)$, but the function $1/r$ is not $L^2$ integrable near the origin. In other words, the strict comparison condition  \eqref{eq:strictinnerproduct} implies that the elements in  $H^1(\fiber,\Bigwedge^*\R^2;B)$ decay\footnote{The rate of decay is essentially determined by the number	$\langle v_1,v_2\rangle - \langle u_1,u_2\rangle$.}  near the origin. Note that, for higher dimensional convex polyhedral corners, if we assume  the strict comparison condition  \eqref{eq:strictinnerproduct}, then  the obvious analogue of Lemma \ref{lemma:iteratedSobolev} also holds starting from codimension two.   

Now we are ready to prove Theorem \ref{thm:ess-saNew}.
 
\begin{proof}[Proof of Theorem \ref{thm:ess-saNew}]
	The verification of the essential self-adjointness  of $D = D^\dR$ can be localized near each singular point of $N\times Z$. By the same discussion from Section \ref{sec:conic} and Section \ref{sec:ess-selfadj-general} (in particular,  Lemma \ref{lemma:iteratedSobolev} together with Lemma \ref{lemma:1/rEmbedding} and the Kato--Rellich perturbation theorem, cf. the proof of Theorem \ref{thm:ess-sa}), it suffices to prove the essential self-adjointness in the following model case. 
	
	Let $\fiber$ be a polyhedral corner in $\R^n$. Consider the bundle $\Bigwedge^*(\R^{n+k})$ over $\fiber\times\R^k$. For each codimension one  face $\overbar F_j$ of $\fiber$, let $u_j$ be its unit inner normal vector. Moreover, for each $\overbar F_j$, let $v_j$ be a unit vector in $\R^{n+k}$. Suppose that
	$$\langle u_i,u_j\rangle<\langle v_i,v_j\rangle$$
	 for each pair of  adjacent codimension one faces $\overbar F_i$ and $\overbar F_j$ of $\fiber$. On each codimension one face $\overbar F_j\times \R^k$ of $\fiber \times \R^k$, the boundary condition $B$ on sections $\varphi$ of $\Bigwedge^*(\R^{n+k})$ is given by
	$$\mathscr E(\overbar c(u_j)\otimes c(v_j))\varphi=-\varphi,$$
	where $\mathscr E$ is the even-odd grading on $\Bigwedge^*(\R^{n+k})$. Let $D$ be the de Rham operator on $\fiber \times \R^k$ acting on $\Bigwedge^*(\R^{n+k})$. To prove the proposition,  it suffice to prove that $D$ subject to the boundary condition $B$ is essentially self-adjoint near $\{0\}\times\R^k$. 
	
	We retain the same notation from the proof of Lemma \ref{lemma:ess-saNew}. Let 
	\[  \Psi \colon L^2(\link\times\R_+ \times \R^k, \Bigwedge^*(\link\times\R_+\times \R^k))\to L^2(\fiber\times \R^k, \Bigwedge^*(\fiber \times\R^n) ),  \] be the isometry via  polar coordinates as constructed in the proof of Lemma \ref{lemma:ess-saNew}, where $\link$ is the link of $\fiber$.  We have 
	$$\Psi^*D\Psi=\overbar c(\partial_r)\partial_r+P+\sum_{\alpha=1}^k\overbar c(\partial_{x_\alpha })\partial_{x_\alpha},$$
	where $\{x_\alpha\}_{1\leq \alpha\leq k}$ is the coordinates of $\R^k$  and
	$$P=\sum_{i=1}^{n-1}\overbar c(e_i)(\nabla^\link_{e_i}+1\otimes c(\partial_r)c(e_i)).$$
	For any fixed $r\in \R_+$ and fixed coordinates $\{x_\alpha\}_{1\leq \alpha\leq k}$, it follows from Lemma \ref{lemma:ess-saNewOneFiber} that the operator $P$ (as an operator  acting on the bundle  $\Bigwedge^*(\link\times\R_+\times\R^k)$ over the space $\link$  subject to the boundary condition $B$)   is essentially self-adjoint with domain  $H^1(\link,\Bigwedge^*(\link\times\R_+\times\R^k);B)$.
	
	By the von Neumann Theorem,  $D$ is essentially self-adjoint if and only if  the deficiency indices of $D$ are zero, where the latter means $\ker(D_{\max} \pm i) = 0 $. We shall prove that $\ker(D_{\max} \pm i) = 0 $ under the given assumption above. Suppose there exists $\psi \in L^2(\link\times\R_+ \times \R^k, \Bigwedge^*(\link\times\R_+\times \R^k))  $ such that $\psi \in  \ker(D_{\max} + i)$, that is, 
	\begin{equation}\label{eq:(D+imu)phi}
		\Big(\overbar c(\partial_r)\partial_r+P+\sum_{\alpha=1}^k\overbar c(\partial_{x_\alpha})\partial_{x_\alpha}+i\Big)\psi=0.
	\end{equation}
We want to show that $\psi = 0$.

	For any $L^2$-section $\varphi$ of $\Bigwedge^\ast(\R^{n+k})$ over $\fiber\times\R^k$, its vector-valued Fourier transform along $\R^k$ is given by
	$$\widehat\varphi(\xi)=\int_{\R^k}e^{-i\langle x,\xi\rangle}\varphi(x)dx,$$
	which is again an $L^2$-section over $\fiber\times\R^k$ with values in $\Bigwedge^\ast (\R^{n+k})$. Under this Fourier transform, Equation \eqref{eq:(D+imu)phi} becomes 
	\begin{equation}\label{eq:(D+imu)HatPhi}
		\Big(\overbar c(\partial_r)\partial_r+\frac 1 rP+\sum_{\alpha = 1}^k\overbar c(\partial_{x_\alpha})i\xi_\alpha+i\Big)\widehat \psi(\xi)=0
	\end{equation}
where $\xi = (\xi_1, \cdots, \xi_k)\in \R^k$. 
	Let us define 
	$$D_\xi\coloneqq \overbar c(\partial_r)\partial_r+\frac 1 rP+\sum_{\alpha=1}^k\overbar c(\partial_{x_\alpha})i\xi_\alpha+i=\overbar c(\partial_r)\partial_r+\frac{1}{r} P+i\overbar c(\xi)+i.$$
	The operators $\{D_\xi\}_{\xi \in \R^k}$ form a smooth family of differential operators acting on $\Bigwedge^*(\link\times\R_+\times\R^k)$ over $\link\times\R_+$. Solving Equation \eqref{eq:(D+imu)phi} is equivalent to solving 
\begin{equation}\label{eq:fourierdiff}
D_\xi\widehat \psi(\xi)=0
\end{equation}
	for every $\xi\in \R^k$. 
	
For simplicity, let us denote  
	$$\gamma_j\coloneqq \mathscr E(\overbar c(u_j)\otimes c(v_j))$$
	to be the operator used in the construction of the boundary condition $B$ at the codimension one face $\overbar F_j\times Z$. Recall that $u_j$ is the unit inner normal vector of $\overbar F_j$ in $\fiber$. As $\fiber$ is polyhedral corner in $\R^n$, the same vector $u_j$ is also the unit inner normal vector of $\link_j=\link\cap\overbar F_j$ in $\link$. In particular,  $u_j$  is orthogonal to the radial vector $\partial_r$ of $\fiber$, and is also orthogonal to $\partial_{x_\alpha}$ for each $1\leq \alpha\leq k$. It follows that  $\overbar c(\partial_r)$ and $\overbar c(\partial_{x_\alpha})$ commute with $\gamma_j$, hence preserve the boundary condition $B$ as well as the domain of $P$. 
	
	Now let us define 
	$$\widetilde D_\xi\coloneqq -\overbar c(\partial_r)\cdot D_\xi=\partial_r-\frac 1 r\overbar c(\partial_r)P-i\overbar c(\partial_r)\overbar c(\xi)-i\overbar c(\partial_r)$$
and set 
	$$\widetilde P=-\overbar c(\partial_r)P,\quad \mu=\sqrt{|\xi|^2+1} \quad \textup{and}\quad  A=-\frac{i}{\mu}\Big(\overbar c(\partial_r)\overbar c(\xi)+ \overbar c(\partial_r)\Big).$$
So we have 
	$$\widetilde D_\xi=\partial_r+\frac 1 r\widetilde P+\mu A.$$
The operator $\widetilde P$ along $\link$ is formally symmetric. As $\overbar c(\partial_r)$ and $\overbar c(\xi)$ preserve the domain of $P$ and $P$ is essentially self-adjoint, the operator  $\widetilde P$ along $\link$ is  essentially self-adjoint with domain $H^1(\link,\Bigwedge^*(\link\times\R_+\times\R^k);B)$. Moreover, the operator $A$ is a self-adjoint endomorphism that anti-commutes with $\widetilde P$ and satisfies  $A^2=1$. 
	
Note that   $\widetilde{\mathscr E} \coloneqq  \mathscr E \cdot \overbar c(\partial_r)$ is a self-adjoint endomorphism  of $\Bigwedge^*(\link\times\R_+\times\R^k)$ such that $\widetilde{\mathscr E}^2=1$. We have the orthogonal decomposition
	$$\Bigwedge^*(\link\times\R_+\times\R^k)=E\oplus E^\perp,$$
	where $E$ is the subbundle corresponding to the positive eigenspace of  $\widetilde{\mathscr E}$ and $E^\perp$ is the subbundle corresponding to the negative eigenspace of  $\widetilde{\mathscr E}$. Since $A^*=A$, $A^2=1$, and $A$ anti-commutes with $\widetilde{\mathscr E}$, it follows that 
	$$A = \begin{pmatrix}
		0&U^*\\U&0
	\end{pmatrix}$$
	with respect to the above decomposition $E\oplus E^\perp$, where   $U\colon E\to E^\perp$ is some unitary.
	
	Let us identify $\Bigwedge^*(\link\times\R_+\times\R^k)$ with $E\oplus E$ via the unitary $1\oplus U^\ast$. Under this identification,  we have 
	$$ A = \begin{pmatrix}
		0&1 \\ 1 & 0
	\end{pmatrix} \textup{ acting on } E\oplus E. $$
	Note that $\widetilde{\mathscr E}$ preserves the domain of $\widetilde P$ and commutes with $\widetilde P$. Under the identification  $\Bigwedge^*(\link\times\R_+\times\R^k) \xrightarrow[\cong]{1\oplus U^\ast} E\oplus E$, the operator $\widetilde P$ is of the form: 
	$$\widetilde P = \begin{pmatrix}
		Q_1 &0\\0&Q_2
	\end{pmatrix}.$$
	As $\widetilde P$ anti-commutes with $A$, it follows that $Q_2=-Q_1$. For simplicity, let us denote $Q_1$ by $Q$. 
	
	 To summarize, under the identification  $\Bigwedge^*(\link\times\R_+\times\R^k) \xrightarrow[\cong]{1\oplus U^\ast} E\oplus E$, the operator $\widetilde D_\xi$ becomes
	\begin{equation}\label{eq:drQeta}
	\widetilde D_\xi = 	\partial_r+\frac 1 r\begin{pmatrix}
			Q&0\\0&-Q
		\end{pmatrix}+\begin{pmatrix}
			0&\mu\\ \mu&0
		\end{pmatrix},
	\end{equation}
	where $Q$ (subject to the boundary condition $B$) is an essentially self-adjoint operator on $L^2(\link,E)$ with domain $H^1(\link,E;B)$. Moreover, it follows from Lemma \ref{lemma:ess-saNewOneFiber} that  $|Q|\geq 1/2$.
	
	Let $\{\phi_\lambda\}$ be an orthonormal basis of $L^2(\link,E)$ such that $Q\phi_\lambda=\lambda\phi_\lambda$. For each $\xi\in \R^k$, we may write 
	$$\widehat\psi(\xi)=\Big(\sum_\lambda \sigma_{0,\lambda}(r)\phi_\lambda\Big)\oplus \Big(\sum_\lambda \sigma_{1,\lambda}(r)\phi_\lambda\Big),$$
where $\sigma_{0, \lambda}$ and $\sigma_{1, \lambda}$ are real-valued functions on $\R_+$. 	Equation \eqref{eq:fourierdiff}  then splits according to the above eigen-function decomposition and becomes the following family of differential equations: 
	\begin{equation*}
		\left(\frac{d}{dr}+\frac 1 r\begin{pmatrix}
			\lambda&0\\0&-\lambda
		\end{pmatrix}+\begin{pmatrix}
			0&\mu\\\mu&0
		\end{pmatrix}\right) \begin{bmatrix}
			\sigma_{0,\lambda}(r) \vspace{0.1cm}\\ \sigma_{1,\lambda}(r)
		\end{bmatrix}=0, 
	\end{equation*}
or equivalently, 
	\begin{equation}\label{eq:sigmaLambdaDef}
		\begin{cases}
			\displaystyle	\frac{d}{dr} \sigma_{0,\lambda}+\frac \lambda r\sigma_{0,\lambda}=-\mu\sigma_{1,\lambda}, \vspace{.2cm}\\ 
			\displaystyle	- \frac{d}{dr}\sigma_{1,\lambda}+\frac \lambda r\sigma_{1,\lambda}=\mu\sigma_{0,\lambda}.\\
		\end{cases}
	\end{equation}
	If we set
	$$t=\mu r, \quad \psi_{0,\lambda}=\sigma_{0,\lambda}\textup{ and }  \psi_{1,\lambda}=i\sigma_{1,\lambda},$$
	then Equation \eqref{eq:sigmaLambdaDef} becomes
	\begin{equation}
		\begin{cases}
			\displaystyle	\frac{d}{dt}\psi_{0,\lambda}+\frac \lambda t\psi_{0,\lambda}=i\psi_{1,\lambda}, \vspace{.2cm} \\
			\displaystyle	-\frac{d}{dt}\psi_{1,\lambda}+\frac \lambda t\psi_{1,\lambda}=i\psi_{0,\lambda},
		\end{cases}
	\end{equation}
	which is exactly Equation \eqref{eq:psi0,1def} in the proof of Lemma \ref{lemm:>=1/2}. In particular,  the solution of Equation \eqref{eq:sigmaLambdaDef} is given by
	$$\begin{bmatrix}
		\sigma_{0,\lambda}(r) \vspace{0.1cm}\\ \sigma_{1,\lambda}(r)
	\end{bmatrix} =a_1\begin{bmatrix}
		\sqrt{\mu r}K_{\lambda+1/2}(\mu r) \vspace{0.1cm}\\
		-\sqrt{\mu r}K_{\lambda-1/2}(\mu r)
	\end{bmatrix}+a_2\begin{bmatrix}
		\sqrt{\mu r}I_{\lambda+1/2}(\mu r) \vspace{0.1cm} \\
		-\sqrt{\mu r}I_{\lambda-1/2}(\mu r)
	\end{bmatrix}$$
for some constants $a_1$ and $a_2$, where $I_\nu$ and $K_\nu$ are modified Bessel functions of the first and the second kind, respectively. If $|\lambda|\geq 1/2$,  such a solution does not lie in $L^2(\R_+)$ unless both $a_1$ and $a_2$ are zero. This shows that $\ker(D_{\max} + i) = 0$. The proof for $\ker(D_{\max} - i) = 0$  is completely similar. This completes the proof of the essential self-adjointness of $D^\dR$ (subject to the boundary condition $B$). 

As for the domain of $D^\dR_B$, note that  Proposition \ref{prop:norm-equivalent} (with the same proof) still holds for the more general boundary condition $B$ of the present case. It follows that the domain of $D^\dR_B$ is $H^1(N\times 
Z,\Bigwedge^*(N\times Z);B)$. By the Rellich lemma, the inclusion $H^1(N\times 
Z,\Bigwedge^*(N\times Z);B) \hookrightarrow L^2(N\times 
Z,\Bigwedge^*(N\times Z))$ is a compact operator. It follows that $D^\dR_B$ is Fredholm. This finishes the proof. 
\end{proof}

\subsection{Fredholm index of twisted Dirac operators on manifolds with polyhedral boundary}\label{sec:fredholm-index-poly}
 In this subsection, we prove the main theorem of this section: an index theorem for compact manifolds with polyhedral boundary (Theorem \ref{thm:index-poly}). 

We first need a generalization of Lemma \ref{lemma:solutionapproximate}. Let $(N,\overbar g)$ be a Riemannian manifold with polyhedral  boundary and $E$ be a Clifford bundle over $N$. Let $\nabla$ be a connection on $E$ and $D$ the corresponding Dirac operator. Let $B$ be a local boundary condition on the sections of $E$ at the codimension one faces of $N$. Then $D$ satisfies  the following Bochner-type formula
\begin{equation}\label{eq:homogenousIneq-poly}
	\int_N|D\psi|^2=\int_N |\nabla\psi|^2+\int_N \langle R\psi,\psi\rangle+\int_{\partial N}\langle A\psi,\psi\rangle
\end{equation}
for all $\psi\in C_0^\infty(N,E;B)$ (cf. Definition \ref{def:smooth,H1}), 
where $R$ and $A$ are some endomorphisms of $E$ determined by the curvature of $\nabla$ (cf. Section \ref{sec:boundaryconditions}, in particular, Proposition \ref{prop:D^2} for an example that is the most relevant to the main results of the paper).

\begin{lemma}\label{lemma:solutionapproximate-poly}
	With the notation as above, assume both $R$ and $A$ uniformly bounded and  pointwise non-negative. Moreover, Suppose there is  a sequence of connections $\prescript{k}{}\nabla$ and local boundary condition $B_k$ on $E$ such that the associated Dirac operators $D_k$ satisfy 
	\begin{equation}\label{eq:homogenousIneqforn}			
		\int_N|D_k\psi|^2=\int_N |\prescript{k}{}\nabla\psi|^2+\int_N \langle R_k\psi,\psi\rangle+\int_{\partial N}\langle A_k\psi,\psi\rangle,
	\end{equation}
	for $\psi\in C_0^\infty(N,E;B_k)$, where $R_k$ and $A_k$ are endomorphisms of $E$ determined by the curvature of $\prescript{k}{}\nabla$.  Assume that  
	\begin{enumerate}[label=$(\arabic*)$]
		\item  as $k\to \infty$, $\prescript{k}{}\nabla\to \nabla$ and $D_k\to D$ in the  operator norm topology as bounded operators from $H^1(N,E)$ to $L^2(N,E)$,
		\item  as $k\to \infty$,  $R_k\to R$ in the  operator norm topology  as bounded operators from $H^1(N,E)$ to $L^2(N,E)$,
		\item  as $k\to \infty$,  $A_k\to A$ in the  operator norm topology  as bounded operators from $H^{1/2}(\partial N,E)$ to $L^2(\partial N,E)$,
		\item  as $k\to \infty$, $B_k\to B$ in the sense that the orthogonal projection associated to $B_k$ converges uniformly to the orthogonal projection associated to $B$,
		\item for each $k$, there exists a nonzero $\varphi_k\in H^1(N,E;B_k)$ such that $D_k\varphi_k=0$.
	\end{enumerate}  
	Then there exists a nonzero $\varphi\in H^1(N,E;B)$ such that  $D\varphi=0$ and $\nabla\varphi=0$. 	
\end{lemma}

\begin{proof}
	We assume that $\|\varphi_k\|=1$. Since $R$ and $A$ are pointwise nonnegative, it follows from the assumptions (1)--(3)  that  there exists a sequence of positive numbers $\{\varepsilon_k\}$ such that $\varepsilon_k \to 0$,  as $k\to \infty$,  and 
	\begin{align*}
		\int_N\big|\prescript{k}{}{\nabla}\varphi_k\big|^2 & = \int_N\big|D_k\varphi_k\big|^2 - \int_N \langle R_k\varphi_k,\varphi_k\rangle - \int_{\partial N}\langle A_k\varphi_k,\varphi_k\rangle \\
		&  \leq \int_N \langle (-R_k+R)\varphi_k,\varphi_k\rangle - \int_{\partial N}\langle (-A_k+A)\varphi_k,\varphi_k\rangle\\
		&\leq \varepsilon_k(\|\varphi_k\|^2+\|\prescript{k}{}\nabla\varphi_k\|^2)
	\end{align*}
	It follows that
	$$\|\prescript{k}{}\nabla\varphi_k\|\leq \sqrt{\frac{\varepsilon_k}{1-\varepsilon_k}}.$$
	Hence the $H^1$-norm of $\varphi_k$ is uniformly bounded for all $k$. By the Rellich lemma, there is a subsequence of $\{\varphi_k\}$ that converges in $L^2(N,E)$. Without loss of generality, we may assume that $\{\varphi_k\}$ converges to $\varphi$ in $L^2(N,E)$. It follows that $\|\varphi\|=1$. In particular,  $\varphi$ is nonzero. 
	
	Note that the $L^2$-norm of ${\nabla}\varphi_k$ also converges to zero since $\prescript{k}{}\nabla\to\nabla$. 
	Therefore, $\{\varphi_k\}$ forms a Cauchy sequence in $H^1(N,E)$. Consequently,  $\varphi_k$ converges  to  $\varphi$ with respect to the $H^1$-norm. In particular, we have 
	\[ \int_{N}|\nabla\varphi|^2 =  \lim_{k\to \infty} \int_{N}|\prescript{k}{}\nabla\varphi_k|^2 = 0,\]
	It follows that $\nabla \varphi =0$, hence $D\varphi = 0$. By the trace theorem for Sobolev spaces, we also have $\varphi_k|_{\overbar F} \to \varphi|_{\overbar F}$ in $H^{1/2}(\overbar F, E)$ for each codimension one face $\overbar F$ of $N$. Since each $\varphi_k$ satisfies the local boundary condition $B_k$ and $B_k\to B$ as $k\to \infty$, it follows that $\varphi$ satisfies the local boundary condition $B$. This finishes the proof. 
	
\end{proof}
We emphasize the assumption on the pointwise non-negativity of $R$ and $A$ is essential for Lemma \ref{lemma:solutionapproximate-poly} to hold. If $D_B$ only satisfies the usual G\r{a}rding's inequality as in Proposition \ref{prop:norm-equivalent}, then we would only be able to obtain a convergence in $L^2$ but not in $H^1$ in general. We also remark that Lemma \ref{lemma:solutionapproximate-poly} does not require the essential self-adjointness of $D$ or $D_k$. By the same argument, we also have the obvious analogue of Corollary \ref{coro:nosolution} for manifolds with polyhedral boundary.

Now we prove the main theorem of this section (Theorem \ref{thm:index-poly}): an index theorem for twisted Dirac operators on manifolds with polyhedral boundary in all dimensions. 

\begin{proof}[Proof of Theorem \ref{thm:index-poly}] 
	The main strategy of the proof is similar to that of Theorem \ref{thm:index}. Roughly speaking, we will first  ``deform" $N$ and $M$  (and their Riemannian metrics) so that they have almost the same geometry of their fiberwise polyhedral corners,   while keeping the Fredholm index of the associated twisted Dirac operator unchanged, and then compute the Fredholm index by the same cutting-and-pasting argument as in the proof of of Theorem \ref{thm:index}. However, a key difference is that we need to introduce auxiliary dimensions in order to construct such a deformation in the case of manifolds with polyhedral boundary. The auxiliary space dimensions allow us to circumvent the geometric deformation problem for general spherical polyhedra (cf. Remark \ref{remark:poly-deform}) and to  use more general (and more algebraic) local boundary conditions (that do not necessarily come from the geometry of any manifold with polyhedral boundary) to deform the twisted Dirac operator while keeping the Fredholm index unchanged.  
	
   The theorem is nontrivial only when $\dim N$ and $\dim M$ have the same parity. Let us first consider the case where both $N$ and $M$ are \emph{odd} dimensional. 	Let $\mathbb I$ be the interval $[-1,1]$ . Consider the product map 
	$$f\times \id\colon N\times \mathbb I\to M\times \mathbb I,$$
	where $\id\colon \mathbb I \to \mathbb I$ is the identity map. 
	Let us denote 
	\[ E\coloneqq S_{N\times \mathbb I}\otimes f^*S_{M\times \mathbb I} = S_{T(N\times \mathbb I)\oplus f^\ast T(M\times \mathbb I)}.  \]The bundle $E$ over $N\times \mathbb I$ naturally identifies  with
	$$(S_{TN\oplus f^\ast TM})\hotimes \Bigwedge^*\R, $$
	since $\mathbb I$ is flat. A face of $N\times \mathbb I$ is   either $\overbar F_i\times \mathbb I$ for some face $\overbar F_i$ of $N$, or $N\times \{\pm1\}$, where $\{\pm1\}$ are the end points of $\mathbb I$.  The local boundary condition $B$ on $S_{TN\oplus f^\ast TM}$ naturally induces the following local boundary condition on sections of $E$ at the codimension one faces of $N\times \mathbb I$.   For a smooth section $\varphi$ of $E$ over $N\times \mathbb I$, the boundary condition at $\overbar F_i\times \mathbb I$ is given by 
	$$\mathscr E(\overbar c( u_i\oplus 0)\otimes c(v_i\oplus 0))\varphi=-\varphi,$$
	where $\mathscr E$ is the $\Z_2$-grading on $E$, $ u_i$ is the unit inner normal vector of $\overbar F_i$ in $N$, and $v_i$ is the unit inner normal vector of the corresponding face $F_i$ in $M$. Here the notation $u_i\oplus 0$ means that $ u_i\oplus 0$ lies in $TN\subset TN\oplus T\mathbb I$.  Now at the codimension face $N\times \partial \mathbb I$,  the boundary condition is given by 
	$$\mathscr E(\overbar c(0\oplus w)\otimes c(0\oplus w))\varphi=-\varphi,$$
	where  $w$ is  the unit inner normal vector of $\partial \mathbb I$ in  $\mathbb I$. For simplicity, we shall still denote this boundary condition by $B$. 
	
	We shall construct the desired continuous deformation of twisted Dirac operators  as follows. 
	
	\begin{enumerate}[label= \textbf{Step (\arabic*)}]
		\item Let $u_i$ be the unit inner normal vector field of each codimension one face $\overbar F_i$ of $N$. Let $v_i$ be the unit inner normal vector field of the corresponding codimension one face $F_i$ of $M$. We construct a continuous deformation $\{\widetilde v_i(t)\}_{t\in [0, 1]}$ of  $\{v_i\}$ with the help of the extra space direction of $\mathbb I$ such that the inner products between every pair of adjacent $\widetilde v_i(t)$ and $\widetilde v_j(t)$ is always strictly larger than the corresponding inner product between $u_i$ and $u_j$ throughout the deformation, and furthermore  $\widetilde v_i(t)$ and $\widetilde v_j(t)$ converge to the same vector field as $t\to 1$. Note that the geometry of $N$ remains constant along $t\in[0,1]$ in the current step. 
		\item We deform the metric $\overbar g$ on $N$ so that $\overbar g$ eventually has product structure\footnote{More precisely, the product structure means the following. Let us start with the top codimension faces. If $\overbar F$ is a top codimension face, then product structure means that a sufficiently small tubular neighborhood $\mathcal N_\varepsilon(\overbar F)$ of $\overbar F$ is a direct product  $\overbar F\times \fiber$, where $\fiber$ is a certain convex Euclidean polyhedral corner. Now by induction, near  a codimension $k$ face $\overbar F_k$, product structure means that $\mathcal N_\varepsilon(\overbar F_k) - \bigcup_{|\lambda| \geq k +1}\mathcal N_\delta (\overbar F_\lambda )$ has a direct product structure.} near each face of every codimension, while keeping the strict comparison condition on dihedral angles throughout the deformation (thanks to the first step). 
		
		\item  Consider a family of Riemannian metrics $\{g_t\}_{t\in [2, 3]}$ (of manifolds with polyhedral boundary) on $M$  such that 
		\begin{enumerate}
			\item $g_2 = g$, 
			\item $g_3$ has product structure near each face of every codimension, 
				\item all dihedral angles of $(M, g_t)$ are $<\pi$ for all $t\in [2, 3]$,  
			\item and the map $f\colon (N, \overbar g_2) \to (M, g_3)$  is a fiberwise isometry near each codimension $k$ face, where the fiber is some convex Euclidean polyhedral corner in $\mathbb R^k$. 
		\end{enumerate} We construct a continuous family of vector fields $\{\widetilde v_i(t)\}_{t\in [2, 3]}$ along each codimension one face $F_i$ such that\footnote{Here  $\widetilde v_i(1)$ is the vector field from \textbf{Step (1)}.}   $\widetilde v_i(2) = \widetilde v_i(1)$   and  $\widetilde v_i(3)$ coincides with the corresponding unit inner normal vector field of $F_i$ in $M$ with respect to the metric $g_3$,  while keeping the strict comparison condition on dihedral angles (or more precisely inner products) for all $t\in [2, 3)$. Again, the geometry of $(N, \overbar g_2)$ remains unchanged in this step. 
		\item By a similar argument as the case of manifolds with corners (in the proof of Theorem \ref{thm:index}), we show that the Fredholm index of $D_E$ (subject to the corresponding boundary condition) remains unchanged throughout the above deformations. 
	\end{enumerate}

Let us start with \textbf{Step (1)}. We  fix some sufficiently small $\varepsilon >0$. For each codimension $\ell$ face  $\overbar F_{(\ell)}$ of $N$, we consider  
\begin{equation}\label{eq:smallneighborhood}
	\mathcal U(\overbar F_{(\ell)}) = \mathcal U_\varepsilon (\overbar F_{(\ell)}) \coloneqq  \mathcal N_\varepsilon (\overbar F_{(\ell)}) - \bigcup_{|\lambda| > \ell }\mathcal N_{\varepsilon/2}(\overbar F_{\lambda}) 
\end{equation} 
where $|\lambda|$ stands for codimension of the face  $\overbar F_{\lambda}$ in $N$  and $\mathcal N_\varepsilon (\overbar F_{\lambda})$ is the $\varepsilon$-tubular neighborhood of $\overbar F_{\lambda}$. The collection $\{\mathcal U(\overbar F_{\lambda})\}$ for all faces with codimension $\geq 1$ is an open cover of $\partial N$. Topologically, $\mathcal U(\overbar F_{(\ell)})$ carries a natural fiber bundle structure (see the discussion at the beginning of Section \ref{sec:conic}), where the  base space $\mathcal U(\overbar F_{(\ell)})$ is an open manifold  in the interior of $\overbar F_{(\ell)}$. In particular,  the tangent bundle $TN$ restricted on $\mathcal U(\overbar F_{(\ell)})$  decompose into the vertical part $T_VN$ along the fibers and the horizontal part $T_HN = (T_VN)^\perp$. By the definition of manifolds with polyhedral boundary, the vertical subbundle $T_VN$ is a trivial vector bundle over $\mathcal U(\overbar F_{(\ell)})$ . Similar remarks also apply to $TM$ near the corresponding codimension $\ell$ face $F_{(\ell)}$ in $M$. In particular, we have a decomposition $TM=T_HM\oplus T_VM$ on a similar neighborhood $\mathcal U(F_{(\ell)}) $ of $F_{(\ell)}$, where the vertical subbundle $T_VM$ is again a trivial vector bundle over  $\mathcal U(F_{(\ell)}) $ of $F_{(\ell)}$ and has the same rank as $T_VN$. 

Let $v_i$ be the unit inner normal vector field of each codimension one face $F_i$ of $M$. For  each $x\in \partial M$,  we see the corresponding tangent cone at $x$ is a \emph{convex} polyhedral corner in the tangent space $T_xM$, since  all dihedral angles of $M$ are $<\pi$ by assumption. It follows that there exists a smooth unit vector field $\normal\in TM$ over $\partial M$ such that 
\[ \langle v_i, \normal\rangle_x >0  \] 
for all $x\in \partial M$ and for all $i$. According to the above discussion on the horizontal and vertical subbundle decompositions, it is not difficult to construct a continuous family of vector fields\footnote{For example, such a family of vector fields $\{v_i(s)\}_{s\in [0, 1]}$  can be obtained as follows. By the definition of manifolds with polyhedral boundary,  $(M, g)$ is a codimension zero submanifold of an open smooth manifold $(M^\dagger, g^\dagger)$. Now we choose $M_s$ to be a continuous family of codimension zero submanifolds with polyhedral boundary in $(M^\dagger, g^\dagger)$ such that the fiberwise convex polyhedral corner structure of $M_s$ near each codimension $k$ face (but away from small neighborhoods of faces with codimension $\geq k+1$) converges to the half ball in the positive direction of $\normal$.} $\{v_i(s)\}_{s\in [0, 1]}$ of $TM$ along each face $F_i$  such that 
\begin{enumerate}[label=(\alph*)]
	\item $v_i(s)$ is fiberwise asymptotically flat near all faces of all codimensions,
	\item $v_i(s)$ converges to $\normal$ as $s\to 1$, for all $i$, 
	\item and $\langle v_i(s), \normal\rangle_x >0$ 
	for all $x\in \partial M$, all $s\in [0, 1]$ and  all $i$.
\end{enumerate} 
Since $M$ is compact, property (c) above implies that  there exists a positive number $\theta>0$ such that 
\begin{equation}\label{eq:lowerbound}
	\langle v_i(s),\normal\rangle_x\geq \theta
\end{equation}
all $x\in \partial M$, all $s\in [0, 1]$ and  all $i$. 
 
We emphasize that the  vector fields $\{v_i(s)\}$  only serves as a background of the actual deformation that we are about to construct. In general, we have no control over the dihedral angles (or more precisely the inner products) between $v_i(s)$ and $v_j(s)$.  In particular, if we use $\{v_i(s)\}$ to define a new local boundary condition,  the comparison condition on dihedral angles (cf. line \eqref{eq:dihedralstrict} and \eqref{eq:dihedralequal}) may fail for some  $s\in[0,1]$. Consequently, the essential self-adjointness and the Fredholmness of the associated twisted Dirac operator  may also fail for some  $s\in[0,1]$.  The key part of the proof is to ``modify" the above deformation so that the new deformation satisfies the required dihedral angle comparison condition. This will be achieved by making use of the auxiliary space direction $\mathbb I$ and forming  more general local boundary conditions such as those from Theorem \ref{thm:ess-saNew}.

Recall that we have a  smooth unit vector field $\normal$ defined over $\partial M$.  We may extend the vector field $\normal$ to a smooth vector field over $M$, still denoted by $\normal$,  such that  $\normal$ has unit length in a tubular neighborhood of  $\partial M$. Of course, $\normal$ may equal zero in the interior of $M$. Let $\Y$ be a constant unit vector in $\mathbb I$, say, pointing in the positive direction of $\mathbb I$. Recall that we have denoted the unit inner normal vector at either end point $\{1\}$ or $\{-1\}$ of $\mathbb I$ by $w$. With the above choice of $\Y$, we see that $\Y = -w$ at the end point $\{1\} \in \mathbb I$ and $\Y = w$ at the end point $\{-1\} \in \mathbb I$.   For $t\in [0, 1]$,   we define vector fields $\widetilde v_i(t)$ and $w(t)$ in $T(M\times \mathbb I)$ as follows: 
 \begin{equation}\label{eq:newvector}
\widetilde v_i(t)\coloneqq \frac{1}{\sqrt{1+p(t)^2}}(v_i(t)\oplus p(t)\Y)
 \end{equation}
 along the face $F_i\times \mathbb I$ of $M\times \mathbb I$, and 
\begin{equation}
w(t)\coloneqq \frac{1}{\sqrt{1+q(t)^2|\normal|^2}}(q(t)\normal\oplus w) 
\end{equation}
 along the faces $M\times \{\pm 1\}$ of $M\times \mathbb I$. 
 Here the function $p\colon [0, 1]\to \mathbb R$ is given by
 $$p(t)=\begin{cases}
 	\sqrt{Ct}& \textup{ if } 0\leq t\leq 1-\delta, \vspace{0.3cm}\\
 	\displaystyle \frac{\sqrt{C(1-\delta)}}{\delta}(1-t)& \textup{ if } 1-\delta\leq t\leq 1,
 \end{cases}$$
 with a  sufficiently large $C>0$ and a sufficiently small $\delta>0$, and the function $q\colon [0, 1]\to \mathbb R$ is given by 
 $$q(t)=\frac{2p(t)}{\theta},$$
 where $\theta$ is the positive number from line \eqref{eq:lowerbound}. 
 
 Let $u_i$ be the unit inner normal vector of a codimension one face $\overbar F_i$ of $N$. As before, for each codimension one face  $\overbar F_i$ of $N$, we denote its corresponding codimension one face in $M$ by $F_i$. 
 
 \begin{claim*} If $C >0$ is sufficiently large and $\delta >0$ is sufficiently small, we have the following comparison of inner products: for all $t\in [0, 1]$, 
 	\begin{equation}\label{eq:angleComparision0-3}
 		\begin{split}
 			\langle \widetilde v_i(t),\widetilde v_j(t)\rangle > \langle u_i,u_j\rangle&\text{ for  all adjacent faces }\overbar F_i\times\mathbb I\text{ and }\overbar F_j\times\mathbb I,  \\
 			\langle \widetilde v_i(t),w(t)\rangle\geq 0&\text{ along }(\overbar F_i\times\mathbb I)\cap(N\times\partial\mathbb I). 
 		\end{split}
 	\end{equation}
  Moreover, the second inequality $\langle \widetilde v_i(t),w(t) \rangle \geq 0$ attains the equal sign only at $t=0$ and $t=1$. In other words, the second inequality is a strict inequality $\langle \widetilde v_i(t),w(t) \rangle  > 0$ for all $t\in (0, 1)$.  
 \end{claim*}
 
 Now we verify the first inequality of the comparison condition \eqref{eq:angleComparision0-3}. Note that 
 $$\langle\widetilde v_i(t),\widetilde v_{j}(t) \rangle=\frac{\langle v_i(t),v_j(t)\rangle+(p(t))^2}{1+(p(t))^2}$$
 for all $t\in [0, 1]$. We divide the computation into two cases: for $t\in [0, 1-\delta]$ and for $t\in [1-\delta, 1]$.  
 
 For $0\leq t\leq 1-\delta$, we have
 $$\langle\widetilde v_i(t),\widetilde v_{j}(t) \rangle=\frac{\langle v_i(t),v_j(t)\rangle+Ct}{1+Ct}.$$
Therefore $\langle\widetilde v_i(t),\widetilde v_{j}(t) \rangle>\langle u_i,u_j\rangle$ is equivalent to 
$$C\geq \frac 1 t\cdot \frac{\langle u_i,u_j\rangle-\langle v_i(t), v_{j}(t) \rangle}{1-\langle  v_i, v_j\rangle}.$$
 By the assumption on dihedral angles (cf. line \eqref{eq:strictdihedral-poly}), we have  
 \[ \langle u_i,u_j\rangle - \langle v_i(0),v_j(0)\rangle = \langle u_i,u_j\rangle - \langle v_i,v_j\rangle<0 \text{ and } 1-\langle v_i,v_j\rangle>0\]
 along all pairs of adjacent faces $\overbar F_i\times\mathbb I$ and $\overbar F_j\times\mathbb I$.  It follows that
 $$\sup_{t\in[0,1]}\frac{\langle u_i,u_j\rangle-\langle v_i(t), v_{j}(t) \rangle}{t(1-\langle  v_i, v_j\rangle)}<+\infty.$$
 Therefore, the first inequality of the comparison condition \eqref{eq:angleComparision0-3} holds for all $t\in [0, 1-\delta]$,  as long as $C>0$ (independent of $\delta$) is sufficiently large.

 Recall that $v_i(t)$ converges to $\normal$ as $t\to 1$, for all codimension one faces $F_i$. It follows that for any $\varepsilon >0$, there exists $\delta>0$ such that 
 \[  |\langle v_i(t),v_j(t)\rangle - 1|< \varepsilon  \]
 for all $ t\in [1-\delta, 1]$.  In particular, there exists a $\delta>0$ such that, for all $t\in [1-\delta, 1]$, 
 \[  \langle v_i(t),v_j(t)\rangle > \langle u_i, u_j\rangle  \]
  along all adjacent faces $\overbar F_i\times\mathbb I$ and $\overbar F_j\times\mathbb I$. Let us fix such a $\delta >0$.  Observe that for any $\lambda\geq 0$, we have 
 $$\langle v_i(t),v_j(t)\rangle\leq \frac{\langle v_i(t),v_j(t)\rangle+\lambda}{1+\lambda}\leq 1.$$
 It follows that 
 $$\langle\widetilde v_i(t),\widetilde v_{j}(t) \rangle=\frac{\langle v_i(t),v_j(t)\rangle+(p(t))^2}{1+(p(t))^2} \geq \langle v_i(t),v_j(t)\rangle > \langle u_i, u_j\rangle $$
 for all $t\in [1-\delta, 1]$. To summarize, we have proved that 
  the first inequality of the comparison condition \eqref{eq:angleComparision0-3} holds as long as $C>0$ is sufficiently large and $\delta >0$ is sufficiently small. 
 
 Now we verify the second inequality of the comparison condition \eqref{eq:angleComparision0-3}. We have 
 \begin{align*}
 	\langle\widetilde v_i(t),w(t)\rangle=&\frac{q(t)\langle v_i(t),\normal\rangle+p(t)\langle \Y,w\rangle}{\sqrt{(1+(p(t))^2)(1+(q(t))^2)}}\\
 	=&\frac{p(t)}{\sqrt{(1+(p(t))^2)(1+\frac{(p(t))^2}{\theta^2/4})}}\Big(\frac{\langle v_i(t),\normal\rangle}{\theta/2}-\langle \Y,w\rangle\Big).
 \end{align*}
Recall that  the vector field $\normal$ is of unit length and satisfies 
 \[ \langle \normal,v_i(t)\rangle\geq \theta>0\]  along $( F_i\times\mathbb I)\cap (M\times\partial\mathbb I)$.  
It follows that 
 $$\frac{\langle v_i(t),\normal\rangle}{\theta/2}-\langle \Y,w\rangle\geq 2-\langle \Y,w\rangle>0$$
 for all $t\in [0, 1]$. Therefore, 
 we have 
 \[ \langle\widetilde v_i(t),w(t)\rangle\geq 0 \]
 for all $t\in [0, 1]$, and the equality holds only if $p(t)=0$, that is, only if $t=0$ or $t=1$. This finishes the proof of the claim and completes the construction of \textbf{Step (1)}. 
 
 Now let us consider \textbf{Step (2)}. We deform $(N,\overbar g)$ to a manifold $(N,\overbar g_2)$ with polyhedral boundary via a  family of Riemannian metrics $\{\overbar g_t\}_{t\in [1, 2]} $ (of manifolds with polyhedral boundary) such that $\overbar g_1 = \overbar g$ and  $\overbar g_2$ has product structure near each face of $N$ with any codimension, while maintaining the strict comparison condition 
 \[  \langle u_i(t), u_j(t) \rangle < \langle \widetilde v_i(1),\widetilde v_j(1)\rangle \] 
 along all pairs of  adjacent faces $\overbar F_i$ and $\overbar F_j$ of $N$, and for all $t\in [1, 2]$. Here $u_i(t)$ is the unit inner normal vector field of $\overbar F_i$ in $N$ with respect to the metric $\overbar g_t$.  It is clear that there exists a continuous family of metrics $\{\overbar g_t\}_{t\in [1, 2]}$ (of manifolds with polyhedral boundary) on the underlying topological manifold $N$ such that $\overbar g_1 = \overbar g$ and $\overbar g_2$ has product structure near each face of $N$ of every codimension. Of course, we generally do not have much control of the dihedral angles of $(N, \overbar g_t)$ (as regard to whether they increase or decrease). But in any case, as both $N$ and the interval $[0, 1]$ are compact, there exists a constant $\alpha <\pi$ such that all dihedral angles of $(N, \overbar g_t)$ are strictly less than $\alpha$ for all $t\in [1, 2]$. 
By the construction of $\textbf{Step (1)}$, when $t=1$,  the vector fields  $\{ \widetilde v_i(1)\}$ all coincide with $\normal$. In particular, we have 
 \[ \langle \widetilde v_i(1),\widetilde v_j(1)\rangle=1\]  for all pairs of  adjacent faces $F_i$ and $F_j$ of $M$. 
 Therefore,  the strict comparison condition 
 \[  \langle u_i(t), u_j(t) \rangle < \langle \widetilde v_i(1),\widetilde v_j(1)\rangle \] is always  satisfied for all  $t\in [1, 2]$. This completes \textbf{Step (2)}. 
 
 Now let us move to \textbf{Step (3)}. Consider a family of Riemannian metrics $\{g_t\}_{t\in [2, 3]}$ (of manifolds with polyhedral boundary) on $M$  such that 
 \begin{enumerate}[label=(\alph*)]
 	\item $g_2 = g$, 
 	\item $g_t$ has product structure near each face of every codimension for $t$ near $3$, 
 		\item all dihedral angles of $(M, g_t)$ are $<\pi$ for all $t\in [2, 3]$, 
 	\item and the map $f\colon (N, \overbar g_2) \to (M, g_3)$  is a fiberwise isometry near each codimension $k$ face, where the fiber is some convex Euclidean polyhedral corner in $\mathbb R^k$. 
 \end{enumerate} 
   Let $v_i(t)$ be the unit inner normal vector field of each codimension one face $F_i$ of $M$ with respect to the metric $g_t$. For  each $x\in \partial M$,  we see the corresponding tangent cone at $x$ is a \emph{convex} polyhedral corner in the tangent space $T_xM$, since  all dihedral angles of $(M, g_t)$ are $<\pi$ by assumption. It follows that there exists a smooth unit vector field $\normal_t\in TM_{g_t}$ over $\partial M$ such that 
 \[ \langle v_i(t), \normal_t\rangle_x >0  \] 
 for all $x\in \partial M$ and for all $i$. 
 
 Similar to the set defined in line \eqref{eq:smallneighborhood}, we consider a family of open sets $\{\mathcal U_\varepsilon(F_\lambda)\}$ of $M$ for  faces $F_\lambda$ of $M$ with codimension $\geq 1$. More precisely, for each face $F_\lambda$ of $M$, we denote by $|\lambda|$ its codimension. By induction (starting from the top codimension), we may choose $\varepsilon_\lambda >0$ for each $F_\lambda$ such that the sets $\{ \mathcal U_{\varepsilon_\lambda}(F_\lambda) \}$ cover $\partial M$ and,  for any pair of distinct faces $F_\lambda$ and $F_\rho$ of the same codimension $|\lambda| = |\rho|$, we have $\mathcal U_{\varepsilon_\lambda}(F_\lambda)\cap \mathcal U_{\varepsilon_\rho}(F_\rho)  = \emptyset$. Here we have 
 \begin{equation}\label{refinedsets}
 	\mathcal U_{\varepsilon_\lambda}(F_\lambda)\coloneqq \mathcal N_{\varepsilon_\lambda}( F_\lambda)- \bigcup_{|\mu|\geq|\lambda|+1}\mathcal N_{\varepsilon_{\mu}/2}( F_{\mu})
 \end{equation}
 where $\mathcal N_{\varepsilon_\lambda} (F_{\lambda})$ is the $\varepsilon_\lambda$-tubular neighborhood of $F_{\lambda}$. For $t$ sufficiently close to $3$, we may assume without loss of generality  that $g_t$ has  product structure on every $\mathcal U_{\varepsilon_\lambda} (F_\lambda)$, which is the direct product of a Euclidean polyhedral corner in $\mathbb R^{|\lambda|}$ and a subspace of $F_\lambda$.   In particular, there is a canonical choice of the vertical part $(T_VM)_{F_\lambda}$ of $TM$ on every $\mathcal U_{\varepsilon_\lambda}(F_\lambda)$ and a canonical choice of the radial variable $r$ for the fiber. For the choice of the smooth vector field $\normal_t$ above, we may assume without loss of generality that  $\normal_t(x)$ lies in $(T_VM)_{F_\lambda}$ for each $x\in \mathcal U_{\varepsilon_\lambda}(F_\lambda)\cap\partial M$. This can be, for example, achieved by induction starting from the top codimension faces.  See Figure \ref{fig:normal} for how $\normal_t(x)$ varies from a codimension two face to a codimension one face. To distinguish  the vector field $\normal_3 \in TM_{g_3}$ over $\partial M$ at $t=3$ from the rest, we shall denote $\normal_3$ by $\normalnew$. 
 \begin{figure}
 	\begin{tikzpicture}[scale=2]
 		\draw[thick] ({2*cos(40)},{2*sin(40)}) -- (0,0) -- ({-2*cos(40)},{2*sin(40)});
 		\draw[->,blue] (0,0) -- (0,0.3);
 		\draw[->,blue] ({0.3*cos(40)},{0.3*sin(40)}) -- ({0.3*cos(40)-0.3*cos(77)},{0.3*sin(40)+0.3*sin(77)});
 		\draw[->,blue] ({0.6*cos(40)},{0.6*sin(40)}) -- ({0.6*cos(40)-0.3*cos(63)},{0.6*sin(40)+0.3*sin(63)});
 		\draw[->,blue] ({0.9*cos(40)},{0.9*sin(40)}) -- ({0.9*cos(40)-0.3*cos(50)},{0.9*sin(40)+0.3*sin(50)});
 		\draw[->,blue] ({1.2*cos(40)},{1.2*sin(40)}) -- ({1.2*cos(40)-0.3*cos(50)},{1.2*sin(40)+0.3*sin(50)});
 		\draw[->,blue] ({1.5*cos(40)},{1.5*sin(40)}) -- ({1.5*cos(40)-0.3*cos(50)},{1.5*sin(40)+0.3*sin(50)});
 		\draw[->,blue] ({-0.3*cos(40)},{0.3*sin(40)}) -- ({-0.3*cos(40)+0.3*cos(77)},{0.3*sin(40)+0.3*sin(77)});
 		\draw[->,blue] ({-0.6*cos(40)},{0.6*sin(40)}) -- ({-0.6*cos(40)+0.3*cos(63)},{0.6*sin(40)+0.3*sin(63)});
 		\draw[->,blue] ({-0.9*cos(40)},{0.9*sin(40)}) -- ({-0.9*cos(40)+0.3*cos(50)},{0.9*sin(40)+0.3*sin(50)});
 		\draw[->,blue] ({-1.2*cos(40)},{1.2*sin(40)}) -- ({-1.2*cos(40)+0.3*cos(50)},{1.2*sin(40)+0.3*sin(50)});
 		\draw[->,blue] ({-1.5*cos(40)},{1.5*sin(40)}) -- ({-1.5*cos(40)+0.3*cos(50)},{1.5*sin(40)+0.3*sin(50)});
 		\draw[dashed] ({0.9*cos(40)},{0.9*sin(40)}) arc (40:130:0.9);
 		\draw ({1.2*cos(40)},0.4) node {$r=\varepsilon/2$};
 	\end{tikzpicture}
 	\caption{The smooth vector field $\normal$.}
 	\label{fig:normal}
 \end{figure}
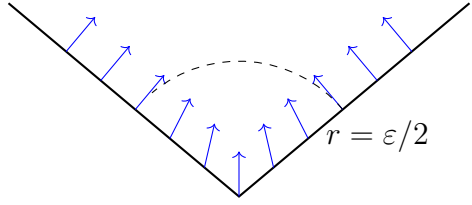
 
  Consider the map $f\colon (N, \overbar g_t) \to (M, g_t)$ with $t\in [2, 3]$.  If  $t, s\in [2, 3]$ are sufficiently close, then there exists an bundle isometry $\Phi_{t, s}$ between  the pullback bundles $f^\ast(TM_{g_t})$ and  $f^\ast(TM_{g_{s}})$ such that the isometry $\Phi_{t, s}$ preserves the decompositions of vertical and horizontal subbundles over
 $\mathcal U_{\varepsilon}(\overbar F_\lambda)$ for each face $\overbar F_\lambda$ of $(N, \overbar g_2)$, as long as $\varepsilon$ is sufficiently small, where $\mathcal U_{\varepsilon}(\overbar F_\lambda)$ is the set defined in line  \eqref{eq:smallneighborhood}. Moreover, we may assume without loss of generality that $\Phi_{t, s}$ maps $\normal_s$ to $\normal_t$.
 By the compactness of the interval $[0, t]$, it follows that there exists a bundle isometry $\Phi_t$ between  the pullback bundles $f^\ast(TM_{g})$ and  $f^\ast(TM_{g_t})$ such that $\Phi_t$ preserves the decompositions of vertical and horizontal subbundles over
 $\mathcal U_{\varepsilon}(\overbar F_\lambda)$ for each face $\overbar F_\lambda$ of $(N, \overbar g_2)$, as long as $\varepsilon$ is sufficiently small, and $\Phi_t$ maps the vector field $\normal \in f^\ast(TM_g)$ to $ \normal_t\in f^\ast(TM_{g_t})$, for all $t\in [2, 3]$.  
 
 By the construction of $\textbf{Step (1)}$, when $t=1$,  the vector fields  $\{ \widetilde v_i(1)\}$ all coincide with $\normal$. In particular, we have 
 \[ \langle \widetilde v_i(1),\widetilde v_j(1)\rangle=1\]  for all pairs of  adjacent faces $F_i$ and $F_j$ of $M$. We define
 \[ \widetilde v_i(t-1) = \normal_t \textup{ along $F_i$ in } (M, g_t) \textup{ for all }  t\in [2, 3].\] 
Via the isometries $\Phi_t$, we have transformed the boundary condition defined using $\{ \widetilde v_i(1)\}$ for $f\colon (N, \overbar g_2) \to (M, g_2)$  to a local boundary condition defined using $\{ \widetilde v_i(2)\}$ for $f\colon (N, \overbar g_2) \to (M, g_3)$, while keeping the strict dihedral angle comparison, hence not changing the Fredholm index of the associated twisted Dirac operator. 

Now we have arrived at the following geometric situation: $(N, g_2)$ has product structure near each face of $N$ of every codimension and the map 
\[  f\times \id \colon (N, \overbar g_2)\times \mathbb I \to (M, g_3)\times \mathbb I \]
	is a fiberwise isometry near each codimension $k$ face for all $k\geq 1$. For  a smooth section $\varphi$ of $E = S_{T(N\times \mathbb I)\oplus f^\ast T(M\times \mathbb I)}$ over $N\times \mathbb I$, the boundary condition at $\overbar F_i\times \mathbb I$ is given by 
	$$\mathscr E(\overbar c( u_i(2)\oplus 0)\otimes c(\widetilde v_i(2)\oplus 0))\varphi=-\varphi,$$
	where $\mathscr E$ is the $\Z_2$-grading on $E$, $ u_i(2)$ is the unit inner normal vector of $\overbar F_i$ in $(N, \overbar g_2)$, and $\widetilde v_i(2) =  \normalnew = \normal_3$.   Now at the codimension face $N\times \{1\}$ (resp. $N\times \{-1\}$),  the boundary condition is given by 
	$$\mathscr E(\overbar c(0\oplus \overbar w)\otimes c(0\oplus w))\varphi=-\varphi,$$
	where $\overbar w$ and $w$ is  the unit inner normal vector of $\{1\}$ (resp. $\{-1\}$) with respect to $\mathbb I$.

Recall that $v_i(t)$ is the unit inner normal vector field of each codimension one face $F_i$ of $M$ with respect to the metric $g_t$ for $t\in [2, 3]$. Let $\normalnew$ be  the smooth unit vector field in  $TM_{g_3}$ over $\partial M$  such that 
\[ \langle v_i(t), \normalnew \rangle_x >0  \] 
for all $x\in \partial M$ and for all $i$. By the identifications via  the isometries $\Phi_{t, s}$ above, we may in fact view $\{v_i(t)\}_{t\in [2, 3]}$ as a family of vector fields in $f^\ast (TM_{g_3})$ over each codimension one face $\overbar F_i$ of $(N, \overbar g_2)$ such that 
\[ \langle v_i(t), \normalnew\rangle_x >0 \]
for all $x\in \partial M$, all $t\in [2, 3]$ and  all $i$. Now we apply the exact same construction of \textbf{Step (1)} to  $\{v_i(t)\}_{t\in [2, 3]}$ and the unit inner normal vector $w$ of $\partial \mathbb I$ in $\mathbb I$. To be precise, we apply the construction of \textbf{Step (1)} to the family $\{v_i(3-t)\}_{t\in [2, 3]}$. In other words, we consider the  family $\{v_i(t)\}_{t\in [2, 3]}$  in the reversed direction, and apply the construction starting at $t=3$ and ending at $t=2$.   As a result, we obtain  vector fields $\{\widetilde v_i(t)\}_{t\in [2, 3]}$ and $\{w(t)\}_{t\in [2, 3]}$ in $T(M\times \mathbb I)$ such that
\begin{equation}
	\begin{split}
		\langle \widetilde v_i(t),\widetilde v_j(t)\rangle \geq  \langle u_i(2),u_j(2)\rangle&\text{ for all adjacent faces }\overbar F_i\times\mathbb I\text{ and }\overbar F_j\times\mathbb I,  \\
		\langle \widetilde v_i(t),w(t)\rangle\geq 0&\text{ along }(\overbar F_i\times\mathbb I)\cap(N\times\partial\mathbb I). 
	\end{split}
\end{equation}
for all $t\in [2, 3]$. 
Moreover, the first inequality $\langle \widetilde v_i(t),\widetilde v_j(t)\rangle \geq  \langle u_i(2),u_j(2)\rangle$ attains the equal sign only at $t=3$; and the second inequality $\langle \widetilde v_i(t),w(t) \rangle \geq 0$ attains the equal sign only at $t=2$ and $t=3$. This completes \textbf{Step (3)}.

Now we come to the final step, \textbf{Step (4)},  of the proof. For the deformation constructed in \textbf{Step (1)} -- \textbf{Step (3)} that is parametrized by $t\in [0, 3]$, let us denote by $N_t$ and $M_t$ the manifolds at time $t$, and denote by $B_t$ the corresponding local boundary.   By Theorem \ref{thm:ess-sa} and Theorem \ref{thm:ess-saNew}, the twisted Dirac operator $D_{B_t}^E$ acting on $E$ with boundary condition $B_t$ is essentially self-adjoint and Fredholm, and moreover its domain is $H^1(N\times \mathbb I, E; B_t)$ for every $t\in[0,3]$.  
\begin{claim*}
	$\ind(D^E_{B_t})$ remains constant for all $t\in[0,3]$.
\end{claim*} 
For $t\in(0,1)$, the corresponding dihedral angles (or equivalently the corresponding inner products) satisfy the strict inequality comparison. 	Then for any  $0< t_1<t_2<1$,   there exists a bounded unitary map (cf. the proof of Theorem \ref{thm:ess-sa} and use Lemma \ref{lemma:1/rEmbedding} in place of Lemma \ref{lemma:sobolevEmbedding})
\[ \mathbf U \colon L^2(N_{t_1}\times \mathbb I, E) \to  L^2(N_{t_2}\times \mathbb I, E) \] 
such that $\mathbf U$ maps the boundary condition  $B_{t_1}$ to the boundary condition $B_{t_2}$. The unitary $\mathbf U$ is induced by a bundle isometry that  may not be defined near faces with codimension $\geq 2$, but is asymptotically conical near all faces of all codimensions. By the same argument in  the proof of Theorem \ref{thm:index},  it follows from the Kato-Rellich perturbation theorem, Lemma \ref{lemma:1/rEmbedding} and\footnote{More precisely, for higher dimensional convex polyhedral corners, if we assume  the strict comparison condition  \eqref{eq:strictinnerproduct}, then  the obvious analogue of Lemma \ref{lemma:iteratedSobolev} also holds starting from codimension two.   }  Lemma \ref{lemma:iteratedSobolev}    that $\ind(D^E_{B_t})$  is constant for $t\in (0, 1)$.

For $t\in (2,3)$, the corresponding dihedral angles (or equivalently the corresponding inner products) also satisfy the strict inequality comparison. Hence the same argument above shows that $\ind(D^E_{B_t})$  is constant for $t\in (2, 3)$. 

For $t\in (1, 2)$, the boundary condition $B_t$ satisfies the following: 
\begin{equation}
	\begin{split}
		\langle \widetilde v_i(t),\widetilde v_j(t)\rangle > \langle u_i(t),u_j(t)\rangle&\text{ for  all adjacent faces }\overbar F_i\times\mathbb I\text{ and }\overbar F_j\times\mathbb I,  \\
		\langle \widetilde v_i(t), w\rangle = 0&\text{ along }(\overbar F_i\times\mathbb I)\cap(N\times\partial\mathbb I). 
	\end{split}
\end{equation}
In particular, $D^E$ with the boundary condition $B_t$ becomes a  product 
\[  D^E=D\hotimes 1+1\hotimes D^{\dR}_{\mathbb I}, \] 
where $D^{\dR}_{\mathbb I}$ is the de Rham operator of $\mathbb I$ subject to the usual absolute boundary condition\footnote{The absolute boundary condition is given as follows. Suppose $\mathbb I =[-1, 1]$ is parameterized by $y$. A smooth differential form $\omega$ on $\mathbb I$ can be written as  $\omega = \omega_0 + \omega_1dy$, where $\omega_0$ and $\omega_1$ are smooth functions on $\mathbb I$. We say $\omega$ satisfies the absolute boundary condition if $\omega_1$ vanishes at $\partial \mathbb I = \{\pm 1\}$.}  on differential forms. Now for $t\in (1, 2)$, the corresponding dihedral angles (or equivalently the corresponding inner products)  satisfy the strict inequality comparison for  $f\colon N_t\to M_t$. Again, the same argument above shows that $\ind(D^E_{B_t})$  is constant for $t\in (1, 2)$.

Now it remains to prove the invariance of the Fredholm index near $t=0,1,2, 3$.  Near these points, the dihedral angle comparison (or the corresponding inner product comparison) may go from equality to strict inequality. In these cases, Lemma \ref{lemma:1/rEmbedding} does not  directly apply, so we shall instead use the cutting-and-pasting method developed in the proof of Theorem \ref{thm:index} (cf. Proposition \ref{prop:gluing}). 

Let us first consider the case of $t=0$.  The main technical difficulty to prove the invariance of the Fredholm index $\ind(D^E_{B_t})$ at $t=0$ is that the inner product  $\langle \widetilde v_i(t), w(t)\rangle$ changes from $0$ to a positive number as $t$ moves along $[0, 1)$, while the inner product of the corresponding vectors on $N\times\mathbb I$ remains constantly $0$. In other words, the dihedral angle comparison along $(\overbar F_i\times\mathbb I)\cap(N\times\partial\mathbb I)$ goes from the equality to a strict inequality.

For $t$ sufficiently close to $0$,  we first deform $\widetilde v_i(t)$ (from line \eqref{eq:newvector}) on the subspace 
$F_i\times(-2/3,2/3)$ of the codimension  one face $F_i\times\mathbb I$ so that $\widetilde v_i(t)$ is perturbed  back to $v_i(0) \oplus 0= v_i\oplus 0$ over the subspace  $F_i\times [-1/2, 1/2]$. We emphasize  that this new perturbation is only performed over 
$F_i\times(-2/3,2/3)$.  As long as $t$ is sufficiently small, such a perturbation preserves the strict dihedral angle comparison for all pairs of adjacent faces for $t\neq 0$. In particular, for each fixed $t$, the Fredholm index of $D^E$ (subject to this new perturbed boundary condition) coincides with  the Fredholm index of $D^E$ (subject to the original boundary condition $B_t$). Therefore, it suffices to prove the invariance of the Fredholm index of $D^E$ subject to this new family of perturbed boundary conditions near $t=0$.  For simplicity, for each $t$ sufficiently close to $0$, we still denote this new perturbed  boundary condition by $B_t$.

Let $\mathbb I_{1/2}=[1/2,1]$. Consider the submanifold $N\times \mathbb I_{1/2}$ of $N\times \mathbb I$. We define a local boundary condition   $\widetilde B_t$ for sections of $E$ at the codimension one faces of $N\times \mathbb I_{1/2}$ as follows. At a codimension one face of $N\times \mathbb I_{1/2}$ that is \emph{not} $N\times \{1/2\}$, we define    $\widetilde B_t = B_t$. At the face $N\times \{1/2\}$ of $N\times \mathbb I_{1/2}$,  the boundary condition $\widetilde B_t$ is  given by
$$\mathscr E(\overbar c(0\oplus w)\otimes c(0\oplus w))\varphi=\varphi$$
for sections $\varphi$ of $E$, where $w$ is the inner normal vector of $N\times \{1/2\}$. In other words, at the face  $N\times \{1/2\}$,   $\widetilde B_t$ is equal to the orthogonal complement of the absolute boundary condition. It follows from Theorem \ref{thm:ess-sa} and Theorem \ref{thm:ess-saNew} that  the twisted  Dirac operator acting on $E$ over $N\times \mathbb I_{1/2}$ (subject to the boundary condition  $\widetilde B_t$)  is essentially self-adjoint and Fredholm, and its domain is  $H^1(N\times \mathbb I_{1/2}, E; \widetilde B_t)$.

When $t=0$, the Dirac operator $D^E_{N\times \mathbb I_{1/2}}$ with the boundary condition $\widetilde B_0$ is a product 
\[  D^E_{N\times \mathbb I_{1/2}}=D_{N}\hotimes 1+1\hotimes D^{\dR}_{\mathbb I}, \] 
where $D^{\dR}_{\mathbb I}$ is the de Rham operator of $\mathbb I$ but subject to a mixed boundary condition as follows. Let us write a smooth differential form $\omega$  on $\mathbb I$ as $\omega = \omega_0 + \omega_1dy$, where $\omega_0$ and $\omega_1$ are smooth functions on $\mathbb I$ and $y$ is the coordinate of $\mathbb I$. Then $\omega$ is required to satisfy the following mixed boundary condition: $\omega_1$ vanishes at the end point $\{1\}$ of $\mathbb I$, and $\omega_0$ vanishes at the end point $\{1/2\}$ of $\mathbb I$. A direct computation shows the operator $D^\dR_{\mathbb I}$ subject to the above mixed boundary condition is invertible. Consequently, by the product formula, we see that  the operator $ (D^E_{N\times \mathbb I_{1/2}})_{\widetilde B_0}$ is invertible.

 Now we show that $(D^E_{N\times \mathbb I_{1/2}})_{\widetilde B_t}$ is also invertible for $t$ sufficiently close to $0$. Assume by contradiction that there is a sequence of positive number $\{t_n\}$ converging to  $0$ such that there is a non-zero  $\varphi_n \in H^1(N\times \mathbb I_{1/2}, E; \widetilde B_{t_n})$ such that $D^E_ {N\times \mathbb I_{1/2}}(\varphi_n)=0$. We may furthermore assume that $\|\varphi_n\|=1$. Note that  $N\times \mathbb I_{1/2}$ has the direct product geometry for all $t_n$. At $t=0$, $\widetilde B_{0}$ respects the product structure of $N\times\mathbb I_{1/2}$. Since  $\widetilde B_{t_n}$ converges  to $\widetilde B_{0}$ as $t_n\to 0$,  it follows from the same proof of Lemma \ref{lemma:solutionapproximate-poly} that there exist positive numbers $\varepsilon_n$ such that  $\varepsilon_n\to 0$ as $n\to\infty$ and 
$$\|D^\dR_{\mathbb I}\varphi_n\|^2+\|D_N\varphi_n\|^2\leq\varepsilon_n.$$  In particular, we have 
$$\Big\|\frac{\partial}{\partial y}\varphi_n\Big\|^2=\|D^\dR_{\mathbb I}\varphi_n\|^2 \leq \varepsilon_n,$$
where $y$ is the coordinate of $\mathbb I$. 
Therefore, for any $y_0,y_1\in\mathbb I_{1/2}$, we have
\begin{align*}
	&\int_N |\varphi_n(x,y_0)-\varphi_n(x,y_1)|^2dx\\
	\leq &\int_N\Big(\int_{y_0}^{y_1}\Big|\frac{\partial}{\partial y}\varphi_n(x,y)\Big|dy \Big)^2 dx\\
	\leq &\int_N|y_0-y_1|\Big(\int_{y_0}^{y_1}\Big|\frac{\partial}{\partial y}\varphi_n(x,y)\Big|^2dy\Big) dx\\
	\leq &\Big\|\frac{\partial}{\partial y}\varphi_n\Big\|^2\leq\varepsilon_n
\end{align*}
by the Cauchy--Schwarz inequality. If we fix $y_0$ and integrate 
\[  \int_N |\varphi_n(x,y_0)-\varphi_n(x,y_1)|^2dx \]
over $y_1$, then we obtain
$$\int_{\mathbb I_{1/2}}\int_N|\varphi_n(x,y_0)-\varphi_n(x,y)|^2dxdy\leq \varepsilon_n,$$
for any $y_0\in \mathbb I_{1/2}$, where $\varphi_n(x,y_0)$ should be regarded as a section of $E$ over  $N\times \mathbb I_{1/2}$ by parallel transporting $\varphi_n(x,y_0)$ along $\mathbb I_{1/2}$. In particular, we have
$$\int_{\mathbb I_{1/2}}\int_N|\varphi_n(x,1/2)-\varphi_n(x,y)|^2dxdy\leq \varepsilon_n.$$
The boundary condition $\widetilde B_{t_n}$ at $N\times \{1/2\}$ corresponds to a projection onto  a subbundle of $E$ over $N\times {1/2}$. Since $N\times \mathbb I_{1/2}$ has the product metric and the bundle $E = (S_{TN\oplus f^\ast TM})\hotimes \Bigwedge^*\R $ over $N\times \mathbb I_{1/2}$,   the above subbundle of $E$ over $N\times \{1/2\}$ naturally induces (by parallel transporting along $\mathbb I_{1/2}$,  for example) a subbundle of $E$ over $N\times \mathbb I_{1/2}$, which will be denoted by $E_{t_n, 1/2}$ from now on.   Since $\varphi_n$ satisfies the boundary condition $\widetilde B_{t_n}$ at $N\times \{1/2\}$, $\varphi_n(x,1/2)$ may be viewed a section of the subbundle $E_{t_n, 1/2} \subset E$ over $N\times \mathbb I_{1/2}$.  By construction, the boundary condition $\widetilde B_{t_n}$ at  $N\times \{1/2\}$ stays the same for all $t$ close to $0$. In particular, we have $E_{t_n,1/2}=E_{0,1/2}$.

Similarly, the boundary condition $\widetilde B_{t_n}$ at $N\times \{1\}$ corresponds to a projection onto  a subbundle of $E$ over $N\times {1}$, which again  naturally induces a subbundle of $E$ over $N\times \mathbb I_{1/2}$, which will be denoted by $E_{t_n, 1}$ from now on. The same argument above shows that 
$$\int_{\mathbb I_{1/2}}\int_N|\varphi_n(x,1)-\varphi_n(x,y)|^2dxdy\leq \varepsilon_n,$$
where $\varphi_n(x,1)$ is viewed as a section of the subbundle $E_{t_n,1}\subset E$ over $N\times \mathbb I_{1/2}$. Since the boundary condition $\widetilde B_{t_n}$ converges to $\widetilde B_0$ as $t_n\to 0$, there is a sequence of bundle isometries  $\{U_n\}$ on $E$ such that $U_n$  maps $\widetilde B_{t_n,1}$ to $\widetilde B_{0,1}$,  and $U_n$ converges to the identity operator on $L^2(N\times\mathbb I_{1/2},E)$ with respect to the operator norm, as $n\to \infty$. Since $\|\varphi_n\|=1$ for all $n$, we have
$$\lim_{n\to\infty}\int_{\mathbb I_{1/2}}\int_N|U_n\varphi_n(x,1)-\varphi_n(x,y)|^2dxdy=0.$$

Since $D^E_{N\times \mathbb I_{1/2}}(\varphi_n)=0$ and $\|\varphi_n\|=1$, it follows from Proposition \ref{prop:norm-equivalent} that the $H^1$-norms of 
$\varphi_n$'s  are  uniformly bounded. By the Rellich Lemma, there is convergent subsequence of $\{\varphi_n\}$ in $L^2(N\times\mathbb I_{1/2},E)$. Without loss of generality, we may assume that $\varphi_n\to \varphi$  in $ L^2(N\times\mathbb I_{1/2},E)$. In particular,  we have $\|\varphi\|=1$. However, since $\varepsilon_n\to 0$, it follows that both $\{U_n\varphi_n(x,1)\}$ and $\{\varphi_n(x,1/2)\}$ also converge to the same limit $\varphi$. Therefore we have
$$\varphi\in L^2(N\times\mathbb I_{1/2}, E_{0,1/2})\cap L^2(N\times\mathbb I_{1/2}, E_{0,1}).$$
This leads to a contradiction since the two subbundles $E_{0,1/2}$ and $E_{0,1}$ are orthogonal to each other. To summarize, we have proved that $(D^E_{N\times \mathbb I_{1/2}})_{\widetilde B_t}$ is invertible for $t$ sufficiently close to $0$. In particular, $\ind \big( (D^E_{N\times \mathbb I_{1/2}})_{\widetilde B_t}\big) $ is constantly zero near $t=0$.

Now we consider the  submanifold $N\times \mathbb I_{-1/2}$ of $N\times \mathbb I$, where $\mathbb I_{-1/2} = [-1, -1/2]$. We construct  a new boundary condition on the codimension one faces of  $N\times \mathbb I_{-1/2}$ that is completely similar to $\widetilde B_{t}$ above, which will still denoted by $\widetilde B_{t}$ for simplicity.  The exact same argument shows that  $D^E_{N\times \mathbb I_{-1/2}}$ (subject to the boundary condition $\widetilde B_{t}$) is invertible for $t$ sufficiently close to $0$. In particular, $\ind \big( (D^E_{N\times \mathbb I_{-1/2}})_{\widetilde B_t}\big) $ is constantly zero near $t=0$.

Now we cut $N\times \mathbb I$ into three parts 
\[ (N\times \mathbb I_{-1/2}) \bigcup (N\times [-1/2, 1/2]) \bigcup (N\times \mathbb I_{-1/2}) \] 
where each part is equipped with the corresponding local boundary condition, which will all be denoted by $\widetilde B_t$. By construction, for all $t$ sufficiently close to $0$,  the Dirac operator $D^E_{N\times [-1/2, 1/2]}$ with the boundary condition $\widetilde B_t$ is the product 
\[  D^E_{N\times [-1/2, 1/2]}=D_{N}\hotimes 1+1\hotimes D^{\dR}_{[-1/2, 1/2]}, \] 
where $D^{\dR}_{[-1/2, 1/2]}$ is the de Rham operator of $[-1/2, 1/2]$ subject to  the absolute boundary condition on differential forms. Consequently, the operator  $(D^E_{N\times [-1/2, 1/2]})_{\widetilde B_t}$ remains the same for all $t$ sufficiently close to $0$, which implies  $\ind\big((D^E_{N\times [-1/2, 1/2]})_{\widetilde B_t}\big)$ stays constant for $t$ sufficiently close to $0$. Now it follows from the gluing formula in   \ref{prop:gluing} (cf. Remark \ref{rk:gluingBs}) that $\ind(D^
E_{B_t})$ is constant near $t=0$.

The exact same argument for the case $t=0$ proves that  $\ind(D^E_{B_t})$ is also constant near  $t=1$ and $t=2$. To summarize, we have proved that  $\ind(D^E_{B_t})$ is constant for $t\in[0,3)$. 

Recall that, near $t=3$, $N_t$ and $M_t$ have product structures near all faces of all codimensions. Now we apply the same cutting-and-pasting argument from the proof of Theorem \ref{thm:index}.  Let us retain the same notation from the proof of Theorem \ref{thm:index}. We have the following decomposition of $N$
$$N_t= N_{t,0}\cup \bigcup_{k=2}^n\bigcup_{\ |\lambda|=k}N_{t,\lambda},$$
where $N_{t, \lambda}$ is the direct product of a convex polyhedral corner in $\R^{|\lambda|}$ and a subspace of $\overbar F_\lambda$ that is away from   the codimension $(k+1)$ faces. Here $|\lambda|$ is the codimension of $\overbar F_\lambda$ in $N$.  Similarly, we have a decomposition of $M$
$$M_t= M_{t,0}\cup\bigcup_{k=2}^n\bigcup_{\ |\lambda|=k} M_{t,\lambda},$$ 
such that $f\colon N_t \to M_t$ restricts  to a map $N_{t, \lambda} \to M_{t,\lambda}$ that is of product type. According to  the cutting construction of Theorem \ref{thm:index}, the manifold $N_{t,\lambda}$ is obtained inductively by cutting along a hypersurface $\overbar X_{t,\lambda}$ (which is fiberwise a  spherical polyhedron)  near the face $\overbar F_\lambda$. Similarly, the manifold $M_{t,\lambda}$ is obtained inductively by cutting along a hypersurface $X_{t,\lambda}$ (which is fiberwise a  spherical polyhedron)  near the face $F_\lambda$ (cf. Figure \ref{fig:decomposition}). 
We  choose  $X_{t,\lambda}$ to be at least ($\varepsilon_\lambda/2)$-away from $F_\lambda$ so that the vertical tangent space  $(T_VM)_{F_\mu}$ on each $\mathcal U_{\varepsilon_\mu}(F_\mu)$ is contained in the tangent space  $TX_{t,\lambda}$ at every point $x\in X_{t, \lambda} \cap F_\mu$, for  all $F_\mu$ with $|\mu|<|\lambda|$. For example, we may choose $X_{t, \lambda}$ fiberwise to be the spherical polyhedron at radius $r = 2\varepsilon_\lambda/3$. The same remark applies to $\overbar X_{t,\lambda}$  in $N_{t}$.  Since the cutting introduces new codimension one faces, the reason main for the above choices of  $\overbar X_{t,\lambda}$  and $X_{t, \lambda}$ is that the inner normal vectors at the new codimension one faces will be automatically orthogonal to the existing vectors that have  already been used in defining the local boundary conditions. 

More precisely, we construct a boundary condition $B_{t,\lambda}$ on each codimension one faces of $N_{t,\lambda}\times\mathbb I$ as follows.  On the face $\overbar X_{t,\lambda}\times\mathbb I$, the boundary condition $B_{t,\lambda}$ is given by 
$$\mathscr E(\overbar c(\overbar e_n\oplus 0)\otimes c(e_n\oplus 0))\varphi=\varphi,$$
where $\mathscr E$ is the $\Z_2$-grading on $E$, $\overbar e_n$ is the unit inner normal vector of $\overbar X_{t,\lambda}$ in $N_{t, \lambda}$, and $e_n$ is the unit inner normal vector of $X_{t,\lambda}$ in $M_{t, \lambda}$. On all other codimension one faces of $N_{t, \lambda}\times \mathbb I$ that are not $\overbar X_{t,\lambda}\times\mathbb I$,  the boundary condition $B_{t,\lambda}$ is defined to be the canonical restriction of $B_t$ (cf. the proof of Theorem \ref{thm:index}).

Recall that the vector field $\normal_t$ lies in the vertical tangent bundle $T_VM$ on each $\mathcal U_{\varepsilon_\lambda}( F_\lambda)$ for each face $F_{\lambda}$ of $M$ (cf. line \eqref{refinedsets}). In  \textbf{Step (3)}, we have extended $\normal_t$ from $\partial M$ to a vector field over  $\partial M$. We may assume without loss of generality that this extended vector field, still denoted by $\normal_t$,  vanishes outside $\bigcup_\lambda \mathcal U_{\varepsilon_\lambda}(F_\lambda)$ and lies in the vertical part $T_VM$ on each $\mathcal U_{\varepsilon_\lambda}(F_\lambda)$. In \textbf{Step (3)}, the vectors $\widetilde v_i(t)$ and $w(t)$ that define the boundary condition $B_t$ are linear combinations of the original inner normal vectors $v_i(t)$, the constant vector $\Y$ along the interval $\mathbb I$, and the vector field $\normal_t$ on $M$. With the above specific choice of  $X_{t,\lambda}$, the unit inner normal vector field of $X_{t,\lambda}$ is orthogonal to $v_i(t)$, $\normal_t$, and $\Y$. Consequently, the unit inner normal vector field of $X_{t,\lambda}$  is orthogonal to both $\widetilde v_i(t)$ and $w(t)$. Hence the new dihedral angles near $X_{t,\lambda}$ introduced by the cutting process  are always $\pi/2$.
Now it follows from Theorem \ref{thm:ess-sa} and Theorem \ref{thm:ess-saNew} that the Dirac operator $D^E$  on $N_{t,\lambda}\times\mathbb I$ subject to the  boundary condition $B_{t,\lambda}$ is essentially self-adjoint and Fredholm, and its domain is  $H^1(N_{t,\lambda}\times\mathbb I,E;B_{t,\lambda})$.

By the gluing formula in Proposition \ref{prop:gluing} and the same argument in the proof of Theorem \ref{thm:index}, in order to  show that $\ind(D^E_{B_t})$ is constant near $t=3$, it suffices to show that $\ind(D^E_{B_{t,\lambda}})=0$ for each $\lambda\ne 0$. This follows from exactly the same argument of the proof of Theorem \ref{thm:index}, together with the product formula of Fredholm index and the obvious analogue of Corollary \ref{coro:nosolution} of manifolds with polyhedral boundary. This finishes the proof of the claim.

Since at $t=3$, $f_3\times \id\colon N_3\times\mathbb I\to M_3\times\mathbb I$ is fiberwise isometry, the same cutting-and-pasting argument in the proof of Theorem \ref{thm:index} applies, which implies that  
\[ \ind(D^E_{B_3})=\deg_{\widehat A}(f)\cdot \chi(M\times \mathbb I)= \deg_{\widehat A}(f)\cdot \chi(M). \] 
This shows that $\ind(D^E_{B_t})=\deg_{\widehat A}(f)\cdot \chi(M)$ for all $t\in[0,3]$. Now the standard product formula for Fredholm index implies that 
\[  \ind(D^E_{B_0}) = \ind(D_B) \cdot \ind(D_{\mathbb I, a}),  \]
where $D_B$ is the Dirac operator acting on $S_{TN\oplus f^\ast TM}$  over $N$  (subject to the local  condition $B$) and $D_{\mathbb I, a}$ is the de Rham operator of $\mathbb I$ subject to the usual absolute boundary condition on differential forms. Since  $\ind(D_{\mathbb I, a}) = 1$, it follows that 
	\[ \ind({D}_B)=\ind(D^E_{B_0}) = \deg_{\widehat A}(f)\cdot \chi(M).\]
This finishes the proof of the theorem for the case where both $N$ and $M$ are odd dimensional. 

If both $N$ and $M$ are even dimensional, we consider the product 
$$f\times \id\colon N\times\mathbb I\to M\times\mathbb I$$
which essentially turns it into an odd dimensional case. To avoid confusion, let us write 
\[ \widehat N = N\times \mathbb I, \quad \widehat M= M\times \mathbb I \quad \textup{ and } \quad   \widehat f= f\times \id. \]
Now both $\widehat N$ and $\widehat M$ are odd dimensional. We shall consider 
$$\widehat f\times \id\colon \widehat N\times\mathbb I\to \widehat M\times\mathbb I$$
and apply the same argument above. Note that taking direct product with $\mathbb I$ introduces new codimension one faces $N\times \partial \mathbb I$ (resp. $M\times \partial \mathbb I$)  in $\widehat N$ (resp. $\widehat M$).  Although the dihedral angles of the original $N$ and $M$ satisfy the strict comparison inequality,  the new dihedral angles of both $\widehat N$ and $\widehat M$ are $\pi/2$, hence do not satisfy the strict comparison inequality.  This extra technical problem can easily be overcome by performing the exact same argument as in the odd dimensional case above but only to the first component $N$ in $\widehat N = N\times \mathbb I$ and $M$ in $\widehat M = M\times \mathbb I$. The rest of the proof is completely similar. This completes the proof of the theorem. 
\end{proof}

\section{Proofs of Theorem \ref{thm:special}, Theorem \ref{thm:extremal-rigid} and Theorem \ref{thm:polyhedra} }\label{sec:proofmain}
In this section, we prove the main theorems of the paper: 
the dihedral extremality conjecture (Theorem \ref{thm:special}) and the comparison theorem of scalar curvatures, mean curvatures and dihedral angles   for compact manifolds with polyhedral boundary (Theorem \ref{thm:extremal-rigid}). As mentioned in the introduction, Theorem \ref{thm:polyhedra} on dihedral rigidity conjecture is an immediate consequence of either Theorem \ref{thm:special} or Theorem \ref{thm:extremal-rigid}. 

Roughly speaking, we shall apply Theorem \ref{thm:index-poly} to obtain a nontrivial parallel solution of the associated twisted Dirac operator arising from the geometric setup of Theorem \ref{thm:special} and Theorem \ref{thm:extremal-rigid}. However, the dihedral angles are assumed to satisfy the strict comparison condition \eqref{eq:strictdihedral-poly} in Theorem \ref{thm:index-poly}, while  the dihedral angles in Theorem \ref{thm:special} and Theorem \ref{thm:extremal-rigid} are only assumed to satisfy the non-strict comparison condition. Therefore, strictly speaking, Theorem \ref{thm:index-poly} does not directly apply the geometric setup of Theorem \ref{thm:special} and Theorem \ref{thm:extremal-rigid}.  To remedy this, we shall approximate the boundary condition by a sequence of more general local boundary conditions such as those in the proof of Theorem \ref{thm:index-poly} (by again introducing auxiliary space dimensions) such that these new  boundary conditions satisfy  the strict comparison condition on dihedral angles. In particular, Theorem \ref{thm:index-poly} (or rather its proof) applies to each of the approximation and computes the correct Fredholm index. Then we apply Lemma \ref{lemma:solutionapproximate-poly} below to conclude that the twisted Dirac operator subject to the original local boundary condition (arising from the original geometry) admits a nontrivial parallel solution.

\begin{proof}[Proof of Theorem $\ref{thm:extremal-rigid}$]
	The assumption that $\deg_{\widehat A}(f) \neq 0$ implies that $\dim N$ and $\dim M$ are of the same parity. 
	
	The odd dimension can be reduced to the even dimensional case as follows. Suppose both $N$ and $M$ are odd dimensional. For part $(i), (ii), (iii)$ and (III), we simply consider the map $f\times \id\colon N\times [0,1]\to M\times [0,1]$, where $\id\colon [0, 1]\to [0, 1]$ is the identity map on $[0, 1]$, and both $N\times [0,1]$ and $M\times [0,1]$ are equipped with the product metric. For part (I) and (II), let $Q$ be a contractible polyhedron  in the $3$-dimensional round unit sphere $\mathbb S^3$ such that every codimension one face of $Q$ is  totally geodesic. For example, let $Q$ be the intersection of $\mathbb S^3$ with the first hyperoctant $\{a\in \mathbb R^4 \mid a_i \geq 0, i=1, \cdots, 4 \}$ of $\mathbb R^4$. Consider the map $f\times \id\colon N\times Q\to M\times Q$, where $\id\colon Q \to Q$ is the identity map on $Q$, and both $N\times Q$ and $M\times Q$ are equipped with the product metric. Hence without loss of generality, we assume that  $N$ and $M$ are even dimensional. 
	
	Let $D$ be the Dirac operator on $S_N\otimes f^*S_M$ over $N$ and $B$ the local  boundary condition from  Definition \ref{def:boundarycondition}. 	
	 
	 \begin{claim}\label{claim:solution}
	 	Under the assumptions of Theorem \ref{thm:extremal-rigid}, $D_B$ admits a nontrivial solution in $H^1(N,S_N\otimes f^*S_M;B)$, that is,  there is a non-zero $\varphi\in H^1(N,S_N\otimes f^*S_M;B)$ such that $\nabla\varphi=0$.
	 \end{claim}
	 The general strategy is to apply Theorem \ref{thm:index-poly} to obtain a nontrivial solution of $D_B$. However, since Theorem \ref{thm:index-poly} assumes the strict comparison condition on dihedral angles (cf. Inequality \eqref{eq:strictdihedral-poly}), we shall approximate the unit inner vector fields of codimension one faces of $M$ by a set of new unit vector fields that  satisfy the strict comparison  condition on dihedral angles, which achieved by introducing auxiliary space dimensions. Then we apply Lemma \ref{lemma:solutionapproximate-poly} to eventually obtain a nontrivial solution. 
	 
	 Consider the direct product of our Geometric Setup \ref{setup} with the interval $\mathbb I=[-1,1]$. More precisely, we consider the map 
	 $$f\times \id \colon N\times\mathbb I\to M\times \mathbb I,$$
	 where $\id\colon \mathbb I \to \mathbb I$ is the identity map on $\mathbb I$. 
	 For brevity, let us denote by $E$ the spinor bundle 
	 $$E\coloneqq S_{T(N\times\mathbb I)\oplus (f\times \id)^*T(M\times\mathbb I)}$$
	 Note that we have the natural identification: 
	 $$E\cong (S_N\otimes f^*S_M)\hotimes \Bigwedge^*\mathbb I.$$
	 The associated Dirac operator $D_E$ on $E$ over $N\times \mathbb I$ is of the form: 
	 $$D_E=D\hotimes 1+1\hotimes D^{\dR}_{\mathbb I},$$
	 where $D$ is the Dirac operator on $S_N\otimes f^\ast S_M$ over $N$ and $D^{\dR}_{\mathbb I}$ is the de Rham operator of $\mathbb I$.
	 
	 For  each codimension one face $\overbar F_i$ of $N$, we denote by $u_i$ the  unit inner normal vector of $\overbar F_i$ and $v_i$  the corresponding unit inner normal vector of $F_i$ in $M$. Let $\mathscr E$ be the grading operators on $S_N\otimes f^*S_M$  respectively. The boundary condition $B$ at $\overbar F_i$ is given by 
	 \begin{equation}
	 	\mathscr E(\overbar c( u_i)\otimes c(v_i))\varphi=-\varphi
	 \end{equation}
	 for sections $\varphi$  of $S_N\otimes f^*S_M$,	where we regard $v_i$ as vector fields over $\overbar F_i$ with values in $f^*TM$.
	 This boundary condition  canonically induces a local boundary condition at the codimension one face $\overbar F_i\times \mathbb I$ given by 
	 \begin{equation}\label{eq:ujvj}
	 	\mathscr E(\overbar c(u_i\oplus 0)\otimes c(v_i\oplus 0))\psi=-\psi,
	 \end{equation}
 and at the codimension one faces $N\times \partial \mathbb I$
 	 \begin{equation}\label{eq:ujvj}
 	\mathscr E(\overbar c(0\oplus w)\otimes c(0\oplus w))\psi=-\psi,
 \end{equation}
	 for sections $\psi$ of $E$, where $\mathscr E$ is the $\mathbb Z_2$-grading operator and $w$ is the unit  inner normal vector of the end points $\{\pm1\}$ in  $\mathbb I$. For brevity, we shall still denote by $B$ this local boundary condition on codimension one faces of $N\times \mathbb I$. 
	 
	 Now we perform the same construction in the proof of Theorem \ref{thm:index-poly} to  the boundary condition $B$ of $E$, and obtain a continuous family of boundary conditions $\{B_t\}_{t\in[0,3]}$ with the same properties. It follows from the same proof of Theorem \ref{thm:index-poly} that the Fredholm index of $(D_E)_{B_t}$ remains constant for all  $t\in (0,3]$, and equals 
	 	\[ \ind((D_E)_{B_t})=\deg_{\widehat A}(f)\cdot \chi(M),\] which is non-zero by assumption. Therefore, for each  $t\in(0,3]$, there is a non-zero $\psi_t \in H^1(N\times \mathbb I, E; B_t)$ such that  $D_E(\psi_t) = 0$. Note that the operator $D_E$ subject to the boundary condition $B$ satisfies the non-negativity condition in Lemma \ref{lemma:solutionapproximate-poly} (cf. Lemma \ref{lemma:solutionapproximate}). Therefore, it follows from Lemma \ref{lemma:solutionapproximate-poly} that there exists a non-trivial element $\psi\in H^1(N\times\mathbb I,E;B)$ such that $\nabla^E\psi=0$. In particular, we have 
	 $$\frac{\partial}{\partial y}\psi=0,$$
	 where $y$ is the parameter of $\mathbb I$. By checking the local  boundary condition $B$ at $N\times \{\pm1\}$, we see that   
	 \[ \psi=\varphi\hotimes 1 \]
	  as a section of $E=(S_N\otimes f^*S_M)\hotimes\Bigwedge^*\R$. It follows that $\varphi$ is a non-trivial parallel section in $H^1(N,S_N\otimes f^*S_M;B)$ as desired. This completes the proof of the claim.
   
  Since $\nabla\varphi=0$, it follows that $\varphi$ is smooth in the interior of $N$. We identify sections of $S_N\otimes f^* S_M$ over $N$ with functions from $N$ to $\R^{\ell}$ by choosing a smooth embedding of $S_N\otimes f^* S_M$ into a trivial bundle over $N$ with rank $\ell$. Since $\varphi$ is parallel, $\varphi$ is a smooth section in the sense of Definition \ref{def:smooth,H1}. Furthermore, as the connection $\nabla$ preserves the metric, $|\varphi|$ is a non-zero constant everywhere. 
   
  Since $C_0^\infty(N,S_N\otimes f^*S_M;B)$ defined in Definition \ref{def:smooth,H1} is dense in $H^1(N,S_N\otimes f^*S_M;B)$, there is a sequence of sections $\varphi_k\in C_0^\infty(N,S_N\otimes f^*S_M;B)$ that converges to $\varphi$ in the $H^1$-norm. By Proposition \ref{prop:D^2} and Lemma \ref{lemma:formallysa}, we obtain that 
  \begin{equation*}
  	\begin{aligned}
  		\int_N|D\varphi_k|^2\geq& \int_N |\nabla\varphi_k|^2 + \int_N\frac{\overbar{\Sc}}{4}|\varphi_k|^2-\int_N\|\wedge^2 df\|\cdot\frac{f^*\Sc}{4}|\varphi_k|^2\\
  		&+\int_{\partial N}\frac{\overbar H}{2}|\varphi_k|^2-\int_{\partial N}\|df\|\cdot\frac{f^*H}{2}|\varphi_k|^2
  \end{aligned}\end{equation*}
Each term in the above inequality is continuous with respect to the $H^1$-norm. Therefore by letting $k\to\infty$, we have
\begin{equation*}
	\begin{aligned}
		0\geq& \int_N\frac{\overbar{\Sc}}{4}|\varphi|^2-\int_N\|\wedge^2 df\|\cdot\frac{f^*\Sc}{4}|\varphi|^2\\
		&+
		\int_{\partial N}\frac{\overbar H}{2}|\varphi|^2-\int_{\partial N}\|df\|\cdot\frac{f^*H}{2}|\varphi|^2
\end{aligned}\end{equation*}
This proves items	$(i)$ and $(ii)$ of the theorem.

Now we prove $(iii)$ of the theorem by the following observation from \cite{Wang:2022vf}. For each interior point $x$ of a  codimension two face of $N$, let $\overbar e_{1}$ and $\overbar e_{2}$ be the two unit inner normal vectors of the two adjacent codimension one faces intersecting at $x$. Let $e_{1}$ and $e_{2}$ be the unit inner normal vectors of the corresponding codimension one faces of $M$ meeting at $f(x)$. Note that  $\varphi$ is smooth at $x$ by the above discussion. It follows that $\varphi$ has to  satisfy two boundary conditions at $x$, that is, we have 
$$(\overbar\epsilon\otimes\epsilon)(\overbar c(\overbar e_{1})\otimes c(e_{1}))\varphi(x)=-\varphi(x) \textup{ and }(\overbar\epsilon\otimes\epsilon)(\overbar c(\overbar e_{2})\otimes c(e_{2}))\varphi(x)=-\varphi(x).$$
Equivalently, we have 
$$(\overbar\epsilon\otimes\epsilon)\overbar c(\overbar e_{1})\varphi(x)=- c(e_{1})\varphi(x)\text{ and }(\overbar\epsilon\otimes\epsilon)\overbar c(\overbar e_{2}) \varphi(x)=-c(e_{2})\varphi(x).$$
Therefore, for any complex numbers $a_1,a_2\in\mathbb C$, we have 
$$(\overbar\epsilon\otimes\epsilon)\overbar c(a_1\overbar e_{1}+a_2\overbar e_{2})\varphi(x)=- c(a_1e_{1}+a_2e_{2})\varphi(x).$$
In particular, both sides have the same  norm, which implies   
$$|a_1\overbar e_{1}+a_2\overbar e_{2}|=|a_1e_{1}+a_2e_{2}|$$
for all $a_1, a_2\in \mathbb C$, since $|\varphi(x)|$ is non-zero. 
By the  polarization identity,  it follows that  
\[ \langle\overbar e_{1},\overbar e_{2}\rangle=\langle e_{1}, e_{2}\rangle, \] which implies that the dihedral angles at $x$ is equal to the dihedral angle at $f(x)$. This proves $(iii)$ of the theorem.

	Now let us prove the rigidity part, that is, (I), (II) and (III) of the theorem. To prove (I), we first make the following observation (cf. \cite[Proposition 1]{Listing:2010te}). 
	\begin{claim*}
		Assume the conditions in (I) hold, that is, $\dim M = \dim N$,  $\mathrm{Ric}(g)>0$ and 
		\[ \Sc(\overbar{g})_x \geq \|df\|^2 \cdot \Sc(g)_{f(x)} \]
		for all $x\in N$,  then for each $x\in N$, either $(df)_x\colon T_x N \to T_{f(x)} M$ is  a homothety or $(df)_x = 0$.  
	\end{claim*} 
	We have already proved the equality $\overbar{\Sc}=\|\wedge^2 df\|\cdot f^*\Sc$ in the above. Since clearly $\|df\|^2\geq \|\wedge^2 df\|$, the condition  $\overbar\Sc\geq\|df\|^2\cdot f^*\Sc$ implies that  
	$$\overbar{\Sc}=\|df\|^2\cdot f^*\Sc=\|\wedge^2 df\|\cdot f^*\Sc,$$
	Choose a local $\overbar g$-orthonormal frame $\overbar e_1,\ldots,\overbar e_n$ of $TN$ and a local $g$-orthonormal frame $e_1,\ldots,e_n$ of $TM$ such that $f_*\overbar e_i=\mu_i e_i$ with $\mu_i\geq 0$.
	For any $i$, the condition $\mathrm{Ric}(g)>0$ implies that $\sum_j R_{ijji}^M>0$. That is,  for any $i$,  there is a $j\neq i$ such that $R_{ijji}^M>0$.
	The proof of Lemma \ref{lemma:curvature>=} shows that 
	$$\sum_{i,j}(\|df\|^2-\mu_i^2\mu_j^2)R^M_{ijji}=0.$$
	Hence for any $i$, there exists $j\neq i$ such that
	$$\|df\|^2 =  \mu_i\mu_j.$$
	It follows that $\mu_i=\|df\|$ for all $1\leq i\leq n$. This proves the claim.  
	
	Now let us define $h \coloneqq  \|df\|$ and $U$ the open subset of $N$ consisting of points where $h>0$. Since by assumption $\deg f$ is nonzero, we see that  $U$ is nonempty.  On $U$,   we have $f^* g=h^2\cdot \overbar g$. Therefore on $U$ we have
	$$f^*\Sc=\frac{\overbar Sc}{h^2}-\frac{2(n-1)}{h^3}\Delta h-\frac{(n-1)(n-4)}{h^4}|dh|^2.$$
	Since we have $\overbar \Sc=\|\wedge^2 df\|\cdot f^*\Sc$ and that $h=0$ on $N-U$, it follows from the above equation that
	\begin{equation}\label{eq:conformal}
	2h^k\Delta h=-(n-4)h^{k-1}|dh|^2
	\end{equation}
	on the whole $N$, for all $k\geq 1$. Furthermore,  on $U\cap (\partial N)$,  we have
	$$f^*H=\frac{\overbar H}{h}-(n-1)\frac{1}{h^2}\frac{\partial h}{\partial\overbar e_n}.$$
	where $\overbar e_n$ is the unit inner normal vector of the codimension one faces of $N$. Since we have $\overbar H=\|df\|\cdot f^*H$ on $\partial N$ and that $h = 0$ on $N-U$, it follows that 
	\[  \frac{\partial h}{\partial\overbar e_n} \equiv 0 \textup{ on } \partial N. \]
	Hence it follows from the Stokes' theorem that
	$$0=\int_N h^k\Delta h+\int_N \langle d(h^k), dh\rangle=
	\int_N h^k\Delta h+k\int_N h^{k-1}|dh|^2. $$
	Applying Equation \eqref{eq:conformal}, we obtain 
	$$\big(k-\frac{n-4}{2}\big)\int_N h^{k-1}|dh|^2=0$$
	for all $k\geq 1$. Therefore,  $dh\equiv 0$ on $N$. Recall that $U$ is nonempty. It follows that $h = \|df\|$ is a non-zero constant, say $a >0$,  on $N$ and  $f\colon (N, a\cdot \overbar{g}) \to (M, g)$ is a Riemannian covering map. This finishes of the proof of part (I).

	Now let us prove part (II). Let $U'$ be the open subset of $N$ consisting of points where $\|\wedge^2 df\| \neq0$. 
	If we have $\mathrm{Ric}(g)<\Sc\cdot g/2$, then for any fixed $k$, we have
	$$\sum_j R_{kjjk}^M+\sum_i R_{ikki}^M< \sum_{i,j} R_{ijji}^M.$$
	Therefore
	$$\sum_{i\ne k,j\ne k} R_{ijji}^M>0.$$
	Hence, for any $k$, there exists at least one pair\footnote{Here we need the fact that $n$ is at least $3$. Indeed, this is implied by the assumption that $\mathrm{Ric}(g)<\Sc\cdot g/2$ in (II), as in dimension two the Ricci curvature is equal to $K\cdot g=\Sc\cdot  g/2$, where $K$ is the Gaussian curvature.} $(i,j)$ with $i\ne k$ and $j\ne k$ such that $R_{ijji}^M>0$. Since we have the equality $\overbar{\Sc}=\|\wedge^2 df\|\cdot f^*\Sc$, it follows from the proof of Lemma \ref{lemma:curvature>=} that 
	$$\sum_{i,j}(\|\wedge^2 df\|-\mu_i^2\mu_j^2)R^M_{ijji}=0.$$
	Therefore, for each fixed $k$, there exist $i\ne k$ and $j\ne k$ such that  $\mu_i\mu_j=\|\wedge^2 df\|$. Since  $\|\wedge^2 df\|\neq 0$ at $x\in U'$, it follows that $\mu_i \neq 0$ and $\mu_j \neq 0$. 
	Note that $\mu_i\mu_k\leq\|\wedge^2 df\|$ and $\mu_j\mu_k\leq\|\wedge^2 df\|$. Thus $\mu_k\leq \|\wedge^2 df\|^{1/2}$ for all  $1\leq k\leq n$. This shows that $\|df\|^2=\|\wedge^2 df\|$ on $U'$.
	Now the same proof for part (I) shows that  $\mathrm{Ric}(g)>0$ implies $\mu_k=\|\wedge^2 df\|^{1/2}$ for all $1\leq k\leq n$. Now let $ h \coloneqq  \sqrt{\|\wedge^2 df\|}$ on $N$. Then by the above discussion, we have $h= 0$ on $N-U'$ and  $h = \|df\|$ on $U'$. The same proof for part (I) shows that $h =  \sqrt{\|\wedge^2 df\|}$ is a nonzero constant, say $c >0$, on $N$ and  $f\colon (N, \sqrt{c}\cdot \overbar{g}) \to (M, g)$ is a Riemannian covering map. This finishes the proof of part (II).

	Now let us prove part (III). If $M$ is and even dimensional flat manifold, then the connection $\nabla^M$ on $S_M$ is flat. Hence locally we can write 
	$$\varphi=\sum_\alpha \varphi_\alpha\otimes s_\alpha$$
	where $\{s_\alpha\}$ is a parallel basis of $f^*S_M$ and $\varphi_\alpha$ are local sections of $S_N$. Since there exists a nonzero section $\varphi$ of $S_N\otimes f^\ast S_M$ such that  $\nabla\varphi=0$. It follows that $$0=\sum_{i,j}R_{ijkl}\overbar c(\overbar e_i)\overbar c(\overbar e_j)\varphi_\alpha=-\frac{1}{2}\mathrm{Ric}_{kl}\varphi_\alpha$$ for any $\alpha,k,l$ in this case, cf. \cite[Corollary 2.8]{spinorialapproach}. Hence $N$ is Ricci flat. If $M$ is an odd dimensional flat manifold, then $M\times[0,1]$ is even dimensional and flat. The above discussion shows that $N\times[0,1]$ is Ricci-flat,  which implies that $N$ is also Ricci-flat. 
	
	Note that in dimension three, Ricci-flatness coincides with flatness. So for \mbox{$3$-dimensional} manifolds, the flatness of $M$ implies the flatness of $N$.  This finishes the proof of part (III), hence completes the proof of the theorem. 
\end{proof}

Now let us prove Theorem \ref{thm:special}. 
\begin{proof}[Proof of Theorem $\ref{thm:special}$]
	Except part (I) and (II), Theorem \ref{thm:special} is simply a special case of  Theorem \ref{thm:extremal-rigid}.

	Now for part (I), choose a local $\overbar g$-orthonormal frame $\overbar e_1,\ldots,\overbar e_n$ of $TN$ and a local $g$-orthonormal frame $e_1,\ldots,e_m$ of $TM$ such that 
	\[ f_*\overbar e_i= \begin{cases}
	\mu_i e_i  & \textup{ if } i\leq \min(m, n)\\
	0 & 	\textup{ otherwise}
	\end{cases} \] with $\mu_i\geq 0$.
	For any $1\leq i \leq m$,  we see that $\mathrm{Ric}(g)>0$ implies that $\sum_j R_{ijji}^M>0$. That is,  for any $1\leq i \leq m$,  there is a $j\neq i$ such that $R_{ijji}^M>0$.
	The proof of Lemma \ref{lemma:curvature>=} shows that 
	$$\sum_{i,j}(1-\mu_i^2\mu_j^2)R^M_{ijji}=0.$$
	Hence for any $1\leq i \leq m$, there exists $j\neq i$ such that
	$$1 =  \mu_i\mu_j.$$
	It follows that $\mu_i=1$ for all $1\leq i\leq m$. This shows that $m = \dim M\leq n = \dim N$ and $f$ is a Riemannian submersion. The proof for part (II) is completely similar. This finishes the proof.

\end{proof}

\begin{remark}
	Note that 	part (I) and (II) of Theorem \ref{thm:extremal-rigid} requires the extra assumption that $\dim N = \dim M$, while  part (I) and (II) of Theorem \ref{thm:special} hold even if the dimensions of $N$ and $M$ are different. In general,  one  needs extra geometric assumptions on the metrics $\overbar g$ on $N$ in order to have a version of part (I) or (II) of Theorem \ref{thm:extremal-rigid} for the case where $N$ and $M$ have different dimensions (cf. \cite[Section 4]{Listing:2010te}). 
\end{remark}

\appendix
\section{Clifford bundles}\label{sec:Cliffordbundles}
In this appendix, we review some standard identifications of  spinor bundles and Clifford bundles. 

Let $X$ be an $n$-dimensional smooth spin manifold with boundary and $S$ the complexified spinor bundle over $X$. We assume that $n$ is even. The Clifford bundle of $TX$ acts on $S$ by left multiplication. On the dual bundle $S^*$, there is naturally a Clifford right action given by
\begin{equation}\label{eq:rightaction}
\langle a \cdot c^*(v), b\rangle\coloneqq \langle a, c(v)\cdot b\rangle,~ \textup{ for all } a\in S^*,~b\in S.
\end{equation}
The left action of $\cl(TX)$ on $S$ induces a natural bundle isomorphism 

\begin{equation}\label{eq:spinortensorspinor}
\begin{aligned}
\cl(TX) &\longrightarrow S\otimes S^*\cong\mathrm{End}(S,S)\\
\alpha &\longmapsto \sum_{j} (\alpha s_j)\otimes s_j^\ast,
\end{aligned}
\end{equation}
where $\{s_j\}$ is a local orthonormal basis of $S$ and $\{s_j^\ast\}$ the corresponding dual basis of $S^\ast$ (cf. \cite[I.5.18]{spingeometry}). 

The spinor bundle $S$ admits a natural $\Z_2$-grading given by Clifford multiplication of the volume element
\[ (\sqrt{-1})^{n/2} e_1 \cdots e_{n}.\] We denote the grading operators on $S$ and $S^\ast$ by $\epsilon$ and $\epsilon^*$, respectively. Let $S_\pm$ and $S^*_\pm$ be the $\pm$-part of $S$ and $S^*$, respectively. The two grading operators induce a bi-grading on $S\otimes S^*$. In particular, the $\Z_2$-grading on $S\otimes S^*$ given by $\epsilon\otimes\epsilon^*$ corresponds to the even-odd grading on $\cl(TX)$,  under the isomorphism in line \eqref{eq:spinortensorspinor}.

Furthermore, under the natural isomorphism   $\cl(TX)\cong \Bigwedge^\ast X = \Bigwedge^\ast (T^\ast X)$, the even/odd grading on $\cl(TX)$ corresponds to the usual even/odd grading (with respect to the degrees of differential forms) on $\Bigwedge^\ast X$. In particular, the Dirac operator on $S\otimes S^*$ with respect to the $\mathbb Z_2$-grading given by $\epsilon\otimes\epsilon^*$ can be  naturally identified with the de Rham operator of  $X$ (cf. \cite[II.5.12]{spingeometry}).   

Let $e_n$ be the unit inner normal vector field of $\partial X$. As $c(e_n)$ anti-commutes with $\epsilon$,  multiplication by $c(e_n)$ switches the $+$ and $-$ parts of $S$. Similarly,  multiplication by $c^*(e_n)$ switches the $+$ and $-$ parts of $S^*$.
\begin{definition}\label{def:boundaryB}
	Let us restrict the bundle $S\otimes S^\ast$ over $\partial X$.	We define  $B$ to be the sub-bundle of $ S\otimes S^*$ over $\partial X$ consisting of the $(-1)$-eigenspace of the operator $(\epsilon\otimes \epsilon^*) (c(e_n)\otimes c^*(e_n))$, that is, 
	$$B\coloneqq \ker(1+(\epsilon\otimes \epsilon^*) (c(e_n)\otimes c^*(e_n))).$$
	Equivalently, if a sections $\varphi$ decomposed into
	$$\varphi=\varphi_{++}+\varphi_{+-}+\varphi_{-+}+\varphi_{--} $$
	with respect to the bi-grading on $S\otimes S^\ast$, then $\varphi$ lies in $B$ if and only if
	$$\varphi_{--}=-( c(e_n)\otimes c^*(e_n))\varphi_{++} \textup{ and } \varphi_{+-}=(c(e_n)\otimes c^*(e_n))\varphi_{-+}.$$
\end{definition}
\begin{proposition}\label{prop:tangentpart}
	Under the natural bundle isomorphisms 
	\[ S\otimes S^\ast \cong \cl(TX) \cong \Bigwedge^\ast X\] over $\partial X$, the sub-bundle $B$ corresponds to the sub-bundle of $\Bigwedge^\ast X$ generated by forms that are tangential to $\partial X$.
\end{proposition}
\begin{proof}
	Choose a local orthonormal basis $\{s_i\}$ of $S_+$. Then $\{s_i, -c(\overbar e_n) s_i\}$ is a  local orthonormal basis of $S$. By definition, we have $s_i^*c^*(e_n)=(-c( e_n)s_i)^*$. For simplicity, let us write $c_n$ and $c_n^\ast$ in place of $c(e_n)$ and $c^\ast(e_n)$.  The map in line \eqref{eq:spinortensorspinor} from $\cl TN$ to $S_N\otimes S_N^*$ becomes 
	$$\alpha  \mapsto \sum_{i}\alpha s_i\otimes s_i^*+ \sum_i  \alpha (-c_ns_i)\otimes s_i^*c^*_n.$$ 
	
	Let $\alpha$ be an element in $\cl(TN)$ consisting of vectors tangential to $\partial X$. Let us assume first that $\alpha $ has even degree. In this case, $\alpha$ commutes with the grading operator $\epsilon$ on $S$ and preserves the grading. Equivalently,  when viewed as an element  $S\otimes S^*$,  $\alpha$ has only $++$ and $--$ components. Since $\alpha$ does not contain the vector $e_n$, we see that  $c_n$ commutes with $\alpha$. Therefore, we have
	$$\alpha (-c_ns_i)\otimes s_i^*c^*_n
	=-c_n \alpha (s_i)\otimes s_i^*c^*_n=-
	(c_n\otimes c^*_n)(\alpha(s_i)\otimes s_i).$$
	In other words, we have 
	\[  \alpha_{--}=-( c_n\otimes c^*_n)\alpha_{++}.\]

	Similarly, now assume that $\alpha$ has odd degree. In this case, $\alpha$ anti-commutes with the grading operator $\epsilon$ on $S$, and as an element in  $S\otimes S^*$,  $\alpha$ have only $+-$ and $-+$ components. Also,  $c_n$ anti-commutes with $\alpha$. Therefore, we have
	$$\alpha(-c_ns_i)\otimes s_ic^*_n
	=c_n\alpha (s_i)\otimes s_ic^*_n=
	(c_n \otimes c^*_n)(\alpha(s_i)\otimes s_i).$$
	In other words, we have 
	\[  \alpha_{+-}=( c_n\otimes c^*_n)\alpha_{-+}\]
	in this case. 
	
	In conclusion, we have shown that $B$ is mapped to the sub-bundle\footnote{The bundle $\Bigwedge^*(\partial X)$ over $\partial X$ is a subbundle of $\Bigwedge^*X$ over $\partial X$ in a canonical way.} $\Bigwedge^*(\partial X)$ of $\Bigwedge^*X$ generated by  forms tangential to $\partial X$. Since $B$ and $\Bigwedge^*(\partial X)$ have the same rank,  the bundle isomorphism $S\otimes S^\ast \cong \Bigwedge^\ast X$ over $\partial X$ restricts to an isomorphism between $B$ and $\Bigwedge^*(\partial X)$. This finishes the proof.
\end{proof}

\section{Deformations of two dimensional convex spherical polygons}\label{sec:appendixdeform}

In this section, for the convenience of the reader,  we give some details on how to continuously deform a two dimensional convex spherical polygon through a continuous family of two dimensional convex spherical polygons of the same combinatorial type while increasing all corresponding dihedral angles. 

First, we have the following elementary lemma. 

\begin{lemma}\label{lemma:sphericalTriangle}
	Let $\mathbb S^2$ be the unit sphere in $\R^3$ and $\triangle ABC$ be a convex spherical triangle in $\mathbb S^2$. Then there exists a spherical triangle $\triangle AB'C'$ nearby such that $B'$ is on the geodesic line $AB$, $C'$ is on the geodesic line $AC$,  $\angle B'>\angle B$, and $\angle C'>\angle C$.
\end{lemma}
\begin{proof}
	Without loss of generality, we assume that $A$ is the point $(0,0,1)\in \mathbb R^3$, and the unit outer normal vectors of $AB$ and $AC$ (as edges of the spherical triangle $\triangle ABC$) are given by the vectors $\langle 0,-1,0\rangle $ and $\langle -\sin\theta,\cos\theta,0\rangle$, where $\theta = \angle A$. Let $\langle u,v,w\rangle $ be the unit outer normal vector of $BC$. Note that we have $w<0$. Moreover, we  have that
	$$\cos\angle B=v \textup{ and }\cos\angle C=u\sin\theta - v\cos\theta.$$
	Since $\triangle ABC$ are convex, we have $\sin\theta>0$. To make $\angle B$ and $\angle C$ larger, we first decrease  $v$ by a small amount and then decrease  $u$ by a sufficient amount accordingly so that $\cos\angle C=u\sin\theta - v\cos\theta$ decreases.  
\end{proof}

The following lemma shows that we may continuously deform a two dimensional convex spherical polygon through a  family of two dimensional convex spherical polygons of the same combinatorial type while increasing all corresponding dihedral angles. 

\begin{lemma}\label{lemma:deformsphericalpolygon}
	Let $\mathbb P$ be a two dimensional convex spherical polygon. Then for any $\varepsilon >0$,  there exists a continuous family of two dimensional convex spherical polygons $\{\mathbb P_t\}_{0\leq t \leq 1}$ in $\mathbb S^2$  with the following properties: 
	\begin{itemize}
		\item $\mathbb P_0 = \mathbb P$,
		\item $\mathbb P_t$ has the same combinatorial type as that of $\mathbb  P$, 
		\item for each $t>0$,  the dihedral angles of $\mathbb P_t$ are strictly larger than the corresponding dihedral angles of $\mathbb P_0$, and 
		\item furthermore at $t=1$, all dihedral angles of $\mathbb P_1$ are $\geq \pi -\varepsilon$. 
	\end{itemize} 
\end{lemma}
\begin{proof}
	Note that for each given edge $E_1$ of $\mathbb P$, there are precisely two other edges $E_2$ and $E_3$ adjacent to $E_1$.  We extend $E_2$ and $E_3$ along the geodesic circles in which they lie so that these extended geodesic lines together with  $E_1$ form a spherical triangle. By Lemma \ref{lemma:sphericalTriangle}, we may replace the edge $E_1$ by another geodesic line $E'_1$ so that $E'_1$ intersects with $E_2$ (resp. $E_3$) to form a new dihedral angle that is strictly larger than  the original dihedral angle between $E_1$ and $E_2$  (resp. $E_1$ and $E_3$). As a consequence, we obtained a new convex spherical polygon $\mathbb P'$ of the same combinatorial type as that of  $\mathbb P$ such that precisely two dihedral angles of $\mathbb P'$ are strictly greater than the corresponding dihedral angles of $\mathbb P$, and all the  other dihedral angles  of $\mathbb P'$ coincide with the corresponding dihedral angles of $\mathbb P$. In other words, if $\mathbb P$ is a non-degenerate convex spherical polyhedron, we are always able to deform it to another non-degenerate convex spherical polyhedron, while two of the dihedral angles increase and all the other dihedral angles remain the same.
	
	Suppose that $\mathbb P$ has $n$ vertices. Let $S$ be the supremum of the sum of all dihedral angles of non-degenerate convex spherical polygons (with $n$ vertices) that are obtained from $\mathbb P$ by the above deformation. Since a spherical polygon is completely determined by the normal vectors of its edges, the space of all such polygons are compact. Therefore, there is a spherical polygon $\widehat{\mathbb P}$ that attains the supremum $S$. However, $\widehat{\mathbb P}$ may contain degenerate faces, which is of either of the following two types:
	\begin{enumerate}
		\item the length of an edge is equal to zero, or
		\item the angle between two non-trivial adjacent edges is equal to $\pi$.
	\end{enumerate}
	By the Gauss--Bonnet formula for surfaces with piecewise smooth boundary, the area of the polygon $\widehat{\mathbb P}$ is equal to $S-(n-2)\pi$, which is larger than the area of $\mathbb P$. Therefore there is at least one non-degenerate edge of $\widehat{\mathbb P}$. In other words, the shape of $\widehat{\mathbb P}$ is still a convex spherical polygon, but some edges of $\widehat{\mathbb P}$ may degenerate (and become vertices) and some adjacent edges may intersect at angle $\pi$ (that is, they are contained in the same line). 
	
	We claim that the shape of $\widehat{\mathbb P}$ is a hemisphere. If not, one can always move an edge of $\widehat{\mathbb P}$ using Lemma \ref{lemma:sphericalTriangle} to increase two of its dihedral angles. Now by the maximality of $S$, there exists a non-degenerate spherical polygon $\mathbb P_a$ that is obtained  from $\mathbb P$ by the above edge-moving deformation and is very close to $\widehat{\mathbb P}$. If we apply the same move (done on $\widehat{\mathbb P}$) to $\mathbb P_a$, then we obtain a non-degenerate spherical polygon $\mathbb P_b$ whose dihedral angles sum up to exceed $S$.  This contradicts to the maximality of $S$.	Therefore, the only possible shape of $\widehat{\mathbb P}$ is a hemisphere. To summarize, we obtain a continuous family of two dimensional convex spherical polygons $\{\mathbb P_t\}_{0\leq t \leq 1}$ in $\mathbb S^2$  with the following properties: 
	\begin{itemize}
		\item $\mathbb  P_0 = \mathbb P$,
		\item $\mathbb  P_t$ has the same combinatorial type as that of $\mathbb  P$, 
		\item for $t\in[0,1)$,  the angles of $\mathbb  P_t$ is non-decreasing and strictly smaller than $\pi$, and
		\item furthermore at $t=1$, all dihedral angles of $\mathbb  P_1$ are all equal to $\pi$. 
	\end{itemize} 
Now consider the family of spherical polygons $\{\mathbb P_t\}_{t\in [0, 1-\delta]}$ for some sufficiently small $\delta$, and reparameterize the interval $[0, 1-\delta]$ to $[0, 1]$ to give the desired family of spherical polygons. This finishes the proof. 
\end{proof} 

\section{Estimates of scalar curvature, mean curvature and dihedral angles for smooth sections}\label{sec:appendixestimate} 

In this section, under an extra regularity condition on polytope maps, we prove a stronger version of Proposition \ref{prop:D^2} that not only detects the comparisons of scalar curvatures and mean curvatures, but also the comparison of dihedral angles. 

For simplicity, we shall focus on the case of manifolds with corners in Proposition \ref{prop:D^2WithAngles}. In order to streamline the proof, we will work with the following notion of corner maps. 
\begin{definition}\label{def:corner} A  map $f\colon N\to M$ between manifolds with corners is called a corner map if 
	\begin{enumerate} 
		\item $f$ is smooth map\footnote{Recall that $f$ is  a smooth map between manifolds with corners if and only if $f$ is the restriction of a smooth map $\varphi\colon N' \to M'$, where $N'$ and $M'$ are two open manifolds which respectively contain $N$ and $M$ as their submanifolds.} between manifolds with corners;
		\item $f$ maps faces to faces, that is, for each codimension one face $\overbar{F}_i$ of $N$, we have $f(\overbar{F}_i) \subseteq F_i$ for some codimension one face $F_i$ of $M$.
		\item  for any collection of codimension one faces, say,   $\{\overbar F_{1}, \cdots, \overbar F_{k} \}$ and any $x\in \cap_{j=1}^k \overbar F_{j}$, the tangent map 
		\[   f_\ast  \colon T_xN \to T_{f(x)} M \]
		restricted to the linear subspace $\mathfrak N_x \subset T_xN$ is injective and 
		\[  \cap_{j=1}^k TF_j  \bigcap f_\ast(\mathfrak N_x) =  0, \] where $\mathfrak N_x $ is the linear subspace of $T_xN$ spanned by the normal vectors of $\{\overbar F_{j}\}_{j=1}^k$ at $x$, and $F_j$ is the corresponding codimension one face in $M$ such that $f(\overbar F_j) \subseteq F_j$. 
	\end{enumerate}
\end{definition}

\begin{proposition}\label{prop:D^2WithAngles}
	Let $(M, g)$ and $(N, \overbar{g})$ be two oriented compact Riemannian manifolds with corners.  Suppose   $f\colon N \to M$ is a spin corner map (Definition \ref{def:spin} and Definition \ref{def:corner}). Let $D$ be the Dirac operator on $S_N\otimes f^* S_M$.  If both the curvature operator of $M$ and the second fundamental form of $\partial M$ are non-negative, then  we have 
	\begin{equation}\label{eq:>=0cornerWithAngles}
	\begin{aligned}
	\int_N|D\varphi|^2\geq& \int_N |\nabla\varphi|^2 + \int_N\frac{\overbar{\Sc}}{4}|\varphi|^2-\int_N\|\wedge^2 df\|\cdot\frac{f^*\Sc}{4}|\varphi|^2\\
	&+\int_{\partial N}\langle D^\partial\varphi,\varphi\rangle+
	\int_{\partial N}\frac{\overbar H}{2}|\varphi|^2-\int_{\partial N}\|df^\partial\|\cdot\frac{f^*H}{2}|\varphi|^2\\
	&+\frac{1}{2}\sum_{i,j}\int_{\overbar F_{ij}}(\pi-\overbar\theta_{ij})\cdot  |\varphi|^2-\frac{1}{2}\sum_{i,j}\int_{\overbar F_{ij}}|\pi -\theta_{ij}|\cdot |\varphi|^2
	\end{aligned}\end{equation}
	for all smooth sections  $\varphi$ of $S_N\otimes f^* S_M$, where $\overbar\theta_{ij}$ \textup{(}resp. $\theta_{ij}$\textup{)} are the dihedral angles of $N$ \textup{(}resp. $M$\textup{)}, cf. Definition $\ref{def:exterior}$.  
\end{proposition}
\begin{proof} 
 Our general strategy is to approximate $N$ by a family of manifolds with smooth boundary, then apply the computation in Section \ref{sec:smoothboundary} on each of these approximations, and finally obtain the inequality in line \eqref{eq:>=0corner} by taking the limit. 
	
	The choice of these approximations, which are manifolds with smooth boundary, has to be made in a careful way.  As a matter of fact, sometimes an approximation of $N$ may not be contained inside $N$. For this reason, we introduce the following ambient manifolds. 	Let $N^\dagger$ be  a manifold  that contains $N$ in its interior. We extend the Riemannian metric $\overbar g$ of $N$ to a smooth Riemannian metric  on $N^\dagger$.  Similarly let $M^\dagger$ be  a manifold  that contains $M$ in its interior. We also extend the Riemannian metric of $M$ to a smooth Riemannian metric on $M^\dagger$. For example, we can choose $N^\dagger$ to be the space obtained by attaching a small cylinder $\partial N\times [0, \varepsilon]$ to   $N$ along $\partial N$, and similarly  choose $M^\dagger = M\cup_{\partial M} (\partial M\times [0, \varepsilon])$. In this case, $f\colon N \to M$ also extends to a from $N^\dagger$ to $M^\dagger$. 
	
	We will construct a family of manifolds $\{N_r\}_{r>0}$ so that the following holds:
	\begin{itemize}
		\item $N_r$ is a manifold with smooth boundary,
		\item $N_r$ coincides with $N$ away from the $r$-neighborhood of the codimension two faces of $N$,
		\item Within the $r$-neighborhood of each codimension two face of $N$, the second fundamental form of $\partial N_r$  is dominated by the  curvature of a family of curves that are asymptotically orthogonal to the tangent space of the given codimension two face, 
		\item For each curve above, the signed curvature of the image curve in $M^\dagger$ does not change its sign.
	\end{itemize}
	
	For each $r>0$, we construct $N_r$ inductively as follows.
	\begin{enumerate}[label=$(\arabic*)$, ref=$(\arabic*)$]
		\item \label{smooth-1}Away from the $r$-neighborhood of all codimension two faces,  we simply choose $N_r$ to coincide with $N$. 
		\item \label{smooth-2} Now we shall smooth out the part of codimension two faces that is away from the codimension three faces. Intuitively speaking, we will smooth out the dihedral angle at each point by a nice smooth curve while the dihedral angle will be remembered as the integral of the signed curvature over this curve. Here the signed curvature at a point of a curve is positive if the normal vector of the curve at that point is inner (i.e., this normal vector points inward of $N$).
		
		If $\overbar{F}_{ij} =  \overbar F_i \cap \overbar F_j$ is a codimension two face of $N$, we denote its interior by $\interior{\overbar{F}}_{ij}$,  that is, all points in  $ \overbar{F}_{ij}$,  but not in any codimension three faces. For each $x \in \interior{\overbar{F}}_{ij}$, let $\mathfrak N_x\subset T_xN$ be the linear subspace of $T_xN$ spanned by the normal vectors to $\overbar F_i $ and $\overbar F_j$ at $x$. For any smooth curve  $\alpha$ in the $r$-disc of  $\mathfrak N_x$, we denote by  $f_\ast(\alpha)$ its  image in $f_\ast(\mathfrak N_x)\subset T_{f(x)}M$. Furthermore, we denote by $\alpha_N$  the image curve of $\alpha$  in $N$ under the exponential map $\exp\colon T_{x}N \to N$.   Similarly, $\alpha_M$ is the image  of $f_\ast(\alpha)$ in $M$ under the exponential map $\exp\colon T_{f(x)}M \to M$.

		Now we choose a curve $\gamma$ in the $r$-disc of  $\mathfrak N_x$ such that the corresponding curve $\gamma_M$ in $M$ intersects the faces $F_i$ and $F_j$ tangentially, and such that the signed curvature $k_{\gamma_M}$ of $\gamma_M$ does not change sign, i.e., either  $k_{\gamma_M} \geq 0$ on the whole curve or $k_{\gamma_M} \leq 0$ on the whole curve. Furthermore, we can assume without loss of generality that $|k_{\gamma_M}|\leq 1/r$. The curve $\gamma_N$ could lie inside of $N$ or outside of $N$ depending on the dihedral angle at $f(x)$. See Figure \ref{fig:inside} for when the curve $\gamma_N$ lies inside of $N$; and see Figure \ref{fig:outside} for when the curve $\gamma_N$ lies in $N^\dagger$ but outside of $N$. Let us also denote $\gamma_N$ by $\gamma^x_N$ if we want to specify its relevance to $x$. 
		\begin{figure}[h]
			\centering
			\begin{tikzpicture}[scale=0.8]
			\fill[gray!20] plot[domain=-2:0] ({\x}, {sqrt(25-(\x-3)^2)}) --
			plot[domain=0:2] ({\x}, {sqrt(25-(\x+3)^2)}) -- cycle;
			\draw [black,very thick,domain=-2:0] plot ({\x}, {sqrt(25-(\x-3)^2)});
			\draw [red,thick,domain=-1.95:-1] plot ({\x}, {sqrt(24.5025-(\x-3)^2)});
			\draw [black,very thick,domain=0:2] plot ({\x}, {sqrt(25-(\x+3)^2)});
			\draw [red,thick,domain=1:1.95] plot ({\x}, {sqrt(24.5025-(\x+3)^2)});
			\filldraw[red] ({-1}, {sqrt(24.5025-(-1-3)^2)}) circle[radius=0.05];
			\filldraw[red] ({1}, {sqrt(24.5025-(-1-3)^2)}) circle[radius=0.05];
			\draw [red,thick,domain=-1:1] plot ({\x}, {3*sqrt(24.5025-(-1-3)^2)/4+sqrt(1-(\x)^2+(24.5025-(-1-3)^2)/16)});
			\draw [-stealth] (3,2) -- node[anchor=south] {$f$} (4,2);
			\fill[gray!20] (5,0) -- (9,3) -- (13,0) -- cycle;
			\draw[very thick] (9,3) -- (5,0);
			\draw[very thick] (9,3) -- (13,0);
			\draw[red,thick,domain=0:1.95] plot ({1+(\x-1.95)*4/3+20/3},{\x});
			\draw[red,thick,domain=0:1.95] plot ({1-(\x-1.95)*4/3+28/3},{\x});
			\filldraw[red] ({20/3+1}, {1.95}) circle[radius=0.05];
			\filldraw[red] ({28/3+1}, {1.95}) circle[radius=0.05];
			\filldraw[red] (9,2) node {$\gamma_M$};
			\filldraw[red] (0,3) node {$\gamma_N$};
			\draw[red,thick,domain=20/3:28/3] plot({\x+1},{sqrt(400/81-(\x-8)^2)+1.95-16/9});
			\end{tikzpicture}
			\caption{ $\gamma_N$ lies inside of $N$}
			\label{fig:inside}
		\end{figure}
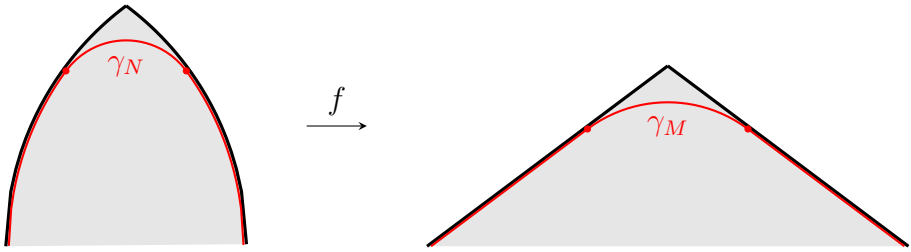
		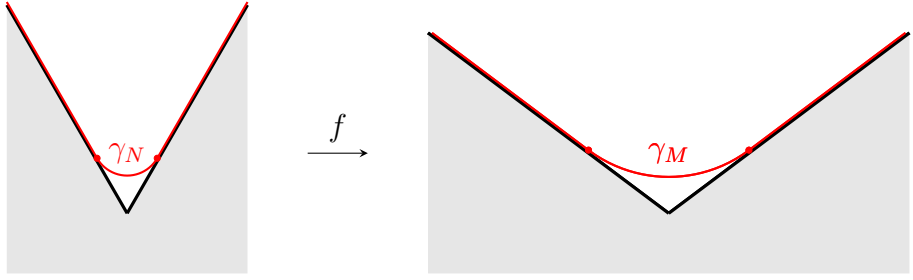
\begin{figure}[h]
			\centering
			\begin{tikzpicture}[scale=0.8]
			\fill[gray!20] (0,1) -- ({4*cos(60)},{1+4*sin(60)}) -- ({4*cos(60)},0) -- ({-4*cos(60)},0)-- ({-4*cos(60)},{1+4*sin(60)}) -- cycle;
			\draw[very thick] (0,1) -- ({4*cos(60)},{1+4*sin(60)});
			\draw[very thick] (0,1) -- ({-4*cos(60)},{1+4*sin(60)});
			\draw[red,thick]  ({1*cos(60)},{1+1*sin(60)+0.05}) --  ({4*cos(60)},{1+4*sin(60)+0.05});
			\draw[red,thick]  ({-1*cos(60)},{1+1*sin(60)+0.05}) --  ({-4*cos(60)},{1+4*sin(60)+0.05});
			\filldraw[red] ({-1*cos(60)},{1+1*sin(60)+0.05}) circle[radius=0.05];
			\filldraw[red] ({1*cos(60)},{1+1*sin(60)+0.05}) circle[radius=0.05];
			\draw[red,thick,domain=-30:-150] plot ({cos(\x)/sqrt(3)},{1+1*sin(60)+0.05+(1/2+sin(\x))/sqrt(3)});
			\filldraw[red] (0,2) node {$\gamma_N$};
			\draw[very thick] (9,1) -- (5,4);
			\draw[very thick] (9,1) -- (13,4);
			\draw[red,thick,domain=0:1.95] plot ({1+(\x-1.95)*4/3+20/3},{4-\x});
			\draw[red,thick,domain=0:1.95] plot ({1-(\x-1.95)*4/3+28/3},{4-\x});
			\filldraw[red] ({20/3+1}, {4-1.95}) circle[radius=0.05];
			\filldraw[red] ({28/3+1}, {4-1.95}) circle[radius=0.05];
			\draw[red,thick,domain=20/3:28/3] plot({\x+1},{4-(sqrt(400/81-(\x-8)^2)+1.95-16/9)});
			\filldraw[red] (9,2) node {$\gamma_M$};
			\draw [-stealth] (3,2) -- node[anchor=south] {$f$} (4,2);
			\fill[gray!20] (5,4) -- (9,1) -- (13,4) -- (13,0) -- (5,0) -- cycle;
			\draw[very thick] (9,1) -- (5,4);
			\draw[very thick] (9,1) -- (13,4);
			\draw[red,thick,domain=0:1.95] plot ({1+(\x-1.95)*4/3+20/3},{4-\x});
			\draw[red,thick,domain=0:1.95] plot ({1-(\x-1.95)*4/3+28/3},{4-\x});
			\filldraw[red] ({20/3+1}, {4-1.95}) circle[radius=0.05];
			\filldraw[red] ({28/3+1}, {4-1.95}) circle[radius=0.05];
			\draw[red,thick,domain=20/3:28/3] plot({\x+1},{4-(sqrt(400/81-(\x-8)^2)+1.95-16/9)});
			\filldraw[red] (9,2) node {$\gamma_M$};
			\end{tikzpicture}
			\caption{$\gamma_N$ lies outside of $N$}
			\label{fig:outside}
		\end{figure}
		
		In the above choice of $\gamma^x_N$, we can in fact choose the curves $\gamma^x_N$ (as $x$ varies in $ \interior{\overbar{F}}_{ij}$)  so that
		these curves together form a smooth codimension one  submanifold $\overbar Y_{ij}$ of $N^\dagger$ (cf. Figure \ref{fig:smooth-codim-two}). Roughly speaking, $\overbar Y_{ij}$ is  the boundary of some tubular neighborhood of $\interior{\overbar F}_{ij}$ in $N$.  Furthermore, we can arrange the curves so that the second fundamental form $\overbar A_r$ of $\overbar Y_{ij}$ is of the form: 
		\begin{equation}\label{eq:second-fund}
		\overbar A_r=\overbar k_r \cdot  d\overbar{e}_1^2+\sum_{\lambda >1 \text{ or }\mu >1}O(1)d \overbar e_\lambda d\overbar e_\mu, 
		\end{equation}
		where  $\{\overbar e_\lambda\}$ is a local orthonormal basis of $T\overbar{Y}_{ij}$ with $\overbar e_1$ being the unit tangent vector of the curve $\gamma_N^x$ and $\overbar k_r$ is the signed curvature of $\gamma_N^x$. Here $O(1)$ means a quantity that is uniformly bounded as $r\to 0$. 
		
		\begin{figure}[h!]
			\begin{tikzpicture}
			\draw[very thick] (0,0) -- (0,3) -- (-2,2) -- (-2,-1) -- cycle;
			\draw[very thick] (0,3) -- (-1.6,3.8) -- (-1.6,2.2);
			\draw[red,thick,domain=0.1:2.7] plot({-0.2+0.05*sin(\x*300)},{\x});
			\draw[red,thick,domain=-1:0] plot({\x-0.2+0.05*sin(0.5*300)},{0.5-sqrt(-0.1*\x)});
			\draw[red,thick,dashed,domain=-1:0] plot({\x-0.2+0.05*sin(0.5*300)},{0.5+sqrt(-0.1*\x)});
			\draw[red,thick,domain=-1:0] plot({\x-0.2+0.05*sin(1*300)},{1-sqrt(-0.1*\x)});
			\draw[red,thick,dashed,domain=-1:0] plot({\x-0.2+0.05*sin(1*300)},{1+sqrt(-0.1*\x)});
			\draw[red,thick,domain=-1:0] plot({\x-0.2+0.05*sin(1.5*300)},{1.5-sqrt(-0.1*\x)});
			\draw[red,thick,dashed,domain=-1:0] plot({\x-0.2+0.05*sin(1.5*300)},{1.5+sqrt(-0.1*\x)});
			\draw[red,thick,domain=-1:0] plot({\x-0.2+0.05*sin(2*300)},{2-sqrt(-0.1*\x)});
			\draw[red,thick,dashed,domain=-1:0] plot({\x-0.2+0.05*sin(2*300)},{2+sqrt(-0.1*\x)});
			\end{tikzpicture}
			\caption{Smoothing out codimension two faces}
			\label{fig:smooth-codim-two}
		\end{figure}
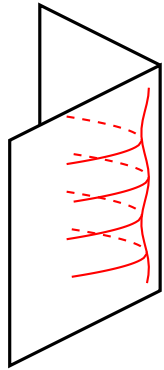

		\item \label{smooth-3} Now we shall smooth out the part of codimension three faces that is away from codimension four faces. In order to make our exposition more transparent, we will in fact first choose an approximation that roughly speaking turns codimension three singularities in $N$ into codimension two singularities.  
		
		Let $ \interior{\overbar{F}}_{ijk}$ be the interior of a codimension three face $\overbar F_{ijk} = \overbar F_i \cap \overbar F_j\cap \overbar F_k$ of $N$. Roughly speaking, we shall choose $\overbar Y_{ijk}$ to be the boundary of some tubular neighborhood of $ \interior{\overbar{F}}_{ijk}$ in $N$. More precisely, for each $x\in \interior{\overbar{F}}_{ijk}$, let  $\mathfrak N_x$ be the linear subspace of $T_xN$ spanned by $\{v_i, v_j, v_j\}$, where $v_i, v_j, v_k$ are the inner normal vectors to $\overbar F_i, \overbar F_j$ and $\overbar F_k$ respectively. Let $\mathfrak S_r$ be the sphere of radius $r$ in $\mathfrak N_x$. 
		Let $\exp(\mathfrak S_r)$ be the image  of $\mathfrak S_r$ under the exponential map $\exp\colon T_y N\to N$. Then the intersection of $ \exp(\mathfrak S_r)$ with $\overbar Y_{ij}\cup \overbar Y_{jk}\cup \overbar Y_{ik}$   is a smooth curve $\beta$. Let $\mathfrak R_x$ be the part of  $\exp(\mathfrak S_r)$ that is enclosed by the curve $\beta$. As $x$ varies in $\interior{\overbar F}_{ijk}$, these $\mathfrak R_x$'s trace out a smooth codimension one submanifold in $N^\dagger$, which we shall denote by $\overbar Y_{ijk}$.   Note that in general the exterior angle is non-trivial at a point where  $\overbar Y_{ijk}$ and $\overbar Y_{ij}\cup \overbar Y_{jk} \cup \overbar Y_{ik}$ intersect. In any case,  the intersection of $\overbar Y_{ijk}$ and $\overbar Y_{ij}\cup \overbar Y_{jk} \cup \overbar Y_{ik}$ contains only singularities of codimension $\leq 2$, which will eventually be smoothed out in a subsequent step.   In the following,  we shall use $\overbar Y_{\Theta}$ as a generic notation for referring to $\overbar Y_{ij}$ or $\overbar Y_{ijk}$ without specifying the actual sub-index. 
		
		Roughly speaking,  we have partially resolved the codimension three singularities in the original space by codimension two singularities. We will eventually smooth out these new codimension two singularities by repeating the construction in Step \ref{smooth-2}. But for the moment, let us continue this partial resolution process. Inductively, suppose we have partially resolved the codimension $\ell$ singularities by codimension $(\ell-1)$ singularities. Let $ \interior{\overbar{F}}_{\Lambda}$ be the interior of a codimension $(\ell+1)$ face $\overbar F_{\Lambda}$ of $N$.  Again, for each $x\in \interior{\overbar{F}}_{\Lambda}$,  let  $\mathfrak S_s$ be the sphere of sufficiently small radius $s$ in the linear subspace $\mathfrak N_x\subset T_xN$ spanned by the normal vectors of all codimension one faces that intersect $\interior{\overbar{F}}_{\Lambda}$ nonempty. Let $\exp(\mathfrak S_s)$ be the image  of $\mathfrak S_s$ under the exponential map $\exp\colon T_x N\to N$. Then the intersection of $ \exp(\mathfrak S_s)$ with the previously constructed $Y_\Theta$'s  is a closed $(\ell-1)$-dimensional manifold $\Sigma$. In general the inherited Riemannian metric on $\Sigma$ is not smooth, but with singularities. Let $\mathfrak R_x$ be the part of  $\exp(\mathfrak S_s)$ that is enclosed by  $\Sigma$. As $x$ varies in $\interior{\overbar F}_{\Lambda}$, these $\mathfrak R_x$'s trace out a codimension one submanifold in $N^\dagger$, which we shall denote by $\overbar Y_{\Lambda}$. Again, the dihedral angle is non-trivial at a point where  $\overbar Y_{\Lambda}$ and the previous $\overbar Y_{\Theta}$'s intersect. On the other hand, the intersection of  $\overbar Y_{\Lambda}$ and the previous $\overbar Y_{\Theta}$'s  contains only singularities of codimension $\leq \ell$. By repeating this partial resolution process inductively, we eventually obtain an  approximation of $N$ by another $n$-dimensional manifold  $N_{(1)}$ with corners such that  $N_{(1)}$ only has singularities of codimension $\leq (n-1)$. That is,  $N_{(1)}$ does not have codimension $n$ faces, or equivalently,  any collection of $n$ distinct codimension one faces of $N_{(1)}$ has empty intersection.  
		
		\item \label{smooth-4} We repeat the same process above on $N_{(1)}$ to obtain  an approximation of $N$ by an $n$-dimensional manifold $N_{(2)}$ with corners  such that $N_{(2)}$ only has singularities of codimension $\leq (n-2)$. Now by induction, we finally arrive at an approximation of $N$ by a manifold  $N_r$ with smooth boundary. See Figure \ref{fig:steps} for the effect of  these smoothing steps when performed near a  vertex of a $3$-dimensional manifold with corners. 
	\end{enumerate}
	
	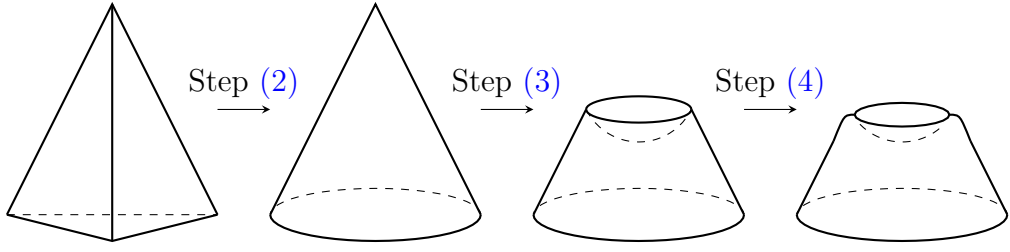
\begin{figure}
		\begin{tikzpicture}[scale=0.7]
		\draw[thick] (-7,0) -- (-5,4) -- (-3,0);
		\draw[thick] (-5,4) -- (-5,-0.5);
		\draw[thick] (-7,0) -- (-5,-0.5) -- (-3,0);
		\draw [-stealth] (-3,2) -- node[anchor=south] {Step \ref{smooth-2}} (-2,2);
		\draw[dashed] (-7,0) -- (-3,0);
		\draw[thick] (-2,0) -- (0,4) -- (2,0);
		\draw[thick,domain=-180:0] plot({2*cos(\x)},{0.5*sin(\x)});
		\draw[dashed,domain=0:180] plot({2*cos(\x)},{0.5*sin(\x)});
		\draw [-stealth] (2,2) -- node[anchor=south] {Step \ref{smooth-3}} (3,2);
		\draw[thick] (3,0) -- (4,2);
		\draw[thick] (6,2) -- (7,0);
		\draw[dashed,domain=-1:1] plot({5+\x},{2.5-sqrt(1.25-(\x)^2)});
		\draw[thick,domain=-180:0] plot({5+cos(\x)},{2+0.25*sin(\x)});
		\draw[thick,domain=0:180] plot({5+cos(\x)},{2+0.25*sin(\x)});
		\draw[thick,domain=-180:0] plot({5+2*cos(\x)},{0.5*sin(\x)});
		\draw[dashed,domain=0:180] plot({5+2*cos(\x)},{0.5*sin(\x)});
		\draw [-stealth] (7,2) -- node[anchor=south] {Step \ref{smooth-4}} (8,2);
		\draw[thick] (8,0) -- (8.7,1.4);
		\draw[thick] (11.3,1.4)-- (12,0);
		\draw[dashed,domain=-0.8:0.8] plot({10+\x},{2.5-sqrt(1.25-(\x)^2)});
		\draw[thick,domain=-180:180] plot({10+0.45*2*cos(\x)},{1.9+0.45*0.5*sin(\x)});
		\draw[thick] (8.7,1.4) .. controls (8.9,1.9) .. (9.1,1.9);
		\draw[thick] (11.3,1.4) .. controls (11.1,1.9) .. (10.9,1.9);
		\draw[thick,domain=-180:0] plot({10+2*cos(\x)},{0.5*sin(\x)});
		\draw[dashed,domain=0:180] plot({10+2*cos(\x)},{0.5*sin(\x)});
		\end{tikzpicture}	
		\caption{An illustration of the steps in the construction of the smooth approximation $N_r$}
		\label{fig:steps}
	\end{figure}

	Set $\overbar C_{r}= \partial N_r-\partial N$, i.e., $\overbar C_r$ is the part of $\partial N_r$ that is not in $\partial N$. We denote by $\overbar H_r$ the mean curvature of $\overbar C_r$. Given a smooth section $\varphi$  of $S_N\otimes f^\ast S_M$ over $N$, we extend it to a smooth section of the corresponding bundle over $N^\dagger$, which will still be denoted by $\varphi$. By the proof of Proposition \ref{prop:D^2smooth}, we have
	\begin{equation}\label{eq:>=0forN_r}
	\begin{aligned}
	\int_{N_r}|D\varphi|^2\geq& \int_{N_r} |\nabla\varphi|^2 + \int_{\partial N_r}\langle \slashed D^\partial\varphi,\varphi\rangle + o(1) \\
	& +  \int_{N_r\cap N }\frac{\overbar{\Sc}}{4}|\varphi|^2-\int_{N_r\cap N}\|\wedge^2 df\|\cdot\frac{f^*\Sc}{4}|\varphi|^2  \\
	&+\int_{\partial N\cap\partial N_r}\frac{\overbar H}{2}|\varphi|^2-\int_{\partial N\cap\partial N_r}\|df^\partial\|\cdot\frac{f^*H}{2}|\varphi|^2\\
	&+\int_{\overbar C_r}\sum_{\lambda} \langle\overbar c_\partial(\overbar e_\lambda)\nabla_{\overbar e_\lambda}\varphi,\varphi\rangle,
	\end{aligned}\end{equation}
	where $\slashed D^\partial$ is an appropriate substitute of  $D^\partial$ and  $o(1)$ means a quantity that goes to zero, as $r\to 0$. Recall that in our definition of the boundary Dirac operator $D^\partial$ from line \eqref{eq:boundarydirac}, we used both the normal vectors to $\partial N$ and $\partial M$. 	Since the image $f(\partial N_r)$ of $\partial N_r$ may not be a manifold in general, the notion of normal vectors to $f(\partial N_r)$ does not quite make sense. However, there exists a natural substitute $\slashed D^\partial$ of $D^\partial$ that makes all the estimates work. The precise definition of $\slashed D^\partial$ will be given in line  \eqref{eq:Dscript}. Also, the reason for having the extra term $o(1)$ is the following.  The curvature operator on $f(N_r)-M$ (i.e., the part of $f(N_r)$ that is outside of $M$) is generally not non-negative, hence the curvature estimate in Lemma \ref{lemma:curvature>=} does not apply. This  results in an extra error term coming from a certain integral over $N_r- N$. On the other hand, the volume of $N_r - N$ goes to zero, as $r\to 0$. We see this extra error term goes to zero, as $r\to 0$.

	We need to show that  each term on the right hand side of the inequality \eqref{eq:>=0forN_r} converge to the corresponding term of line \eqref{eq:>=0corner}.  This is relatively  straightforward for all terms except the last term. So we shall mainly focus on  estimating the last term 
	\[ \int_{\overbar C_r}\sum_{\lambda} \langle\overbar c_\partial(\overbar e_\lambda)\nabla_{\overbar e_\lambda}\varphi,\varphi\rangle.  \]
	
	To illustrate the key idea of our estimation, we first prove the special case where $\dim N=\dim M=2.$ In this case, by applying Equation \eqref{eq:secondf}, we have 
	\begin{equation}\label{eq:>=0forN_r2}
	\begin{aligned}
	\int_{\overbar C_r} \langle\overbar c_\partial(\overbar e_1)\nabla_{\overbar e_1}\varphi,\varphi\rangle = &
	\int_{\overbar C_r}\langle D^\partial\varphi,\varphi\rangle+
	\int_{\overbar C_{r}}\frac{\overbar H_r}{2}|\varphi|^2\\
	&-\frac{1}{2}\int_{\overbar C_{r}}A_r( e_1,f_*\overbar e_1) \cdot \langle\overbar c_\partial(\overbar e_1)\otimes c_\partial(e_1)\varphi,\varphi\rangle,
	\end{aligned}
	\end{equation}
	where $A_r$ is the second fundamental form of the curve $C_r = f(\overbar C_r)$ in $M^\dagger$,  and  $e_1$ is the unit tangent vector of $C_r$.  We denote by $\overbar v_{ij}$ the vertices of $N$. Then we have $\overbar C_r=\cup \overbar C_{ij,r}$, where  $\overbar C_{ij,r}$ is the curve near $\overbar v_{ij}$ coming from the construction of the approximation $N_r$. In this case, the mean curvature $\overbar H_r$ of $\overbar C_{ij,r}$ is simply the signed curvature of $\overbar C_{ij,r}$, which measures how fast the angle of the tangent vector of the curve changes with respect to the arc length of the curve. Therefore we have
	$$\int_{\overbar C_{ij,r}}\overbar H_r|\varphi|^2=(\pi-\overbar\theta_{ij})\cdot |\varphi(\overbar v_{ij})|^2+o(1)  \textup{ as } r \to 0, $$
	where $\overbar \theta_{ij}$ is the dihedral angle at the vertex $\overbar v_{ij}$, and $(\pi-\overbar\theta_{ij})$ is the jump angle of the tangent vector. 
	
	Set $C_{ij,r}=f(\overbar C_{ij,r})$, which is a smooth curve near the vertex $v_{ij} = f(\overbar v_{ij})$ in $M^\dagger$. We have $A_r(u,v)=k_r\cdot \langle u,v\rangle_M$, where $k_r$ is the signed curvature of the curve $C_{ij,r}$. 
	Therefore, when restricted on each curve $\overbar C_{ij, r}$, the last term on the right hand side of  \eqref{eq:>=0forN_r2} becomes 
	$$-\frac{1}{2}\int_{\overbar F_{ij,r}} k_r \cdot \langle e_1,f_* \overbar e_1\rangle_M \cdot \langle\overbar c_\partial(\overbar e_1)\otimes c_\partial(e_1)\varphi,\varphi\rangle d\overbar{s},$$
	where $d\overbar{s}$ is the infinitesimal arc length of $\overbar C_{ij,r}$. Since $e_1$ is a unit vector in $TM^\dagger$ and is parallel to $f_*\overbar e_1$, we have
	$$f_*\overbar e_1=\sqrt{\langle f_*\overbar e_1,f_*\overbar e_1\rangle_M}\cdot e_1.$$
	Therefore
	$$\langle e_1,f_* \overbar e_1\rangle_M \cdot d\overbar s=\langle e_1,e_1 \rangle_M \cdot ds=ds,$$
	where $ds$ is  the infinitesimal arc length of $C_{ij,r}$.

	Furthermore, since $\overbar c_\partial(\overbar e_1)\otimes 1+1\otimes c_\partial(e_1)$ is skew-symmetric, it follows that 
	\begin{align*}
	-\overbar c_\partial(\overbar e_1)\otimes c_\partial(e_1)=&
	\frac{1}{2}\big( -(\overbar c_\partial(\overbar e_1)\otimes 1+1\otimes c_\partial(e_1))^2 +\overbar c_\partial(\overbar e_1)^2\otimes 1+1\otimes c_\partial(e_1)^2\big)\\
	\geq& -1.
	\end{align*}
	Therefore, as $r\to 0$, we have 
	\begin{align*}
	-\int_{\overbar C_{ij,r}} k_r \langle e_1,f_* \overbar e_1\rangle_M \cdot \langle\overbar c_\partial(\overbar e_1)\otimes c_\partial(e_1)\varphi,\varphi\rangle d\overbar s \geq&
	-\int_{C_{ij,r}} |k_r| \cdot \langle\varphi,\varphi\rangle ds\\
	=& 
	-|\pi -f^*\theta_{ij}|\cdot |\varphi(v_{ij})|^2+o(1),
	\end{align*}
	where the last  equality follows from   the fact that $k_r$ does not change sign and that  $\varphi$ is smooth.

	Now we derive a similar estimate for the general case. Consider the decomposition 
	\[ \overbar C_r=(\bigcup_{ij} \overbar C_{ij,r})\cup \overbar V_r, \] where $\overbar V_r$ is the part of $\partial N_r$ that is sufficiently close to the codimension three faces of $N$,  and  $\overbar C_{ij,r}$ is the part that lies in the complement of $\overbar V_r$ and near the codimension two face $\overbar F_{ij}$ of $N$. A main observation in the estimation for the general case is that only the mean curvature of $\overbar C_{ij,r}$ will eventually contribute to the final limit, the integral of which converges to become the dihedral angle contribution.

	Recall that in our definition of the boundary Dirac operator $D^\partial$ from line \eqref{eq:boundarydirac}, we used both the normal vectors to $\partial N$ and $\partial M$. 	Since the image $f(\partial N_r)$ of $\partial N_r$ may not be a manifold in general, the notion of normal vectors to $f(\partial N_r)$ does not quite make sense. Hence our first step is to make sense of the ``boundary Dirac operator" in this general case.   Indeed, there is a natural substitute for the normal vector even at points  of $f(\partial N_r)$ where it fails to have a manifold structure.  First, for any point  $z \in  \partial N_r\cap \partial N$, we have $f(z)\in \interior{F}_j$ for some codimension one face $F_j$ of $M$. In this case, we simply choose $e_n \in T_{f(z)}M^\dagger $ to be the unit inner normal vector to $F_j$.

	Now consider a point  $z\in \gamma_N^x\subset  \overbar Y_{ij}$, where $x\in \interior{\overbar F}_{ij}$, and $\gamma_N^x$ and $\overbar Y_{ij}$ are the curve and submanifold constructed in Step \ref{smooth-2} above. At $f(x) \in \interior{F}_{ij} \subset  M$, let $T_{f(x)} F_{ij}$ be the tangent space of $ F_{ij}$ at $f(x)$. We identify $T_{f(x)} F_{ij}$ with a linear subspace $\mathfrak M_{f(z)}$ of $T_{f(z)}M^\dagger$ via the exponential map $\exp\colon T_{f(x)} M^\dagger \to M^\dagger$. Let $\overbar e_1$ be the unit tangent vector of   the curve $\gamma_N^x$ at the point $z \in \gamma_N^x$. It follows that $f_*(\overbar e_1)$ is non-zero and does not lie in  $\mathfrak M_{f(z)}$. Therefore, there exists a unique  unit inner vector in $T_{f(z)}M^\dagger $ that is orthogonal to the linear span of $\mathfrak M_{f(z)}$ and $f_*(\overbar e_1)$. Again, we denote this unique unit inner vector  by $e_n$. More generally, for a point $z\in \overbar Y_{\Lambda}$, we simply replace $F_{ij}$ and $\gamma_N^x$ by $F_{\Lambda}$ and $\mathcal R_x$ as in Step \ref{smooth-3}, and carry out the same construction as above. This produces the desired unit vector $e_n\in T_{f(z)} M^\dagger$. 
	
	Here we shall also introduce an extra notation to be used in later part of the proof.  We define $e_1 \in T_{f(z)}M^\dagger$ to be 
	\begin{equation}\label{eq:e_1}
	e_1 \coloneqq  \frac{f_\ast(\overbar e_1) - v}{\|f_\ast(\overbar e_1) - v\|}
	\end{equation} 
	where $v$ is the projection of  $f_\ast(\overbar e_1)$ onto $\mathfrak M_{f(z)}$. In particular, $e_1$ is orthogonal to $\mathfrak M_{f(z)}$ in $T_{f(z)}M^\dagger$.
	
	Now let $\overbar e_n \in TN_r$ be the unit inner normal vector to $\partial N_r$ and $e_n\in TM^\dagger$ the unit vector chosen above. Similar to  line \eqref{eq:nablapartial}, we define the following boundary Dirac operator over $\partial N_r$: 
	\begin{equation}\label{eq:Dscript}
	\begin{aligned}
	\slashed{D}^\partial\coloneqq \sum_\lambda \overbar c_\partial (\overbar e_\lambda)\nabla_{\overbar e_\lambda}
	&-\frac 1 2 \sum_{\lambda,\mu}\langle\prescript{N}{}\nabla_{\overbar e_\lambda} \overbar e_n,\overbar e_\mu \rangle_N\overbar c_\partial (\overbar e_\lambda)\overbar c_\partial(\overbar e_\mu)\otimes 1\\
	&-\frac 1 2\sum_{\lambda,\mu}\overbar c_\partial (\overbar e_\lambda)\otimes \langle\prescript{M}{}\nabla_{f_*(\overbar e_\lambda)}  e_n, e_\mu \rangle_M c_\partial(e_\mu)
	\end{aligned}
	\end{equation}
	where $\{\overbar e_\lambda\}$ is a local orthonormal basis of $T(\partial N_r)$ and $\{e_\lambda\}$ is local orthonormal basis of $f^\ast(\mathbb R e_1 \oplus \mathfrak M) \subset f^\ast(TM^\dagger)$. 
	Set $\overbar H_r$ to be the mean curvature of  $\partial N_r$ and $A_r$ to be as follows
	\begin{equation}\label{eq:substitute}
	A_{r}(e_\lambda,e_\mu)\coloneqq \langle \prescript{M}{}{\nabla} _{ e_\lambda}e_\mu, e_n\rangle_M=-\langle \prescript{M}{}{\nabla} _{ e_\lambda}e_n, e_\mu\rangle_M.
	\end{equation}
	for $e_\lambda, e_\mu$ in $f^\ast TM^\dagger$ that are orthogonal to $e_n$. 
	Then we have
	\begin{equation}
	\sum_{\lambda} \overbar c_\partial(\overbar e_\lambda)\nabla_{\overbar e_\lambda}=\slashed D^\partial+\frac{\overbar H_r}{2}-\frac{1}{2}\sum_{\lambda,\mu}A_r(f_*\overbar e_\lambda,e_\mu)\overbar c_\partial(\overbar e_\lambda)\otimes c_\partial(e_\mu).
	\end{equation}
	
	By construction, the second fundamental form of $\overbar C_{ij,r}$ is of the form \eqref{eq:second-fund}. It follows that essentially only the curvatures of the curves $\gamma_N^x$ in Step \ref{smooth-2} contribute to the mean curvature $\overbar H_r$. In particular,   as $r\to 0$, we have 
	$$\int_{\overbar C_{ij,r}}\overbar H_r|\varphi|^2=\int_{\overbar F_{ij}}(\pi -\overbar\theta_{ij} )\cdot |\varphi|^2+o(1).$$

	Again, consider a point $z\in \gamma_N^x\subset  \overbar Y_{ij}$, where $x\in \interior{\overbar F}_{ij}$, and $\gamma_N^x$ and $\overbar Y_{ij}$ are the curve and submanifold constructed in Step \ref{smooth-2} above. By assumption on $f$, the tangent map $f_\ast \colon T_x N\to T_{f(x)}M$ maps $T_x \overbar{F}_{ij}$ to $T_{f(x)} F_{ij}$. By construction,   $\mathfrak M_{f(z)}$ is a copy of $T_{f(x)} F_{ij}$ in $T_{f(z)}M^\dagger$ via  the exponential map $\exp\colon T_{f(x)} M^\dagger \to M^\dagger$. Since the second fundamental form $\overbar A_r$ of $\partial N_r$ is of the form  $\eqref{eq:second-fund}$, it follows that $A_r$ defined in line \eqref{eq:substitute} is also of the form
	\begin{equation}\label{eq:asympsecondform}
	A_r= k_r\cdot de_1^2+\sum_{\lambda>1\text{ or }\mu>1}O(1)de_\lambda de_\mu,
	\end{equation} 
	where $k_r$ is the signed curvature of local flow curves in $M^\dagger$ generated by the vector field $e_1$ from line \eqref{eq:e_1}.

	Recall that by construction $\overbar C_{ij,r}$ is part of $\overbar Y_{ij}$, where $\overbar Y_{ij}$ can be viewed as a  fiber bundle over $\interior {\overbar F}_{ij}$ with fibers being the curves $\gamma_N^x$, cf. Step \ref{smooth-2}. Let $d\overbar \omega$ be the infinitesimal volume element of $\overbar C_{ij,r}$. Then as $r\to 0$, we have asymptotically
	\[ d\overbar w =d\overbar s \, d\overbar \sigma, \] where $d\overbar \sigma$ is the infinitesimal volume element of $\interior {\overbar F}_{ij}$, and $d\overbar{s}$ is the  infinitesimal length element along the curves $\gamma_N^x$. 
	
	By the definition of $e_1$ in line \eqref{eq:e_1}, when $r\to 0$, the local flow curves  generated by the vector field $e_1$  asymptotically coincide with the curves obtained by projecting $f(\gamma_N^x)$ to  $\exp(\mathfrak L_{f(x)})$, where $\mathfrak L_{f(x)}$ is the linear subspace of $T_{f(x)}M$ spanned by the normal vectors of $F_{i}$ and $F_j$, and  $\exp(\mathfrak L_{f(x)})$ is its image in $M^\dagger$ under the exponential map. For simplicity, let us denote this local flow curve by $\zeta_M^{f(x)}$. See Figure \ref{fig:px}. In particular, if $ds$ is the infinitesimal length element of the curve $\zeta_M^{f(x)}$, then we have 
	$$ds=\langle f_*(\overbar e_1),e_1\rangle_M \cdot d\overbar s.$$
	Since the total length of each curve $\gamma_N^x$ goes to zero as $r\to 0$, it follows that,  as $r\to 0$, we have
	\begin{equation}\label{eq:>=angleM}
	\begin{split}
	&-\int_{\overbar C_{ij,r}}\sum_{\lambda,\mu}A_r( e_\lambda,f_*\overbar e_\mu)\langle\overbar c_\partial(\overbar e_\mu)\otimes c_\partial(e_\lambda)\varphi,\varphi\rangle d\overbar\omega\\
	=&-\int_{\overbar C_{ij,r}}\Big(k_r\cdot \langle e_1,f_*\overbar e_1\rangle_M\cdot \langle\overbar c_\partial(\overbar e_1)\otimes c_\partial(e_1)\varphi,\varphi\rangle +O(1)\Big)d\overbar s\, d\overbar \sigma\\
	\geq &  -\int_{\overbar F_{ij}}\int_{\gamma_M^{f(x)}}\langle\varphi,\varphi \rangle\cdot |k_r| \cdot ds d\overbar \sigma \quad + \quad o(1) \\
	=&-\int_{\overbar F_{ij}} |\pi -f^*\theta_{ij}| \cdot |\varphi|^2 d\overbar{\sigma} \quad + \quad o(1),\\
	\end{split}
	\end{equation}
	where  the last equality follows from the fact that $\varphi$ is smooth and that $k_r$ does not change its sign along the curve $\zeta_M^{f(x)}$.
	\begin{figure}[h]
		\centering
		\begin{tikzpicture}
		\draw[thick,black] (0,-1) -- (0,3);
		\filldraw (0,0) circle[radius=0.05] node[anchor=east] {$f(x)$};
		\filldraw[blue] (2.6, 3.5) node {$f(\gamma_N^x)$};
		\draw[dashed,black] (0,0) -- (4,0);
		\draw[dashed,black] (0,0) -- (-2.8,-2.8);
		\draw[very thick,black] (3.5,0).. controls (0.35,-0.15).. (-2.1,-2.1);
		\draw[dashed,black] (3.5,0) -- (3.5,3.5);
		\draw[dashed,black] (-2.1,-2.1) -- (-2.1,0.9);
		\draw[dashed,blue] (0,0) -- (4,4);
		\draw[dashed,blue] (0,0) -- (-2.8,1.2);
		\draw[very thick,blue] (3.5,3.5).. controls (0.35,0.2) .. (-2.1,0.9);
		\draw[-stealth,black] (0.5,-0.32) -- (2.5,0.22);
		\filldraw (2,0.3) node {$e_1$};
		\draw[-stealth,blue] (0.5,0.75) -- (2,1.7);
		\filldraw[blue] (1.8,1) node {$f_*(\overbar{e}_1)$};
		\draw[very thick,-stealth,red] (0.35,0.5) -- (0.35,-0.25);
		\filldraw[black] (2.6, -.6) node {$\zeta_M^{f(x)}$};
		\end{tikzpicture}
		\caption{The curve $f(\gamma_N^x)$ projects onto the local flow curve $\zeta_M^{f(x)}$ generated by the vector field $e_1$.}
		\label{fig:px}
	\end{figure}
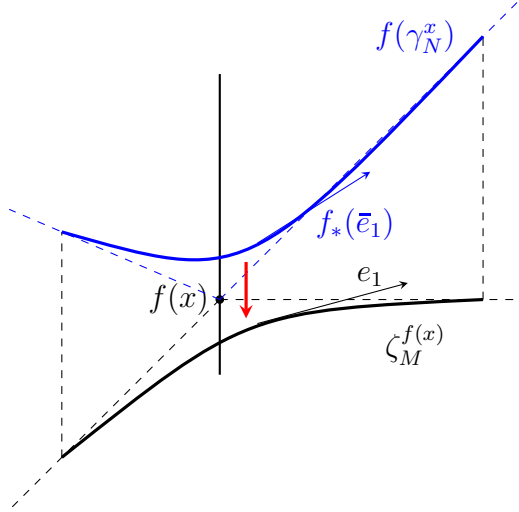

	It remains to estimate the following integral over $\overbar V_r$: 
	\begin{equation}\label{eq:small}
	\begin{aligned}
	\int_{\overbar V_r} \langle\overbar c_\partial(\overbar e_1)\nabla_{\overbar e_1}\varphi,\varphi\rangle = &
	\int_{\overbar V_r}\langle \slashed D^\partial\varphi,\varphi\rangle+
	\int_{\overbar V_{r}}\frac{\overbar H_r}{2}|\varphi|^2\\
	&-\frac{1}{2}\int_{\overbar 
		V_{r}}\sum_{\lambda,\mu}A_r( e_\lambda,f_*\overbar e_\mu)\langle\overbar c_\partial(\overbar e_\mu)\otimes c_\partial(e_\lambda)\varphi,\varphi\rangle.
	\end{aligned}
	\end{equation} 
	By repetitively applying the same estimation on $\overbar C_{ij, r}$ above to the last two terms of the right hand side of Equation \eqref{eq:small}, it follows that, as $r\to 0$, 
	\[ \int_{\overbar V_{r}}\frac{\overbar H_r}{2}|\varphi|^2 \to 0, \] 
	\[ \frac{1}{2}\int_{\overbar 
		V_{r}}\sum_{\lambda,\mu}A_r( e_\lambda,f_*\overbar e_\mu)\langle\overbar c_\partial(\overbar e_\mu)\otimes c_\partial(e_\lambda)\varphi,\varphi\rangle  \to 0.\]
	This is because the decay rate of  the volume of $\overbar V_r$ is greater than the growth rate of the mean curvature function $\overbar H_r$ of $\overbar V_r$, as $r \to 0$. 
	
	To summarize, for each smooth section $\varphi$ of $S_{N^\dagger}\otimes f^\ast S_{M^\dagger}$,  we have proved that
	\begin{equation}\label{eq:corner-approx}
	\begin{aligned}
	\int_{N_r} |D\varphi|^2\geq& \int_{N_r} |\nabla\varphi|^2 + \int_{\partial N_r}\langle \slashed D^\partial\varphi,\varphi\rangle -  o(1)  \\
	& + \int_{N_r}\frac{\overbar{\Sc}}{4}|\varphi|^2-\int_{N_r}\|\wedge^2 df\|\cdot\frac{f^*\Sc}{4}|\varphi|^2
	\\
	&+\int_{\partial N\cap\partial N_r}\frac{\overbar H}{2}|\varphi|^2-\int_{\partial N\cap\partial N_r}\|df^\partial\|\cdot\frac{f^*H}{2}|\varphi|^2\\
	&+\frac{1}{2}\sum_{i,j}\int_{\overbar F_{ij}}(\pi -\overbar\theta_{ij})\cdot  |\varphi|^2 - \frac{1}{2}\sum_{i,j}\int_{\overbar F_{ij}} |\pi -f^*\theta_{ij}| \cdot |\varphi|^2.
	\end{aligned}\end{equation}
	The proof will be complete once we show that each term in line \eqref{eq:corner-approx} converges to the corresponding term in line \eqref{eq:>=0corner}. The only term that needs an explanation is 
	\[  \int_{\partial N_r}\langle \slashed D^\partial\varphi,\varphi\rangle.  \]
	That is, we want to show that 
	\[   \int_{\partial N_r}\langle \slashed D^\partial\varphi,\varphi\rangle\to \int_{\partial N}\langle  D^\partial\varphi,\varphi\rangle\]
	as $r\to 0$. This will be an immediate consequence of the following claim and the fact that the coefficients of $\slashed D^\partial$  converges to the coefficients  of  $D^\partial$ (as differential operators). 
	\begin{claim*}
		For any smooth section $\varphi$ of $S_{N^\dagger} \otimes f^*S_{M^\dagger}$, the supreme norm $|\slashed D^\partial \varphi|$ of $\slashed D^\partial \varphi$ is uniformly bounded (i.e. independent of  $r$). 
	\end{claim*} 
	Recall the definition of $\slashed D^\partial$ from line \eqref{eq:Dscript}.  We first estimate the first term in line \eqref{eq:Dscript}. By the construction of the curves $\gamma_N^x$ in Step \ref{smooth-2}, we have  
	$$\prescript{N}{}\nabla_{\overbar e_1} \overbar e_1=\overbar k_r \overbar e_n+O(1), $$
	where $\overbar k_r$ is the signed curvature of $\gamma_N^x$. 
	Similarly, we have
	$$\prescript{M}{}\nabla_{e_1} e_1=k_r  e_n+O(1).$$ Since the tangent map $f_\ast \colon T_x N\to T_{f(x)}M$ maps $T_x \overbar{F}_{ij}$ to $T_{f(x)} F_{ij}$,  we see that asymptotically $\overbar e_1$ is the only basis vector in  $T_z(\partial N_r)$ such that $\langle f_\ast(\overbar e_1), e_1\rangle \neq 0$. In other words, if $\lambda>1$, then
	\[ \langle f_\ast(\overbar e_\lambda), e_1\rangle = o(1) \textup{ as $r\to 0$, }\] 
	where   $\{\overbar e_\lambda \}$ are basis vectors of $T_z(\partial N_r)$  chosen above. 
	Therefore,  we have
	\begin{align*}
	\sum_\lambda \overbar c_\partial (\overbar e_\lambda)\nabla_{\overbar e_\lambda}\varphi = & \overbar c_\partial (\overbar e_1)\cdot 
	\frac{1}{2}\langle\prescript{N}{}\nabla_{\overbar e_1} \overbar e_1,e_n\rangle_N\cdot\Big(\overbar c(\overbar e_1)\overbar c(\overbar e_n)\otimes 1\Big)\varphi
	\\
	&+\frac{1}{2}\langle\prescript{M}{}\nabla_{f_*(\overbar e_1)} e_1,e_n\rangle_M\cdot \Big(\overbar c_\partial(\overbar e_1)\otimes c(e_1)c(e_n)	\Big)\varphi+O(1)\\
	=&\frac{\overbar k_r}{2}\varphi-\frac{1}{2}k_r\langle f_*(\overbar e_1),e_1\rangle_M\cdot \Big(\overbar c_\partial(\overbar e_1)\otimes c_\partial(e_1)\Big)\varphi+O(1).
	\end{align*}
	Here the first equality follows from the explicit formula of the spinor connection $\nabla$ \cite[Theorem 2.7]{spinorialapproach}, and the $O(1)$ term  depends on the $C^1$-norm of $\varphi$.
	Furthermore, by the expressions in line \eqref{eq:second-fund} and line \eqref{eq:asympsecondform}, the last two terms of the right hand side of Equation \eqref{eq:Dscript} become 
	$$-\frac{\overbar k_r}{2}+\frac{1}{2}k_r\langle f_*(\overbar e_1),e_1\rangle_M\cdot \overbar c_\partial(\overbar e_1)\otimes c_\partial(e_1)+O(1).$$
	To summarize, we have 
	\[\langle \slashed D^\partial \varphi, 
	\varphi\rangle  = O(1).  \] 
	This proves the claim, hence completes the proof of the proposition. 
\end{proof}

\end{document}